\documentclass[11pt, twoside]{article}
\usepackage{amsmath,amsthm,amssymb}
\usepackage{times}
\usepackage{enumerate}

\def\titlerunning#1{\gdef\titrun{#1}}
\makeatletter
\def\author#1{\gdef\autrun{\def\and{\unskip, }#1}\gdef\@author{#1}}
\def\address#1{{\def\and{\\\hspace*{18pt}}\renewcommand{\thefootnote}{}%
\footnote {#1}}%
\markboth{\autrun}{\titrun}} \makeatother
\def\email#1{e-mail: #1}
\def\subjclass#1{{\renewcommand{\thefootnote}{}%
\footnote{\emph{Mathematics Subject Classification (2010):} #1}}}

\usepackage{amssymb,latexsym}
\usepackage{mathrsfs}
\usepackage{amssymb}
\usepackage{amsmath,latexsym}
\usepackage{amscd,latexsym} 
\usepackage[all]{xy}

\newcommand{\N}{{\mathbb N}}

\newcommand{\R}{{\mathbb R}}

\newcommand{\Z}{{\mathbb Z}}

\newtheorem{theorem}{Theorem}[section]

\newtheorem{corollary}[theorem]{Corollary}

\newtheorem{remark}[theorem]{Remark}

\newtheorem{hypothesis}[theorem]{Hypothesis}

\newtheorem{lemma}[theorem]{Lemma}

\newtheorem{proposition}[theorem]{Proposition}

\numberwithin{equation}{section}

\begin{document}
\baselineskip=17pt

\titlerunning{***}

\title{Morse theory methods for quasi-linear elliptic systems\\ of higher order
\thanks{Partially supported by the NNSF  11271044 of China.}}


\author{Guangcun Lu}

\date{February 22, 2017}

\maketitle

\address{F1. Lu: School of Mathematical Sciences, Beijing Normal University,
Laboratory of Mathematics  and Complex Systems,  Ministry of
  Education,    Beijing 100875, The People's Republic
 of China; \email{gclu@bnu.edu.cn}}


\begin{abstract}
We develop the local Morse theory for a class of non-twice continuously differentiable functionals on Hilbert spaces, including a new generalization of the Gromoll-Meyer's splitting theorem
 and a weaker Marino-Prodi perturbation type result. With them some critical point theorems
 and famous bifurcation theorems are generalized. Then we show that these are
 applicable to studies of quasi-linear elliptic equations and systems of higher order given by
 multi-dimensional variational problems  as in (\ref{e:1.3}).
  \end{abstract}

\tableofcontents

\section{Introduction }

Since Palais and Smale \cite{Pal, PaSm, Sma} generalized finite-dimensional Morse theory in \cite{Mor}
to nondegenerate $C^2$ functionals on infinite dimensional Hilbert manifolds
and used it to study multiplicity of solutions for semilinear elliptic boundary value problems,
via Gromoll and Meyer \cite{GrM}, Marino and Prodi \cite{MP} and
many other people's effort,  such a direction has very successful developments, see a few of nice books
\cite{BaSzWi, Ch0, Ch, Ch1, MaWi, MoMoPa, PerAO, Skr1, ZouSc} and references therein for details. The Morse theory
for functionals on an infinite dimensional Hilbert space $H$
have two main aspects:  Morse relations related critical groups to Betti numbers of
underlying spaces (global), computation of critical groups (local). The basic tool for the latter,
Gromoll-Meyer's generalized Morse lemma (or splitting theorem)
 in \cite{GrM}, was only generalized to $C^2$ functionals on Hilbert spaces
(\cite{Ch, MaWi}) until author's recent work \cite{Lu1, Lu2}.
Because of this,  most of applications of the theory to differential equations are restricted to semi-linear elliptic equations  and Hamiltonian systems \cite{Ch, MaWi, MoMoPa}. Skrypnik \cite{Skr1} established Morse inequalities for the functional (\ref{e:1.8})
with $p=2$ and $V=W_0^{m,2}(\Omega)$ provided that
linearizations of the corresponding Euler-Lagrange equation (\ref{e:1.9})
at any solution of it have no nontrivial solutions.
Our new splitting lemmas in \cite{Lu1, Lu2} can be effectively used to study periodic solutions
of Lagrangian systems on compact manifolds which are strongly convex and has quadratic growth on
the fibers, including the case of the system (\ref{e:1.4}) if $n=1$ and
\textsf{Hypothesis} $\mathfrak{F}_{2,N}$ was satisfied. Their ideas were also used to derive  the desired splitting
and shifting lemmas for the Finsler energy functional on the space of $H^1$-curves  in \cite{Lu3, Lu4}.
However, when applying these splitting lemmas to the functional in (\ref{e:1.8}) with $p=2$ we need
that the involved critical points have higher smoothness, which
 can only be guaranteed under more assumptions on Lagrangian $F$ by
  the regularity theory  of differential equations.
  It is this unsatisfactory restriction that motivates us to look for
  a more suitable splitting lemma which is applicable to
  the functional in (\ref{e:1.3}) under \textsf{Hypothesis} $\mathfrak{F}_{p,N}$ with $p=2$.


The following notation will be used throughout this paper.
 For normed linear spaces $X,Y$ we denote
 by $X^\ast$ the dual space of $X$, and by $\mathscr{L}(X,Y)$ the space of
linear bounded operators  from $X$ to $Y$.
We also abbreviate $\mathscr{L}(X):=\mathscr{L}(X,X)$.
The open ball in a normed linear space $X$ with radius $r$ and center in $y \in X$ is denoted by
$B_X(y,r):=\{x\in X \ |\  \|x-y\|_X<r\}$ and the corresponding closed ball
is written as $\bar{B}_X(y,r):=\{x\in X \ |\  \|x-y\|_X\le r\}$.
The (norm)-closure of a set $S \subset X$ will be denoted by $\overline{S}$ or $Cl(S)$.
Let $m$ and $n$ be two positive integers,  $\Omega\subset\R^n$  a bounded domain   with  boundary
$\partial\Omega$. Denote the general point of $\Omega$ by $x=(x_1,\cdots,x_n)\in\R^n$
and the element of Lebesgue $n$-measure on $\Omega$ by $dx$.
A {\it multi-index} is an $n$-tuple
$\alpha=(\alpha_1,\cdots,\alpha_n)\in (\mathbb{N}_0)^n$, where $\mathbb{N}_0=\mathbb{N}\cup\{0\}$.
 $|\alpha|:=\alpha_1+\cdots+\alpha_n$ is called the length of $\alpha$.
 Denote by $M(k)$ the number of such
$\alpha$ of length $|\alpha|\le k$, $M_0(k)=M(k)-M(k-1)$, $k=0,\cdots,m$,
where $M(-1)=\emptyset$. Then $M(0)=M_0(0)$ only consists of
${\bf 0}=(0,\cdots,0)\in (\mathbb{N}_0)^n$.\\

Let $p\in [2, \infty)$ be a real number and $N\ge 1$  an integer.

\noindent{\textsf{Hypothesis} $\mathfrak{F}_{p,N}$}.\quad
 For each multi-index $\gamma$ as above, let
 \begin{eqnarray*}
 &&p_\gamma\in (1,\infty)\;\hbox{if}\;
 |\gamma|=m-n/p,\quad\hbox{and}\; p_\gamma=\frac{np}{n-(m-|\gamma|)p}
\;\hbox{if}\; m-n/p<|\gamma|\le m,\\
 &&q_\gamma=1\;\hbox{if}\;|\gamma|<m-n/p,\quad\hbox{and}\;
 q_\gamma=\frac{p_\gamma}{p_\gamma-1} \;\hbox{if}\;m-n/p\le |\gamma|\le m;
 \end{eqnarray*}
 and for each two multi-indexes $\alpha, \beta$ as above, let
 $p_{\alpha\beta}=p_{\beta\alpha}$ be defined by the conditions
\begin{eqnarray*}
&p_{\alpha\beta}= 1-\frac{1}{p_\alpha}-\frac{1}{p_\beta}\quad
 \hbox{if}\;|\alpha|=|\beta|=m,\\
&p_{\alpha\beta}=  1-\frac{1}{p_\alpha},\quad \hbox{if}\;m-n/p\le |\alpha|\le
 m,\; |\beta|<m-n/p,\\
&p_{\alpha\beta}= 1 \quad \hbox{if}\; |\alpha|, |\beta|<m-n/p, \\
& 0<p_{\alpha\beta}<1-\frac{1}{p_\alpha}-\frac{1}{p_\beta}\quad\hbox{if}\;|\alpha|,\;|\beta|\ge
 m-n/p,\;|\alpha|+|\beta|<2m.
\end{eqnarray*}
For $M_0(k)=M(k)-M(k-1)$, $k=0,1,\cdots,m$ as above,
we write $\xi\in \prod^m_{k=0}\mathbb{R}^{N\times M_0(k)}$ as
$\xi=(\xi^0,\cdots,\xi^m)$, where
 $ \xi^0=(\xi^1_{\bf 0},\cdots,\xi^N_{\bf 0})^T\in \mathbb{R}^{N}$ and
 $$
 \xi^k=\left(\xi^i_\alpha\right)_{\scriptsize\begin{array}{ll}
 1\le i\le N\\
 |\alpha|=k\end{array}}\in \mathbb{R}^{N\times M_0(k)}\quad\hbox{for}\;k=1,\cdots,m.
 $$
 Denote by  $\xi^k_\circ=\{\xi^k_\alpha\,:\,|\alpha|<m-n/p\}$ for $k=1,\cdots,N$. Suppose
$$
\overline\Omega\times\prod^m_{k=0}\mathbb{R}^{N\times M_0(k)}\ni (x,
\xi)\mapsto F(x,\xi)\in\R
$$
is a  Caratheodory function  (i.e., being measurable in $x$ for all values of $\xi$, and  continuous in $\xi$ for almost all $x$) with the following properties:\\
{\bf (i)}  $F(x,\xi)$ is twice continuously differentiable in $\xi$ for almost all $x$, $F(\cdot,0)\in L^1(\Omega)$
and
 $$
 F^i_\alpha(x,\xi):=\frac{\partial F(x,\xi)}{\partial\xi^i_\alpha},\quad i=1,\cdots,N,\quad |\alpha|\le m
$$
satisfy: $F^i_\alpha(\cdot,0)\in L^1(\Omega)$ if $|\alpha|<m-n/p$,
and $F^i_\alpha(\cdot,0)\in L^{q_\alpha}(\Omega)$ if $m-n/p\le |\alpha|\le m$, $i=1,\cdots,N$.\\
{\bf (ii)} There exists a continuous, positive, nondecreasing functions $\mathfrak{g}_1$ such that
for $i,j=1,\cdots,N$, $|\alpha|, |\beta|\le m$ and the above numbers $p_{\alpha\beta}$ functions
$$
\overline\Omega\times\R^{M(m)}\to\R,\; (x, \xi)\mapsto
F^{ij}_{\alpha\beta}(x,\xi)=\frac{\partial^2F(x,\xi)}{\partial
\xi^i_\alpha\partial \xi^j_\beta}
$$
satisfy:
\begin{eqnarray}\label{e:1.1}
 |F^{ij}_{\alpha\beta}(x,\xi)|\le
\mathfrak{g}_1(\sum^N_{k=1}|\xi_\circ^k|)\left(1+
\sum^N_{k=1}\sum_{m-n/p\le |\gamma|\le
m}|\xi^k_\gamma|^{p_\gamma}\right)^{p_{\alpha\beta}}.
\end{eqnarray}
{\bf (iii)} There exists a continuous, positive, nondecreasing functions $\mathfrak{g}_2$ such that
\begin{eqnarray}\label{e:1.2}
\sum_{|\alpha|=|\beta|=m}F^{ij}_{\alpha\beta}(x,\xi)\eta^i_\alpha\eta^j_\beta\ge
\mathfrak{g}_2(\sum^N_{k=1}|\xi^k_\circ|)\Biggl(1+ \sum^N_{k=1}\sum_{|\gamma|=m}|\xi^k_\gamma|\Biggr)^{p-2}
\sum^N_{i=1}\sum_{|\alpha|= m}(\eta^i_\alpha)^2
\end{eqnarray}
for any $\eta=(\eta^{ij}_{\alpha\beta})\in\R^{N\times M_0(m)}$.

{\it Note}: If $m\le n/p$ the functions  $\mathfrak{g}_1$ and $\mathfrak{g}_2$ should be understand as positive constants.\\

For an element of $W^{m, p}(\Omega, \mathbb{R}^N)$,
$\vec{u}=(u^1,\cdots, u^N):\Omega\to\mathbb{R}^N$,
 we shall denote by $D^k\vec{u}$ the set $\{D^\alpha u^i\,:\, |\alpha|=k,\; i=1,\cdots,N\}$
 for each $k=1,\cdots,m$,  and form
 the expression $F(x, \vec{u}(x),\cdots, D^m\vec{u}(x))$, in which
 $\vec{u}(x)$ takes the place of $\xi^0$, and $D^\alpha u^i(x)$
 takes the place of $\xi^i_\alpha$ respectively.
 Let $V\subset W^{m, p}(\Omega, \mathbb{R}^N)$ be
 a closed subspace containing $W^{m,p}_0(\Omega, \mathbb{R}^N)$.
 Consider variational problem
\begin{equation}\label{e:1.3}
\mathfrak{F}(\vec{u})=\int_\Omega F(x, \vec{u},\cdots, D^m\vec{u})dx,\quad \vec{u}\in V.
\end{equation}
 We call critical points of $\mathfrak{F}$ {\it generalized solutions} of the boundary value problem
corresponding to the subspace $V$:
\begin{equation}\label{e:1.4}
\sum_{|\alpha|\le m}(-1)^{|\alpha|}D^\alpha F^i_\alpha(x, \vec{u},\cdots, D^m\vec{u})=0,\quad
i=1,\cdots,N.
\end{equation}

If $N=1$,  \textsf{Hypothesis} $\mathfrak{F}_{p,N}$
can be written as the following simple version, which was first given in \cite{Skr1}.

\noindent{\textsf{Hypothesis} $\mathfrak{f}_p$}.\quad
Let  $p$,  $p_\alpha, q_\alpha$, $p_{\alpha\beta}$ and $\Omega$
be as in \textsf{Hypothesis} $\mathfrak{F}_{p,N}$.
Write $\xi\in\R^{M(m)}$ as  $\xi=\{\xi_\alpha:\,|\alpha|\le m\}$
and  $\xi_\circ=\{\xi_\alpha:\,|\alpha|<m-n/p\}$. Suppose
that $f: \overline\Omega\times\R^{M(m)}\to\R$
 is a Caratheodory function with the following properties:\\
 {\bf (i)}  $f(x,\xi)$ is twice continuously differentiable in $\xi$ for almost all $x$, $f(\cdot,0)\in L^1(\Omega)$
 and each $f_\alpha(x,\xi):=\frac{\partial f(x,\xi)}{\partial\xi_\alpha}$  satisfy:
 $f_\alpha(\cdot,0)\in L^1(\Omega)$ if $|\alpha|<m-n/p$,
and $f_\alpha(\cdot,0)\in L^{q_\alpha}(\Omega)$ if $m-n/p\le |\alpha|\le m$.
 \\
{\bf (ii)} There exists a continuous, positive, nondecreasing functions $\mathfrak{g}_1$ such that
for the above numbers $p_{\alpha\beta}$ functions
$$
\overline\Omega\times\R^{M(m)}\to\R,\; (x, \xi)\mapsto
f_{\alpha\beta}(x,\xi)=\frac{\partial^2f(x,\xi)}{\partial
\xi_\alpha\partial \xi_\beta}
$$
satisfy:
\begin{eqnarray}\label{e:1.5}
 |f_{\alpha\beta}(x,\xi)|\le
\mathfrak{g}_1(|\xi_\circ|)\left(1+
\sum_{m-n/p\le |\gamma|\le
m}|\xi_\gamma|^{p_\gamma}\right)^{p_{\alpha\beta}}.
\end{eqnarray}
{\bf (iii)} There exists a continuous, positive, nondecreasing functions $\mathfrak{g}_2$ such that
\begin{eqnarray}\label{e:1.6}
\sum_{|\alpha|=|\beta|=m}f_{\alpha\beta}(x,\xi)\eta_\alpha\eta_\beta\ge
\mathfrak{g}_2(|\xi_\circ|)\Biggl(1+ \sum_{|\gamma|=m}|\xi_\gamma|\Biggr)^{p-2}
\sum_{|\alpha|= m}\eta^2_\alpha
\end{eqnarray}
for any $\eta\in\R^{M_0(m)}$.\\

Consider the case $m=1$ and $n\ge 2$. Then $M=n+1$ and $f$ becomes
$$
f:\overline{\Omega}\times\R^1\times\R^n\to\mathbb{R},\; (x,\xi_0, \xi_1,\cdots,\xi_n)\mapsto f(x,
\xi_0, \xi_1,\cdots,\xi_n).
$$
The corresponding \textsf{Hypothesis} $\mathfrak{f}_2$ is:  there exist constant numbers $c_1, c_2>0$ such that
\begin{eqnarray}\label{e:1.7}
 |f_{ij}(x,\xi)|\le c_1\Biggl(1+
\sum^n_{k=0}|\xi_k|^{2_k}\Biggr)^{2_{ij}},\quad
\sum_{i,j=1}f_{ij}(x,\xi)\eta_i\eta_j\ge
c_2\Biggl(\sum^n_{i=1}\eta^2_i\Biggr).
\end{eqnarray}
Here, (I) if $n=2$, $2_0=s\in (2,\infty)$, $2_i=2$, $i=1,\cdots,n$,
$2_{ij}=0$ for $i,j=1,\cdots,n$, $2_{0i}=2_{i0}\in (0,1/2-1/s)$ for $i=1,\cdots, n$, and $2_{00}\in (0,1-2/s)$; (II) if $n>2$, $2_0=2n/(n-2)$, $2_i=2$ for $i=1,\cdots,n$, $2_{ij}=0$ for $i,j=1,\cdots,n$, $2_{i0}=2_{0i}\in (0, 1/n)$ for $i=1,\cdots,n$,
$2_{00}\in (0, 2/n)$.

Under \textsf{Hypothesis} $\mathfrak{f}_p$,
 let $V\subset W^{m,p}(\Omega)$ be any closed subspace
containing $W^{m,p}_0(\Omega)$. The critical points of  the variational problem
\begin{equation}\label{e:1.8}
\mathcal{F}(u)=\int_\Omega f(x, u,\cdots, D^mu)dx
\end{equation}
in the Banach space $V$ are called {\it generalized solutions} of the boundary value problem
corresponding to the subspace $V$:
\begin{equation}\label{e:1.9}
\sum_{|\alpha|\le m}(-1)^{|\alpha|}D^\alpha f_\alpha(x, u,\cdots, D^mu)=0,
\end{equation}
where  $D^ku(x)=\{D^\alpha u(x)\,:\,|\alpha|=k\}$,
$k=1,\cdots,m$. For example, when $V$ is $W^{m,p}_0(\Omega)$ (resp. $W^{m,p}(\Omega)$),
the corresponding boundary value problem will be the Dirichlet (resp. Neumann) problem
(cf. \cite[pages 6-7]{Skr3}). Moreover, under \textsf{Hypothesis} $\mathfrak{f}_2$,
if $\dim\Omega=2$ and $f\in C^{k,\alpha}$ for some $\alpha\in (0,1)$ and an integer $k\ge 3$,
it was proved in \cite[Chapter 7, Th.4.4]{Skr3} that  every critical point $u$ of $\mathcal{F}$ on $W^{m,2}_0(\Omega)$ sits in $C^{k+m-1,\alpha}(\overline{\Omega})$; in fact $u$ is also analytic in $\Omega$ provided that $f$ is analytic in its arguments.

As stated on the
pages 118-119 of \cite{Skr3} (see  \cite{Skr1} for detailed
arguments), under \textsf{Hypothesis} $\mathfrak{f}_p$ the functional $\mathcal{F}$ in
(\ref{e:1.8}) is of class $C^1$; and the (derivative) mapping
$\mathcal{F}':W^{m,p}_0(\Omega)\to [W^{m,p}_0(\Omega)]^\ast$ is
 Fr\'echet differentiable if $p>2$, but only G\^ateaux-differentiable if $p=2$.
 A critical point $u$ of $\mathcal{F}$ is said to be {\it nondegenerate}
 if the derivative of $\mathcal{F}'$ at it,
 $\mathcal{F}''(u):W^{m,p}_0(\Omega)\to \mathscr{L}(W^{m,p}_0(\Omega), [W^{m,p}_0(\Omega)]^\ast)$
 is injective. In case $p=2$, if $\mathcal{F}$ has only nondegenerate
 critical points, Skrypnik \cite[Chapter 5]{Skr1} obtained Morse inequalities
 provided that $\mathcal{F}(u)\to+\infty$ as $\|u\|_{m,2}\to\infty$.
  On the other hand he also obtained\\

 \noindent{\bf Skrypnik Theorem}\,(\cite[Chap.5, Sec. 5.1, Theorem~1]{Skr3}).\quad
 {\it If $p=2$, $m=1$ and $f\in C^2(\overline{\Omega}\times\R^1\times\R^n)$
 has uniformly bounded mixed partial derivatives
 $$
 f_{ij}=\frac{\partial^2 f(x,u,\xi)}{\partial \xi_i\partial \xi_j},\quad
f_{i0}=\frac{\partial^2 f(x,u,\xi)}{\partial \xi_i\partial u},\quad
f_{00}=\frac{\partial^2 f(x,u,\xi)}{\partial u^2},
$$
 (and therefore $f$ satisfies \textsf{Hypothesis} $\mathfrak{f}_2$),
  then the functional $\mathcal{F}$ on $W^{1,2}_0(\Omega)$  has Fr\'echet second derivative at zero
  if and only if
$$
f(x,0,\xi)=\sum^n_{i,j=1}a_{ij}(x)\xi_i\xi_j+ \sum^n_{i=1}b_i(x)\xi_i+ c(x).
$$
}
 So, generally speaking, under \textsf{Hypothesis} $\mathfrak{f}_2$
 the known Morse theory method cannot be effectively used to study critical points
 of $\mathcal{F}$ on $W^{m,2}_0(\Omega)$ without nondegenerate conditions.
  A similar question also appears in some optimal control problems \cite{Va1}.\\

The key of this paper is to prove a new splitting theorem (Theorem~\ref{th:S.1.2}) for a class of
non-$C^2$ functionals on a Hilbert space under the following Hypothesis~\ref{hyp:1.1}
(following the notion and terminology in \cite{Lu2} without special statements).
Even if for the Lagrangian systems studied in \cite{Lu1}, we can largely simplify the arguments therein
 with this new theorem. However, the theories in \cite{Lu1,Lu2} may, sometime, provide
 more elaborate results as done in \cite{Lu5, Lu7, Lu8, Lu9}.

\begin{hypothesis}\label{hyp:1.1}
{\rm Let $H$ be a Hilbert space with inner product $(\cdot,\cdot)_H$
and the induced norm $\|\cdot\|$, and let $X$ be a dense linear subspace in $H$.
Let  $V$ be an open neighborhood of the origin $\theta\in H$,
and let $\mathcal{L}\in C^1(V,\mathbb{R})$ satisfy $D\mathcal{L}(\theta)=0$.
Assume that the gradient $\nabla\mathcal{L}$ has a G\^ateaux derivative $B(u)\in \mathscr{L}_s(H)$ at every point
$u\in V\cap X$, and that the map $B:V\cap X\to
\mathscr{L}_s(H)$  has a decomposition
$B=P+Q$, where for each $x\in V\cap X$,  $P(x):H\to H$ is a bounded linear  positive definitive operator and
$Q(x):H\to H$ is a compact linear  operator with the following
properties:
\begin{enumerate}
\item[\bf (D1)]  All eigenfunctions of the operator $B(\theta)$ that correspond
to non-positive eigenvalues belong to $X$;

\item[\bf (D2)] For any sequence $\{x_k\}_{k\ge 1}\subset
V\cap X$ with $\|x_k\|\to 0$ it holds that
$\|P(x_k)u-P(\theta)u\|\to 0$ for any $u\in H$;

\item[\bf (D3)] The  map $Q:V\cap X\to
\mathscr{L}(H)$ is continuous at $\theta$ with respect to the topology
induced from $H$ on $V\cap X$;

\item[\bf (D4)] For any sequence $\{x_n\}_{n\ge 1}\subset V\cap X$ with $\|x_n\|\to 0$ (as $n\to\infty$), there exist
 constants $C_0>0$ and $n_0\in\N$ such that
$$
(P(x_n)u, u)_H\ge C_0\|u\|^2\quad\forall u\in H,\;\forall n\ge n_0.
$$
\end{enumerate}}
\end{hypothesis}

({\it Note}:
In Lemma~\ref{lem:D*} we shall prove that
the condition (D4) is equivalent to (D4*) in \cite{Lu2}. Lemma~\ref{lem:S.2.4}
shows that this property with $X=H$ is hereditary on closed subspaces).

Actually, we prove a more general parameterized splitting theorem, Theorem~\ref{th:S.5.3}.
Using it we complete generalizations of many bifurcation theorems for potential operators
in Section~\ref{sec:Bi}. A weaker Marino-Prodi perturbation type result is also presented in
Section~\ref{sec:MP}. These  constitute abstract theories in Part I of this paper.
 Part II deals with quasi-linear elliptic systems of higher order.
In Section~\ref{sec:Funct}, we study fundamental analytic properties of the functional
 $\mathfrak{F}$ under  Hypothesis~$\mathfrak{F}_{p,N}$. In particular,
  Hypothesis~$\mathfrak{F}_{2,N}$ assures that $\mathfrak{F}$ satisfies  Hypothesis~\ref{hyp:1.1}
on any closed subspace of $W^{m,2}(\Omega, \mathbb{R}^N)$ for a bounded Sobolev domain
$\Omega\subset\mathbb{R}^n$. Because of these, the  Morse theory methods can be used to study
the quasi-linear elliptic boundary value problem (\ref{e:1.4}) under  Hypothesis~$\mathfrak{F}_{2,N}$
as done for the semi-linear elliptic one in \cite{Ch0,  MoMoPa}. In other sections we are only satisfied
to present some direct applications of results in Part I, for example, giving Morse inequalities in Section~\ref{sec:Morse}
and some bifurcation results for quasi-linear elliptic systems in Section~\ref{sec:BifE}.
Further applications will be given in latter papers.



\part{Abstract theories}\label{part:1}
\section{Local Morse theory for a class of non-$C^2$ functionals}\label{sec:S}
\setcounter{equation}{0}

\subsection{Statements of main results}\label{sec:S.1}

Our local Morse theory mainly consist of  a new splitting theorem
and a Marino-Prodi perturbation type result  for a class of non-$C^2$ functionals.

We always assume that Hypothesis~\ref{hyp:1.1} holds
 without special statements. Then it implies that $\nabla\mathcal{L}$ is of class $(S)_+$ near $\theta$
as proved in \cite[p.2966-2967]{Lu2}. In particular, $\mathcal{L}$
satisfies the (PS) condition near $\theta$.

 For the bounded linear self-adjoint operator $B(\theta)$ on the Hilbert space $H$,
 let $H=H^+\oplus H^0\oplus H^-$ be the orthogonal decomposition
according to the positive definite, null and negative definite spaces of it.
Denote by $P^\ast$ the orthogonal projections onto $H^\ast$, $\ast=+,0,-$.
By \cite[Proposition~B.2]{Lu2} the above fundamental assumptions
implies  that there exists a constant $C_0>0$ such that
each $\lambda\in (-\infty, C_0)$ is either not in the spectrum $\sigma(B(\theta))$ or is an
isolated point of $\sigma(B(\theta))$ which is also an eigenvalue of finite multiplicity.
It follows that both $H^0$ and $H^-$ are finitely dimensional, and that
there exists a small $a_0>0$ such that $[-2a_0,
2a_0]\cap\sigma(B(\theta))$ at most contains a point $0$, and hence
\begin{equation}\label{e:S.1.1}
\left.\begin{array}{ll}
  (B(\theta)u, u)_H\ge
2a_0\|u\|^2\quad\forall u\in H^+,\\
  (B(\theta)u, u)_H\le
-2a_0\|u\|^2\quad\forall u\in H^-.
\end{array}\right\}
\end{equation}
Note that (D1) implies $H^-\oplus H^0\subset X$.
$\nu:=\dim H^0$ and $\mu:=\dim H^-$ are called the {\it Morse index} and
{\it nullity} of the critical point $\theta$.
In particular, if $\nu=0$ the critical point $\theta$ is said to be {\it nondegenerate}.
Without special statements, all nondegenerate critical points in this paper are in the sense of this definition.
Moreover, such a critical point must be isolated by (\ref{e:S.3.2}) or (\ref{e:S.3.3})

Our first result is the following
Morse-Palias Lemma. Comparing with that of \cite[Remark~2.2(i)]{Lu2},
 the smoothness of $\mathcal{L}$ is strengthened,
but  other conditions  are suitably  weakened.

\begin{theorem}\label{th:S.1.1}
Under Hypothesis~\ref{hyp:1.1}, if $\nu=0$, i.e., $\theta$
is nondegenerate,  there exist a small $\epsilon>0$, an open neighborhood $W$ of $\theta$ in
$H$ and an origin-preserving homeomorphism, $\phi: B_{H^+}(\theta,\epsilon) +
B_{H^-}(\theta,\epsilon)\to W$,
 such that
$$
\mathcal{L}\circ\phi(u^++ u^-)=\|u^+\|^2-\|u^-\|^2,\quad \forall (u^+, u^-)\in B_{H^+}(\theta,\epsilon)\times
B_{H^-}(\theta,\epsilon).
$$
Moreover, if $\hat{H}$ is a closed subspace containing $H^-$, and $\hat{H}^+$ is the orthogonal
complement of $H^-$ in $\hat{H}$, i.e., $\hat{H}^+=\hat{H}\cap H^+$, then
$\phi$ restricts to a homeomorphism $\hat{\phi}:(B_{\hat{H}^+}(\theta,\epsilon) + B_{H^-}(\theta,\epsilon))
\to\hat{W}:=W\cap\hat{H}$, and $\mathcal{L}\circ\hat{\phi}(u^++ u^-)=\|u^+\|^2-\|u^-\|^2$
for all $(u^+, u^-)\in B_{\hat{H}^+}(\theta,\epsilon)\times B_{H^-}(\theta,\epsilon)$.
\end{theorem}

Under the assumptions of this theorem, if $X=H$ we can prove
that $\nabla\mathcal{L}$ is locally
 invertible at $\theta$ in Theorem~\ref{th:S.4.4}.
Theorem~\ref{th:S.1.1} is also key for us to prove the following degenerate case.

\begin{theorem}[Splitting Theorem]\label{th:S.1.2}
Let Hypothesis~\ref{hyp:1.1}
hold with $X=H$. Suppose $\nu\ne 0$. Then
there exist small positive numbers
$\epsilon, r, s$, a unique continuous map
$\varphi:B_{H^0}(\theta,\epsilon)\to H^+\oplus H^-$
satisfying $\varphi(\theta)=\theta$ and
\begin{equation}\label{e:S.1.2}
 (I-P^0)\nabla\mathcal{L}(z+ \varphi(z))=0\quad\forall z\in B_{H^0}(\theta,\epsilon),
 \end{equation}
an open neighborhood $W$ of $\theta$ in $H$ and an origin-preserving
homeomorphism
$$
\Phi: B_{H^0}(\theta,\epsilon)\times
\left(B_{H^+}(\theta, r) +
B_{H^-}(\theta, s)\right)\to W
$$
of form $\Phi(z, u^++ u^-)=z+ \varphi(z)+\phi_z(u^++ u^-)$ with
$\phi_z(u^++ u^-)\in H^+\oplus H^-$  such that
$$
\mathcal{ L}\circ\Phi(z, u^++ u^-)=\|u^+\|^2-\|u^-\|^2+ \mathcal{
L}(z+ \varphi(z))
$$
for all $(z, u^+ + u^-)\in B_{H^0}(\theta,\epsilon)\times
\left(B_{H^+}(\theta, r) +
B_{H^-}(\theta, s)\right)$.
Moreover, $\varphi$ is of class $C^{1-0}$, and
the homeomorphism $\Phi$ has also properties:
\begin{enumerate}
\item[\bf (a)] For
each $z\in B_{H^0}(\theta,\epsilon)$, $\Phi(z, \theta)=z+ \varphi(z)$,
$\phi_z(u^++ u^-)\in H^-$ if and only if $u^+=\theta$;

\item[\bf (b)] The functional $B_{H^0}(\theta,\epsilon)\ni z\mapsto
\mathcal{L}^\circ(z):=\mathcal{ L}(z+ \varphi(z))$ is of class $C^1$ and
$$
D\mathcal{L}^\circ(z)v=D\mathcal{L}(z+\varphi(z))v,\qquad\forall v\in H^0.
$$
If $\mathcal{L}$ is of class $C^{2-0}$, so is $\mathcal{L}^\circ$.
\end{enumerate}
\end{theorem}

Since the map $\varphi$ satisfying (\ref{e:S.1.2}) is unique,
as \cite{Lu1, Lu2} it is possible to prove in some cases that $\varphi$ and
 $\mathcal{L}^\circ$ are of class $C^1$ and $C^2$, respectively.

Theorem~\ref{th:S.1.2} is a direct consequence of Theorems~\ref{th:S.4.3},\ref{th:S.5.3}
and Proposition~\ref{prop:S.5.2}.

Under the assumptions of Theorem~\ref{th:S.1.2}, we cannot assure that
$\theta$ is an isolated critical point. But, if $x\in H$ is a critical point
of $\mathcal{L}$ and very close to $\theta$, it follows from
(\ref{e:S.1.2}) and (\ref{e:S.4.4})--(\ref{e:S.4.5}) that
$x\in H^0$ and satisfies $\varphi(x)=\theta$.

Theorems~\ref{th:S.1.1},\ref{th:S.1.2}
 cannot be derived from those of \cite{DHK1}.
In fact, according to the conditions (c) and (d) in \cite[Theorem~1.3]{DHK1}
the functional $\mathcal{L}$ in Theorem~\ref{th:S.1.1} should satisfy:
\begin{description}
\item[(c')] $\exists\;\eta>0, \delta>0$ such that
$|(B(u)(u+z)-B(\theta)(u+z), h)|<\eta\|u+z\|\cdot\|h\|$ for all $u\in B_{H}(\theta,\delta)$,
$z\in H^0$ and $h\in H\setminus\{\theta\}$;
\item[(d')] $\exists\; \delta>0$ such that
$\bigr(\nabla\mathcal{L}(z+u^+_1+u^-_1)-\nabla\mathcal{L}(z+u^+_2+u^-_2), (u^+_1-u^+_2)+
(u^-_1-u^-_2)\bigl)>0$
for all $(u^+_1, u^-_1), (u^+_2, u^-_2)\in B_{H^+}(\theta,\delta)\times B_{H^-}(\theta,\delta)$
with $u^+_1+u^-_1\ne u^+_2+u^-_2$.
\end{description}

The former implies
$\|B(u)(u+z)-B(\theta)(u+z)\|\le\eta\|u+z\|$ for all $u\in B_{H}(\theta,\delta)$,
$z\in H^0$; and specially
\begin{eqnarray*}
&&\|B(u)u-B(\theta)u|\le\eta\|u\|\quad\forall u\in B_{H}(\theta,\delta),\\
&&\|B(z)z-B(\theta)z\|\le\eta\|z\|\quad\forall z\in B_{H}(\theta,\delta)\cap H^0.
\end{eqnarray*}

The latter implies that for some $t=t(z, u^+_1,u^-_1,u^+_2,u^-_2)\in (0,1)$,
$$
\bigl(B(z+u^+_2+u^-_2+ tu^++tu^-)(u^++u^-), u^++u^-\bigr)>0
$$
with $u^+=u^+_1-u^+_2$ and $u^-=u^-_1-u^-_2$.

From these it is not hard to see that under our assumptions the conditions
(c') and (d') cannot be satisfied in general.

Let $H_q(A,B;{\bf K})$ denote the $q$th relative singular homology group of
a pair $(A,B)$ of topological spaces with coefficients in Abel group ${\bf K}$.
For each $q\in\N\cup\{0\}$ {\it the $q$th critical group} (with coefficients in ${\bf K}$)
of $\mathcal{L}$ at a point $\theta$
is defined by
$$
C_q(\mathcal{L},\theta;{\bf K})=H_q(\mathcal{L}_c\cap U, \mathcal{L}_c\cap U\setminus\{\theta\};{\bf K}),
$$
where $c=\mathcal{L}(\theta)$, $\mathcal{L}_c=\{\mathcal{L}\le c\}$ and $U$ is a neighborhood of $\theta$ in $H$.

Under the assumptions of Theorem~\ref{th:S.1.1} we have
$C_q(\mathcal{L},\theta;{\bf K})=\delta_{q\mu}{\bf K}$ as usual.
For the degenerate case, though our $\mathcal{L}^{\circ}$ is only of class $C^1$,
 the proofs in \cite[Theorem~8.4]{MaWi} and \cite[Theorem~5.1.17]{Ch1} (or \cite[Theorem~I.5.4]{Ch})
 may be slightly  modify to get  the following shifting theorem, a special case of Theorem~\ref{th:S.5.4}.

\begin{theorem}[Shifting Theorem]\label{th:S.1.3}
Under the assumptions of Theorem~\ref{th:S.1.2},  if $\theta$ is an
isolated critical point  of $\mathcal{ L}$, for any Abel group ${\bf
K}$ it holds that
$$
C_q(\mathcal{L}, \theta;{\bf K})\cong C_{q-\mu}(\mathcal{
L}^{\circ}, \theta; {\bf K})\quad\forall q=0, 1,\cdots,
$$
\end{theorem}

As done for $C^2$ functionals in \cite{Ch,Ch1,MaWi,MoMoPa}
some critical point theorems can be derived from Theorem~\ref{th:S.1.3}.
For example, $C_q(\mathcal{L}, \theta;{\bf K})$
is equal to $\delta_{q\mu}{\bf K}$ (resp.
$\delta_{q(\mu+\nu)}{\bf K}$) if $\theta$ is a local minimizer
(resp. maximizer)  of $\mathcal{L}^{\circ}$, and
$C_q(\mathcal{L}, \theta;{\bf K})=0$ for $q\le\mu$ and $q\ge\mu+\nu$
if $\theta$ is neither a local minimizer nor local  maximizer of
$\mathcal{L}^{\circ}$. Similarly, the corresponding generalizations of  Theorems~2.1, 2.1', 2.2, 2.3
 and Corollary~1.3 in \cite[Chapter II]{Ch} can be obtained with
Theorems~\ref{th:S.1.1}, \ref{th:S.1.2} and their equivariant versions in Section~
\ref{sec:S.6}. In particular, as a generalization of
\cite[Theorem~II.1.6]{Ch} (or \cite[Theorem~5.1.20]{Ch1})
we have

\begin{theorem}\label{th:S.1.4}
Let Hypothesis~\ref{hyp:1.1}
hold with $X=H$, and let  $\theta$ be an isolated
critical point of mountain pass type, i.e.,  $C_1(\mathcal{L}, \theta;{\bf K})\ne 0$.
Suppose that $\nu>0$ and $\mu=0$ imply $\nu=1$. Then
$C_q(\mathcal{L}, \theta;{\bf K})=\delta_{q1}{\bf K}$.
\end{theorem}

When $\nu>0$ and $\mu=1$, $C_{0}(\mathcal{
L}^{\circ}, \theta; {\bf K})\ne 0$ by Theorem~\ref{th:S.1.3}.
We can change $\mathcal{L}^\circ$ outside a very small
neighborhood $\theta\in B_{H^0}(\theta,\epsilon)$ to get
a $C^1$ functional on $H^0$ which is coercive (and so satisfies the
(PS)-condition). Then it follows from $C_{0}(\mathcal{
L}^{\circ}, \theta; {\bf K})\ne 0$ and \cite[Proposition~6.95]{MoMoPa}
that $\theta$ is a local minimizer of $\mathcal{L}^\circ$.

As a generalization of Corollary~3.1 in
\cite[page 102]{Ch} we have also:
Under the assumptions of Theorem~\ref{th:S.1.4},
 if the smallest eigenvalue $\lambda_1$ of $B(\theta)=d^2\mathcal{L}(\theta)$
 is simple whenever $\lambda_1=0$, then $\lambda_1\le 0$, and
  ${\rm index}(\nabla \mathcal{L}, \theta)=-1$.

Theorem~5.1 and Corollary~5.1 in
\cite[page 121]{Ch} are also true if ``$f\in C^2(M,\mathbb{R})$"
and ``Fredholm operators $d^2f(x_i)$" are replaced by
``$f\in C^1(M,\mathbb{R})$ and $\nabla f$ is G\^ateaux differentiable"
and ``under some
chart around $p_i$ the functional $f$ has a representation
that satisfies Hypothesis~\ref{hyp:1.1}", respectively.

 Marino and Prodi  \cite{MP} studied local Morse function approximations for $C^2$
  functionals on Hilbert spaces.  We shall generalize their result to a class of functionals
    satisfying the following stronger assumption than Hypothesis~\ref{hyp:1.1}.

\begin{hypothesis}\label{hyp:MP.1}
{\rm Let  $V$ be an open set of a Hilbert space $H$ with inner product $(\cdot,\cdot)_H$,
and $\mathcal{L}\in C^1(V,\mathbb{R})$.
Assume that the gradient $\nabla\mathcal{L}$ has a G\^ateaux derivative $B(u)\in \mathscr{L}_s(H)$ at every point $u\in V$, and that the map $B:V\to \mathscr{L}_s(H)$  has a decomposition
$B=P+Q$, where for each $u\in V$,  $P(u):H\to H$ is a bounded linear  positive definitive operator and
$Q(u):H\to H$ is a compact linear  operator with the following
properties:\\
{\bf (i)} For any $u\in H$, the map $V\ni x\mapsto P(x)u\in H$
is continuous;\\
{\bf (ii)} The  map $Q:V\to\mathscr{L}(H)$ is continuous;\\
{\bf (iii)} $P$ is local positive definite uniformly, i.e., each $x_0\in V$
 has a neighborhood $\mathscr{U}(x_0)$ such that for some
 constants $C_0>0$,
$$
(P(x)u, u)_H\ge C_0\|u\|^2,\quad\forall u\in H,\;\forall x\in \mathscr{U}(x_0).
$$
}
\end{hypothesis}

 As in the proof of Theorem~\ref{th:3.1}, under Hypothesis~$\mathfrak{F}_{2,N}$,
we can check that the functional $\mathfrak{F}$ in (\ref{e:1.3}) satisfies this hypothesis.
By improving methods in \cite{MP, Ch, CiVa} we may prove

\begin{theorem}\label{th:MP.2}
Under Hypothesis~\ref{hyp:MP.1}, suppose: (a) $u_0\in V$ is
a unique critical point of $\mathcal{L}$, (b) the corresponding maps $\varphi$
and $\mathcal{L}^\circ$ as in Theorem~\ref{th:S.1.2} near $u_0$ are of classes $C^1$ and $C^2$,
respectively, (c) $\mathcal{L}$
satisfies the (PS) condition. Then
for any $\epsilon>0$ and $r>0$ such that
$\bar{B}_H(u_0, r)\subset V$ and
there exists a functional $\tilde{\mathcal{L}}\in C^1(V,\mathbb{R})$ with the following properties:\\
{\bf (i)} $\tilde{\mathcal{L}}$ satisfies Hypothesis~\ref{hyp:MP.1} and the (PS) condition;\\
{\bf (ii)} $\sup_{u\in V}\|\mathcal{L}^{(i)}(u)-\tilde{\mathcal{L}}^{(i)}(u)\|<\epsilon$, $i=0,1,2$;\\
{\bf (iii)} $\mathcal{L}(x)=\tilde{\mathcal{L}}(x)$ if $x\in V$ and $\|u-u_0\|\ge r$;\\
{\bf (iv)} the critical points of $g$, if any, are in ${B}_H(u_0, r)$ and nondegenerate
(so finitely many by the arguments below \ref{e:S.1.1}); moreover the Morse indexes of these critical
points sit in $[m^-, m^-+n^0]$, where $m^-$ and $n^0$ are the Morse index and nullity of $u_0$,
respectively.
\end{theorem}

As showed, the functionals in \cite{Lu1, Lu9}  satisfy the conditions of this theorem.
Marino--Prodi's result has many important applications in the critical point theory,
see \cite{Ch, CiVa, Gh, LazS} and literature therein. With Theorem~\ref{th:MP.2} they may be given in our
framework.

   Marino--Prodi's perturbation theorem in \cite{MP}
was also generalized to the equivariant case under the finite (resp. compact Lie) group
action by Wasserman \cite{Was} (resp. Viterbo \cite{Vit1}), see the proof of Theorem~7.8 in
\cite[Chapter~I]{Ch} for full details.
Similarly, we can present an equivariant version of Theorem~\ref{th:MP.2} for compact
Lie group action, but it is omitted here.\\

\noindent{\bf Strategies of proofs for results in this section and arrangements}.\quad
Under the assumptions of Theorem~\ref{th:S.1.1}, no known implicit function
theorems or contraction mapping principles can be used to get $\varphi$ in (\ref{e:S.1.2}), which
is different from the case in \cite{Lu1,Lu2}.
The methods in \cite{DHK1} provide a possible way to construct such a $\varphi$.
However, as mentioned above our assumptions cannot guarantee
the above conditions (c') and (d'). Fortunately, it is with Lemma~\ref{lem:S.4.1} and Theorem~\ref{th:S.1.1}
that we can complete this construction.

In Section~\ref{sec:S.2} we list some lemmas.   Theorem~\ref{th:S.1.1} will be proved in Section~\ref{sec:S.3}.
 It is necessary for a key implicit function theorem for a family of potential operators, Theorem~\ref{th:S.4.3},
 which is proved in Section~\ref{sec:S.4};   we also give an inverse function theorem,  Theorem~\ref{th:S.4.4}, there.
In Section~\ref{sec:S.5} we shall prove
a parameterized splitting theorem, Theorem~\ref{th:S.5.3},
and a parameterized shifting theorem, Theorem~\ref{th:S.5.4};
 Theorems~\ref{th:S.1.2},~\ref{th:S.1.3} are special cases of
them, respectively. The equivariant case  is considered in
Section~\ref{sec:S.6}.  Theorem~\ref{th:MP.2} will be proved
in Section~\ref{sec:MP}.

\subsection{Lemmas}\label{sec:S.2}

Under Hypothesis~\ref{hyp:1.1} we have
the following two lemmas as proved in \cite{Lu1,Lu2}.

\begin{lemma}\label{lem:S.2.1}
 There exists a function $\omega:V\cap X\to [0, \infty)$  such that $\omega(x)\to 0$ as $x\in V\cap X$ and $\|x\|\to
0$, and that
$$
|(B(x)u, v)_H- (B(\theta)u, v)_H |\le \omega(x) \|u\|\cdot\|v\|
$$
for any $x\in V\cap X$,  $u\in H^0\oplus H^-$ and $v\in H$.
\end{lemma}

\begin{lemma}\label{lem:S.2.2}
There exists a  small neighborhood $U\subset V$ of $\theta$ in $H$
and a number $a_1\in (0, 2a_0]$ such that for any $x\in U\cap X$,
\begin{enumerate}
\item[{\rm (i)}] $(B(x)u, u)_H\ge a_1\|u\|^2\;\forall u\in H^+$;
\item[{\rm (ii)}] $|(B(x)u,v)_H|\le\omega(x)\|u\|\cdot\|v\|\;\forall u\in H^+, \forall v\in
H^-\oplus H^0$;
\item[{\rm (iii)}] $(B(x)u,u)_H\le-a_0\|u\|^2\;\forall u\in H^-$.
\end{enumerate}
\end{lemma}

\begin{lemma}\label{lem:D*}
Actually, (D4) is equivalent to the condition (D*) in \cite{Lu2}, i.e.,
\begin{enumerate}
\item[\bf (D4*)] There exist positive constants $\eta_0>0$ and  $C'_0>0$ such that
$$
(P(x)u, u)\ge C'_0\|u\|^2\quad\forall u\in H,\;\forall x\in
B_H(\theta,\eta_0)\cap X.
$$
\end{enumerate}
\end{lemma}
\begin{proof}
 Indeed, since each $P(x)$ is a positive definite
linear operator, its spectral set is a bounded closed subset in $(0,\infty)$. Moreover, we have
$\sigma(\sqrt{P(x)})=\{\sqrt{\lambda}\,|\, \lambda\in\sigma(P(x))\}$.
So (D4) is equivalent to the statement: For any sequence $\{x_n\}\subset V\cap X$ with $\|x_n\|\to 0$ (as $n\to\infty$), there holds:
$\inf_n\min\sigma(\sqrt{P(x_n)})>0$.
Similarly, (D4*) can be equivalently expressed as: There exist positive constants $\eta_0>0$ such that
$$
\inf\{\min\sigma(\sqrt{P(x)})\,|\, x\in
B_H(\theta,\eta_0)\cap X\}>0.
$$
Suppose (D4) holds. Since $\eta\mapsto \inf\{\min\sigma(\sqrt{P(x)})\,|\, x\in
B_H(\theta,\eta)\cap X\}$ is non-increasing, that (D4) does not hold means that
there exists a sequence $\{x_n\}\subset V\cap X$ with $\|x_n\|\to 0$ (as $n\to\infty$) such that
$\inf_n\min\sigma(\sqrt{P(x_n)})\to 0$,
which contradicts (D4).
\end{proof}

\begin{lemma}\label{lem:S.2.4}
Suppose that {\rm Hypothesis~\ref{hyp:1.1}} with $X=H$ is satisfied. Then for
any closed subspace  $\hat{H}\subset H$, $(\hat{H}, \hat{V}, \hat{\mathcal{L}})$
satisfies {\rm Hypothesis~\ref{hyp:1.1}} with $X=H$, where
 $\hat{V}:=V\cap\hat{H}$ and $\hat{\mathcal{L}}:=\mathcal{L}|_{\hat{V}}$.
\end{lemma}
\begin{proof}
Clearly, $\hat{\mathcal{L}}\in C^1(\hat{V},\mathbb{R})$ and $D\hat{\mathcal{L}}(\theta)=0$.
Denote by $\Pi:H\to\hat{H}$ the orthogonal projection.
Then the gradient of $\hat{\mathcal{L}}$ at $u\in\hat{V}$, $\nabla\hat{\mathcal{L}}(u)$,
is equal to $\Pi\nabla{\mathcal{L}}(u)$. It follows that
 $\nabla\hat{\mathcal{L}}$ at any $u\in \hat{V}$
  has a G\^ateaux derivative $\hat{B}(u)=\Pi\circ B(u)|_{\hat{H}}\in \mathscr{L}_s(\hat{H})$.
For any $u\in \hat{V}$, put $\hat{P}(u)=\Pi\circ P(u)|_{\hat{H}}$ and
$\hat{Q}(u)=\Pi\circ Q(u)|_{\hat{H}}$. Then $\hat{B}=\hat{P}+\hat{Q}:\hat{V}\to
\mathscr{L}_s(\hat{H})$, $\hat{P}(u)$ is positive definite, and $\hat{Q}(u)$
is a compact linear operator. It is easily checked that other conditions are satisfied.
\end{proof}

\subsection{Proof of Theorem~\ref{th:S.1.1}}\label{sec:S.3}

Take a small $\epsilon>0$ so that $\bar
B_{H^+}(\theta,\epsilon)\oplus \bar
B_{H^-}(\theta,\epsilon)$ is contained in the open
neighborhood $U$ in Lemma~\ref{lem:S.2.2}.
Let us prove the $C^1$ functional
$$
\bar
B_{H^+}(\theta,\epsilon)\oplus \bar
B_{H^-}(\theta,\epsilon)\to\mathbb{R},\;u^++ u^-\mapsto\mathcal{L}(u^++u^-)
$$
satisfies the conditions in
\cite[Theorem~1.1]{DHK}.

\noindent{\bf Step 1}. Fix  $u^+\in
\bar B_{H^+}(\theta,\epsilon)\cap X$ and $u^-_1, u^-_2\in\bar
B_{H^-}(\theta,\epsilon)$ (which are contained in $X$ by (D1)).
Since $\nabla\mathcal{L}$ have a G\^ateaux derivative $B(u)\in \mathcal{L}_s(H)$ at every point
$u\in V\cap X$,  the
function
$$
V\to\mathbb{R},\;u\mapsto (\nabla\mathcal{L}(u^++u), u^-_2-u^-_1)_H
$$
is G\^ateaux differentiable at every $u\in V\cap X$.
Using the mean value theorem we have $t\in (0, 1)$ such that
\begin{eqnarray*}
&&(\nabla\mathcal{L}(u^++u^-_2), u^-_2-u^-_1)_H
- (\nabla\mathcal{L}(u^++u^-_1), u^-_2-u^-_1)_H\\
&=&\left(B(u^++ u^-_1+ t(u^-_2-u^-_1))(u^-_2-u^-_1),
u^-_2-u^-_1\right)_H\\
&\le& -a_0\|u^-_2-u^-_1\|^2
\end{eqnarray*}
by Lemma~\ref{lem:S.2.2}(iii).  Note that $\bar
B_{H^+}(\theta,\epsilon)\cap X$ is dense in $\bar
B_{H^+}(\theta,\epsilon)$ and $\nabla\mathcal{L}$ is continuous. For all $u^+\in \bar
B_{H^+}(\theta,\epsilon)$ and $u^-_i\in\bar
B_{H^-}(\theta,\epsilon)$, $i=1, 2$, we deduce
\begin{eqnarray}\label{e:S.3.1}
(\nabla\mathcal{L}(u^++u^-_2), u^-_2-u^-_1)_H
- (\nabla\mathcal{L}(u^++u^-_1), u^-_2-u^-_1)_H \le
-a_0\|u^-_2-u^-_1\|^2.
\end{eqnarray}
 This implies the
condition (ii) of \cite[Theorem~1.1]{DHK}.

\noindent{\bf Step 2}.  Let  $u^+\in
\bar B_{H^+}(\theta,\epsilon)\cap X$ and $u^-\in\bar
B_{H^-}(\theta,\epsilon)$ (which is contained in $X$ by (D1)).
Then since  $D\mathcal{L}(\theta)=0$, by  the mean value
theorem, for some $t\in (0, 1)$ we have
\begin{eqnarray}\label{e:S.3.2}
&&D\mathcal{L}(u^++ u^-)(u^+-u^-)\nonumber\\
&=&(\nabla\mathcal{L}(u^++u^-), u^+-u^-)_H-(\nabla\mathcal{L}(\theta), u^+-u^-)_H\nonumber\\
&=&\left(B(t(u^++u^-))(u^++u^-), u^+-u^-\right)_H\nonumber\\
&=&\left(B(t(u^++u^-))u^+, u^+\right)_H
-\left(B(t(u^++u^-))u^-, u^-\right)_H\nonumber\\
&\ge & a_1\|u^+\|^2+ a_0\|u^-\|^2
\end{eqnarray}
by Lemma~\ref{lem:S.2.2}(i) and (iii).  As above (\ref{e:S.3.2}) also
holds for all $u^+\in \bar B_{H^+}(\theta,\epsilon)$ because
$\bar B_H(\theta,\epsilon)\cap X^+$ is dense in $\bar
B_H(\theta,\epsilon)\cap H^+$. Hence $D\mathcal{L}(u^++u^-)(u^+-u^-)>0$ for $(u^+, u^-)\ne (\theta, \theta)$. ({\it This implies $\theta$ to be an isolated critical point of $\mathcal{L}$}).
The condition (iii) of \cite[Theorem~1.1]{DHK} is satisfied.

\noindent{\bf Step 3}. For  $u^+\in
\bar B_{H^+}(\theta,\varepsilon)\cap X$, as above we have $t\in (0,
1)$ such that
\begin{eqnarray}\label{e:S.3.3}
D\mathcal{L}(u^+)u^+
&=&D\mathcal{L}(u^+)u^+- D\mathcal{L}(\theta)u^+\nonumber\\
&=&(\nabla\mathcal{L}(u^+), u^+)_H-(\nabla\mathcal{L}(\theta), u^+)_H\nonumber\\
&=&\left(B(tu^+)u^+, u^+\right)_H\nonumber\\
&\ge& a_1\|u^+\|^2
\end{eqnarray}
because of Lemma~\ref{lem:S.2.2}(i). It follows that
$$
D\mathcal{L}(u^+)u^+ \ge a_1\|u^+\|^2> p(\|u^+\|)\quad\forall u^+\in \bar
B_{H^+}(\theta,\epsilon)\setminus\{\theta\},
$$
where $p:(0, \varepsilon]\to (0, \infty)$ is a non-decreasing
function given by $p(t)=\frac{a_1}{2}t^2$. Hence the condition
(iv) of \cite[Theorem~1.1]{DHK} is satisfied.

For the second claim, note that (\ref{e:S.3.1})--(\ref{e:S.3.3}) hold
for all $u^+\in\bar{B}_{H^+}(\theta, \epsilon)$ and $u^-, u^-_i\in \bar{B}_{H^+}(\theta, \epsilon)$, $i=1,2$.
Of course, they are still true for all $u^+\in\bar{B}_{\hat{H}^+}(\theta, \epsilon)$.
Carefully checking the proof of \cite[Theorem~1.1]{DHK} the conclusion is easily obtained.
(Note that this claim seems unable to be directly derived from Lemma~\ref{lem:S.2.4}.)
\hfill$\Box$\vspace{2mm}

Actually, from the proof of Theorem~\ref{th:S.1.1} we may get the more general claim, which is
needed for later applications.

\begin{theorem}\label{th:S.3.1}
Under Hypothesis~\ref{hyp:1.1}, let  $\mathcal{G}\in C^1(V,\mathbb{R})$
satisfy: {\rm i)} $\mathcal{G}'(\theta)=\theta$,
{\rm ii)} the gradient $\nabla\mathcal{G}$ has
G\^ateaux derivative $\mathcal{G}''(u)\in \mathscr{L}_s(H)$ at any $u\in V$, and $\mathcal{G}'': V\to \mathscr{L}_s(H)$ are continuous at $\theta$.
Suppose ${\rm Ker}(B(\theta))=\{\theta\}$,
 i.e., $\theta$ is a nondegenerate critical point of $\mathcal{L}$.
 Then there exist $\rho>0$, $\epsilon>0$, a family of open neighborhoods of $\theta$ in
$H$, $\{W_\lambda\,|\, |\lambda|\le\rho\}$
and a family of origin-preserving homeomorphisms, $\phi_\lambda: B_{H^+}(\theta,\epsilon) +
B_{H^-}(\theta,\epsilon)\to W_\lambda$, $|\lambda|\le\rho$,
 such that
$$
(\mathcal{L}+\lambda\mathcal{G})\circ\phi_\lambda(u^++ u^-)=\|u^+\|^2-\|u^-\|^2
$$
for all $(u^+, u^-)\in B_{H^+}(\theta,\epsilon)\times
B_{H^-}(\theta,\epsilon)$. Moreover, $[-\rho,\rho]\times (B_{H^+}(\theta,\epsilon) +
B_{H^-}(\theta,\epsilon))\ni (\lambda, u)\mapsto \phi_\lambda(u)\in H$
is continuous.
\end{theorem}

\begin{proof}
Since $\mathcal{G}'': V\to \mathscr{L}_s(H)$ are continuous at $\theta$, as in the proof
of (\ref{e:S.3.1})  we may shrink $\epsilon>0$ and find $\rho>0$ such that
for all $\lambda\in [-\rho,\rho]$, $u^+\in \bar
B_{H^+}(\theta,\epsilon)$ and $u^-_i\in\bar
B_{H^-}(\theta,\epsilon)$, $i=1, 2$,
\begin{eqnarray*}
|\lambda(\nabla\mathcal{G}(u^++u^-_2), u^-_2-u^-_1)_H
- \lambda(\nabla\mathcal{G}(u^++u^-_1), u^-_2-u^-_1)_H| \le
\frac{a_0}{2}\|u^-_2-u^-_1\|^2.
\end{eqnarray*}
This and (\ref{e:S.3.1}) lead to
\begin{eqnarray}\label{e:S.3.4}
&&(\nabla(\mathcal{L}+\lambda\mathcal{G})(u^++u^-_2), u^-_2-u^-_1)_H
- (\nabla(\mathcal{L}+\lambda\mathcal{G})(u^++u^-_1), u^-_2-u^-_1)_H\nonumber\\
&& \le -\frac{a_0}{2}\|u^-_2-u^-_1\|^2.
\end{eqnarray}
Similarly, as in the proof of (\ref{e:S.3.2}) we may shrink the above $\rho>0$ and $\epsilon>0$
so that
\begin{eqnarray*}
|\lambda D\mathcal{G}(u^++ u^-)(u^+-u^-)|
\le  \frac{a_1}{2}\|u^+\|^2+ \frac{a_0}{2}\|u^-\|^2
\end{eqnarray*}
for all $\lambda\in[-\rho,\rho]$, $u^+\in \bar B_{H^+}(\theta,\epsilon)$ and $u^-\in\bar
B_{H^-}(\theta,\epsilon)$. This and (\ref{e:S.3.2}) yield
\begin{eqnarray*}
D(\mathcal{L}+\lambda\mathcal{G})(u^++ u^-)(u^+-u^-)
\ge  \frac{a_1}{2}\|u^+\|^2+ \frac{a_0}{2}\|u^-\|^2
\end{eqnarray*}
and specially $D(\mathcal{L}+\lambda\mathcal{G})(u^+)(u^+)
\ge  \frac{a_1}{2}\|u^+\|^2$. These and (\ref{e:S.3.4})
show that the conditions of \cite[Theorem~A.1]{Lu2} are satisfied.
The desired conclusions follow immediately.
\end{proof}

\subsection{An implicit function theorem for a family of potential operators}\label{sec:S.4}

In this subsection we shall prove an implicit function theorem, Theorem~\ref{th:S.4.3},
which implies the first claim in Theorem~\ref{th:S.1.2}.
We also give an inverse function theorem, Theorem~\ref{th:S.4.4},
though it is not used in this paper.

Without special statements, we always assume that
Hypothesis~\ref{hyp:1.1} holds in this subsection.

Take $\epsilon>0$, $r>0$ and $s>0$ so small that the closures of both
$$
\mathcal{Q}_{r,s}:=B_{H^+}(\theta,r)\oplus B_{H^-}(\theta,s)\quad\hbox{and}\quad
B_{H^0}(\theta,\epsilon)\oplus\mathcal{Q}_{r,s}
$$
are contained in the neighborhood $U$ in Lemma~\ref{lem:S.2.2}.
Since $H^0\subset X$,  $X\cap \mathcal{Q}_{r,s}$ is also dense in $\mathcal{Q}_{r,s}$.
Let $P^\bot=I-P^0=P^++P^-$. By Lemma~\ref{lem:S.2.2} we obtain
$a_0'>0, a_1'>0$ such that
\begin{eqnarray}
(P^\bot\nabla\mathcal{L}(z+u), u^+)_H=(\nabla\mathcal{L}(u), u^+)_H
\ge a_1'\|u^+\|^2-a_0'[\omega(z+u)]^2\|u^-\|^2,\label{e:S.4.1}\\
(P^\bot\nabla\mathcal{L}(z+u), u^-)_H=(\nabla\mathcal{L}(u), u^-)_H
\le -a_1'\|u^-\|^2+a_0'[\omega(z+u)]^2\|u^+\|^2\label{e:S.4.2}
\end{eqnarray}
for all $u\in \overline{\mathcal{Q}_{r,s}}$ and $z\in\bar B_{H^0}(\theta,\epsilon)$.
 Since $\omega(z+u)\to 0$ as $\|z+u\|\to 0$, by shrinking
$r>0,s>0$ and $\epsilon>0$ we can require
\begin{equation}\label{e:S.4.3}
[\omega(z+u)]^2<\frac{a_1'}{2a_0'},\quad\forall (z,u)\in \bar B_{H^0}(\theta,\epsilon)\times\overline{\mathcal{Q}_{r,s}}.
\end{equation}
Then this and (\ref{e:S.4.1})--(\ref{e:S.4.2}) lead to
\begin{eqnarray}
&&(P^\bot\nabla\mathcal{L}(z+u), u^+)_H\ge a_1'\|u^+\|^2-\frac{a_1'}{2}\|u^-\|^2,\label{e:S.4.4}\\
&&(P^\bot\nabla\mathcal{L}(z+u), u^-)_H\le -a_1'\|u^-\|^2+\frac{a_1'}{2}\|u^+\|^2\label{e:S.4.5}
\end{eqnarray}
for all $u\in \overline{\mathcal{Q}_{r,s}}$ and $z\in\bar B_{H^0}(\theta,\epsilon)$, and hence
\begin{eqnarray}
&&\bigl(tP^\bot\nabla\mathcal{L}(z_1+u)+ (1-t)P^\bot\nabla\mathcal{L}(z_2+u), u^+\bigr)_H
\ge a_1'\|u^+\|^2-\frac{a_1'}{2}\|u^-\|^2,\label{e:S.4.6}\\
&&\bigl(tP^\bot\nabla\mathcal{L}(z_1+u)+ (1-t)P^\bot\nabla\mathcal{L}(z_2+u), u^-\bigr)_H
\le -a_1'\|u^-\|^2+\frac{a_1'}{2}\|u^+\|^2\label{e:S.4.7}
\end{eqnarray}
for all $u\in \overline{\mathcal{Q}_{r,s}}$ and $z_j\in\bar B_{H^0}(\theta,\epsilon)$, $j=1,2$,
and $t\in [0,1]$.

\begin{lemma}\label{lem:S.4.1}
If $r>0,s>0$ and $\epsilon>0$ are so small that (\ref{e:S.4.3}) is satisfied, then
$$
\inf\{\|tP^\bot\nabla\mathcal{L}(z_1+u)+ (1-t)P^\bot\nabla\mathcal{L}(z_2+u)\|\,|\,
 (t,z_1,z_2,u)\in\Omega \}>0,
$$
where $\Omega=[0,1]\times \bar B_{H^0}(\theta,\epsilon)\times \bar B_{H^0}(\theta,\epsilon)\times
\partial\overline{\mathcal{Q}_{r,s}}$.
\end{lemma}

\begin{proof}
Note that $\partial\overline{\mathcal{Q}_{r,s}}$ is union of two closed subsets
of it, i.e.
$$
\partial\overline{\mathcal{Q}_{r,s}}=[(\partial B_{H^+}(\theta,r))\oplus \bar B_{H^-}(\theta,s)]\cup
[ \bar B_{H^+}(\theta,r)\oplus (\partial B_{H^-}(\theta,s))].
$$
Then $\Omega=\Lambda_1\cup\Lambda_2$, where
$\Lambda_1=[0,1]\times \bar B_{H^0}(\theta,\epsilon)\times \bar B_{H^0}(\theta,\epsilon)\times
(\partial B_{H^+}(\theta,r))\oplus \bar B_{H^-}(\theta,s)$ and
$\Lambda_2=[0,1]\times \bar B_{H^0}(\theta,\epsilon)\times \bar B_{H^0}(\theta,\epsilon)\times
B_{H^+}(\theta,r)\oplus (\partial\bar B_{H^-}(\theta,s))$. We firstly prove
\begin{equation}\label{e:S.4.8}
\inf\{\|tP^\bot\nabla\mathcal{L}(z_1+u)+ (1-t)P^\bot\nabla\mathcal{L}(z_2+u)\|\,|\,
 (t,z_1,z_2,u)\in\Lambda_1 \}>0.
\end{equation}
By a contradiction we assume that there exist sequences $\{t_n\}_{n\ge 1}\subset [0,1]$ and
$$
\{z_n\}_{n\ge 1},\,\{z_n'\}_{n\ge 1}\subset \bar B_{H^0}(\theta,\epsilon),\quad\{u_n\}_{n\ge 1}\subset
(\partial B_{H^+}(\theta,r))\oplus \bar B_{H^-}(\theta,s)
$$
 such that $\|t_nP^\bot\nabla\mathcal{L}(z_n+u_n)+(1-t_n)P^\bot\nabla\mathcal{L}(z_n'+u_n)\|\to 0$.
Hence  after removing finite many terms we can assume
\begin{eqnarray}
&&(t_nP^\bot\nabla\mathcal{L}(z_n+u_n)+(1-t_n)P^\bot\nabla\mathcal{L}(z_n'+u_n), u^+_n)_H
\le \frac{a_1'r^2}{4},\quad\forall n\in\mathbb{N},\label{e:S.4.10}\\
&&(t_nP^\bot\nabla\mathcal{L}(z_n+u_n)+(1-t_n)P^\bot\nabla\mathcal{L}(z_n'+u_n), u^-_n)_H
\ge -\frac{a_1'r^2}{4},\quad\forall n\in\mathbb{N}.\label{e:S.4.11}
\end{eqnarray}
Note that $u^+_n\in \partial B_{H^+}(\theta,r))$ and $u^-_n\in\bar B_{H^-}(\theta,s)$.
So (\ref{e:S.4.10}) and (\ref{e:S.4.6}) lead to
\begin{eqnarray*}
\frac{a_1'}{4}r^2\ge (t_nP^\bot\nabla\mathcal{L}(z_n+u_n)+(1-t_n)P^\bot\nabla\mathcal{L}(z_n'+u_n), u^+_n)_H
\ge  a_1'r^2-\frac{a_1'}{2}\|u^-_n\|^2
\end{eqnarray*}
and therefore
\begin{equation}\label{e:S.4.12}
\frac{r^2}{\|u_n^-\|^2}\le\frac{2}{3},\quad\forall n\in\mathbb{N}.
\end{equation}
Moreover, from (\ref{e:S.4.7}) and (\ref{e:S.4.11})  we derive
\begin{eqnarray*}
-\frac{a_1'r^2}{4}\le
(t_nP^\bot\nabla\mathcal{L}(z_n+u_n)+(1-t_n)P^\bot\nabla\mathcal{L}(z_n'+u_n), u^-_n)_H
\le -a_1'\|u^-_n\|^2+\frac{a_1'r^2}{2}
\end{eqnarray*}
and hence
$$
\frac{r^2}{\|u_n^-\|^2}\ge\frac{4}{3},\quad\forall n\in\mathbb{N},
$$
which contradicts  (\ref{e:S.4.12}). (\ref{e:S.4.8}) is proved.

Similarly, suppose that  there  exist sequences $\{t_n\}_{n\ge 1}\subset [0,1]$ and
$$
 \{z_n\}_n,\,\{z_n'\}_n\subset \bar B_{H^0}(\theta,\epsilon),\quad
\{v_n\}_n\subset
B_{H^+}(\theta,r)\oplus (\partial B_{H^-}(\theta,s))
$$
 such that $\|t_nP^\bot\nabla\mathcal{L}(z_n+v_n)+(1-t_n)P^\bot\nabla\mathcal{L}(z_n'+v_n)\|\to 0$.
 As above we can assume
\begin{eqnarray}
&&(t_nP^\bot\nabla\mathcal{L}(z_n+v_n)+(1-t_n)P^\bot\nabla\mathcal{L}(z_n'+v_n), v^+_n)_H
\le \frac{a_1's^2}{4},\quad\forall n\in\mathbb{N},\label{e:S.4.14}\\
&&(t_nP^\bot\nabla\mathcal{L}(z_n+v_n)+(1-t_n)P^\bot\nabla\mathcal{L}(z_n'+v_n), v^-_n)_H
\ge -\frac{a_1's^2}{4},\quad\forall n\in\mathbb{N}.\label{e:S.4.15}
\end{eqnarray}
Note that $v_n^+\in B_{H^+}(\theta,r)$ and $v_n^-\in\partial B_{H^-}(\theta,s)$ for all $n\in\mathbb{N}$.
Then (\ref{e:S.4.7}) and (\ref{e:S.4.15}) imply
\begin{eqnarray*}
-\frac{a_1's^2}{4}&\le& (t_nP^\bot\nabla\mathcal{L}(z_n+v_n)+(1-t_n)P^\bot\nabla\mathcal{L}(z_n'+v_n), v_n^-)_H
\le -a_1's^2+\frac{a_1'}{2}\|v_n^+\|^2
\end{eqnarray*}
and so
\begin{eqnarray}\label{e:S.4.16}
\frac{s^2}{\|v_n^+\|^2}\le\frac{2}{3},\quad\forall n\in\mathbb{N}.
\end{eqnarray}
With the same methods, (\ref{e:S.4.6}) and (\ref{e:S.4.14}) lead to
\begin{eqnarray*}
\frac{a_1's^2}{4}\ge (t_nP^\bot\nabla\mathcal{L}(z_n+v_n)+(1-t_n)P^\bot\nabla\mathcal{L}(z_n'+v_n), v_n^+)_H
\ge a_1'\|v_n^+\|^2-\frac{a_1'}{2}s^2
\end{eqnarray*}
and so
$$
\frac{s^2}{\|v_n^+\|^2}\ge\frac{4}{3},\quad\forall n\in\mathbb{N}.
$$
This contradicts  (\ref{e:S.4.16}). Hence
$$
\inf\{\|tP^\bot\nabla\mathcal{L}(z_1+u)+ (1-t)P^\bot\nabla\mathcal{L}(z_2+u)\|\,|\,
 (t,z_1,z_2,u)\in\Lambda_2 \}>0.
$$
This and (\ref{e:S.4.8}) yield the desired conclusions.
\end{proof}

Since (D4) is equivalent to (D4*), it was proved in \cite[p. 2966--2967]{Lu2} that
$\nabla\mathcal{L}$ is of class $(S)_+$ under the conditions (S), (F), (C) and (D) in
\cite{Lu2}. In particular, this is also true under the assumptions of Theorem~\ref{th:S.1.1}
(without requirement $H^0=\{\theta\}$), of course the conditions of Theorem~\ref{th:S.1.1}
guarantee the same claim.

In the following we always assume that
 $r>0,s>0$ and $\epsilon>0$ are as in Lemma~\ref{lem:S.4.1}.

\begin{lemma}\label{lem:S.4.2}
For each $z\in B_{H^0}(\theta, \epsilon)$, the map
$$
f_z:\overline{\mathcal{Q}_{r,s}}\ni u\mapsto P^\bot\nabla\mathcal{L}(z+u)\in H^+\oplus H^-
$$
is of class $(S)_+$. Moreover, for any two points $z_0, z_1\in B_{H^0}(\theta, \epsilon)$ the map
$\mathscr{H}:[0,1]\times \overline{\mathcal{Q}_{r,s}}\to H^+\oplus H^-$ given by
  $$
  \mathscr{H}(t,u)=(1-t)P^\bot\nabla\mathcal{L}(z_0+u)+ tP^\bot\nabla\mathcal{L}(z_1+u)
  $$
  is a homotopy of class $(S)_+$ (cf. \cite[Definition~4.40]{MoMoPa}).
\end{lemma}

\begin{proof}
By \cite[Proposition~4.41]{MoMoPa} we only need to prove
 the first claim.
Let $\{u_j\}\subset \overline{\mathcal{Q}_{r,s}}$
weakly converge to $u\in H^+\oplus H^-$.
Assume that they satisfy
$$
\overline{\lim}(P^\bot\nabla\mathcal{L}(z+u_j), u_j-u)_H\le 0.
$$
It suffices to prove $u_j\to u$ in $H^+\oplus H^-$.
Note that $u_j\rightharpoonup u$ in $H$ because $\overline{\mathcal{Q}_{r,s}}\subset H^+\oplus H^-$.
So is $z+u_j\rightharpoonup z+u$ in $H$.
Moreover, $u_j-u\in H^+\oplus H^-$ implies
\begin{eqnarray*}
(P^\bot\nabla\mathcal{L}(z+u_j), u_j-u)_H
&=&(\nabla\mathcal{L}(z+u_j), u_j-u)_H\\
&=&(\nabla\mathcal{L}(z+u_j), (z+u_j)-(z+u))_H.
\end{eqnarray*}
It follows that $\overline{\lim}(\nabla\mathcal{L}(z+u_j), (z+u_j)-(z+u))_H\le 0$.
But $\nabla\mathcal{L}$ is of class $(S)_+$ near $\theta\in H$, we have
$z+u_j\to z+u$ and so $u_j\to u$.
\end{proof}

Let $\deg$ denote the Browder-Skrypnik degree for  demicontinuous $(S)_+$-maps
(\cite{Bro0, Bro}, \cite{Skr1, Skr2, Skr3}), see \cite[\S4.3]{MoMoPa} for a nice exposition.
By Lemma~\ref{lem:S.4.1} $\deg(f_0, \mathcal{Q}_{r,s}, \theta)$ is well-defined and
using the Poincar\'e-Hopf theorem (cf. \cite[Theorem~1.2]{CiDe}) we have
 \begin{equation}\label{e:S.4.17}
\deg(f_0, \mathcal{Q}_{r,s}, \theta)=\sum^\infty_{q=0}(-1)^q{\rm rank}C_q(f_0,\theta;G).
 \end{equation}
Note that $\mathcal{L}|_{\mathcal{Q}_{r,s}}$ satisfies the conditions of
Theorem~\ref{th:S.1.1}. It follows that $C_q(f_0,\theta;G)=\delta_{\mu q}G$, where $\mu=\dim H^-$.
Hence (\ref{e:S.4.17}) becomes
 \begin{equation}\label{e:S.4.18}
\deg(f_0, \mathcal{Q}_{r,s}, \theta)=(-1)^{\mu}.
 \end{equation}
For each $z\in B_{H^0}(\theta,\epsilon)$,
we derive from Lemma~\ref{lem:S.4.1} that
\begin{eqnarray*}
&&\inf\{\|f_z(u)\|\,|\, u\in\partial\overline{\mathcal{Q}_{r,s}}\}>0\quad\hbox{and}\\
&&\inf\{\|tf_z(u)+(1-t)f_0(u)\|\,|\,t\in [0,1],\; u\in\partial\overline{\mathcal{Q}_{r,s}}\}>0.
\end{eqnarray*}
The former implies  that $\deg(f_z, \mathcal{Q}_{r,s}, \theta)$ is well-defined,
 the latter and  Lemma~\ref{lem:S.4.2} lead to
 \begin{equation}\label{e:S.4.19}
 \deg(f_z, \mathcal{Q}_{r,s}, \theta)=\deg(f_0, \mathcal{Q}_{r,s}, \theta)=(-1)^{\mu}
 \end{equation}
 by (\ref{e:S.4.18}). So there exists a point $u_z\in  \mathcal{Q}_{r,s}$ such that
  \begin{equation}\label{e:S.4.20}
 P^\bot\nabla\mathcal{L}(z+ u_z)=f_z(u_z)=\theta.
 \end{equation}

Now let us give the main result in this subsection.

\begin{theorem}[Parameterized Implicit Function Theorem]\label{th:S.4.3}
Under the assumptions of Theorem~\ref{th:S.1.2},
suppose further that  $\mathcal{G}_1,\cdots,  \mathcal{G}_n\in C^1(V,\mathbb{R})$
satisfy
\begin{description}
\item[(i)] $\mathcal{G}'_j(\theta)=\theta$, $j=1,\cdots,n$;
\item[(ii)] for each $j=1,\cdots,n$, the gradient $\nabla\mathcal{G}_j$ has  G\^ateaux derivative $\mathcal{G}''_j(u)\in \mathscr{L}_s(H)$ at any $u\in V$, and $\mathcal{G}''_j: V\to \mathscr{L}_s(H)$ are continuous at $\theta$.
\end{description}
Then by shrinking $\epsilon>0$ (if necessary) we have $\delta>0$
and a unique continuous map
 \begin{equation}\label{e:S.4.21}
\psi:[-\delta, \delta]^n\times B_H(\theta,\epsilon)\cap H^0\to \mathcal{Q}_{r,s}\subset
(H^0)^\bot
 \end{equation}
 such that for all $(\vec{\lambda}, z)\in [-\delta, \delta]^n\times B_{H}(\theta,\epsilon)\cap H^0$
 with $\vec{\lambda}=(\lambda_1,\cdots,\lambda_n)$,
$\psi(\vec{\lambda},\theta)=\theta$ and
\begin{equation}\label{e:S.4.22}
 P^\bot\nabla\mathcal{L}(z+ \psi(\vec{\lambda}, z))+
 \sum^n_{j=1}\lambda_j P^\bot\nabla\mathcal{G}_j(z+ \psi(\vec{\lambda}, z))=\theta,
 \end{equation}
 where  $P^\bot$ is  as in (\ref{e:S.4.20}). This $\psi$ also satisfies
 \begin{equation}\label{e:S.4.22.1}
 \|\psi(\vec{\lambda}, z_1)-\psi(\vec{\lambda}, z_2)\|\le 3\|z_1-z_2\|,\quad
 \forall (\vec{\lambda},z)\in [-\delta, \delta]^n\times B_H(\theta,\epsilon)\cap H^0.
 \end{equation}
 Moreover, if $G$ is a compact Lie group acting on $H$ orthogonally,
$V$, $\mathcal{L}$ and all $\mathcal{G}_j$ are $G$-invariant (and hence $H^0$, $(H^0)^\bot$
are $G$-invariant subspaces, and $\nabla\mathcal{L}$, $\nabla\mathcal{G}_j$ are $G$-equivariant),
then  $\psi$ is equivariant on $z$, i.e.,
$\psi(\vec{\lambda}, g\cdot z)=g\cdot\psi(\vec{\lambda},z)$ for $(\vec{\lambda},z)\in
[-\delta, \delta]^n\times B_H(\theta,\epsilon)\cap H^0$ and $g\in G$.
\end{theorem}

\begin{proof} {\bf Step 1}.\quad {\it
 There exist $\rho_1, \delta\in (0,1)$  such that $B_H(\theta, 2\rho_1)\subset V$
 and that if sequences $\vec{\lambda}_k=(\lambda_{k,1},\cdots,\lambda_{k,n})\in [-\delta,\delta]^n$
 converge to  $\vec{\lambda}_0=(\lambda_{0,1},\cdots,\lambda_{0,n})\in [-\delta,\delta]^n$,
 $u_k\in B_H(\theta, 2\rho_1)$ weakly converge to
  $u_0\in B_H(\theta, 2\rho_1)$,  and they also satisfy
\begin{eqnarray}\label{e:S.4.23}
\overline{\lim}(\nabla\mathcal{L}(u_k)+\sum^n_{j=1}\lambda_j\nabla\mathcal{G}_j(u_k)  , u_k-u_0)_H\le 0,
\end{eqnarray}
then $u_k\to u_0$. In particular, for each $\vec{\lambda}\in[-\delta, \delta]^n$, the map
$$
{B}_H(\theta, 2\rho_1)\times [0,1]\to H^+\oplus H^-,\;
(t, u)\mapsto P^\bot\nabla\mathcal{L}(u)+ \sum^n_{j=1}t\lambda_j P^\bot\nabla\mathcal{G}_j(u)
  $$
  is a homotopy of class $(S)_+$ (cf. \cite[Definition~4.40]{MoMoPa}).}

In fact, by \cite[(5.8)]{Lu2} we had found $\rho_1>0$ and $C_0'>0$ such that $B_H(\theta, 2\rho_1)\subset V$
and
\begin{eqnarray}\label{e:S.4.24}
(\nabla\mathcal{L}(u), u-u')_H&\ge&\frac{C_0'}{2}\|u-u'\|^2+ (\nabla\mathcal{L}(u'), u-u')_H\nonumber\\
&+&(Q(\theta)(u-u'), u-u')_H
 \end{eqnarray}
for any $u,u'\in B_H(\theta, 2\rho_1)$. Similarly, for each fixed $j\in\{1,\cdots,n\}$, we have $\tau=\tau(u,u')\in (0,1)$ such that
\begin{eqnarray*}
&&(\nabla\mathcal{G}_j(u), u-u')_H=(\nabla\mathcal{G}_j(u)-\nabla\mathcal{G}_j(u'), u-u')_H+ (\nabla\mathcal{G}_j(u'), u-u')_H\\
&=&(\mathcal{G}''_j(\tau u+ (1-\tau)u')(u-u'), u-u')_H+ (\nabla\mathcal{G}_j(u'), u-u')_H\\
&=&([\mathcal{G}''_j(\tau u+ (1-\tau)u')-\mathcal{G}''_j(\theta)](u-u'), u-u')_H+ (\nabla\mathcal{G}_j(u'), u-u')_H\\
&&+ (\mathcal{G}''_j(\theta)(u-u'), u-u')_H,\quad\forall u, u'\in B_H(\theta, 2\rho_1).
 \end{eqnarray*}
Since $V\ni v\mapsto \mathcal{G}''_j(v)\in\mathscr{L}_s(H)$ is continuous at $\theta$,
we may shrink $\rho_1>0$ so that
\begin{eqnarray}\label{e:S.4.24.1}
\|\mathcal{G}''_j(v)-\mathcal{G}''_j(\theta)\|\le\frac{C_0'}{8n},
\quad\forall v\in B_H(\theta, 2\rho_1),\;j=1,\cdots,n.
\end{eqnarray}
It follows that for all $u, u'\in B_H(\theta, 2\rho_1)$ and $j=1,\cdots,n$,
\begin{eqnarray*}
|(\nabla\mathcal{G}_j(u), u-u')_H|&\le& \frac{C_0'}{8n}\|u-u'\|^2 + |(\nabla\mathcal{G}_j(u'), u-u')_H|\\
&&+ |(\mathcal{G}_j''(\theta)(u-u'), u-u')_H|.
 \end{eqnarray*}
Take $\delta\in (0, 1)$ so that
$$
\delta\sum^n_{j=1}\|\mathcal{G}_j''(\theta)\|<\frac{C_0'}{8}.
$$
These and (\ref{e:S.4.24}) imply that for all $\vec{\lambda}=(\lambda_1,\cdots,\lambda_n)\in [-\delta, \delta]^n$,
\begin{eqnarray*}
&&(\nabla\mathcal{L}(u), u-u')_H+ \sum^n_{j=1}\lambda_j(\nabla\mathcal{G}_j(u), u-u')_H\nonumber\\
&\ge&\frac{C_0'}{4}\|u-u'\|^2+ (\nabla\mathcal{L}(u'), u-u')_H+ (Q(\theta)(u-u'), u-u')_H\nonumber\\
&-& \sum^n_{j=1}|(\nabla\mathcal{G}_j(u'), u-u')_H|.
 \end{eqnarray*}
Replacing $u, u'$ and $\lambda_j$ by $u_k, u_0$ and $\lambda_{k,j}$ in the inequality, we derive from
(\ref{e:S.4.23}) that $u_k\to u_0$
because  (D3)  implies that
$(\nabla\mathcal{L}(u_0), u_k-u_0)_H\to 0$, $(Q(\theta)(u_k-u_0), u_k-u_0)_H\to 0$ and
$(\nabla\mathcal{G}_j(u_0), u_k-u_0)_H\to 0$.

  {\it Note}: The above proof shows that
  the family $\{\mathcal{L}_{\vec{\lambda}}:=\mathcal{L}+ \sum^n_{j=1}\lambda_j\mathcal{G}_j\,|\, \vec{\lambda}\in [-\delta,\delta]^n\}$
  satisfies the (PS) condition on $\bar{B}_H(\theta,\varepsilon)$ for any $\varepsilon<2\rho_1$, that is,
  if sequences $\vec{\lambda}_k\in [-\delta,\delta]^n$ converge to
  $\vec{\lambda}_0\in [-\delta,\delta]^n$,
 and $u_k\in \bar{B}_H(\theta,\varepsilon)$ satisfies
 $\nabla\mathcal{L}_{\vec{\lambda}_k}(u_k)\to\theta$ and $\sup_k|\mathcal{L}_{\vec{\lambda}_k}(u_k)|<\infty$,
 then $\{u_k\}_{k\ge 1}$ has a converging subsequence
 $u_{k_i}\to u_0\in \bar{B}_H(\theta,\varepsilon)$ with
 $\nabla\mathcal{L}_{\vec{\lambda}_0}(u_0)=\theta$.

\noindent{\bf Step 2}.\quad Let $r>0,s>0$ and $\epsilon>0$ be as in Lemma~\ref{lem:S.4.1}.
By shrinking them, we can assume that
$\bar B_{H^0}(\theta,\epsilon)\times\overline{\mathcal{Q}_{r,s}}
\subset{B}_H(\theta, 2\rho_1)$ and
 \begin{eqnarray}\label{e:S.4.24.2}
 \sup\{\|\nabla\mathcal{L}_{\vec{\lambda}}(z,u)\|\,|\,
 (\vec{\lambda},z,u)\in [-1, 1]^n\times \bar{B}_{H^0}(\theta,\epsilon)\oplus\overline{\mathcal{Q}_{r,s}}\}<\infty
 \end{eqnarray}
because $\nabla\mathcal{L}$ and $\nabla\mathcal{G}_1,\cdots,\nabla\mathcal{G}_n$ are all locally bounded.  Then by Lemma~\ref{lem:S.4.1} we may shrink  $\delta\in (0, 1)$ so that
$$
\inf\|tP^\bot(\nabla\mathcal{L}+ \sum^n_{j=1}\lambda_j\nabla\mathcal{G}_j)(z_1+u)+ (1-t)P^\bot(\nabla\mathcal{L}+ \sum^n_{j=1}\lambda_j\nabla\mathcal{G}_j)(z_2+u)\|>0,
 $$
where the infimum is taken for all $(t,z_1,z_2,u)\in[0,1]\times \bar B_{H^0}(\theta,\epsilon)\times \bar B_{H^0}(\theta,\epsilon)\times
\partial\overline{\mathcal{Q}_{r,s}}$ and $(\lambda_1,\cdots,\lambda_n)\in [-\delta,\delta]^n$.
 This implies that
for each $(\vec{\lambda}, z)\in [-\delta,\delta]^n\times B_{H^0}(\theta, \epsilon)$, the map
$$
f_{\vec{\lambda}, z}:\overline{\mathcal{Q}_{r,s}}\ni u\mapsto P^\bot\nabla\mathcal{L}(z+u)+
\sum^n_{j=1}\lambda_j P^\bot\nabla\mathcal{G}_j(z+u)\in H^+\oplus H^-
$$
has a well-defined Browder-Skrypnik degree ${\rm deg}(f_{\vec{\lambda}, z}, \mathcal{Q}_{r,s},\theta)$ and
\begin{eqnarray}\label{e:S.4.25}
{\rm deg}(f_{\vec{\lambda}, z}, \mathcal{Q}_{r,s},\theta)={\rm deg}(f_{\vec{0}, 0}, \mathcal{Q}_{r,s},\theta)
={\rm deg}(f_{0}, \mathcal{Q}_{r,s},\theta)=(-1)^{\mu},
 \end{eqnarray}
where $f_0$ is as in (\ref{e:S.4.18}). Hence for each $(\vec{\lambda}, z)\in [-\delta,\delta]^n\times B_{H^0}(\theta, \epsilon)$
there exists a point $u_{\vec{\lambda},z}\in  \mathcal{Q}_{r,s}$ such that
  \begin{equation}\label{e:S.4.26}
 P^\bot\nabla\mathcal{L}(z+ u_{\vec{\lambda},z})+ \sum^n_{j=1}\lambda_j P^\bot\nabla\mathcal{G}_j(z+ u_{\vec{\lambda},z})=f_{\vec{\lambda},z}(u_{\lambda,z})=\theta.
 \end{equation}
By shrinking the above  $\epsilon>0, r>0$ and $s>0$
 (if necessary) so that $\omega$ and $a_0, a_1$ in Lemma~\ref{lem:S.2.2} can satisfy
 \begin{equation}\label{e:S.4.27}
\omega(z+u)<\min\{a_0,a_1\}/2,\quad\forall
(z,u)\in\bar B_{H^0}(\theta, \epsilon)\times \overline{\mathcal{Q}_{r,s}}.
\end{equation}

 \noindent{\bf Step 3}.\quad {\it
If $\delta>0$ is sufficiently small, then  $u_{\vec{\lambda},z}$ is a unique zero point of $f_{\vec{\lambda},z}$ in $\mathcal{Q}_{r,s}$.}

 In fact, suppose that there exists another different $u_{\vec{\lambda},z}'\in \mathcal{Q}_{r,s}$
  satisfying (\ref{e:S.4.26}). We decompose
 $u_{\vec{\lambda},z}-u_{\vec{\lambda},z}'=(u_{\vec{\lambda},z}-u_{\vec{\lambda},z}')^++ (u_{\vec{\lambda},z}-
 u_{\vec{\lambda},z}')^-$, and prove the conclusion in three cases:\\
$\bullet$ $\|(u_{\vec{\lambda},z}-u_{\vec{\lambda},z}')^+\|>\|(u_{\vec{\lambda},z}-u_{\vec{\lambda},z}')^-\|$,\\
$\bullet$ $\|(u_{\vec{\lambda},z}-u_{\vec{\lambda},z}')^+\|=\|(u_{\vec{\lambda},z}-u_{\vec{\lambda},z}')^-\|$,\\
$\bullet$ $\|(u_{\vec{\lambda},z}-u_{\vec{\lambda},z}')^+\|<\|(u_{\vec{\lambda},z}-u_{\vec{\lambda},z}')^-\|$.

Let us write $\mathcal{L}_{\vec{\lambda}}=\mathcal{L}+\sum^n_{j=1}\lambda_j\mathcal{G}_j$ for conveniences.
Then (\ref{e:S.4.26}) implies
\begin{eqnarray}\label{e:S.4.28}
0&=&(P^\bot\nabla\mathcal{L}_{\vec{\lambda}}(z+u_{\vec{\lambda},z})-P^\bot\nabla\mathcal{L}_{\vec{\lambda}}(z+
u_{\vec{\lambda},z}'), (u_{\vec{\lambda},z}-u_{\vec{\lambda},z}')^+)_H\nonumber\\
&=&(P^\bot\nabla\mathcal{L}(z+u_{\vec{\lambda},z})-P^\bot\nabla\mathcal{L}(z+u_{\vec{\lambda},z}'), (u_{\vec{\lambda},z}-u_{\vec{\lambda},z}')^+)_H
\nonumber\\
&&+\sum^n_{j=1}\lambda_j(P^\bot\nabla\mathcal{G}_j(z+u_{\vec{\lambda},z})-P^\bot\nabla\mathcal{G}_j(z+u_{\vec{\lambda},z}'), (u_{\vec{\lambda},z}-u_{\vec{\lambda},z}')^+)_H.
\end{eqnarray}
For simplicity we write $u_{\vec{\lambda},z}=u_z$ and $u_{\vec{\lambda},z}'=u_z'$.

For the first two cases,  we may use the mean value theorem to get
 $\tau\in (0,1)$ such that
 \begin{eqnarray}\label{e:S.4.29}
&&(P^\bot\nabla\mathcal{L}(z+u_z)-P^\bot\nabla\mathcal{L}(z+u_z'), (u_z-u_z')^+)_H\nonumber\\
&=&(\nabla\mathcal{L}(z+u_z)-\nabla\mathcal{L}(z+u_z'), (u_z-u_z')^+)_H\nonumber\\
&=&(B(z+ \tau u_z+(1-\tau )u_z')(u_z-u_z'), (u_z-u_z')^+)_H\nonumber\\
&=&(B(z+ \tau u_z+(1-\tau )u_z')(u_z-u_z')^+, (u_z-u_z')^+)_H\nonumber\\
&+&(B(z+ \tau u_z+(1-\tau)u_z')(u_z-u_z')^-, (u_z-u_z')^+)_H\nonumber\\
&\ge& a_1\|(u_z-u_z')^+\|^2-\omega(z+ \tau u_z+(1-t)u_z')\|(u_z-u_z')^-\|\cdot\|(u_z-u_z')^+\|\nonumber\\
&\ge& a_1\|(u_z-u_z')^+\|^2-\frac{a_1}{4}[\|(u_z-u_z')^-\|^2+\|(u_z-u_z')^+\|^2]\nonumber\\
&\ge& a_1\|(u_z-u_z')^+\|^2-\frac{a_1}{2}\|(u_z-u_z')^+\|^2\nonumber\\
&=&\frac{a_1}{2}\|(u_z-u_z')^+\|^2,
\end{eqnarray}
where the first inequality comes from  Lemma~\ref{lem:S.2.2}(i)-(ii),
the second one is derived from (\ref{e:S.4.27}) and the inequality
$2|ab|\le |a|^2+|b|^2$, and
the third one is because
$\|(u_z-u_z')^-\|\le \|(u_z-u_z')^+\|$.
It follows from  (\ref{e:S.4.28})--(\ref{e:S.4.29}) that
\begin{eqnarray}\label{e:S.4.30}
0&=&(P^\bot\nabla\mathcal{L}_{\vec{\lambda}}(z+u_{\vec{\lambda},z})-P^\bot\nabla\mathcal{L}_{\vec{\lambda}}(z+
u_{\vec{\lambda},z}'), (u_{\vec{\lambda},z}-u_{\vec{\lambda},z}')^+)_H\nonumber\\
&\ge&\frac{a_1}{2}\|(u_{\vec{\lambda},z}-u_{\vec{\lambda},z}')^+\|^2\nonumber\\
&&+ \sum^n_{j=1}\lambda_j(\mathcal{G}''_j(z+ \tau u_{\vec{\lambda},z}+(1-\tau)u_{\vec{\lambda},z}')(u_{\vec{\lambda},z}-u_{\vec{\lambda},z}'),
(u_{\vec{\lambda},z}-u_{\vec{\lambda},z}')^+)_H.
\end{eqnarray}
By (\ref{e:S.4.24.1}) we have a constant $M>0$ such that
\begin{eqnarray}\label{e:S.4.30.1}
\sup\{\|\mathcal{G}''_j(z+w)\|\,|\,(z,w)\in \bar B_{H^0}(\theta,\epsilon)\times\overline{\mathcal{Q}_{r,s}},\,j=1,\cdots,n\}<M.
\end{eqnarray}
 Hence
\begin{eqnarray}\label{e:S.4.31}
&&\sum^n_{j=1}|\lambda_j(\mathcal{G}''_j(z+ \tau u_{\vec{\lambda},z}+(1-\tau)u_{\vec{\lambda},z}')(u_{\vec{\lambda},z}-u_{\vec{\lambda},z}'), (u_{\vec{\lambda},z}-u_{\vec{\lambda},z}')^+)_H|\nonumber\\
&\le& n\delta M\|(u_{\vec{\lambda},z}-u_{\vec{\lambda},z}')\|\cdot\|(u_{\vec{\lambda},z}-u_{\vec{\lambda},z}')^+\|\nonumber\\
&\le&n\delta M[\|(u_{\vec{\lambda},z}-u_{\vec{\lambda},z}')^+\|^2+
\|(u_{\vec{\lambda},z}-u_{\vec{\lambda},z}')^-\|\cdot\|(u_{\vec{\lambda},z}-u_{\vec{\lambda},z}')^+\|]\nonumber\\
&\le &2n\delta M\|(u_{\vec{\lambda},z}-u_{\vec{\lambda},z}')^+\|^2.
\end{eqnarray}
Let us shrink $\delta>0$ in Step 2 so that $\delta<\frac{a_1}{8nM}$. Then this and
(\ref{e:S.4.30}) lead to
\begin{eqnarray*}
0=(P^\bot\nabla\mathcal{L}_{\vec{\lambda}}(z+u_{\vec{\lambda},z})-P^\bot\nabla\mathcal{L}_{\vec{\lambda}}(z+u_{\vec{\lambda},z}'), (u_{\vec{\lambda},z}-u_{\vec{\lambda},z}')^+)_H
\ge\frac{a_1}{4}\|(u_{\vec{\lambda},z}-u_{\vec{\lambda},z}')^+\|^2.
\end{eqnarray*}
This contradicts  $(u_{\vec{\lambda},z}-u_{\vec{\lambda},z}')^+\ne\theta$.

Similarly, for the third case,
as in (\ref{e:S.4.29}) we may use Lemma~\ref{lem:S.2.2}(ii)-(iii) to obtain
\begin{eqnarray*}
0&=&(P^\bot\nabla\mathcal{L}(z+u_z)-P^\bot\nabla\mathcal{L}(z+u_z'), (u_z-u_z')^-)_H\\
&=&(\nabla\mathcal{L}(z+u_z)-\nabla\mathcal{L}(z+u_z'), (u_z-u_z')^-)_H\\
&=&(B(z+ tu_z+(1-t)u_z')(u_z-u_z'), (u_z-u_z')^-)_H\\
&=&(B(z+ tu_z+(1-t)u_z')(u_z-u_z')^-, (u_z-u_z')^-)_H\\
&+&(B(z+ tu_z+(1-t)u_z')(u_z-u_z')^+, (u_z-u_z')^-)_H\\
&\le& -a_0\|(u_z-u_z')^-\|^2+\omega(z+ tu_z+(1-t)u_z')\|(u_z-u_z')^-\|\cdot\|(u_z-u_z')^+\|\\
&\le& -a_0\|(u_z-u_z')^-\|^2+\frac{a_0}{4}[\|(u_z-u_z')^-\|^2+\|(u_z-u_z')^+\|^2]\\
&\le& -a_0\|(u_z-u_z')^-\|^2+\frac{a_0}{2}\|(u_z-u_z')^-\|^2\\
&=&-\frac{a_0}{2}\|(u_z-u_z')^-\|^2,
\end{eqnarray*}
and hence
\begin{eqnarray}\label{e:S.4.32}
0&=&(P^\bot\nabla\mathcal{L}_{\vec{\lambda}}(z+u_{\vec{\lambda},z})-P^\bot\nabla\mathcal{L}_{\vec{\lambda}}(z+
u_{\vec{\lambda},z}'), (u_{\vec{\lambda},z}-u_{\vec{\lambda},z}')^-)_H\nonumber\\
&\le&-\frac{a_0}{2}\|(u_{\vec{\lambda},z}-u_{\vec{\lambda},z}')^+\|^2\nonumber\\
&&+ \sum^n_{j=1}\lambda_j(\mathcal{G}''_j(z+ \tau u_{\vec{\lambda},z}+(1-\tau)u_{\vec{\lambda},z}')(u_{\vec{\lambda},z}-u_{\vec{\lambda},z}'),
(u_{\vec{\lambda},z}-u_{\vec{\lambda},z}')^-)_H.
\end{eqnarray}
As in (\ref{e:S.4.31}) we may deduce
\begin{eqnarray*}
&&\sum^n_{j=1}|\lambda_j(\mathcal{G}''_j(z+ \tau u_{\vec{\lambda},z}+(1-\tau)u_{\vec{\lambda},z}')(u_{\vec{\lambda},z}-u_{\vec{\lambda},z}'), (u_{\vec{\lambda},z}-u_{\vec{\lambda},z}')^-)_H|\nonumber\\
&\le& n\delta M\|(u_{\vec{\lambda},z}-u_{\vec{\lambda},z}')\|\cdot\|(u_{\vec{\lambda},z}-u_{\vec{\lambda},z}')^-\|\nonumber\\
&\le&n\delta M[\|(u_{\vec{\lambda},z}-u_{\vec{\lambda},z}')^-\|^2+
\|(u_{\vec{\lambda},z}-u_{\vec{\lambda},z}')^-\|\cdot\|(u_{\vec{\lambda},z}-u_{\vec{\lambda},z}')^+\|]\nonumber\\
&\le &2n\delta M\|(u_{\vec{\lambda},z}-u_{\vec{\lambda},z}')^-\|^2.
\end{eqnarray*}
So if the above $\delta>0$ is also shrunk  so that $\delta<\frac{a_0}{8nM}$,  we may derive from this and (\ref{e:S.4.32}) that
\begin{eqnarray*}
0=(P^\bot\nabla\mathcal{L}_{\vec{\lambda}}(z+u_{\vec{\lambda},z})-P^\bot\nabla\mathcal{L}_{\vec{\lambda}}(z+u_{\vec{\lambda},z}'), (u_{\vec{\lambda},z}-u_{\vec{\lambda},z}')^-)_H
\le-\frac{a_0}{4}\|(u_{\vec{\lambda},z}-u_{\vec{\lambda},z}')^-\|^2,
\end{eqnarray*}
which also leads to a contradiction.

As a consequence, we have a well-defined  map
\begin{eqnarray}\label{e:S.4.33}
\psi: [-\delta,\delta]^n\times B_{H^0}(\theta,\epsilon)\to \mathcal{Q}_{r,s},\;(\lambda, z)\mapsto u_{\vec{\lambda},z}.
\end{eqnarray}

\noindent{\bf Step 4}.\quad{\it  $\psi$ is continuous.}

Let sequence $\{\vec{\lambda}_k\}_{k\ge 1}\in [-\delta,\delta]^n$ and $\{z_k\}_{k\ge 1}\subset B_{H^0}(\theta,\epsilon)$ converge to
$\vec{\lambda}_0\in [-\delta,\delta]^n$ and
$z_0\in B_{H^0}(\theta,\epsilon)$, respectively.
We want to prove that $\psi(\vec{\lambda}_k, z_k)\to\psi(\vec{\lambda}_0, z_0)$.
Since $\{\psi(\vec{\lambda}_k, z_k)\}_{k\ge 1}$ is contained in $\mathcal{Q}_{r,s}$,
we can suppose $\psi(\vec{\lambda}_k, z_k)\rightharpoonup u_0\in\overline{\mathcal{Q}_{r,s}}$ in $H$.
Noting  $\psi(\vec{\lambda}_k, z_k)-u_0\in H^+\oplus H^-$, by (\ref{e:S.4.26}) we have
\begin{eqnarray*}
&&\bigr(\nabla\mathcal{L}_{\vec{\lambda}_k}(z_k+\psi(\vec{\lambda}_k, z_k)), \psi(\vec{\lambda}_k,
z_k)-u_0)\bigr)\\
&=&\bigr(P^\bot\nabla\mathcal{L}_{\vec{\lambda}_k}(z_k+\psi(\vec{\lambda}_k, z_k)), \psi(\vec{\lambda}_k,
z_k)-u_0)\bigr)=0.
\end{eqnarray*}
It follows from this and (\ref{e:S.4.24.2}) that
\begin{eqnarray*}
&&|\bigr(\nabla\mathcal{L}_{\vec{\lambda}_k}(z_k+\psi(\vec{\lambda}_k, z_k)), (z_k+\psi(\vec{\lambda}_k,
z_k))-(z_0+u_0)\bigr)\|\\
&=&|\bigl(\nabla\mathcal{L}_{\vec{\lambda}_k}(z_n+\psi(\vec{\lambda}_k, z_k)), z_k-z_0\bigr)|\\
&\le&\|\nabla\mathcal{L}_{\vec{\lambda}_k}(z_k+\psi(\vec{\lambda}_k, z_k))\|\cdot\|z_k-z_0\|\to 0.
\end{eqnarray*}
As in the proof of Step~1, we may derive from this
that $z_k+\psi(\vec{\lambda}_k,
z_k)\to z_0+u_0$ and so $\psi(\vec{\lambda}_k, z_k)\to u_0$.

Moreover, (\ref{e:S.4.26}) implies
 $P^\bot\nabla\mathcal{L}_{\vec{\lambda}_k}(z_k+\psi(\vec{\lambda}_k, z_k))=0$, $k=1,2,\cdots$.
 The $C^1$-smoothness of  $\mathcal{L}$ and all $\mathcal{G}_j$ leads to
 $P^\bot\nabla\mathcal{L}_{\vec{\lambda}_0}(z_0+u_0)=0$.
  By Step~3 we arrive at
  $\psi(\vec{\lambda}_0, z_0)=u_0$ and hence $\psi$ is continuous at $(\vec{\lambda}_0, z_0)$.

\noindent{\bf Step 5}.\quad For any $(\vec{\lambda}, z_i)\in [-\delta, \delta]^n\times B_{H}(\theta,\epsilon)\cap H^0$, $i=1,2$,
by the definition of $\psi$, we have
\begin{eqnarray*}
 P^\bot\nabla\mathcal{L}(z_i+ \psi(\vec{\lambda}, z_i))+
 \sum^n_{j=1}\lambda_j P^\bot\nabla\mathcal{G}_j(z_i+ \psi(\vec{\lambda}, z_i))=\theta,\quad i=1,2,
 \end{eqnarray*}
and hence for $\Xi=z_1-z_2+ \psi(\vec{\lambda}, z_1)-\psi(\vec{\lambda}, z_2)$ we derive
\begin{eqnarray}\label{e:S.4.34}
0&=&(P^\bot\nabla\mathcal{L}(z_1+ \psi(\vec{\lambda}, z_1))-P^\bot\nabla\mathcal{L}(z_2+ \psi(\vec{\lambda}, z_2)),\Xi^+)_H
\nonumber\\
&+& \sum^n_{j=1}\lambda_j (P^\bot\nabla\mathcal{G}_j(z_1+ \psi(\vec{\lambda}, z_1))-
P^\bot\nabla\mathcal{G}_j(z_2+ \psi(\vec{\lambda}, z_2)),\Xi^+)_H.
 \end{eqnarray}
As in the proof of (\ref{e:S.4.29}) we obtain  $\tau\in (0,1)$ such that
 \begin{eqnarray}\label{e:S.4.35}
&&(P^\bot\nabla\mathcal{L}(z_1+ \psi(\vec{\lambda}, z_1))-P^\bot\nabla\mathcal{L}(z_2+ \psi(\vec{\lambda}, z_2)),\Xi^+)_H\nonumber\\
&=&(B(\tau z_1+ \tau\psi(\vec{\lambda}, z_1)+ (1-\tau)z_2+ (1-\tau)\psi(\vec{\lambda}, z_2))\Xi^+,\Xi^+)_H\nonumber\\
&+&(B(\tau z_1+ \tau\psi(\vec{\lambda}, z_1)+ (1-\tau)z_2+ (1-\tau)\psi(\vec{\lambda}, z_2))(\Xi^0+\Xi^-),\Xi^+)_H\nonumber\\
&\ge& a_1\|\Xi^+\|^2-\frac{a_1}{4}[\|\Xi^-+\Xi^0\|^2+\|\Xi^+\|^2]\nonumber\\
&=&\frac{3a_1}{4}\|\Xi^+\|^2-\frac{a_1}{4}\|\Xi^0\|^2-\frac{a_1}{4}\|\Xi^-\|^2.
 \end{eqnarray}
 Let us further shrink $\delta>0$ in Step~3 so that $\delta<\frac{\min\{a_0,a_1\}}{16nM}$.
As in (\ref{e:S.4.31}) we may deduce
\begin{eqnarray}\label{e:S.4.36}
&&\sum^n_{j=1}\lambda_j (P^\bot\nabla\mathcal{G}_j(z_1+ \psi(\vec{\lambda}, z_1))-
P^\bot\nabla\mathcal{G}_j(z_2+ \psi(\vec{\lambda}, z_2)),\Xi^+)_H\\
&\le&\sum^n_{j=1}|\lambda_j(\mathcal{G}''_j(\tau z_1+(1-\tau)z_2+ \tau\psi(\vec{\lambda}, z_1)+(1-\tau)\psi(\vec{\lambda}, z_2))\Xi, \Xi^+)_H|\nonumber\\
&\le& n\delta M\|\Xi\|\cdot\|\Xi^+\|\le 2n\delta M[\|\Xi\|^2+\|\Xi^+\|^2]\nonumber\\
&\le&\frac{a_1}{8}[\|\Xi^-\|^2+\|\Xi^0\|^2+2\|\Xi^+\|^2].
\end{eqnarray}
This and (\ref{e:S.4.34})--(\ref{e:S.4.35}) lead to
$$
0\ge \frac{3a_1}{4}\|\Xi^+\|^2-\frac{a_1}{4}\|\Xi^0\|^2-
\frac{a_1}{4}\|\Xi^-\|^2-\frac{a_1}{8}[\|\Xi^-\|^2+\|\Xi^0\|^2+2\|\Xi^+\|^2]
$$
and so
\begin{eqnarray}\label{e:S.4.37}
0\ge 4\|\Xi^+\|^2-3\|\Xi^0\|^2-3\|\Xi^-\|^2.
\end{eqnarray}

Similarly, replacing $\Xi^+$ by $\Xi^-$ in (\ref{e:S.4.35}) and (\ref{e:S.4.36}) we derive
 \begin{eqnarray*}
&&(P^\bot\nabla\mathcal{L}(z_1+ \psi(\vec{\lambda}, z_1))-P^\bot\nabla\mathcal{L}(z_2+ \psi(\vec{\lambda}, z_2)),\Xi^-)_H\\
&\le& -\frac{3a_0}{4}\|\Xi^-\|^2+\frac{a_0}{4}\|\Xi^0\|^2+ \frac{a_0}{4}\|\Xi^+\|^2,\\
 &&\sum^n_{j=1}\lambda_j (P^\bot\nabla\mathcal{G}_j(z_1+ \psi(\vec{\lambda}, z_1))-
P^\bot\nabla\mathcal{G}_j(z_2+ \psi(\vec{\lambda}, z_2)),\Xi^-)_H\nonumber\\
&\le&\frac{a_0}{8}[\|\Xi^+\|^2+\|\Xi^0\|^2+2\|\Xi^-\|^2].
\end{eqnarray*}
As above these two inequalities and the equality
\begin{eqnarray*}
0&=&(P^\bot\nabla\mathcal{L}(z_1+ \psi(\vec{\lambda}, z_1))-P^\bot\nabla\mathcal{L}(z_2+ \psi(\vec{\lambda}, z_2)),\Xi^-)_H
\nonumber\\
&+& \sum^n_{j=1}\lambda_j (P^\bot\nabla\mathcal{G}_j(z_1+ \psi(\vec{\lambda}, z_1))-
P^\bot\nabla\mathcal{G}_j(z_2+ \psi(\vec{\lambda}, z_2)),\Xi^-)_H
 \end{eqnarray*}
yield: $0\ge 4\|\Xi^-\|^2-3\|\Xi^0\|^2-3\|\Xi^+\|^2$. Combing with (\ref{e:S.4.37}) we obtain
\begin{eqnarray*}
\|\Xi^++\Xi^-\|^2=\|\Xi^+\|^2+\|\Xi^-\|^2\le 6\|\Xi^0\|^2.
\end{eqnarray*}
Note that $\Xi^9=z_1-z_2$ and $\Xi^++\Xi^-=\psi(\vec{\lambda}, z_1)-\psi(\vec{\lambda}, z_2)$.
The desired claim is proved.

{\bf Step 6}.\quad The uniqueness of $\psi$ implies that it is equivariant on $z$.
\end{proof}

As a by-product we have also the following result though it is not used in this paper.

\begin{theorem}[Inverse Function Theorem]\label{th:S.4.4}
If the assumptions of Theorem~\ref{th:S.1.1} hold with $X=H$,
then $\nabla\mathcal{L}$ is locally
 invertible at $\theta$.
\end{theorem}

\begin{proof}
We can assume that $\nabla\mathcal{L}$ is of class $(S)_+$ in
 $\overline{\mathcal{Q}_{r,s}}$.
Since $H^0=\{\theta\}$ and $\nabla\mathcal{L}=f_0$,  we have
\begin{eqnarray}\label{e:S.4.38}
\deg(\nabla\mathcal{L}, \mathcal{Q}_{r,s}, \theta)=(-1)^{\mu}
\end{eqnarray}
by (\ref{e:S.4.18}).
Moreover, $\varrho:=\inf\{\|\nabla\mathcal{L}(u)\|\,|\, u\in\partial\overline{\mathcal{Q}_{r,s}}\}>0$
 by Lemma~\ref{lem:S.4.1}. For any given $v\in B_H(\theta,\varrho)$, let us define
 $$
 \mathscr{H}:[0,1]\times\overline{\mathcal{Q}_{r,s}}\to H,\;(t,u)\mapsto \nabla\mathcal{L}(u)-tv.
 $$
Then $\|\mathscr{H}(t,u)\|=\|\nabla\mathcal{L}(u)-tv\|\ge \|\nabla\mathcal{L}(u)\|-\|v\|\ge\varrho-\|v\|>0$
for all $(t,u)\in [0,1]\times\partial\overline{\mathcal{Q}_{r,s}}$.
Assume that sequences $t_n\to t$ in $[0,1]$, $\{u_n\}_{n\ge 1}\subset\mathcal{Q}_{r,s}$ converges weakly to $u$ in $H$,
and they satisfy $\lim\sup_{n\to\infty}(\mathscr{H}(t_n,u_n), u_n-u)_H\le 0$.
Then
$$
(\nabla\mathcal{L}(u_n),u_n-u)_H=(\mathscr{H}(t_n,u_n),u_n-u)_H+ t_n(v, u_n-u)_H
$$
leads to $\lim\sup_{n\to\infty}(\nabla\mathcal{L}(u_n),u_n-u)_H\le 0$. It follows that $u_n\to u$ in $H$
because $\nabla\mathcal{L}$ is of class $(S)_+$ in $\overline{\mathcal{Q}_{r,s}}$.
Hence  $\mathscr{H}$ is a homotopy of class $(S)_+$, and thus (\ref{e:S.4.38}) gives
$$
\deg(\nabla\mathcal{L}-v, \mathcal{Q}_{r,s}, \theta)
=\deg(\nabla\mathcal{L}, \mathcal{Q}_{r,s}, \theta)=(-1)^{\mu},
$$
which implies $\nabla\mathcal{L}(\xi_v)=v$ for some $\xi_v\in \mathcal{Q}_{r,s}$.
By Step 3 in the proof of Theorem~\ref{th:S.4.4} (taking $\vec{\lambda}=\vec{0}$)
it is easily seen that the equation $\nabla\mathcal{L}(u)=v$ has a unique solution in $\overline{\mathcal{Q}_{r,s}}$,
and in particular, $\xi_v$ is unique. Then
we get a map $B_H(\theta,\varrho)\to \mathcal{Q}_{r,s},\;v\mapsto\xi_v$ to satisfy
$\nabla\mathcal{L}(\xi_v)=v$ for all $v\in B_H(\theta,\varrho)$.
We claim that this map is continuous. Arguing by contradiction,
assume that there exists a sequence $v_n\to v$ in $B_H (\theta,\varrho)$, such that
$\xi_{v_n}\rightharpoonup \xi^\ast$ in $H$ and $\|\xi_{v_n}-\xi_v\|\ge\epsilon_0$
for some $\epsilon_0>0$ and all $n=1,2,\cdots$. Note that
\begin{eqnarray*}
(\nabla\mathcal{L}(\xi_{v_n}), \xi_{v_n}-\xi^\ast)_H&=&(v_n,\xi_{v_n}-\xi^\ast)_H
=(v_n-v,\xi_{v_n}-\xi^\ast)_H+(v,\xi_{v_n}-\xi^\ast)_H\to 0.
\end{eqnarray*}
We derive that $\xi_{v_n}\to\xi^\ast$ in $H$, and so
$\nabla\mathcal{L}(\xi_{v_n})=v_n$ can lead to $\nabla\mathcal{L}(\xi^\ast)=v$.
The uniqueness of solutions implies $\xi^\ast=\xi_v$.
This prove the claim. Hence $\nabla\mathcal{L}$
is a homeomorphism from an open neighborhood $\{\xi_v\,|\,v\in B_H(\theta,\varrho)\}$
of $\theta$ in $\mathcal{Q}_{r,s}$ onto $B_H(\theta,\varrho)$.
\end{proof}

Theorem~\ref{th:S.4.4} cannot be derived from the invariance of domain theorem
(5.4.1) of Berger \cite{Ber} or \cite[Theorem~2.5]{Fe}. Recently, Ekeland proved
an weaker inverse function theorem, \cite[Theorem~2]{Ek}.
Since we cannot insure that $B(u)$ has a right-inverse
$L(u)$ which is uniformly bounded in a neighborhood of $\theta$,
\cite[Theorem~2]{Ek} it cannot lead to Theorem~\ref{th:S.4.4} either.

\subsection{Parameterized splitting and shifting theorems}\label{sec:S.5}

 To shorten the proof of the main theorem, we shall write parts of it into two propositions.

\begin{proposition}\label{prop:S.5.1}
Under the assumptions of Theorem~\ref{th:S.4.3},
 for each $(\vec{\lambda},z)\in [-\delta,\delta]^n\times B_{H^0}(\theta,\epsilon)$,
let $\psi_{\vec{\lambda}}(z)=\psi(\vec{\lambda}, z)$ be given by (\ref{e:S.4.21}).
Then it satisfies
 \begin{eqnarray*}
\mathcal{L}_\lambda(z+\psi_\lambda(z))=\min\{\mathcal{L}_\lambda(z+ u)\,|\, u\in B_H(\theta, r)\cap H^+\}
\end{eqnarray*}
 if $H^-=\{\theta\}$, and
\begin{eqnarray*}
&&\mathcal{L}_{\vec{\lambda}}(z+\psi_{\vec{\lambda}}(z))=\min\{\mathcal{L}_{\vec{\lambda}}(z+ u+ P^-\psi_{\vec{\lambda}}(z))\,|\, u\in B_H(\theta, r)\cap H^+\},\\
&&\mathcal{L}_{\vec{\lambda}}(z+\psi_{\vec{\lambda}}(z))=\max\{\mathcal{L}_{\vec{\lambda}}(z+ P^+\psi_{\vec{\lambda}}(z)+ v)\,|\, v\in B_H(\theta, s)\cap H^-\}
\end{eqnarray*}
if $H^-\ne\{\theta\}$.
\end{proposition}

\begin{proof}\quad
{\bf Case} $H^-=\{\theta\}$.\quad
We have $\mathcal{Q}_{r,s}=B_H(\theta, r)\cap H^+$,
and (\ref{e:S.4.22}) becomes
$$
P^+\nabla\mathcal{L}_{\vec{\lambda}}(z+\psi_{\vec{\lambda}}(z))=0,\quad\forall z\in B_{H^0}(\theta,\epsilon).
$$
By this we may use integral mean value theorem to deduce  that for each $u\in \mathcal{Q}_{r,s}$,
\begin{eqnarray*}
&&\mathcal{L}_{\vec{\lambda}}(z+ u)-\mathcal{L}_{\vec{\lambda}}(z+\psi_{\vec{\lambda}}(z))\\
&=&\int^1_0(\nabla\mathcal{L}_{\vec{\lambda}}(z+\psi_{\vec{\lambda}}(z)+ \tau(u-\psi_{\vec{\lambda}}(z))), u-\psi_{\vec{\lambda}}(z))_Hd\tau\\
&=&\int^1_0(P^+\nabla\mathcal{L}_{\vec{\lambda}}(z+\psi_{\vec{\lambda}}(z)+ \tau(u-\psi_{\vec{\lambda}}(z))), u-\psi_{\vec{\lambda}}(z))_Hd\tau\\
&=&\int^1_0(P^+\nabla\mathcal{L}_{\vec{\lambda}}(z+\psi_{\vec{\lambda}}(z)+ \tau(u-\psi_{\vec{\lambda}}(z)))-P^+\nabla\mathcal{L}_{\vec{\lambda}}(z+\psi_{\vec{\lambda}}(z)), u-\psi_{\vec{\lambda}}(z))_Hd\tau\\
&=&\int^1_0(\nabla\mathcal{L}_{\vec{\lambda}}(z+\psi_{\vec{\lambda}}(z)+ \tau(u-\psi_{\vec{\lambda}}(z)))-\nabla\mathcal{L}_{\vec{\lambda}}(z+\psi_{\vec{\lambda}}(z)), u-\psi_{\vec{\lambda}}(z))_Hd\tau\\
&=&\int^1_0(\nabla\mathcal{L}(z+\psi_{\vec{\lambda}}(z)+ \tau(u-\psi_{\vec{\lambda}}(z)))-\nabla\mathcal{L}(z+\psi_{\vec{\lambda}}(z)), u-\psi_{\vec{\lambda}}(z))_Hd\tau\\
&&+\sum^n_{j=1}\lambda_j\int^1_0(\nabla\mathcal{G}_j(z+\psi_{\vec{\lambda}}(z)+ \tau(u-\psi_{\vec{\lambda}}(z)))-\nabla\mathcal{G}_j(z+\psi_{\vec{\lambda}}(z)), u-\psi_{\vec{\lambda}}(z))_Hd\tau\\
&=&\int^1_0\tau d\tau\int^1_0\big(B(z+\psi_{\vec{\lambda}}(z)+ \rho\tau(u-\psi_{\vec{\lambda}}(z)))(u-\psi_{\vec{\lambda}}(z)), u-\psi_{\vec{\lambda}}(z)\bigr)_Hd\rho\\
&&+\sum^n_{j=1}\lambda_j\int^1_0\tau d\tau\int^1_0\big(\mathcal{G}''_j(z+\psi_{\vec{\lambda}}(z)+ \rho\tau(u-\psi_{\vec{\lambda}}(z)))(u-\psi_{\vec{\lambda}}(z)), u-\psi_{\vec{\lambda}}(z)\bigr)_Hd\rho\\
&\ge&\frac{a_1}{2}\|u-\psi_{\vec{\lambda}}(z)\|^2\\
&&+\sum^n_{j=1}\lambda_j\int^1_0\tau d\tau\int^1_0\big(\mathcal{G}''_j(z+\psi_{\vec{\lambda}}(z)+ \rho\tau(u-\psi_{\vec{\lambda}}(z)))(u-\psi_{\vec{\lambda}}(z)), u-\psi_{\vec{\lambda}}(z)\bigr)_Hd\rho.
\end{eqnarray*}
Here the final inequality comes from Lemma~\ref{lem:S.2.2}(i).
For the final sum, as in (\ref{e:S.4.31}) we have
 \begin{eqnarray*}
&&\left|\sum^n_{j=1}\lambda_j\int^1_0\tau d\tau\int^1_0\big(\mathcal{G}''_j(z+\psi_{\vec{\lambda}}(z)+ \rho\tau(u-\psi_{\vec{\lambda}}(z)))(u-\psi_{\vec{\lambda}}(z)), u-\psi_{\vec{\lambda}}(z)\bigr)_Hd\rho\right|\\
&\le& 2n\delta M\|u-\psi_{\vec{\lambda}}(z))\|^2\le \frac{a_1}{4}\|u-\psi_{\vec{\lambda}}(z))\|^2.
\end{eqnarray*}
These lead to
\begin{eqnarray}\label{e:S.5.1}
\mathcal{L}_{\vec{\lambda}}(z+ u)-\mathcal{L}_{\vec{\lambda}}(z+\psi_{\vec{\lambda}}(z))
\ge\frac{a_1}{4}\|u-\psi_{\vec{\lambda}}(z)\|^2,
\end{eqnarray}
which implies the desired conclusion.

 {\bf Case} $H^-\ne\{\theta\}$.\quad
  For each $u\in B_H(\theta, r)\cap H^+$ we have $u+P^-\psi_{\vec{\lambda}}(z)\in \mathcal{Q}_{r,s}$.
  As above we can use (\ref{e:S.4.22}) to derive
  \begin{eqnarray}\label{e:S.5.2}
\mathcal{L}_{\vec{\lambda}}(z+ u+ P^-\psi_{\vec{\lambda}}(z))-\mathcal{L}_{\vec{\lambda}}(z+ \psi_{\vec{\lambda}}(z))
\ge\frac{a_1}{4}\|u-P^+\psi_{\vec{\lambda}}(z)\|^2,
\end{eqnarray}
and therefore  the second equality.

 Finally, for
each $v\in B_H(\theta, r)\cap H^-$ we have $P^+\psi_{\vec{\lambda}}(z)+v\in \mathcal{Q}_{r,s}$.
As above, using (\ref{e:S.4.22}) and  Lemma~\ref{lem:S.2.2}(ii)-(iii) we may deduce
\begin{eqnarray*}
&&\mathcal{L}_{\vec{\lambda}}(z+ P^+\psi_{\vec{\lambda}}(z)+v)-\mathcal{L}_{\vec{\lambda}}(z+\psi_{\vec{\lambda}}(z))\\
&=&\int^1_0(\nabla\mathcal{L}_{\vec{\lambda}}(z+ \psi_{\vec{\lambda}}(z)+ t(u-P^+\psi_{\vec{\lambda}}(z))), v-P^-\psi_{\vec{\lambda}}(z))_Hdt\\
&=&\int^1_0(\nabla\mathcal{L}_{\vec{\lambda}}(z+ \psi_{\vec{\lambda}}(z)+ t(v-P^-\psi_{\vec{\lambda}}(z)))- \nabla\mathcal{L}_{\vec{\lambda}}(z+ \psi_{\vec{\lambda}}(z)), v-P^-\psi_{\vec{\lambda}}(z))_Hdt\\
&=&\int^1_0t\int^1_0(B(z+ \psi_{\vec{\lambda}}(z)+ \tau t(v-P^-\psi_{\vec{\lambda}}(z)))(v-P^-\psi_{\vec{\lambda}}(z)), v-P^-\psi_{\vec{\lambda}}(z))_Hdtd\tau\\
&+&\sum^n_{j=1}\lambda_j\int^1_0t dt\int^1_0\big(\mathcal{G}''_j(z+ \psi_{\vec{\lambda}}(z)+ \tau t(v-P^-\psi_{\vec{\lambda}}(z)))(v-P^-\psi_{\vec{\lambda}}(z)), v-P^-\psi_{\vec{\lambda}}(z)\bigr)_Hd\tau\\
&\le&-\frac{a_0}{2}\|v-P^-\psi_{\vec{\lambda}}(z)\|^2\\
&+&\sum^n_{j=1}\lambda_j\int^1_0t dt\int^1_0\big(\mathcal{G}''_j(z+ \psi_{\vec{\lambda}}(z)+ \tau t(v-P^-\psi_{\vec{\lambda}}(z)))(v-P^-\psi_{\vec{\lambda}}(z)), v-P^-\psi_{\vec{\lambda}}(z)\bigr)_Hd\tau\\
&\le&-\frac{a_0}{4}\|v-P^-\psi_{\vec{\lambda}}(z)\|^2,
\end{eqnarray*}
 and hence  the third equality.
\end{proof}

\begin{proposition}\label{prop:S.5.2}
Under the assumptions of Theorem~\ref{th:S.4.3},
 for each $(\vec{\lambda},z)\in [-\delta,\delta]^n\times B_{H^0}(\theta,\epsilon)$,
let $\psi_{\vec{\lambda}}(z)=\psi(\vec{\lambda}, z)$ be given by (\ref{e:S.4.21}).
 Then the functional
$\mathcal{L}_{\vec{\lambda}}^\circ: B_H(\theta, \epsilon)\cap H^0\to \mathbb{R}$
   given by
 \begin{equation}\label{e:S.5.3}
 \mathcal{L}^\circ_{\vec{\lambda}}(z):=\mathcal{L}_{\vec{\lambda}}(z+\psi_{\vec{\lambda}}(z))
 =\mathcal{L}(z+\psi(\vec{\lambda}, z))+ \sum^n_{j=1}\lambda_j\mathcal{G}_j(z+\psi(\vec{\lambda}, z))
  \end{equation}
  is of class $C^1$, and its differential is given by
\begin{eqnarray}\label{e:S.5.4}
D\mathcal{L}_{\vec{\lambda}}^\circ(z)h&=&D\mathcal{L}_{\vec{\lambda}}(z+\psi_{\vec{\lambda}}(z))h\nonumber\\
&=&D\mathcal{L}(z+\psi(\vec{\lambda},z))h+
\sum^n_{j=1}\lambda_j D\mathcal{G}_j(z+\psi(\vec{\lambda},z))h,\quad\forall h\in H^0.
\end{eqnarray}
(Clearly, this implies that $[-\delta,\delta]^n\ni\vec{\lambda}\mapsto\mathcal{L}^\circ_{\vec{\lambda}}\in C^1(\bar{B}_H(\theta, \epsilon)\cap H^0)$
is continuous by shrinking $\epsilon>0$ since $\dim H^0<\infty$).
\end{proposition}

 \begin{proof} {\bf Case} $H^-\ne\{\theta\}$.\quad
  For fixed $z\in B_H(\theta, \epsilon)\cap H^0$, $h\in H^0$, and
$t\in\mathbb{R}$ with sufficiently small $|t|$, the last two equalities in
Proposition~\ref{prop:S.5.1} imply
 \begin{eqnarray}\label{e:S.5.5}
&&\mathcal{L}_{\vec{\lambda}}(z+th+ P^+\psi_{\vec{\lambda}}(z+th)+P^-\varphi(z))-\mathcal{L}_{\vec{\lambda}}(z+P^+\psi_{\vec{\lambda}}(z+th)+ P^-\psi_{\vec{\lambda}}(z))\nonumber\\
&\le& \mathcal{L}_{\vec{\lambda}}(z+th+\psi_{\vec{\lambda}}(z+th))-\mathcal{L}_{\vec{\lambda}}(z+\psi_{\vec{\lambda}}(z))\nonumber\\
&\le&\mathcal{L}_{\vec{\lambda}}(z+th+ P^+\psi_{\vec{\lambda}}(z)+ P^-\psi_{\vec{\lambda}}(z+th))-
\mathcal{L}_{\vec{\lambda}}(z+ P^+\psi_{\vec{\lambda}}(z)+ P^-\psi_{\vec{\lambda}}(z+th)).\nonumber\\
\end{eqnarray}
Since $\mathcal{L}_{\vec{\lambda}}$ is $C^1$  and  $\psi_{\vec{\lambda}}$ is continuous we deduce,
\begin{eqnarray}\label{e:S.5.6}
&&\lim_{t\to 0}\frac{\mathcal{L}_{\vec{\lambda}}(z+th+P^+\psi_{\vec{\lambda}}(z+th)+P^-\psi_{\vec{\lambda}}(z))-
\mathcal{L}_{\vec{\lambda}}(z+
P^+\psi_{\vec{\lambda}}(z+th)+P^-\psi_{\vec{\lambda}}(z))}{t}\nonumber\\
&=& \lim_{t\to 0}\int^1_0D\mathcal{L}_{\vec{\lambda}}(z+sth+P^+\psi_{\vec{\lambda}}(z+th)+P^-\psi_{\vec{\lambda}}(z))h ds\nonumber\\
&=&D\mathcal{L}_{\vec{\lambda}}(z+\psi_{\vec{\lambda}}(z))h.
\end{eqnarray}
Here the last equality follows from the Lebesgue's Dominated Convergence Theorem
since
$$
\{D\mathcal{L}_{\vec{\lambda}}(z+sth+P^+\psi_{\vec{\lambda}}(z+th)+P^-\psi_{\vec{\lambda}}(z))h\,|\, 0\le s\le 1,\;|t|\le 1\}
$$
is bounded by the compactness of
 $\{z+ sth+P^+\psi_{\vec{\lambda}}(z+th)+P^-\psi_{\vec{\lambda}}(z)\,|\, 0\le s\le 1,\; |t|\le 1\}$.

Similarly, we have
\begin{eqnarray}\label{e:S.5.7}
&&\lim_{t\to 0}\frac{\mathcal{L}_{\vec{\lambda}}(z+th+ P^+\psi_{\vec{\lambda}}(z)+P^-\psi_{\vec{\lambda}}(z+th))-\mathcal{L}_{\vec{\lambda}}(z+P^+\psi_{\vec{\lambda}}(z)+
P^-\psi_{\vec{\lambda}}(z+th))}{t}\nonumber\\
&=&D\mathcal{L}_{\vec{\lambda}}(z+\psi_{\vec{\lambda}}(z))h.
\end{eqnarray}
Then it follows from (\ref{e:S.5.5})-(\ref{e:S.5.7}) that
 \begin{eqnarray*}
\lim_{t\to 0}\frac{\mathcal{L}_{\vec{\lambda}}(z+th+\psi_{\vec{\lambda}}(z+th))-\mathcal{L}_{\vec{\lambda}}(z+\psi_{\vec{\lambda}}(z))}{t}=
D\mathcal{L}_{\vec{\lambda}}(z+\psi_{\vec{\lambda}}(z))h.
\end{eqnarray*}
 That is, $\mathcal{L}_{\vec{\lambda}}^\circ$ is G\^ateaux differentiable  and $D\mathcal{L}_{\vec{\lambda}}^\circ(z)=D\mathcal{L}_{\vec{\lambda}}(z+\psi_{\vec{\lambda}}(z))|_{H^0}$.
 The latter implies that  $\mathcal{L}_{\vec{\lambda}}^\circ$ is of class $C^1$
 because both $D\mathcal{L}_{\vec{\lambda}}$ and $\psi_{\vec{\lambda}}$ are continuous.

 {\bf Case} $H^-=\{\theta\}$.\quad
For fixed $z\in B_H(\theta, \epsilon)\cap H^0$ and $h\in H^0$, and
$t\in\mathbb{R}$ with sufficiently small $|t|$, the first equality in Proposition~\ref{prop:S.5.1} implies
\begin{eqnarray}\label{e:S.5.8}
&&\mathcal{L}_{\vec{\lambda}}(z+th+\psi_{\vec{\lambda}}(z+th))-\mathcal{L}_{\vec{\lambda}}(z+\psi_{\vec{\lambda}}(z+th))\nonumber\\
&\le& \mathcal{L}_{\vec{\lambda}}(z+th+\psi_{\vec{\lambda}}(z+th))-\mathcal{L}_{\vec{\lambda}}(z+\psi_{\vec{\lambda}}(z))\nonumber\\
&\le&\mathcal{L}_{\vec{\lambda}}(z+th+\psi_{\vec{\lambda}}(z))-\mathcal{L}_{\vec{\lambda}}(z+\psi_{\vec{\lambda}}(z)).
\end{eqnarray}
By  the continuity of  $\nabla\mathcal{L}_{\vec{\lambda}}$ and
$\psi_{\vec{\lambda}}$ we have
\begin{eqnarray}\label{e:S.5.9}
&&\lim_{t\to 0}\frac{\mathcal{L}_{\vec{\lambda}}(z+th+\psi_{\vec{\lambda}}(z+th))-
\mathcal{L}_{\vec{\lambda}}(z+\psi_{\vec{\lambda}}(z+th))}{t}\nonumber\\
&=& \lim_{t\to 0}\int^1_0D\mathcal{L}_{\vec{\lambda}}(z+sth+\psi_{\vec{\lambda}}(z+th))h ds\nonumber\\
&=&D\mathcal{L}_{\vec{\lambda}}(z+\psi_{\vec{\lambda}}(z))h.
\end{eqnarray}
(As above this follows from the Lebesgue's Dominated Convergence Theorem
because
$\{z+ sth+\psi_{\vec{\lambda}}(z+th)\,|\, 0\le s\le 1,\; 0\le t\le 1\}$
is compact
and thus $\{\nabla\mathcal{L}_{\vec{\lambda}}(z+sth+\psi_{\vec{\lambda}}(z+th))h\,|\, 0\le s\le 1,\;|t|\le 1\}$
is bounded).
Similarly, we may prove
\begin{eqnarray}\label{e:S.5.10}
\lim_{t\to 0}\frac{\mathcal{L}_{\vec{\lambda}}(z+th+\psi_{\vec{\lambda}}(z))-
\mathcal{L}_{\vec{\lambda}}(z+\psi_{\vec{\lambda}}(z))}{t}=
D\mathcal{L}_{\vec{\lambda}}(z+\psi_{\vec{\lambda}}(z))h,
\end{eqnarray}
and thus
\begin{eqnarray*}
\lim_{t\to 0}\frac{\mathcal{L}_{\vec{\lambda}}(z+th+\psi_{\vec{\lambda}}(z+th))-
\mathcal{L}_{\vec{\lambda}}(z+\psi_{\vec{\lambda}}(z))}{t}=
D\mathcal{L}_{\vec{\lambda}}(z+\psi_{\vec{\lambda}}(z))h
\end{eqnarray*}
by (\ref{e:S.5.8})--(\ref{e:S.5.10}).
The desired claim follows immediately.
 \end{proof}

\begin{theorem}[Parameterized Splitting Theorem]\label{th:S.5.3}
Under the assumptions of Theorem~\ref{th:S.4.3},
by shrinking $\delta>0$, $\epsilon>0$ and $r>0, s>0$, we obtain
an open neighborhood $W$ of $\theta$ in $H$ and an origin-preserving
homeomorphism
\begin{eqnarray}\label{e:S.5.11}
&&[-\delta, \delta]^n\times B_{H^0}(\theta,\epsilon)\times
\left(B_{H^+}(\theta, r) +
B_{H^-}(\theta, s)\right)\to [-\delta, \delta]^n\times W,\nonumber\\
&&(\vec{\lambda}, z, u^++u^-)\mapsto (\vec{\lambda},\Phi_{\vec{\lambda}}(z, u^++u^-))
\end{eqnarray}
such that
\begin{equation}\label{e:S.5.12}
\mathcal{ L}_{\vec{\lambda}}\circ\Phi_{\vec{\lambda}}(z, u^++ u^-)=\|u^+\|^2-\|u^-\|^2+ \mathcal{
L}_{\vec{\lambda}}(z+ \psi(\vec{\lambda}, z))
\end{equation}
for all $(\vec{\lambda}, z, u^+ + u^-)\in [-\delta,\delta]^n\times B_{H^0}(\theta,\epsilon)\times
\left(B_{H^+}(\theta, r) +
B_{H^-}(\theta, s)\right)$, where $\psi$ is given by (\ref{e:S.4.21}).
 The functional functional
$\mathcal{L}_{\vec{\lambda}}^\circ: B_H(\theta, \epsilon)\cap H^0\to \mathbb{R}$
   given by (\ref{e:S.5.3}) is of class $C^1$, and its differential is given by (\ref{e:S.5.4}).
Moreover, (i) if $\mathcal{L}$ and $\mathcal{G}_j$, $j=1,\cdots,n$, are of class $C^{2-0}$,
then so is $\mathcal{L}_{\vec{\lambda}}^\circ$ for each $\vec{\lambda}\in [-\delta, \delta]^n$;
(ii) if a compact Lie group $G$  acts on $H$ orthogonally, and
$V$, $\mathcal{L}$ and $\mathcal{G}$ are $G$-invariant (and hence $H^0$, $(H^0)^\bot$
are $G$-invariant subspaces), then for each $\vec{\lambda}\in [-\delta, \delta]^n$, $\psi(\vec{\lambda}, \cdot)$
 and $\Phi_{\vec{\lambda}}(\cdot,\cdot)$  are $G$-equivariant, and $\mathcal{L}^\circ_{\vec{\lambda}}(z)=\mathcal{
L}_{\vec{\lambda}}(z+ \psi(\vec{\lambda}, z))$ is $G$-invariant.
\end{theorem}

Sometimes, for example, the corresponding conditions with \cite[Theorem~1.1]{Lu1} or \cite[Remark~3.2]{Lu2} are also satisfied,
  we can prove that $\psi(\vec{\lambda}, \cdot)$ is of class $C^1$ and that
  $\mathcal{L}^\circ_{\vec{\lambda}}$ is of class $C^2$; moreover,
   \begin{eqnarray}
   &&D\mathcal{L}^\circ_{\vec{\lambda}}(z)u=(\nabla\mathcal{L}_{\vec{\lambda}}(z+ \psi(\vec{\lambda}, z)), u)_H,\label{e:S.5.12.1}\\
   &&d^2\mathcal{L}^\circ_{\vec{\lambda}}(z)(u,v)=(\mathcal{L}''_{\vec{\lambda}}(z+ \psi(\vec{\lambda}, z))
   (u+ D_z\psi(\vec{\lambda}, z)u),v)_H\label{e:S.5.12.2}
  \end{eqnarray}
  for all $z\in B_H(\theta, \epsilon)\cap H^0$ and $u,v\in H^0$. In particular, since $\psi(\vec{\lambda}, \theta)=\theta$ and $D_z\psi(\vec{\lambda}, \theta)=\theta$,
     \begin{equation}\label{e:S.5.12.3}
  d^2\mathcal{L}^\circ_{\vec{\lambda}}(\theta)(z_1,z_2)=(\mathcal{L}''_{\vec{\lambda}}(\theta)z_1,z_2)_H=
  -\sum^n_{j=1}\lambda_j(\mathcal{G}''_j(\theta)z_1,z_2)_H,\quad\forall z_1, z_2\in H^0.
  \end{equation}
\noindent{\bf Claim}.\quad{\it In this situation, if $\theta\in H$ is a nondegenerate critical point of $\mathcal{L}_{\vec{\lambda}}$ then $\theta\in H^0$ such a critical point of
$\mathcal{L}^\circ_{\vec{\lambda}}$ too.}

In fact, suppose for some $z_1\in H^0$ that $d^2\mathcal{L}^\circ_{\vec{\lambda}}(\theta)(z_1,z_2)=0\;
\forall z_2\in H^0$. (\ref{e:S.5.12.3}) implies
$(P^0\mathcal{L}''_{\vec{\lambda}}(\theta)z_1,u)_H=(P^0\mathcal{L}''_{\vec{\lambda}}(\theta)z_1,P^0u)_H
=0$ for all $u\in H$. Hence $P^0\mathcal{L}''_{\vec{\lambda}}(\theta)z_1=\theta$.
Moreover, since $(I-P^0)\nabla\mathcal{L}_{\vec{\lambda}}(z+ \psi(\vec{\lambda}, z))=\theta$
for all $z\in B_H(\theta, \epsilon)\cap H^0$. Differentiating this equality with respect to $z$ we get
$(I-P^0)\mathcal{L}''_{\vec{\lambda}}(z+ \psi(\vec{\lambda}, z))(u+ D_z\psi(\vec{\lambda}, z)u)=\theta$
for all $u\in H^0$. In particular, $(I-P^0)\mathcal{L}''_{\vec{\lambda}}(\theta)z=\theta$
for all $z\in H^0$. It follows that $\mathcal{L}''_{\vec{\lambda}}(\theta)z_1=\theta$
and hence $z_1=\theta$.\\

\noindent{\it Proof of Theorem~\ref{th:S.5.3}}. \quad
Let $N=H^0$, and for each $\vec{\lambda}\in [-\delta,\delta]^n$
 we define a  map $F_{\vec{\lambda}}:B_{N}(\theta, \epsilon)\times \mathcal{Q}_{r,s}\to\R$ by
 \begin{equation}\label{e:S.5.13}
 F_{\vec{\lambda}}(z, u)=\mathcal{L}_{\vec{\lambda}}(z+ \psi(\vec{\lambda},z)+ u)-\mathcal{L}_{\vec{\lambda}}(z+ \psi(\vec{\lambda},z)).
\end{equation}
 Then $D_2F_{\vec{\lambda}}(z,u)v=(P^\bot\nabla{\mathcal L}_\lambda(z+ \psi(\vec{\lambda},z)+ u), v)_H$
for  $z\in \bar B_{N}(\theta, \epsilon)$,  $u\in \mathcal{Q}_{r,s}$ and $v\in N^\bot$.
 Moreover it holds that
\begin{eqnarray}\label{e:S.5.14}
F_{\vec{\lambda}}(z, \theta)=0\quad\hbox{and}\quad D_2F_{\vec{\lambda}}(z,
\theta)(v)=0\quad\;\forall v\in N^\bot.
\end{eqnarray}
Since  $B_{N}(\theta,\epsilon)\oplus\mathcal{Q}_{r,s}$ has the closure
contained in the neighborhood $U$ in Lemma~\ref{lem:S.2.2}, and
 $\psi(\vec{\lambda},\theta)=\theta$, we can shrink $\nu>0$,
$\epsilon>0$, $r>0$ and $s>0$ so small that
 \begin{equation}\label{e:S.5.15}
 z+ \psi(\vec{\lambda},z)+
u^++ u^-\in U
\end{equation}
for all $\vec{\lambda}\in [-\delta,\delta]^n$, $z\in\bar B_{N}(\theta,\epsilon)$ and $u^+ + u^-\in \overline{\mathcal{Q}_{r,s}}$.

Let us verify that each $F_{\vec{\lambda}}$ satisfies conditions (ii)-(iv) in \cite[Theorem~A.1]{Lu2}.

\noindent{\bf Step 1}. For $\vec{\lambda}\in [-\delta,\delta]^n$, $z\in\bar B_{N}(\theta,\epsilon)$, $u^+\in
\bar B_{H^+}(\theta, r)$ and $u^-_1, u^-_2\in\bar
B_{H^-}(\theta,\epsilon)$,  we have
\begin{eqnarray}\label{e:S.5.16}
&&[D_2F_{\vec{\lambda}}(z, u^+ + u^-_2)-D_2F_{\vec{\lambda}}(z, u^++ u^-_1)](u^-_2-u^-_1)\nonumber\\
&=&(\nabla\mathcal{L}_{\vec{\lambda}}(z+ \psi_{\vec{\lambda}}(z)+ u^++u^-_2), u^-_2-u^-_1)_H\nonumber\\
&&- (\nabla\mathcal{L}_{\vec{\lambda}}(z+ \psi_{\vec{\lambda}}(z)+ u^++u^-_1), u^-_2-u^-_1)_H.
\end{eqnarray}
Since $\nabla\mathcal{L}_{\vec{\lambda}}$ is G\^ateaux differentiable so is the
function
$$
u\mapsto (\nabla\mathcal{L}_{\vec{\lambda}}(z+ \psi_{\vec{\lambda}}(z)+ u^++u), u^-_2-u^-_1)_H.
$$
By the mean value theorem we have $t\in (0, 1)$ such that
\begin{eqnarray}\label{e:S.5.17}
&&(\nabla\mathcal{L}_{\vec{\lambda}}(z+ \psi_{\vec{\lambda}}(z)+ u^++u^-_2), u^-_2-u^-_1)_H
- (\nabla\mathcal{L}_{\vec{\lambda}}(z+ \psi_{\vec{\lambda}}(z)+ u^++u^-_1), u^-_2-u^-_1)_H\nonumber\\
&=&\left(B(z+ \psi_{\vec{\lambda}}(z)+ u^++ u^-_1+ t(u^-_2-u^-_1))(u^-_2-u^-_1),
u^-_2-u^-_1\right)_H\nonumber\\
&&+\sum^n_{j=1}\lambda_j\left(\mathcal{G}''_j(z+ \psi_{\vec{\lambda}}(z)+ u^++ u^-_1+ t(u^-_2-u^-_1))(u^-_2-u^-_1),
u^-_2-u^-_1\right)_H\nonumber\\
&\le& \sum^n_{j=1}\lambda_j\left(\mathcal{G}''_j(z+ \psi_{\vec{\lambda}}(z)+ u^++ u^-_1+ t(u^-_2-u^-_1))(u^-_2-u^-_1),
u^-_2-u^-_1\right)_H\nonumber\\
&&-a_0\|u^-_2-u^-_1\|^2
\end{eqnarray}
because of Lemma~\ref{lem:S.2.2}(iii).   Recall that we have assumed $\delta<\frac{\min\{a_0,a_1\}}{8nM}$
in Step~3 of the proof of Theorem~\ref{th:S.4.3}. From this and
 (\ref{e:S.4.30.1}) it follows that
\begin{eqnarray*}
&&\sum^n_{j=1}|\lambda_j\left(\mathcal{G}''_j(z+ \psi_{\vec{\lambda}}(z)+ u^++ u^-_1+ t(u^-_2-u^-_1))(u^-_2-u^-_1),
u^-_2-u^-_1\right)_H|\\
&&\le n\delta M\|u^-_2-u^-_1\|^2\le\frac{a_0}{8}\|u^-_2-u^-_1\|^2.
\end{eqnarray*}
This and (\ref{e:S.5.16})--(\ref{e:S.5.17}) lead to
\begin{eqnarray*}
 [D_2F_{\vec{\lambda}}(z, u^+ + u^-_2)-D_2F_{\vec{\lambda}}(z, u^++ u^-_1)](u^-_2-u^-_1)\le
-\frac{a_0}{2}\|u^-_2-u^-_1\|^2.
\end{eqnarray*}
 This implies the condition (ii) of \cite[theorem~A.1]{Lu2}.

\noindent{\bf Step 2}. For $\vec{\lambda}\in [-\delta,\delta]^n$, $z\in\bar B_{N}(\theta,\epsilon)$, $u^+\in\bar B_{H^+}(\theta, r)$ and $u^-\in\bar
B_{H^-}(\theta, s)$, by (\ref{e:S.5.14}) and the mean value
theorem, for some $t\in (0, 1)$ we have
\begin{eqnarray}\label{e:S.5.18}
&&D_2F_{\vec{\lambda}}(z, u^++u^-)(u^+-u^-)\nonumber\\
&=&D_2F_{\vec{\lambda}}(z, u^++u^-)(u^+-u^-)- D_2F_{\vec{\lambda}}(z, \theta)(u^+-u^-)\nonumber\\
&=&(\nabla\mathcal{L}_{\vec{\lambda}}(z+ \psi_{\vec{\lambda}}(z)+ u^++u^-), u^+-u^-)_H-(\nabla\mathcal{L}_{\vec{\lambda}}(z+ \psi_{\vec{\lambda}}(z)), u^+-u^-)_H\nonumber\\
&=&\left(B(z+ \psi_{\vec{\lambda}}(z)+ t(u^++u^-))(u^++u^-), u^+-u^-\right)_H\nonumber\\
&+&\sum^n_{j=1}\lambda_j\left(\mathcal{G}''_j(z+ \psi_{\vec{\lambda}}(z)+ t(u^++u^-))(u^++u^-), u^+-u^-\right)_H\nonumber\\
&=&\left(B(z+ \psi_{\vec{\lambda}}(z)+ t(u^++u^-))u^+, u^+\right)_H
-\left(B(z+ \psi_{\vec{\lambda}}(z)+
t(u^++u^-))u^-, u^-\right)_H\nonumber\\
&+&\sum^n_{j=1}\lambda_j\left(\mathcal{G}''_j(z+ \psi_{\vec{\lambda}}(z)+ t(u^++u^-))(u^++u^-), u^+-u^-\right)_H\nonumber\\
&\ge & a_1\|u^+\|^2+ a_0\|u^-\|^2\nonumber\\
&+&\sum^n_{j=1}\lambda_j\left(\mathcal{G}''_j(z+ \psi_{\vec{\lambda}}(z)+ t(u^++u^-))(u^++u^-), u^+-u^-\right)_H
\end{eqnarray}
because of Lemma~\ref{lem:S.2.2}(i) and (iii).  As above we have
\begin{eqnarray*}
&&\sum^n_{j=1}|\lambda_j\left(\mathcal{G}''_j(z+ \psi_{\vec{\lambda}}(z)+ t(u^++u^-))(u^++u^-), u^+-u^-\right)_H|\\
&&\le n\delta M\|u^++u^-\|\cdot\|u^+-u^-\|\le\frac{\min\{a_0,a_1\}}{4}(\|u^+\|^2+\|u^-\|^2)\\
&&\le\frac{a_1}{4}\|u^+\|^2+ \frac{a_0}{4}\|u^-\|^2.
\end{eqnarray*}
From this and (\ref{e:S.5.18}) we deduce
\begin{eqnarray*}
D_2F_{\vec{\lambda}}(z, u^++u^-)(u^+-u^-)\ge & \frac{a_1}{2}\|u^+\|^2+ \frac{a_0}{2}\|u^-\|^2.
\end{eqnarray*}
The condition (iii) of \cite[Theorem~A.1]{Lu2} is satisfied.

\noindent{\bf Step 3}. For $\vec{\lambda}\in [-\delta,\delta]^n$, $z\in\bar B_{N}(\theta,\epsilon)$ and $u^+\in
\bar B_{H^+}(\theta,r)$, as above we have $t\in (0,
1)$ such that
\begin{eqnarray*}
&&D_2F_{\vec{\lambda}}(z, u^+)u^+=D_2F_{\vec{\lambda}}(z, u^+)u^+- D_2F_{\vec{\lambda}}(z, \theta)u^+\\
&=&(\nabla\mathcal{L}_{\vec{\lambda}}(z+ \psi_{\vec{\lambda}}(z)+ u^+), u^+)_H-(\nabla\mathcal{L}_{\vec{\lambda}}(z+ \psi_{\vec{\lambda}}(z)), u^+)_H\\
&=&\left(B(z+ \psi_{\vec{\lambda}}(z)+ tu^+)u^+, u^+\right)_H+ \sum^n_{j=1}\lambda_j\left(\mathcal{G}''_j(z+
\psi_{\vec{\lambda}}(z)+ tu^+)u^+, u^+\right)_H\\
&\ge& a_1\|u^+\|^2+\sum^n_{j=1}\lambda_j\left(\mathcal{G}''_j(z+ \psi_{\vec{\lambda}}(z)+ tu^+)u^+, u^+\right)_H
\end{eqnarray*}
because of Lemma~\ref{lem:S.2.2}(i). Moreover, it is proved as before that
\begin{eqnarray*}
\sum^n_{j=1}|\lambda_j\left(\mathcal{G}''_j(z+ \psi_{\vec{\lambda}}(z)+ tu^+)u^+, u^+\right)_H|\le
\frac{\min\{a_0,a_1\}}{8}\|u^+\|^2.
\end{eqnarray*}
 Hence we obtain
 $$
D_2F_{\vec{\lambda}}(z, u^+)u^+ \ge a_1\|u^+\|^2> p(\|u^+\|)\quad\forall u^+\in \bar
B_{H^+}(\theta,s)\setminus\{\theta\},
$$
where $p:(0, \varepsilon]\to (0, \infty)$ is a non-decreasing
function given by $p(t)=\frac{a_1}{4}t^2$. Namely, $F_{\vec{\lambda}}$ satisfies the condition
(iv) of \cite[Theorem~A.1]{Lu2} (the parameterized version of \cite[Theoren~1.1]{DHK}).

The other arguments are as before.

\noindent{\bf Step 4}. The claim (i) in the part of ``Moreover" follows from
(\ref{e:S.4.22.1}) directly. For the second one,
since $\psi(\lambda, \cdot)$ is $G$-equivariant,
and $\mathcal{L}_\lambda$ is $G$-invariant, we derive from (\ref{e:S.5.13}) that
$F_{\vec{\lambda}}$ is $G$-invariant. By the construction of
$\Phi_{\vec{\lambda}}(\cdot,\cdot)$ (cf. \cite{DHK} and \cite[Theorem~A.1]{Lu1}),
it is expressed by $F_{\vec{\lambda}}(z, \cdot)$, one easily sees that
$\Phi_{\vec{\lambda}}(\cdot,\cdot)$ is $G$-equivariant.
\hfill$\Box$\vspace{2mm}

\begin{theorem}[Parameterized Shifting Theorem]\label{th:S.5.4}
Suppose for some $\vec{\lambda}\in [-\delta, \delta]^n$ that $\theta\in H$ is an isolated critical point of $\mathcal{L}_{\vec{\lambda}}$ (thus $\theta\in H^0$ is that of $\mathcal{L}_{\vec{\lambda}}^\circ$). Then
\begin{equation}\label{e:S.5.19}
C_q(\mathcal{L}_{\vec{\lambda}},\theta;{\bf K})=C_{q-\mu}(\mathcal{L}^\circ_{\vec{\lambda}},\theta;{\bf K})\quad\forall
q\in\mathbb{N}\cup\{0\},
\end{equation}
where $\mathcal{L}^\circ_{\vec{\lambda}}(z)=\mathcal{
L}_{\vec{\lambda}}(z+ \psi(\vec{\lambda}, z))= \mathcal{L}(z+\psi(\vec{\lambda}, z))+
\sum^n_{j=1}\lambda_j\mathcal{G}_j(z+\psi(\vec{\lambda}, z))$ is
 as in (\ref{e:S.5.3}).
\end{theorem}

\begin{proof}
Though $\mathcal{L}_{\vec{\lambda}}$ and $\mathcal{L}_{\vec{\lambda}}^\circ$
are only of class $C^1$, the construction of the Gromoll-Meyer pair
on the pages 49-51 of \cite{Ch1} is also effective for them
(see \cite{ChGh}). Hence the result can be obtained by repeating
the proof of \cite[Theorem~I.5.4]{Ch1}.

Of course,  with a stability theorem of critical groups the present case can also be reduced to  
that of \cite[Theorem~I.5.4]{Ch1}. In fact, by Theorem~\ref{th:S.5.3} we have
$C_\ast(\mathcal{L}_{\vec{\lambda}},\theta;{\bf K})=C_\ast(\widehat{\mathcal{L}}_{\vec{\lambda}},\theta;{\bf K})$, where
$$
\widehat{\mathcal{L}}_{\vec{\lambda}}: B_{H^0}(\theta,\epsilon)\times
(H^+\oplus H^-)\to \mathbb{R},\quad
(z,u^++u^-)\mapsto \|u^+\|^2-\|u^-\|^2+ {\mathcal{L}}_{\vec{\lambda}}^\circ(z).
$$
By a smooth cut function we can construct a $C^1$ functional
$f:H^0\to\mathbb{R}$ such that
$f(z)=\mathcal{L}_{\vec{\lambda}}^\circ(z)$ for $\|z\|<\epsilon/2$,
and $f(z)=\|z\|^2$ for $\|z\|\ge 3\epsilon/4$.
Define a functional
$$
\widetilde{\mathcal{L}}: H^0\times
(H^+\oplus H^-)\to \mathbb{R},\;(z,u^++u^-)\mapsto \|u^+\|^2-\|u^-\|^2+ f(z).
$$
Clearly, $C_\ast(\widehat{\mathcal{L}}_{\vec{\lambda}},\theta;{\bf K})=C_\ast(\widetilde{\mathcal{L}},\theta;{\bf K})$.
Since $\theta\in H^0$ is an isolated critical point of $\mathcal{L}_{\vec{\lambda}}^\circ$,
by shrinking $\epsilon>0$ we can assume that $\|\nabla\mathcal{L}_{\vec{\lambda}}^\circ(z)\|>0$
for all $z\in B_{H^0}(\theta,\epsilon)\setminus\{\theta\}$.
A standard result in differential topology claims that
$C^\infty(H^0,\mathbb{R})$ is dense
in $C^1_S(H^0,\mathbb{R})$ (equipped with strong topology).
Hence we can choose a function $g\in C^\infty(H^0,\mathbb{R})$ to satisfy
\begin{eqnarray*}
&&\|\nabla f(z)-\nabla g(z)\|<\frac{1}{2}\|\nabla f(z)\|,\quad\forall z\in
B_{H^0}(\theta,\epsilon)\setminus\{\theta\},\\
&&\|f(z)- g(z)\|<\frac{1}{2}\|f(z)\|,\quad\forall z\in
H^0\setminus B_{H^0}(\theta, 10\epsilon).
\end{eqnarray*}
They imply respectively that
\begin{eqnarray}
&&\|\nabla ((1-t)f+tg)(z)\|>\frac{1}{2}\|\nabla\mathcal{L}^\circ(z)\|>0,\quad\forall z\in
B_{H^0}(\theta,\epsilon/2)\setminus\{\theta\},\label{e:S.5.20}\\
&&\|(1-t)f(z)+ tg(z)\|>\frac{1}{2}\|f(z)\|=\frac{1}{2}\|z\|^2,\quad\forall z\in
H^0\setminus B_{H^0}(\theta, 10\epsilon)\label{e:S.5.21}
\end{eqnarray}
for all $t\in [0,1]$.  Hence  each functional
$f_t:H^0\to\mathbb{R},\;z\mapsto (1-t)f(z)+tg(z)$ has a unique critical point
$\theta$ in $B_{H^0}(\theta,\epsilon/2)$ by (\ref{e:S.5.20}), and satisfies
(PS) condition by (\ref{e:S.5.21}) and finiteness of $\dim H^0$.
It follows from the stability theorem of critical groups (\cite[Theorem~5.2]{CorH}) that
\begin{eqnarray}\label{e:S.5.22}
C_\ast(\mathcal{L}_{\vec{\lambda}}^\circ,\theta;{\bf K})=C_\ast(f,\theta;{\bf K})=C_\ast(f_t,\theta;{\bf K})=
C_\ast(g,\theta;{\bf K}),\quad\forall t\in [0, 1].
\end{eqnarray}
Since the functionals $f_t$ and
$H^+\oplus H^-\ni u^++u^-\mapsto \|u^+\|^2-\|u^-\|^2\in\mathbb{R}$
satisfies
(PS) condition, so is each functional
$$
\widetilde{\mathcal{L}}_t: H^0\times
(H^+\oplus H^-)\to \mathbb{R},\;(z,u^++u^-)\mapsto \|u^+\|^2-\|u^-\|^2+ f_t(z).
$$
As above we derive from \cite[Theorem~5.2]{CorH} that
$C_\ast(\widehat{\mathcal{L}},\theta;{\bf K})
=C_\ast(\widetilde{\mathcal{L}}_1,\theta;{\bf K})$.
But \cite[Theorem~I.5.4]{Ch} or \cite[Theorem~8.4]{MaWi} implies
$C_\ast(\widetilde{\mathcal{L}}_1,\theta;{\bf K})=C_{\ast-\mu}(g,\theta;{\bf K})$.
Hence
$$
C_\ast(\mathcal{L}_{\vec{\lambda}},\theta;{\bf K})=C_\ast(\widehat{\mathcal{L}}_{\vec{\lambda}},\theta;{\bf K})=
C_\ast(\widetilde{\mathcal{L}}_1,\theta;{\bf K})=C_{\ast-\mu}(g,\theta;{\bf K})=
C_{\ast-\mu}(\mathcal{L}_{\vec{\lambda}}^\circ,\theta;{\bf K}).
$$
\end{proof}

\subsection{Splitting and shifting theorems around critical orbits}\label{sec:S.6}

 Let $H$ be a Hilbert space with inner product $(\cdot,\cdot)_H$ and
let $({\cal H}, (\!(\cdot, \cdot)\!))$ be a $C^3$ Hilbert-Riemannian  manifold modeled on $H$.
Let $\mathcal{ O}\subset{\cal H}$
be a compact $C^3$ submanifold without boundary, and let $\pi:
N\mathcal{ O}\to \mathcal{ O}$ denote the normal bundle of it. The
bundle is a $C^2$-Hilbert vector bundle over $\mathcal{ O}$, and can
be considered as a subbundle of $T_\mathcal{ O}{\cal H}$ via the
Riemannian metric $(\!(\cdot, \cdot)\!)$. The metric $(\!(\cdot, \cdot)\!)$
induces a natural $C^2$ orthogonal bundle
projection ${\bf \Pi}:T_{\mathcal{O}}\mathcal{H}\to N\mathcal{O}$. For $\varepsilon>0$ we
denote by
\begin{eqnarray*}
N\mathcal{ O}(\varepsilon):=\{(x,v)\in N\mathcal{
 O}\,|\,\|v\|_{x}<\varepsilon\},
\end{eqnarray*}
the so-called normal disk bundle of radius $\varepsilon$.
 If $\varepsilon>0$ is sufficiently small   the exponential map $\exp$ gives a $C^2$-diffeomorphism
 $\digamma$ from  $N\mathcal{ O}(\varepsilon)$ onto an open
neighborhood of $\mathcal{ O}$ in ${\cal H}$, $\mathcal{
N}(\mathcal{ O},\varepsilon)$.

For $x\in\mathcal{ O}$, let  $\mathscr{L}_s(N\mathcal{O}_x)$ denote the space
of those operators $S\in \mathscr{L}(N\mathcal{ O}_x)$ which are self-adjoint
with respect to the inner product $(\!(\cdot, \cdot)\!)_x$, i.e.
$(\!(S_xu, v)\!)_x=(\!(u, S_xv)\!)_x$ for all $u, v\in N\mathcal{
O}_x$. Then we have a $C^1$ vector bundle $\mathscr{L}_s(N\mathcal{ O})\to
\mathcal{O}$ whose fiber at $x\in\mathcal{ O}$ is given by
$\mathscr{L}_s(N\mathcal{ O}_x)$.

 Let  $\mathcal{L}:{\cal H}\to\mathbb{R}$ be a $C^1$  functional.
  A connected $C^3$ submanifold $\mathcal{O}\subset
{\cal H}$ is called a {\it critical manifold} of $\mathcal{L}$ if
$\mathcal{L}|_\mathcal{ O}={\rm const}$ and $D\mathcal{L}(x)v=0$
for any $x\in\mathcal{ O}$ and $v\in T_x{\cal H}$. If there exists a
neighborhood ${\cal V}$ of $\mathcal{O}$ such that ${\cal V}\setminus \mathcal{O}$
contains no critical points of $\mathcal{L}$ we say  $\mathcal{O}$  to be {\it isolated}.

 Furthermore, we make the following

\begin{hypothesis}\label{hyp:S.6.1}
{\rm The gradient field
 $\nabla\mathcal{L}:\mathcal{H}\to T\mathcal{H}$ is G\^ateaux differentiable
 and thus  we have a bounded linear self-adjoint
operator $d^2\mathcal{L}(x)\in \mathscr{L}_s(T_x\mathcal{H})$
 for each $x\in\mathcal{O}$; moreover,
  $\mathcal{O}\ni x\mapsto d^2\mathcal{L}(x)$ is a continuous section of
$\mathscr{L}_s(T\mathcal{H})\to\mathcal{O}$ is continuous,
$\dim{\rm Ker}(d^2\mathcal{L}(x))={\rm const}\;\forall
x\in\mathcal{O}$, and there exists $a_0>0$ such that
\begin{equation}\label{e:S.6.1}
\sigma(d^2\mathcal{L}(x))\cap([-2a_0,
2a_0]\setminus\{0\})=\emptyset\quad\forall x\in\mathcal{ O}.
\end{equation}}
\end{hypothesis}

This implies that
$\mathcal{O}\ni x\mapsto\mathcal{B}_x(\theta_x):=
{\bf \Pi}_x\circ d^2\mathcal{L}(x)|_{N{\cal O}_x}=d^2(\mathcal{L}\circ\exp_x|_{N{\cal O}_x})(\theta_x)$
        is a continuous section of  $\mathscr{L}_s(N\mathcal{O}\to\mathcal{O}$,
$\dim{\rm Ker}({\cal B}_x(\theta_x))={\rm const}\;\forall
x\in\mathcal{O}$, and
$$
\sigma({\cal B}_x(\theta_x))\cap([-2a_0,
2a_0]\setminus\{0\})=\emptyset\quad\forall x\in\mathcal{ O}.
$$

Let $\chi_\ast$ ($\ast=+, -, 0$) be the characteristic function of
the intervals $[2a_0, +\infty)$, $(-2a_0, a_0)$ and $(-\infty,
-2a_0]$, respectively. Then we have the orthogonal bundle
projections on the normal bundle $N\mathcal{ O}$, $P^\ast$ (defined
by $P^\ast_x(v)=\chi_\ast({\cal B}_x(\theta_x))v$), $\ast=+, -, 0$.
Denote by $N^\ast\mathcal{ O}=P^\ast N\mathcal{O}$, $\ast=+, -, 0$.
(Clearly,  ${\cal B}_x(\theta_x)(N^\ast\mathcal{ O}_x)\subset
N^\ast\mathcal{O}_x$ for any $x\in\mathcal{ O}$ and $\ast=+, -, 0$).
By \cite[Lem.7.4]{Ch}, we have
$N\mathcal{ O}=N^+\mathcal{ O}\oplus N^-\mathcal{ O}\oplus
N^0\mathcal{O}$. If ${\rm rank}N^0\mathcal{O}=0$, the critical orbit $\mathcal{O}$ is called {\it nondegenerate}.

In the following we only consider the case $\mathcal{O}$ is a critical orbit of a compact Lie group.
The general case can be treated  as in \cite{Lu3+}.
The following assumption implies naturally Hypothesis~\ref{hyp:S.6.1}.

\begin{hypothesis}\label{hyp:S.6.2}
{\rm {\bf (i)} Let $G$ be a compact
Lie group, and let  ${\cal H}$ be a $C^3$ $G$-Hilbert manifold. (So
$T{\cal H}$ is a $C^2$ $G$-vector bundle, i.e. for any $g\in G$ the
induced action $G\times T{\cal H}\to T{\cal H}$ given by
$g\cdot(x,v)=(g\cdot x, dg(x)v)$ is a $C^1$ bundle map satisfying
$gT_x{\cal H}=T_{g\cdot x}{\cal H}\;\forall x\in{\cal H}$).
Furthermore, this action also preserves the Riemannian-Hilbert
structure $(\!(\cdot,\cdot)\!)$, i.e.
$$(\!(g\cdot
u, g\cdot v)\!)_{g\cdot x}=(\!(u, v)\!)_x,\quad\forall x\in{\cal
H},\quad\forall u, v\in T_x{\cal H}.
$$
 (In this case $\bigl({\cal H},
(\!(\cdot,\cdot)\!)\bigr)$ is said to be a $C^2$ $G$-{\it
Riemannian-Hilbert manifold}). \\
 {\bf (ii)} The $C^1$ functional $\mathcal{ L}:\mathcal{H}\to\mathbb{R}$ is $G$-invariant,
 $\nabla\mathcal{L}:\mathcal{H}\to T\mathcal{H}$ is G\^ateaux differentiable,
 (i.e., under any $C^3$ local chart the functional $\mathcal{L}$
 has a representation that is $C^1$ and has a G\^ateaux differentiable gradient map),      and
$\mathcal{ O}$ is an isolated critical orbit which is a $C^3$ critical
submanifold with  Morse index $\mu_\mathcal{O}$.}
\end{hypothesis}

Since $\exp_{g\cdot x}(g\cdot v)=g\cdot\exp_x(v)$ for any $g\in G$
and $(x,v)\in T{\cal H}$, we have
$$
\mathcal{L}\circ\exp(g\cdot x,g\cdot v)=\mathcal{ L}(\exp(g\cdot x,
g\cdot v))=\mathcal{ L}(g\cdot\exp(x,v))=\mathcal{ L}(\exp(x,v)).
$$
 It follows  that $\nabla\mathcal{L}(g\cdot x)= g^{-1}\cdot \nabla\mathcal{L}(x)g$ and
\begin{equation}\label{e:S.6.2}
\nabla\left(\mathcal{L}\circ\exp|_{N\mathcal{O}(\varepsilon)_{gx}}\right)(g\cdot v)= g\cdot \nabla\left(\mathcal{L}\circ\exp|_{N\mathcal{O}(\varepsilon)_x}\right)(v)
\end{equation}
for any $g\in G$ and $(x,v)\in N\mathcal{O}(\varepsilon)_x$, which leads to
\begin{equation}\label{e:S.6.3}
d^2\left(\mathcal{L}\circ\exp|_{N\mathcal{O}(\varepsilon)_{gx}}\right)(g\cdot v)\cdot g= g\cdot d^2\left(\mathcal{L}\circ\exp|_{N\mathcal{O}(\varepsilon)_x}\right)(v)
\end{equation}
as bounded linear operators from $N\mathcal{O}_x$ onto $N\mathcal{O}_{gx}$.

\begin{theorem}\label{th:S.6.0}
Under   Hypothesis~\ref{hyp:S.6.2},
let for some $x_0\in\mathcal{ O}$ the pair $\bigl(\mathcal{L}\circ\exp_{x_0},
B_{T_{x_0}\mathcal{H}}(\theta,\varepsilon)\bigr)$
(and so the pair $(\mathcal{L}\circ\exp|_{N\mathcal{O}(\varepsilon)_{x_0}},  N\mathcal{O}(\varepsilon)_{x_0})$
by Lemma~\ref{lem:S.2.4}) satisfies the corresponding conditions in Hypothesis~\ref{hyp:1.1} with $X=H$.
Suppose that the critical orbit $\mathcal{O}$ is nondegenerate. Then there exist  $\epsilon>0$
and a  $G$-equivariant homeomorphism onto an open neighborhood of
the zero section preserving fibers,
$\Phi:  N^+\mathcal{O}(\epsilon)\oplus N^-\mathcal{ O}(\epsilon)\to N\mathcal{ O}$,
such that for any $x\in\mathcal{O}$ and  $(u^+, u^-)\in  N^+\mathcal{ O}(\epsilon)_x\times N^-\mathcal{
O}(\epsilon)_x$,
\begin{eqnarray}\label{e:S.6.4}
\mathcal{L}\circ\exp\circ\Phi(x,  u^++
u^-)=\|u^+\|^2_x-\|u^-\|^2_x+ \mathcal{L}|_{\mathcal{O}}.
\end{eqnarray}
\end{theorem}

It naturally leads to a Morse relation if $\mathcal{L}$ satisfies the (PS)
condition and has only nondegenerate critical orbits between regular levels.
Theorem~\ref{th:S.6.0} will be proved after the proof of the following theorem.

\begin{theorem}\label{th:S.6.1}
Under   Hypothesis~\ref{hyp:S.6.2},
let for some $x_0\in\mathcal{ O}$ the pair
$(\mathcal{L}\circ\exp|_{N\mathcal{O}(\varepsilon)_{x_0}},  N\mathcal{O}(\varepsilon)_{x_0})$
satisfy the corresponding conditions with Hypothesis~\ref{hyp:1.1} with $X=H$.
Suppose that the critical orbit $\mathcal{O}$ is degenerate,
i.e., ${\rm rank}N^0\mathcal{O}>0$. Then there exist  $\epsilon>0$, a  $G$-equivariant topological  bundle
morphism that preserves the zero section,
 $$
\mathfrak{h}:N^0\mathcal{O}(3\epsilon)\to N^+\mathcal{O}\oplus N^-\mathcal{O},\;(x,v)\mapsto \mathfrak{h}_x(v),
$$
and a  $G$-equivariant homeomorphism onto an open neighborhood of
the zero section preserving fibers,
$\Phi: N^0\mathcal{ O}(\epsilon)\oplus N^+\mathcal{O}(\epsilon)\oplus N^-\mathcal{ O}(\epsilon)\to N\mathcal{O}$,
such that  the following properties hold:\\
\noindent{\bf (I)} The induced map,
$\mathfrak{h}_x:N^0\mathcal{O}(\epsilon)_x\to T_x{\cal H}$,  satisfies
\begin{eqnarray}\label{e:S.6.6}
(P^+_x+P^-_x)\circ{\bf \Pi}_x\nabla(\mathcal{L}\circ\exp_x)(v+ \mathfrak{h}_x(v))=0\quad\forall
(x,v^0)\in N^0\mathcal{O}(\epsilon).
\end{eqnarray}
\noindent{\bf (II)} $\Phi$ has the form $\Phi(x, v+ u^++
u^-)=\bigl(x, v+ \mathfrak{h}_x(v)+\phi_{(x, v)}(u^++ u^-)\bigr)$
with $\phi_{(x, v)}(u^++ u^-)\in (N^+\mathcal{ O}\oplus N^-\mathcal{
O})_x$, and satisfies
\begin{eqnarray}\label{e:S.6.7}
\mathcal{L}\circ\exp\circ\Phi(x, v, u^++
u^-)=\|u^+\|^2_x-\|u^-\|^2_x+ \mathcal{ L}\circ\exp_x(v+
\mathfrak{h}_x(v))
\end{eqnarray}
for any $x\in\mathcal{O}$ and  $(v, u^+, u^-)\in N^0\mathcal{
O}(\epsilon)_x\times N^+\mathcal{ O}(\epsilon)_x\times N^-\mathcal{
O}(\epsilon)_x$. Moreover
\begin{description}
\item[(II.1)] $\Phi(x, v)=(x, v+ \mathfrak{h}_x(v))\;\forall v\in N^0\mathcal{
O}(\epsilon)_x$,
\item[(II.2)] $\phi_{(x,v)}(u^++ u^-)\in N^-\mathcal{O}$ if and only if
$u^+=\theta_x$.
\end{description}
\noindent{\bf (III)} For each $x\in\mathcal{O}$ the function
\begin{eqnarray}\label{e:S.6.8}
N^0\mathcal{O}(\epsilon)_x\to\R,\;v\mapsto \mathcal{ L}_x^\circ(v):=
\mathcal{ L}\circ\exp_x(v+ \mathfrak{h}_x(v))
\end{eqnarray}
is $G_x$-invariant, of class $C^{1}$,  and satisfies
$$
D\mathcal{L}_x^\circ(v)v':=
(\nabla(\mathcal{ L}\circ\exp_x)(v+ \mathfrak{h}_x(v)),v'),\quad\;\forall v'\in N^0\mathcal{O}_x.
$$
Moreover, each $\mathfrak{h}_x$ is of class $C^{1-0}$, and hence $\mathcal{L}_x^\circ$
is of class $C^{2-0}$ if $\mathcal{L}$ is of class $C^{2-0}$.
\end{theorem}

\begin{proof}\quad
We only outline main procedures. By the assumption and  (\ref{e:S.6.3})
we deduce that each pair
$(\mathcal{L}\circ\exp|_{N\mathcal{O}(\varepsilon)_{x}},  N\mathcal{O}(\varepsilon)_{x})$
satisfies the corresponding conditions with Hypothesis~\ref{hyp:1.1} with $X=H$ too,
and that there exists $a_0>0$ such that
\begin{eqnarray}\label{e:S.6.10}
\sigma\left(d^2\left(\mathcal{L}\circ\exp|_{N\mathcal{O}(\varepsilon)_x}\right)(\theta_x)\right)\cap([-2a_0,
2a_0]\setminus\{0\})=\emptyset,\quad\forall x\in\mathcal{O}.
\end{eqnarray}

By Theorem~\ref{th:S.1.2} we get a small $\epsilon\in (0, \varepsilon/3)$ and a continuous map
$$
\mathfrak{h}_{x_0}: N^0\mathcal{ O}(3\epsilon)_{x_0}\to
 N^\pm\mathcal{ O}(\varepsilon/2)_{x_0},
$$
 such that $\mathfrak{h}_{x_0}(g\cdot v)=g\cdot \mathfrak{h}_{x_0}(v)$,
 $\mathfrak{h}_{x_0}(\theta_{x_0})=\theta_{x_0}$ and
$$
(P^+_{x_0}+ P^-_{x_0})\nabla\left(\mathcal{L}\circ\exp|_{N\mathcal{O}(\varepsilon)_{x_0}}\right)(v+
\mathfrak{h}_{x_0}(v))=0,\quad\forall v\in N^0\mathcal{
O}(3\epsilon)_{x_0}.
$$
Furthermore, the function
$$
\mathcal{ L}^\circ_{x_0}: N^0\mathcal{O}(\epsilon)_{x_0}\to
\R,\;v\mapsto \mathcal{ L}\circ\exp_{x_0}(v+ \mathfrak{h}_{x_0}(v))
$$
is of class $C^{1}$, and
$D\mathcal{ L}^\circ_{x_0}(v)u=\bigl(\nabla(\mathcal{L}\circ\exp|_{N\mathcal{O}(\varepsilon)_{x_0}})(v+
\mathfrak{h}_{x_0}(v)),u\bigr)$. Define
\begin{eqnarray*}
\mathfrak{h}: N^0\mathcal{ O}(3\epsilon)\to T{\cal H},\quad
(x,v)\mapsto g\cdot\mathfrak{h}_{x_0}(g^{-1}\cdot v),
\end{eqnarray*}
where $g\cdot x_0=x$. It is continuous. Otherwise, there exists a sequence
$\{(x_j, v_j)\}_j\subset  N^0\mathcal{ O}(3\epsilon)$ which converges to
a point $(\bar{x},\bar{v})\in  N^0\mathcal{O}(3\epsilon)$, but
 $\mathfrak{h}(x_j, v_j)\nrightarrow\mathfrak{h}(\bar{x},\bar{v})$.
 Passing to a subsequence we may assume that
 $\{\mathfrak{h}(x_j, v_j)\}_j$ has no intersection with an open neighborhood ${\bf U}$ of $\mathfrak{h}(\bar{x},\bar{v})$
 in $T{\cal H}$. Let $\bar{g}, g_j\in G$ be such that
 $\bar{g}\cdot x_0=\bar{x}$ and $g_j\cdot x_0=x_j$, $j=1,2,\cdots$.
 Then $\mathfrak{h}(\bar{x},\bar{v})=\bar{g}\cdot\mathfrak{h}_{x_0}(\bar{g}^{-1}\cdot\bar{v})$
 and $\mathfrak{h}(x_j, v_j)=g_j\cdot\mathfrak{h}_{x_0}(g_j^{-1}\cdot{v}_j)$
 for each $j\in\mathbb{N}$. Note that $\bar{g}^{-1}\cdot{\bf U}$ is an open neighborhood of
 $\mathfrak{h}_{x_0}(\bar{g}^{-1}\cdot\bar{v})=\bar{g}^{-1}\cdot\mathfrak{h}(\bar{x},\bar{v})$
 and that $\{\bar{g}^{-1}\cdot\mathfrak{h}(x_j, v_j)=\bar{g}^{-1}\cdot g_j\cdot\mathfrak{h}_{x_0}(g_j^{-1}\cdot{v}_j)\}_j$
 has no intersection with $\bar{g}^{-1}\cdot{\bf U}$.
  Since $G$ is compact,  we may assume $\bar{g}^{-1}\cdot g_j\to \hat{g}\in G$
  and so $g_j^{-1}\to (\bar{g}\hat{g})^{-1}\in G$  after passing to a
  subsequence (if necessary). Then
 $\bar{g}^{-1}\cdot\mathfrak{h}(x_j, v_j)=\bar{g}^{-1}\cdot g_j\cdot\mathfrak{h}_{x_0}(g_j^{-1}\cdot{v}_j)
 \to \hat{g}\cdot\mathfrak{h}_{x_0}((\bar{g}\hat{g})^{-1}\cdot\bar{v})=\mathfrak{h}_{x_0}(\bar{g}^{-1}\cdot\bar{v})$.
It follows that $\mathfrak{h}_{x_0}(\bar{g}^{-1}\cdot\bar{v})$ does not belong to
$\bar{g}^{-1}\cdot{\bf U}$. This contradicts the fact that $\bar{g}^{-1}\cdot{\bf U}$ is an open neighborhood of
 $\mathfrak{h}_{x_0}(\bar{g}^{-1}\cdot\bar{v})$.

By the definition of $\mathfrak{h}$, it is clearly $G$-equivariant and
satisfies
\begin{equation}\label{e:S.6.11}
(P^+_{x}+ P^-_{x})\nabla\left(\mathcal{L}\circ\exp|_{N\mathcal{O}_{x}(\varepsilon)}\right)(v+
\mathfrak{h}_{x}(v))=0,\quad\forall (x,v)\in N^0\mathcal{
O}(3\epsilon).
\end{equation}
 Moreover, the map $\mathcal{F}: N^0\mathcal{O}(\epsilon)\oplus N^+\mathcal{O}(\epsilon)\oplus N^-\mathcal{O}(\epsilon)\to\R$ defined by
\begin{eqnarray}\label{e:S.6.12}
\mathcal{ F}(x, v, u^++u^-)&=&\mathcal{F}_x(v, u^++u^-)\nonumber\\
&=&\mathcal{ L}\circ\exp_x(v+\mathfrak{h}_x(v)+ u^++u^-)-\mathcal{ L}\circ\exp_x(v+
\mathfrak{h}_x(v)),
\end{eqnarray}
is $G$-invariant, and satisfies for any $(x,v)\in N^0\mathcal{O}(\epsilon)$ and $u\in N^+\mathcal{ O}_x\oplus
N^-\mathcal{O}_x$,
\begin{eqnarray}\label{e:S.6.13}
\mathcal{F}_x(v, \theta_x)=0\quad\hbox{and}\quad D_2\mathcal{F}_x(v, \theta_x)[u]=0.
\end{eqnarray}
By (\ref{e:S.6.2})--(\ref{e:S.6.3}) and Lemmas~\ref{lem:S.2.1},~\ref{lem:S.2.2}  we can
immediately obtain:

\begin{lemma}\label{lem:S.6.2}
There exist positive numbers $\varepsilon_1\in (0, \varepsilon)$ and
$a_1\in (0, 2a_0)$, and a function $\Omega:N\mathcal{
O}(\varepsilon_1)\to [0, \infty)$ with the property that
$\Omega(x,v)\to 0$ as $\|v\|_x\to 0$,
 such that for any $(x,v)\in N\mathcal{ O}(\varepsilon_1)$ the
 following conclusions hold with ${\cal B}_x=d^2\left(\mathcal{L}\circ\exp|_{N\mathcal{O}_x(\varepsilon)}\right)$:
\begin{description}
\item[(i)]  $|(\!({\cal B}_x(v)u, w)\!)_x- (\!({\cal B}_x(\theta_x)u, w)\!)_x |\le \Omega(x,v)
\|u\|_x\cdot\|w\|_x$ for any $u\in N^0\mathcal{ O}_x\oplus
N^-\mathcal{ O}_x$ and $w\in N\mathcal{ O}_x$;
\item[(ii)] $(\!({\cal B}_x(v)u, u)\!)_x\ge a_1\|u\|^2_x$ for all $u\in N^+\mathcal{ O}_x$;
\item[(iii)] $|(\!({\cal B}_x(v)u,w)_x\!)|\le\Omega(x,v)\|u\|_x\cdot\|w\|_x$
for all $u^+\in N^+\mathcal{ O}_x,  w\in N^-\mathcal{
O}_x\oplus N^0\mathcal{ O}_x$;
\item[(iv)] $(\!({\cal B}_x(v)u,u)_x\le-a_0\|u\|^2$ for all $u\in N^-\mathcal{ O}_x$.
\end{description}
\end{lemma}

Let us choose $\varepsilon_2\in (0,\epsilon/2)$ so small that
$$
(x, v^0+ \mathfrak{h}_x(v^0)+ u^++ u^-)\in  N\mathcal{
O}(\varepsilon_1)
$$
for $(x, v^0)\in N^0\mathcal{ O}(2\varepsilon_2)$ and $(x, u^\ast)\in
N^\ast\mathcal{ O}(2\varepsilon_2)$, $\ast=+, -$.
As in the proof of \cite[Lemma~3.5]{Lu2}, we may use
 \cite[Lemma~2.4]{Lu2}
to derive

\begin{lemma}\label{lem:S.6.3}
Let the constants $a_1$ and $a_0$  be given by
Lemma~\ref{lem:S.6.2}(ii),(iv).
For the above $\varepsilon_2>0$ and each $x\in\mathcal{O}$ the
restriction of the functional $\mathcal{ F}_x$  to
$\overline{N^0\mathcal{
O}(2\varepsilon_2)_x}\oplus[\overline{N^+\mathcal{
O}(2\varepsilon_2)_x}\oplus\overline{N^-\mathcal{
O}(2\varepsilon_2)_x}]$
 satisfies:
\begin{description}
\item[(i)] $[D_2\mathcal{ F}_x(v^0, u^++ u^-_2)- D_2\mathcal{ F}_x(v^0, u^++ u^-_1)](u^-_2-u^-_1) \le -a_1\|u^-_2-u^-_1\|^2_x$ for any $(x, v^0)\in\overline{N^0\mathcal{ O}(2\varepsilon_2)}$,
$(x, u^+)\in \overline{N^+\mathcal{ O}(2\varepsilon_2)}$
and $(x, u^-_j)\in \overline{N^-\mathcal{ O}(2\varepsilon_2)}$, $j=1,2$;

\item[(ii)] $D_2\mathcal{ F}_x(v^0, u^++ u^-)(u^+-u^-) \ge a_1\|u^+\|^2_x+
a_0\|u^-\|^2_x$ for any $(x, v^0)\in\overline{N^0\mathcal{ O}(2\varepsilon_2)}$ and
$(x, u^\ast)\in \overline{N^\ast\mathcal{ O}(2\varepsilon_2)}$, $\ast=+,-$;

\item[(iii)] $D_2\mathcal{ F}_x(v^0, u^+)u^+ \ge a_1\|u^+\|^2_x$
for any $(x, v^0)\in\overline{N^0\mathcal{ O}(2\varepsilon_2)}$ and
$(x, u^+)\in \overline{N^+\mathcal{ O}(2\varepsilon_2)}$.
\end{description}
\end{lemma}

Denote by bundle projections $\Pi_0:\overline{N^0\mathcal{ O}(\varepsilon_2)}\to\mathcal{ O}$ and
$$
\Pi_{\pm}:N^+\mathcal{O}\oplus N^-\mathcal{O}\to\mathcal{O},\quad
\Pi_{\ast}:N^\ast\mathcal{O}\to\mathcal{O},\;\ast=+,-.
$$
Let $\Lambda=\overline{N^0\mathcal{ O}(2\varepsilon_2)}$, $p:\mathcal{E}\to\Lambda$
and $p_\ast:\mathcal{E}^\ast\to\Lambda$ be the pullbacks of
$N^+\mathcal{O}\oplus N^-\mathcal{O}$ and
$N^\ast\mathcal{O}$ via $\Pi_0$, $\ast=+,-$. Then $\mathcal{E}=\mathcal{E}^+\oplus\mathcal{E}^-$,
and for $\lambda=(x,v^0)\in\Lambda$ we have $\mathcal{E}_\lambda=
N^+\mathcal{O}_x\oplus N^-\mathcal{O}_x$ and
$\mathcal{E}^\ast_\lambda=N^\ast\mathcal{O}_x$, $\ast=+,-$. Moreover, for each $\eta>0$ we write
\begin{eqnarray*}
&&B_\eta(\mathcal{E})=\left\{(\lambda, w)\,|\, \lambda=(x,v^0)\in\Lambda\;\&\;w\in
(N^+\mathcal{O}\oplus N^-\mathcal{O})_x(\eta)\right\},\\
&&\bar{B}_\eta(\mathcal{E})=\left\{(\lambda, w)\,|\, \lambda=(x,v^0)\in\Lambda\;\&\;w\in
\overline{(N^+\mathcal{O}\oplus N^-\mathcal{O})_x(\eta)}\right\}.
\end{eqnarray*}
Similarly,  $B_\eta(\mathcal{E}^\ast)$ and $\bar{B}_\eta(\mathcal{E}^\ast)$ ($\ast=+,-$)
are defined. Let  $\mathcal{J}: B_{2\varepsilon_2}(\mathcal{E})\to\R$ be given by
\begin{eqnarray}\label{e:S.6.14}
&&\mathcal{J}(\lambda,  v^\pm)=\mathcal{J}_\lambda(v^\pm) =\mathcal{F}(x, v^0, v^\pm),\\
&&\quad\forall\lambda=(x, v^0)\in\Lambda\;\&\;\forall v^\pm\in B_{2\varepsilon_2}(\mathcal{E})_\lambda.\nonumber
\end{eqnarray}
It is continuous, and $C^1$ in $v^\pm$. From (\ref{e:S.6.13}) and Lemma~\ref{lem:S.6.3} we directly obtain

\begin{lemma}\label{lem:S.6.4}
The functional $\mathcal{J}_\lambda$ satisfies the conditions in
Theorem~A.2 of \cite{Lu2} (the bundle parameterized version of \cite[Theoren~1.1]{DHK}).  Precisely, for each $\lambda\in\Lambda$  the functional
$\mathcal{J}_\lambda: B_{2\varepsilon_2}(\mathcal{E})\to\mathbb{R}$
 satisfies:
\begin{description}
\item[(i)] $\mathcal{J}_\lambda(\theta_\lambda)=0$ and $D\mathcal{J}_\lambda(\theta_\lambda)=0$;
\item[(ii)] $[D\mathcal{J}_\lambda(u^++ u^-_2)- D\mathcal{J}_\lambda(u^++ u^-_1)](u^-_2-u^-_1) \le -a_1\|u^-_2-u^-_1\|^2_x$ for any
$\lambda=(x,v^0)\in\Lambda$, $u^+\in \bar{B}_{\varepsilon_2}(\mathcal{E}^+)_\lambda$ and
$u^-_j\in \bar{B}_{\varepsilon_2}(\mathcal{E}^-)_\lambda$, $j=1,2$;

\item[(iii)] $D\mathcal{J}_\lambda(\lambda, u^++ u^-)(u^+-u^-) \ge a_1\|u^+\|^2_x+
a_0\|u^-\|^2_x$ for any $\lambda=(x,v^0)\in\Lambda$ and
$u^\ast\in \bar{B}_{\varepsilon_2}(\mathcal{E}^\ast)_\lambda$, $\ast=+,-$;

\item[(iv)] $D\mathcal{J}_\lambda(u^+)u^+ \ge a_1\|u^+\|^2_x$
for any $\lambda=(x, v^0)\in\Lambda$ and
$u^+\in \bar{B}_{\varepsilon_2}(\mathcal{E}^+)_\lambda$.
\end{description}
\end{lemma}

By this  we can use Theorem~A.2 of \cite{Lu2}
to get $\epsilon\in (0, \varepsilon_2)$, an open neighborhood
$U$ of the zero section $0_\mathcal{ E}$ of $\mathcal{ E}$ in $B_{2\varepsilon_2}(\mathcal{E})$ and a
homeomorphism
\begin{eqnarray}\label{e:S.6.15}
\phi: B_{\epsilon}(\mathcal{ E}^+) \oplus
B_{\epsilon}(\mathcal{ E}^-)\to U,\;(\lambda, u^++ u^-)\mapsto (\lambda, \phi_\lambda(u^++u^-))
\end{eqnarray}
such that for all $(\lambda, u^++ u^-)\in  B_{\epsilon}(\mathcal{
E}^+) \oplus B_{\epsilon}(\mathcal{ E}^-)$ with $\lambda=(x,v^0)\in\Lambda$,
\begin{eqnarray}\label{e:S.6.16}
J(\phi(\lambda, u^++
u^-))=\|u^+\|^2_{x}-\|u^-\|^2_{x}.
\end{eqnarray}
 Moreover,
for each $\lambda\in\Lambda$,
$\phi_\lambda(\theta_\lambda)=\theta_\lambda$, $\phi_\lambda(x+y)\in
\mathcal{E}^-_\lambda$ if and only if $x=\theta_\lambda$, and
$\phi$ is a homoeomorphism from $B_{\epsilon}(\mathcal{
E}^-)$ onto $U\cap \mathcal{ E}^-$.

Note that $B_{\epsilon}(\mathcal{ E}^+) \oplus
B_{\epsilon}(\mathcal{ E}^-)= \overline{N^0\mathcal{
O}(2\varepsilon_2)}\oplus N^+\mathcal{O}(\epsilon)\oplus N^-\mathcal{
O}(\epsilon)$ and 
$U=\overline{N^0\mathcal{O}(2\varepsilon_2)}\oplus\widehat{U}$,
where $\widehat{U}$ is an open neighborhood
 of the zero section of $N^+\mathcal{O}\oplus N^-\mathcal{
O}$ in $N^+\mathcal{O}(2\varepsilon_2)\oplus N^-\mathcal{
O}(\varepsilon_2)$. Let $\mathcal{W}=N^0\mathcal{O}(\epsilon)\oplus\widehat{U}$,
which is an open neighborhood
 of the zero section of $N\mathcal{O}$ in
$N^0\mathcal{O}(2\varepsilon_2)\oplus N^+\mathcal{O}(2\varepsilon_2)\oplus N^-\mathcal{
O}(\varepsilon_2)$. By (\ref{e:S.6.15}) we get a
 homeomorphism
\begin{eqnarray}\label{e:S.6.17}
\phi: &&N^0\mathcal{O}(\epsilon)\oplus N^+\mathcal{O}(\epsilon)\oplus N^-\mathcal{
O}(\epsilon)\to \mathcal{W},\\
&&(x,v, u^++ u^-)\mapsto (x,v, \phi_{(x,v)}(u^++u^-)),\nonumber
\end{eqnarray}
and therefore a topological embedding bundle morphism that preserves the
zero section,
\begin{eqnarray}\label{e:S.6.18}
\Phi:&& N^0\mathcal{ O}(\epsilon)\oplus N^+\mathcal{
O}(\epsilon)\oplus N^-\mathcal{ O}(\epsilon)\to N\mathcal{ O},\\
&&(x,v, u^++ u^-)\mapsto (x,v+ \mathfrak{h}_{x}(v), \phi_{(x,v)}(u^++u^-)).\nonumber
\end{eqnarray}
 From (\ref{e:S.6.12}), (\ref{e:S.6.14}) and (\ref{e:S.6.16}) it follows that $\Phi$ and $\phi$ satisfy
 \begin{eqnarray*}
 \mathcal{ L}\circ\exp\circ\Phi(x,v+u^++u^-)&=&
\mathcal{ L}\circ\exp_x(v+ \mathfrak{h}_{x}(v)+ \phi_{(x,v)}(u^++u^-))\\
&=&\|u^+\|^2_x-\|u^-\|_x^2+\mathcal{ L}\circ\exp_x(v+
\mathfrak{h}_{x}(v))
\end{eqnarray*}
for all $(x,v+u^+,u^-)\in N^0\mathcal{O}(\epsilon)\oplus N^+\mathcal{O}(\epsilon)\oplus N^-\mathcal{
O}(\epsilon)$. The other conclusions  easily follow from the above arguments.
Theorem~\ref{th:S.6.1} is proved.
\end{proof}

\noindent{\it Proof of Theorem~\ref{th:S.6.0}}.\quad
In the present case Lemma~\ref{lem:S.6.2} also holds with $N^0\mathcal{O}_x=\{\theta_x\}\;\forall x\in\mathcal{O}$. But we need to replace
the map $\mathcal{F}$ in (\ref{e:S.6.12}) by
$$
\mathcal{F}(x,  u^++u^-): N^+\mathcal{O}(\epsilon)\oplus N^-\mathcal{
O}(\epsilon)\to\R,\; (x, u^++u^-)\mapsto\mathcal{ L}\circ\exp_x( u^++u^-).
$$
For any $x\in\mathcal{O}_x$, let $\mathcal{F}_x$ be the restriction of $\mathcal{F}$ to $N^+\mathcal{O}(\epsilon)_x\oplus N^-\mathcal{O}(\epsilon)_x$. As in the
proof of Theorem~\ref{th:S.1.1}, Lemma~\ref{lem:S.6.3}
is still true with $\overline{N^0\mathcal{
O}(2\varepsilon_2)_x}=\{\theta_x\}$.
Then the desired conclusions can be obtained
by  applying \cite[Theorem~A.2]{Lu2} to $\Lambda=\mathcal{O}$
and $J_\lambda=\mathcal{F}_x$ with $\lambda=x\in\mathcal{O}$.
\hfill$\Box$\vspace{2mm}

Many results in \cite{Ch, MaWi, Wa, Wa1} also hold in our setting. Here are a few of them,
which are needed in this paper.

\begin{corollary}[Shifting Theorem]\label{cor:S.6.5}
Let ${\bf K}$ be any commutative ring. Then\\
{\bf i)} Under the assumptions of Theorem~\ref{th:S.6.0},  it holds that
\begin{eqnarray}
&&C_\ast(\mathcal{ L}, \mathcal{O};\mathbb{Z}_2)\cong
C_{\ast-\mu_\mathcal{O}}(\mathcal{O};\mathbb{Z}_2),\label{e:S.6.18.1}\\
&&C_\ast(\mathcal{L}, \mathcal{O};{\bf K})\cong
C_{\ast-\mu_\mathcal{O}}(\mathcal{O};\theta^-\otimes{\bf K}),\label{e:S.6.18.2}
\end{eqnarray}
 where $\theta^-$ is the orientation bundle of $N^-\mathcal{O}$. \\
{\bf ii)} Under the assumptions of Theorem~\ref{th:S.6.1}, if $\mathcal{ O}$ has
trivial normal bundle then  for any commutative group ${\bf K}$ and
$x\in\mathcal{ O}$,
\begin{eqnarray}
C_q(\mathcal{ L}, \mathcal{ O};{\bf K})\cong
\oplus^q_{j=0}C_{q-j-\mu_\mathcal{O}}(\mathcal{ L}^{\circ}_x,
\theta_x; {\bf K})\otimes H_j(\mathcal{ O};{\bf K})\quad\forall q=0,
1,\cdots.\label{e:S.6.18.3}
\end{eqnarray}
{\rm (}Consequently, every  $C_q(\mathcal{ L}, \mathcal{ O};{\bf
K})$ is isomorphic to finite direct sum $r_1{\bf
K}\oplus\cdots\oplus r_s{\bf K}\oplus H_j(\mathcal{O};{\bf K})$,
where each $r_i\in\{0,1\}$, see \cite[Remark 4.6]{Lu4}. {\rm )}
\end{corollary}

As in \cite{Bot,Was}, (\ref{e:S.6.18.1})--(\ref{e:S.6.18.2}) follow from (\ref{e:S.6.4}).

\begin{corollary}\label{cor:S.6.6}
Under the assumptions of Corollary~\ref{cor:S.6.5}, we have:
\begin{description}
\item[(i)] $\mathcal{ O}$ is a local minimum (so $\mu_\mathcal{ O}=0$) if and only if
$$ C_q(\mathcal{ L}, \mathcal{ O};{\bf K})\cong \delta_{q0}{\bf K}\quad\forall
q\in\Z\Longleftrightarrow C_0(\mathcal{ L},  \mathcal{ O};{\bf
K})\ne 0.
$$

\item[(ii)] $C_1(\mathcal{ L}, \mathcal{ O};{\bf K})\ne 0$ and ${\rm rank}N^0\mathcal{ O}=1$
 then  $\mu_\mathcal{O}=0$ and
\begin{equation}\label{e:S.6.19}
C_q(\mathcal{ L}, \mathcal{ O};{\bf K})\cong {\bf K}\otimes
H_{q-1}(\mathcal{ O};{\bf K})\;\forall q\in\Z.
\end{equation}

\item[(iii)] If  ${\rm rank}N^0\mathcal{ O}=1$ in the case $\mu_\mathcal{ O}=0$, then $\theta$ is of mountain pass type
(in the sense that some (and hence any) $\theta_x$ with
$x\in\mathcal{ O}$ is a critical point of $\mathcal{ L}\circ\exp_x$
on $N\mathcal{ O}_x$ of mountain pass type) if and only if
(\ref{e:S.6.19}) holds;

\item[(iv)] If $C_k(\mathcal{ L}, \mathcal{ O};{\bf K})\ne 0$ for $k={\rm rank}N^-\mathcal{ O}$
  then for any $q\in\Z$
$$
C_q(\mathcal{L}, \mathcal{ O};{\bf K})\cong
\left\{\begin{array}{ll}
  {\bf K}\otimes{\bf K}\;&\hbox{if}\; q\ge k,\\
  0&\hbox{if}\;q<k.
  \end{array}\right.
  $$
\end{description}
\end{corollary}

By the Peter-Weyl theorem, the compact Lie group has a faithful representation into
the real orthogonal group $O(m)$ (i.e., a injective Lie group homomorphism into $O(m)$)
for some integer $m>0$. Hence $G$ can be viewed as a subgroup of $O(m)$.
Let $E$ be the Hilbert manifold consisting of all $m$-orthogonal frames in the Hilbert
space $l^2$. It is a contractible space on which $G$ acts freely.
Let $B_G=E/G$ denote the classifying space, which is a Hilbert manifold. Then
$E\to B_G$ is a universal smooth principal $G$-bundle.
Let $E\times_G\mathcal{H}$ be the quotient of $E\times\mathcal{H}$
by the free diagonal action $g\cdot (p, u)=(gp, g\cdot u)$. This Hilbert manifold
 is a fiber space on $B_G$ with fiber $\mathcal{H}$.

The $G$-invariant functional $\mathcal{L}$ lifts a natural one
on $E\times\mathcal{H}$, $(p,u)\mapsto \mathcal{L}(u)$, and hence induces a functional $\mathcal{L}^E$
on $E\times_G\mathcal{H}$ with same smoothness as $\mathcal{L}$.
The critical orbit $\mathcal{O}$ of $\mathcal{L}$ corresponds to
a critical manifold $E\times_G\mathcal{O}$ of $\mathcal{L}^E$,
and they have the same Morse indexes. Moreover,
for a $G$-invariant Gromoll-Meyer pair $(W, W^-)$ of $\mathcal{O}$ (see \cite{Wa} for
its existence), $(E\times_G W, E\times_G W^-)$ is a Gromoll-Meyer of $E\times_G\mathcal{O}$.

Let $c=\mathcal{L}|_{\mathcal{O}}$ and $U$ be a $G$-invariant neighborhood of
${\cal O}$ with $K(\mathcal{L})\cap\mathcal{L}_c\cap U=\mathcal{O}$,
where $\mathcal{L}_c=\{x\in\mathcal{H}\,|\,\mathcal{L}(x)\le c\}$.
For any coefficient ring ${\bf K}$ and $q\in\mathbb{N}\cup\{0\}$, the $q^{\rm th}$ $G$ critical group
of $\mathcal{O}$ is defined by
\begin{eqnarray*}
C^\ast_G(\mathcal{L},\mathcal{O};{\bf K})&=&H^\ast_G(\mathcal{L}_c\cap U,
(\mathcal{L}_c\setminus\mathcal{O})\cap U;{\bf K})\\
&=&H^\ast(E\times_G(\mathcal{L}_c\cap U),
E\times_G((\mathcal{L}_c\setminus\mathcal{O})\cap U);{\bf K}).
\end{eqnarray*}
It is equal to $H^\ast_G(W, W^-;{\bf K})$, see \cite[page 76]{Ch}.
Moreover, if $\mathcal{O}$ is nondegenerate, it follows from (\ref{e:S.6.4})
and  the Thom isomorphism theorem (with twisted coefficients) that
\begin{equation}\label{e:S.6.20}
C_G^\ast(\mathcal{L}, \mathcal{ O};{\bf K})\cong
H_G^{\mu_{\cal O}-1}(\mathcal{O};\theta^-\otimes{\bf K}),
\end{equation}
where $\mu_{\cal O}$ is the Morse index of $\mathcal{O}$ and $\theta^-$
is the orientation bundle of $N^-{\cal O}$, see Theorem~7.5 on the page 75 of \cite{Ch}.

More generally, the corresponding versions of
Theorems~\ref{th:S.4.3},\ref{th:S.5.3} and \ref{th:S.5.4} can also be proved.
We  only write the following  since it is needed in the proof of Theorem~\ref{th:Bi.3.20} later.

\begin{theorem}[Parameterized Splitting Theorem around Critical Orbits]\label{th:S.6.7*}
Under the assumptions of Theorem~\ref{th:S.6.1},
suppose further that $G$-invariant functionals $\mathcal{G}_j\in C^1(\mathcal{H},\mathbb{R})$,
$j=1,\cdots,n$, have value zero and vanishing derivative at each point of $\mathcal{O}$, and also satisfy:
\begin{description}
\item[(i)] gradients $\nabla\mathcal{G}_j$ have  G\^ateaux derivatives
$\mathcal{G}''_j(u)$ at each point $u$ near $\mathcal{O}$,
\item[(ii)] $\mathcal{G}''_j(u)$ are continuous at each point $u\in\mathcal{O}$.
\end{description}
 If the critical orbit $\mathcal{O}$ is degenerate,
i.e., ${\rm rank}N^0\mathcal{O}>0$, then for sufficiently small $\epsilon>0$, $\delta>0$,\\
\noindent{\bf (I)}\;  there exists a unique continuous map
 \begin{equation}\label{e:S.6.21}
\mathfrak{h}:[-\delta,\delta]^n\times N^0\mathcal{O}(3\epsilon)\to N^+\mathcal{O}\oplus N^-\mathcal{O},\;
(\vec{\lambda}, x,v)\mapsto \mathfrak{h}_x(\vec{\lambda}, v),
\end{equation}
such that $\mathfrak{h}(\vec{\lambda},\cdot):N^0\mathcal{O}(3\epsilon)\to N^+\mathcal{O}\oplus N^-\mathcal{O},
\;(x,v)\mapsto \mathfrak{h}_x(\vec{\lambda}, v)$ is a $G$-equivariant topological  bundle
morphism that preserves the zero section and satisfies
\begin{eqnarray}\label{e:S.6.22}
(P^+_x+P^-_x)\circ{\bf \Pi}_x\nabla(\mathcal{L}_{\vec{\lambda}}\circ\exp_x)(v+ \mathfrak{h}_x(\vec{\lambda}, v))=0\quad\forall
(x,v^0)\in N^0\mathcal{O}(\epsilon),
\end{eqnarray}
where $\mathcal{L}_{\vec{\lambda}}=\mathcal{L}+ \sum^n_{j=1}\mathcal{G}_j$;\\
\noindent{\bf (II)}\;  there exists a continuous map
\begin{eqnarray}\label{e:S.6.23}
\Phi: [-\delta,\delta]^n\times N^0\mathcal{ O}(\epsilon)\oplus N^+\mathcal{
O}(\epsilon)\oplus N^-\mathcal{ O}(\epsilon)\to N\mathcal{ O}
\end{eqnarray}
such that $\Phi(\vec{\lambda},\cdot):N^0\mathcal{ O}(\epsilon)\oplus N^+\mathcal{
O}(\epsilon)\oplus N^-\mathcal{ O}(\epsilon)\to N\mathcal{ O}$ is
 a  $G$-equivariant homeomorphism onto an open neighborhood of
the zero section preserving fibers, and that
\begin{eqnarray}\label{e:S.6.24}
\mathcal{L}_{\vec{\lambda}}\circ\exp\circ\Phi(\vec{\lambda}, x, v, u^++
u^-)=\|u^+\|^2_x-\|u^-\|^2_x+ \mathcal{ L}_{\vec{\lambda}}\circ\exp_x(v+
\mathfrak{h}_x(\vec{\lambda}, v))
\end{eqnarray}
for any $\vec{\lambda}\in [-\delta,\delta]^n$, $x\in\mathcal{O}$ and  $(v, u^+, u^-)\in N^0\mathcal{
O}(\epsilon)_x\times N^+\mathcal{ O}(\epsilon)_x\times N^-\mathcal{
O}(\epsilon)_x$; \\
\noindent{\bf (III)}\; for each $(\vec{\lambda}, x)\in [-\delta,\delta]^n\times\mathcal{O}$ the functional
\begin{eqnarray}\label{e:S.6.25}
N^0\mathcal{O}(\epsilon)_x\to\R,\;v\mapsto \mathcal{L}_{\vec{\lambda},x}^\circ(v):=
\mathcal{L}_{\vec{\lambda}}\circ\exp_x(v+ \mathfrak{h}_x(\vec{\lambda},v))
\end{eqnarray}
is $G_x$-invariant, of class $C^{1}$,  and satisfies
$$
D\mathcal{L}_{\vec{\lambda}, x}^\circ(v)v':=
(\nabla(\mathcal{L}_{\vec{\lambda}}\circ\exp_x)(v+ \mathfrak{h}_x(\vec{\lambda},v)),v'),\quad\;\forall v'\in N^0\mathcal{O}_x.
$$
Moreover,  if the critical orbit $\mathcal{O}$ is nondegenerate,
i.e., ${\rm rank}N^0\mathcal{O}=0$, then
$\mathfrak{h}$ does not appear,
$\Phi$ is from $[-\delta,\delta]^n\times N^+\mathcal{
O}(\epsilon)\oplus N^-\mathcal{ O}(\epsilon)$ to $N\mathcal{O}$, and (\ref{e:S.6.24}) becomes
\begin{eqnarray}\label{e:S.6.26}
\mathcal{L}_{\vec{\lambda}}\circ\exp\circ\Phi(\vec{\lambda}, x,  u^++
u^-)=\|u^+\|^2_x-\|u^-\|^2_x+ \mathcal{L}_{\vec{\lambda}}|_{\mathcal{O}}
\end{eqnarray}
for any $\vec{\lambda}\in [-\delta,\delta]^n$, $x\in\mathcal{O}$ and  $(u^+, u^-)\in  N^+\mathcal{ O}(\epsilon)_x\times N^-\mathcal{O}(\epsilon)_x$.
\end{theorem}

\subsection{Proof of Theorem~\ref{th:MP.2}}\label{sec:MP}
\setcounter{equation}{0}

 Without loss of generality we may assume $\theta\in V$ and $u_0=\theta$.
By the assumption we have a $C^2$ reduction functional $\mathcal{L}^\circ: B_H(\theta,\delta)\cap H^0\to
\mathbb{R}$ such that $\theta$ is the unique critical point of it and $\mathcal{L}^\circ(z)=o(\|z\|^2)$.
Clearly, we can shrink $\delta>0$ so that $\delta<\min\{r,1\}$, $\bar{B}_H(\theta,\delta)\cap H^0\subset V$
and $\omega$ in Lemma~\ref{lem:S.2.2} satisfies
\begin{eqnarray}\label{e:MP.0}
\omega(z+\varphi(z))<\frac{1}{2}\min\{a_0,a_1\},\quad\forall z\in B_H(\theta,\delta)\cap H^0.
\end{eqnarray}
By the uniqueness of solutions we can also require that if
$v\in {B}_{H}(\theta,\delta)$ satisfies $(I-P^0)\nabla{\mathcal{L}}(v)=0$ then
$v=z+\varphi(z)$ for some $z\in B_H(\theta,\delta)\cap H^0$.

Take a smooth function $\rho:[0,\infty)\to\mathbb{R}$ satisfying: $0\le\rho\le 1$,
$\rho(t)=1$ for $t\le\delta/2$, $\rho(t)=0$ for $t\ge\delta$, and $|\rho'(t)|<4/\delta$.
For $b\in H^0$ we set $\mathcal{L}_b^\circ(z)=\mathcal{L}^\circ(z)+ \rho(\|z\|)(b,z)_H$. Then
\begin{eqnarray}\label{e:MP.1}
D\mathcal{L}_b^\circ(z)\xi&=&D\mathcal{L}(z+\varphi(z))(\xi+\varphi'(z)\xi)+ \rho(\|z\|)(b,\xi)_H\nonumber\\
&&+\rho'(\|z\|)(b, z)_H(z/\|z\|,\xi)_H,\quad\forall\xi\in H^0.
\end{eqnarray}
Note that $\|D\mathcal{L}^\circ(z)\|\ge \nu$ for some $\nu>0$ and for all $z\in
\bar{B}_{H^0}(\theta,\delta)\setminus {B}_{H^0}(\theta, {\delta/2})$.
Suppose $\|b\|<\nu/5$. Then
\begin{eqnarray}\label{e:MP.2}
\|D\mathcal{L}_b^\circ(z)\|&=&\big\|D\mathcal{L}^\circ(z)+ \rho(\|z\|)b+ (b,z)_H\rho'(\|z\|)z/\|z\|\big\|
\ge  \nu- 5\|b\|>0,
\end{eqnarray}
and therefore $\mathcal{L}_b^\circ$ has no critical point in
$\bar{B}_{H^0}(\theta,\delta)\setminus {B}_{H^0}(\theta,{\delta/2})$.
By Sard's theorem we may take arbitrary small $b$ such that the critical points of
$\mathcal{L}_b^\circ$, if any, are nondegenerate.
Choose a $C^2$ function $\beta:H\to\mathbb{R}$ such that $\beta(u)=0$ for $u\in H\setminus B_H(\theta,r)$,
and $\beta(u)=1$ for $u\in B_H(\theta,\delta)$. Clearly,
we can require $\|\beta^{(i)}(u)\|\le M$ for some $M>0$, $i=0,1,2$ and for all $u\in H$. Define
\begin{equation}\label{e:MP.3}
\tilde{\mathcal{L}}_b(u)={\mathcal{L}}(u)+ \beta(u)\rho(\|P^0u\|)(b,P^0u)_H
\end{equation}
Clearly, it satisfies (iii). If $b$ is taken to be very small, then (ii) is satisfied.
Moreover, since $\mathcal{L}$ satisfies the (PS) condition, $\|D\mathcal{L}(u)\|\ge c$
for some $c>0$ and for all $u\in {B}_{H}(\theta, r)\setminus {B}_{H}(\theta,\delta)$.
Hence all critical points of $\tilde{\mathcal{L}}_b$ belong to ${B}_{H}(\theta,\delta)$
as long as $b$ is small enough.

Let $v\in {B}_{H}(\theta,\delta)$ be a critical point of
$\tilde{\mathcal{L}}_b$. Then
\begin{eqnarray}\label{e:MP.4}
0=\tilde{\mathcal{L}}'_b(v)\xi&=&(\nabla{\mathcal{L}}(v),\xi)_H+ \rho(\|P^0v\|)(b, P^0\xi)_H\nonumber\\
&&+\rho'(\|P^0v\|)(b, P^0v)_H(P^0v, P^0\xi)_H/\|P^0v\|,\quad\forall\xi\in H.
\end{eqnarray}
Since $\rho(\|P^0v\|)=1$ for $\|P^0v\|\le\|v\|<\delta$,
 this implies $(\nabla{\mathcal{L}}(v),\xi)_H=0$ for any $\xi\in H^+\oplus H^-$,
i.e., $((I-P^0)\nabla{\mathcal{L}}(v),\xi)_H=0$ for any $\xi\in H$. It follows that
$v=z+\varphi(z)$ for some $z\in B_H(\theta,\delta)\cap H^0$. [This $z$ is nonzero. Otherwise,
$v=\theta$. But $\theta$ is not a critical point of $\tilde{\mathcal{L}}_b$ if $b\ne\theta$].
Note that  $(\nabla{\mathcal{L}}(z+\varphi(z)),\varphi'(z)\xi)_H=0\;\forall\xi\in H^0$
because $\varphi'(z)\xi\in H^+\oplus H^-$. (\ref{e:MP.4}) leads to
\begin{eqnarray*}
0&=&(\nabla{\mathcal{L}}(z+\varphi(z)),\xi)_H+ \rho(\|z\|)(b,\xi)_H
+\rho'(\|z\|)(b, z)_H(z, \xi)_H/\|z\|\\
&=&(\nabla{\mathcal{L}}(z+\varphi(z)),\xi)_H+ (\nabla{\mathcal{L}}(z+\varphi(z)),\varphi'(z)\xi)_H\\
&&+\rho(\|z\|)(b,\xi)_H +\rho'(\|z\|)(b, z)_H(z, \xi)_H/\|z\|\quad\forall \xi\in H^0,
\end{eqnarray*}
and therefore $D\mathcal{L}_b^\circ(z)=0$ by (\ref{e:MP.1}). That is, $z$ is a critical point of $\mathcal{L}_b^\circ$,
and so $z\in B_{H^0}(\theta,{\delta/2})$ by (\ref{e:MP.2}). It follows from (\ref{e:MP.3}) that
\begin{eqnarray}\label{e:MP.5}
\tilde{\mathcal{L}}''_b(v)(\xi,\eta)=({\mathcal{L}}''(v)\xi,\eta)_H,\quad\forall\xi,\eta\in H.
\end{eqnarray}

Let us prove that $v$ is a nondegenerate critical point of $\tilde{\mathcal{L}}_b$.
Suppose $\xi\in{\rm Ker}(\tilde{\mathcal{L}}''_b(v))$. Then $(\tilde{\mathcal{L}}''_b(v)\xi,\eta)_H=0$ for any $\eta\in H$.
By (\ref{e:MP.5}), we have
\begin{eqnarray}\label{e:MP.5.1}
\tilde{\mathcal{L}}''(v)(\xi,\eta)=({\mathcal{L}}''(z+\varphi(z))\xi,\eta)_H=0,\quad\forall\eta\in H.
\end{eqnarray}
Decompose $\xi$ into $\xi^0+\xi^\bot$, where $\xi^0\in H^0$ and $\xi^\bot\in H^+\oplus H^-$. A direct computation yields
\begin{eqnarray}\label{e:MP.6}
({\mathcal{L}}''(z+\varphi(z))\xi^0, \eta+\varphi'(z)\eta)_H+
({\mathcal{L}}''(z+\varphi(z))\xi^\bot, \eta+\varphi'(z)\eta)_H
=0,\quad\forall\eta\in H^0.
\end{eqnarray}
Note that $(I-P^0)\nabla{\mathcal{L}}(w+\varphi(w))=0$ for any $w\in B_{H^0}(\theta,\delta)$.
Hence $(\nabla{\mathcal{L}}(w+\varphi(w)), \zeta)_H=0$ for any $w\in B_{H^0}(\theta,\delta)$
and $\zeta\in H^+\oplus H^-$. Differentiating this equality with respect to $w$ yields
$$
(\mathcal{L}''(w+\varphi(w))(\tau+\varphi'(w)\tau), \zeta)_H=0,\quad \forall\tau\in H^0,\;
\forall w\in B_{H^0}(\theta,\delta),\;\forall\zeta\in H^+\oplus H^-.
$$
In particular, we have $({\mathcal{L}}''(z+\varphi(z))\xi^\bot, \eta+\varphi'(z)\eta)_H
=0$ for all $\eta\in H^0$. This and (\ref{e:MP.6}) yield
\begin{eqnarray}\label{e:MP.7}
d^2\mathcal{L}^\circ(z)(\xi^0,\eta)=({\mathcal{L}}''(z+\varphi(z))\xi^0, \eta+\varphi'(z)\eta)_H=0,\quad\forall\eta\in H^0.
\end{eqnarray}
Moreover, $d^2\mathcal{L}_b^\circ(z')=d^2\mathcal{L}^\circ(z')$ for any $z'\in B_{H^0}(\theta,\delta/2)$
by the construction of $\mathcal{L}_b^\circ$. We obtain that
$d^2\mathcal{L}_b^\circ(z)(\xi,\eta)=0$ for all $\eta\in H^0$. Hence $\xi^0=\theta$ since $z$ is a nondegenerate critical point of $\mathcal{L}_b^\circ$ by the choice of $b$, and thus $\xi=\xi^\bot$. By (\ref{e:MP.5})--(\ref{e:MP.5.1}) and $(\tilde{\mathcal{L}}''_b(v)\xi,\eta)_H=0\;\forall\eta\in H$, we get
\begin{eqnarray}\label{e:MP.8}
({\mathcal{L}}''(z+\varphi(z))\xi^\bot,\eta)_H= ({\mathcal{L}}''(z+\varphi(z))\xi,\eta)_H=0,\quad\forall\eta\in H.
\end{eqnarray}
Hence ${\mathcal{L}}''(z+\varphi(z))\xi^\bot=0$. 
Decompose $\xi^\bot$ into $\xi^++\xi^-$, where $\xi^+\in H^+$ and $\xi^-\in H^-$.
Then ${\mathcal{L}}''(z+\varphi(z))\xi^+=-{\mathcal{L}}''(z+\varphi(z))\xi^-$.
By Lemma~\ref{lem:S.2.2} and (\ref{e:MP.0}) we derive
\begin{eqnarray*}
&&a_1\|\xi^+\|^2\le ({\mathcal{L}}''(z+\varphi(z))\xi^+,\xi^+)_H= -({\mathcal{L}}''(z+\varphi(z))\xi^-,\xi^+)_H\le
\frac{a_1}{2}\|\xi^+\|\cdot\|\xi^-\|,\\
&&-a_0\|\xi^-\|^2\ge ({\mathcal{L}}''(z+\varphi(z))\xi^-,\xi^-)_H=-({\mathcal{L}}''(z+\varphi(z))\xi^+,\xi^-)_H\ge-
\frac{a_0}{2}\|\xi^-\|\cdot\|\xi^+\|.
\end{eqnarray*}
These imply $\|\xi^+\|\le\|\xi^-\|/2$ and $\|\xi^-\|\le\|\xi^+\|/2$. Hence $\xi^+=\xi^-=\theta$ and so $\xi^\bot=\theta$.
This shows that $v$ is a nondegenerate critical point of $\tilde{\mathcal{L}}_b$.
Lemma~\ref{lem:S.2.2} and (\ref{e:MP.5}) give rise to
\begin{eqnarray*}
&&\tilde{\mathcal{L}}''_b(v)(\xi,\xi)=({\mathcal{L}}''(v)\xi,\xi)_H\ge a_1\|\xi\|^2,\quad\forall\xi\in H^+,\\
&&\tilde{\mathcal{L}}''_b(v)(\xi,\xi)=({\mathcal{L}}''(v)\xi,\xi)_H\le -a_0\|\xi\|^2,\quad\forall\xi\in H^-.
\end{eqnarray*}
But $H=H^+\oplus H^0\oplus H^-$, $\dim H^-=m^-$ and $\dim H^0=n^0$. These show that
the Morse index of $\tilde{\mathcal{L}}''_b(v)$ must sit in $[m^-, m^-+n^0]$. (iv) is proved.

We also need to show that $\tilde{\mathcal{L}}_b$ satisfies Hypothesis~\ref{hyp:MP.1}
on $V$ if $b$ is small enough. By (\ref{e:MP.3}) we have for all $\xi,\eta\in H$,
\begin{eqnarray}\label{e:MP.9}
\tilde{\mathcal{L}}'_b(u)\xi&=&{\mathcal{L}}'(u)\xi+ (\beta'(u)\xi)\rho(\|P^0u\|)(b,P^0u)_H+
\beta(u)\rho(\|P^0u\|)(b, P^0\xi)_H\nonumber\\
&&+\beta(u)\rho'(\|P^0u\|)(b, P^0u)_H(P^0u, P^0\xi)_H/\|P^0u\|
\end{eqnarray}
and
\begin{eqnarray*}
&&(\tilde{\mathcal{L}}''_b(u)\eta,\xi)_H=({\mathcal{L}}''(u)\eta,\xi)_H+ (\beta''(u)\eta,\xi)_H\rho(\|P^0u\|)(b,P^0u)_H\\
&&+(\beta'(u)\xi)\rho(\|P^0u\|)(b,P^0\eta)_H+ (\beta'(u)\xi)(b,P^0u)_H\rho'(\|P^0u\|)(P^0u, P^0\eta)_H/\|P^0u\|\\
&&+(\beta'(u)\eta)\rho(\|P^0u\|)(b, P^0\xi)_H+ \beta(u)(b, P^0\xi)_H\rho'(\|P^0u\|)(P^0u, P^0\eta)_H/\|P^0u\|\\
&&+(\beta'(u)\eta)\rho'(\|P^0u\|)(b, P^0u)_H(P^0u, P^0\xi)_H/\|P^0u\|\\
&&+\beta(u)\big(\rho''(\|P^0u\|)(P^0u, P^0\eta)_H/\|P^0u\|\big)(b, P^0u)_H(P^0u, P^0\xi)_H/\|P^0u\|\\
&&+\beta(u)\rho'(\|P^0u\|)(b, P^0\eta)_H(P^0u, P^0\xi)_H/\|P^0u\|\\
&&+\beta(u)\rho'(\|P^0u\|)(b, P^0u)_H\left[(P^0\eta, P^0\xi)_H/\|P^0u\|- (P^0u, P^0\xi)_H(P^0u, P^0\eta)_H/\|P^0u\|^3\right] \\
&&=({\mathcal{L}}''(u)\eta,\xi)_H+ \Upsilon(u,b,\xi,\eta).
\end{eqnarray*}
By the constructions of $\beta$ and $\rho$ we have a constant $M_2>0$ such that
$$
|\Upsilon(u,b,\xi,\eta)|\le M_2\|b\|\cdot\|\xi\|\cdot\|\eta\|
$$
for all $u\in V$ and $\xi,\eta\in H$. Since we may require that  the support of $\beta$ can be contained
a neighborhood of $\theta$ on which (iii) of Hypothesis~\ref{hyp:MP.1} holds, for sufficiently small $b$
the positive definite part of $\tilde{\mathcal{L}}''_b$, $\tilde P$ given by
 $(\tilde P(u)\xi,\eta)_H=(P(u)\xi,\eta)_H+ \Upsilon(u,b,\xi,\eta)$,
   is also uniformly positive definite on this neighborhood.
Hence  $\tilde{\mathcal{L}}_b$ satisfies Hypothesis~\ref{hyp:MP.1}.

By (\ref{e:MP.3}) and (\ref{e:MP.9}) we have positive numbers $M_i$, $i=0,1$, such that
$|\tilde{\mathcal{L}}_b(u)-{\mathcal{L}}(u)|\le M_0\|b\|$ and
$\|\tilde{\mathcal{L}}'_b(u)\xi-{\mathcal{L}}'(u)\xi\|\le M_1\|b\|\cdot\|\xi\|$ for
all $u\in V$ and $\xi\in H$. So (ii) and (iii) can be satisfied if $b$ is small.

Finally, let us prove that $\tilde{\mathcal{L}}_b$ satisfies  the (PS) condition for small $b$.
By (ii) and (iii) in Hypothesis~\ref{hyp:MP.1}, there exists  $\epsilon\in (0,\delta/2)$ such that
for all $u\in B_H(\theta,\epsilon)$ and $\xi\in H$,
\begin{eqnarray}\label{e:MP.10}
(P(u)\xi,\xi)_H\ge C_0\|\xi\|^2\quad\hbox{and}\quad \|Q(u)-Q(\theta)\|<C_0/2.
\end{eqnarray}
 Since $\mathcal{L}$ satisfies  the (PS) condition and $\theta$ is a unique critical point
 of $\mathcal{L}$ in $V$, we have $\nu_0>0$ such that $\|\mathcal{L}'(u)\|\ge\nu_0$ for
 all $u\in V\setminus B_H(\theta,\epsilon)$.
Let us choose $b$ so small that $\|\tilde{\mathcal{L}}'_b(u)\|\ge \nu_0/2$ for all $u\in V\setminus B_H(\theta,\epsilon)$.
Then if $\{u_n\}_n\subset V$ satisfies $\tilde{\mathcal{L}}'_b(u_n)\to 0$ and $\sup_n|\tilde{\mathcal{L}}_b(u_n)|<\infty$,
then $\{u_n\}_n$ muse be contained in $B_H(\theta,\epsilon)$. It follows that
$\nabla\tilde{\mathcal{L}}_b(u_n)=\nabla\tilde{\mathcal{L}}(u_n)+b$ for all $n$.
 For any two natural numbers $n$ and $m$,
 using the mean value theorem we have $\tau\in (0,1)$ such that
 \begin{eqnarray*}
 &&(\nabla\mathcal{L}(u_n)-\nabla\mathcal{L}(u_m), u_n-u_m)_H=(B(\tau u_n+ (1-\tau)u_m)(u_n-u_m), u_n-u_m)_H\\
&&=(P(\tau u_n+ (1-\tau)u_m)(u_n-u_m), u_n-u_m)_H+ (Q(\theta)(u_n-u_m), u_n-u_m)_H\\
&&+([Q(\tau u_n+ (1-\tau)u_m)-Q(\theta)](u_n-u_m), u_n-u_m)_H\\
&&\ge C_0\|u_n-u_m\|^2-\frac{C_0}{2}\|u_n-u_m\|^2+(Q(\theta)(u_n-u_m), u_n-u_m)_H
\end{eqnarray*}
by (\ref{e:MP.10}).
Passing to a subsequence we may assume $u_n\rightharpoonup u_0$. Since $Q(\theta)$ is compact,
$Q(\theta)u_n\to Q(\theta)u_0$ and so $(Q(\theta)(u_n-u_m), u_n-u_m)_H\to 0$ as $n,m\to\infty$.
Note that $\nabla\mathcal{L}(u_n)-\nabla\mathcal{L}(u_m)=
(\nabla\mathcal{L}(u_n)+b)-(\nabla\mathcal{L}(u_m)+b)\to 0$ as $n,m\to\infty$. It follows from
the above inequalities that $\|u_n-u_m\|\to 0$  as $n,m\to\infty$. This implies
$u_n\to u_0$.  Theorem~\ref{th:MP.2} is proved.


\section{Bifurcations for potential operators}\label{sec:Bi}
\setcounter{equation}{0}

In this section,  some previous bifurcation theorems,
such as those by Rabinowitz \cite{Rab}, by Fadelll and Rabinowitz  \cite{FaRa1, FaRa2},  and
by Chang and Wang \cite{Ch2, Wa,Wa1},  by Chow and Lauterbach \cite{ChowLa} and
 by Bartsch and Clapp \cite{ Ba1, BaCl},
were generalized so that they can be used to study
variational bifurcation  for the functional $\mathcal{F}$
in (\ref{e:1.8}). Our methods are mainly based our Morse lemma, Theorem~\ref{th:S.1.1},
and the parameterized splitting and shifting theorems, Theorems~\ref{th:S.5.3},
~\ref{th:S.5.4}. The latter suggest that multiparameter bifurcations can be studied similarly;
we here give two, Theorems~\ref{th:Bi.2.2},~\ref{th:Bi.2.3}.

\subsection{Generalizations of a bifurcation theorem by Chow and Lauterbach}\label{sec:B.1}

Let $H$ be a real Hilbert space,  $I$  an open interval containing $0$ in $\mathbb{R}$,
and $\{B_\lambda\}_{\lambda\in I}$ a family of bounded linear self-adjoint operators on $H$
such that $\|B_\lambda-B_0\|\to 0$ as $\lambda\to 0$.
Suppose that  $0$ is an isolated point of the spectrum $\sigma(B_0)$
with $n=\dim{\rm Ker}(B_0)\in (0, \infty)$, and that
${\rm Ker}(B_\lambda)=\{0\}\;\forall\pm\lambda\in (0,\varepsilon_0)$ for some
positive number $\varepsilon_0\ll 1$. By the arguments on the pages 107 and 203 in \cite{Ka},
for each $\lambda\in (-\varepsilon_0,\varepsilon_0)\setminus\{0\}$, $B_\lambda$ has
$n$ eigenvalues near zero, and none of them is zero. In Kato's terminology in \cite[page 107]{Ka},
we have the so-called $0$-group ${\rm eig}_0(B_\lambda)$ consisting
 of eigenvalues of $B_\lambda$ which approach $0$ as $\lambda\to 0$.
 Let $r(B_\lambda)$ be the number of elements in ${\rm eig}_0(B_\lambda)\cap\mathbb{R}^-$ and
 \begin{equation}\label{e:Bi.1.1}
 r^+_{B_\lambda}=\lim_{\lambda\to 0+}r(B_\lambda),\qquad
 r^-_{B_\lambda}=\lim_{\lambda\to 0-}r(B_\lambda).
 \end{equation}

\begin{theorem}\label{th:Bi.1.1}
 Let $U$ be an open neighborhood of the origin of a real Hilbert space $H$,
 and $I$  an open interval containing $0$ in $\mathbb{R}$,
  $\mathcal{F}\in C^0(I\times V,\mathbb{R})$ such that
    $\mathcal{L}:=\mathcal{F}_\lambda=\mathcal{F}(\lambda,\cdot)$ satisfy  Hypothesis~\ref{hyp:1.1} with $X=H$ for each $\lambda\in I$. Suppose that one of the following two conditions is satisfied.
    \begin{description}
\item[(1)]  For some small $\delta>0$, $\lambda\mapsto \mathcal{F}_\lambda$
    is continuous at $\lambda=0$ in $C^1(\bar{B}_H(\theta, \delta))$ topology.
\item[(2)] For some small $\delta>0$, $\lambda\mapsto \mathcal{F}_\lambda$
    is continuous at $\lambda=0$ in $C^0(\bar{B}_H(\theta, \delta))$ topology; and
    for every sequences $\lambda_n\to\lambda_0$ in $I$ and  $\{u_n\}_{n\ge 1}\subset\bar{B}_H(\theta, \delta)$ with $\mathcal{F}'_{\lambda_n}(u_n)\to\theta$ and $\{\mathcal{F}_{\lambda_n}(u_n)\}_{n\ge 1}$ bounded, there exists  a subsequence $u_{n_k}\to u_0\in \bar{B}_H(\theta, \delta)$
with $\mathcal{F}'_{\lambda_0}(u_0)=0$.
 \end{description}
    Then
  \begin{description}
\item[(I)] If  $(\theta,0)$ is not a bifurcation point of the equation
\begin{equation}\label{e:Bi.1.2}
\mathcal{F}'_u(\lambda, u)=0,\quad(\lambda,u)\in I\times V,
\end{equation}
(i.e., $(0,\theta)$ is not in the closure of $\{(\lambda,u)\in I\times V\,|\,
\mathcal{F}'_u(\lambda,u)=0,\,u\ne\theta\}$),
then critical groups  $C_\ast(\mathcal{F}_\lambda, \theta;{\bf K})$ are well-defined and have no changes as $\lambda$ varies in a small neighborhood of $0$.
\item[(II)]  $(0,\theta)$ is  a bifurcation point of the equation
(\ref{e:Bi.1.2}) if the following conditions are also satisfied:
\begin{description}
\item[(a)]  ${\rm Ker}(d^2\mathcal{F}_\lambda(\theta))=\{\theta\}$
 for small $|\lambda|\ne 0$;
\item[(b)] $d^2\mathcal{F}_\lambda(\theta)\to d^2\mathcal{F}_0(\theta)$  as $\lambda\to 0$;
\item[(c)] $0\in\sigma(d^2\mathcal{F}_0(\theta))$ (and so is an isolated point of the spectrum $\sigma(d^2\mathcal{F}_0(\theta))$ and an eigenvalue of $d^2\mathcal{F}_0(\theta)$ of the finite multiplicity by \cite[Lemma~2.2]{BoBu});
\item[(d)]  $r^+_{d^2\mathcal{F}_\lambda(\theta)}\ne
r^-_{d^2\mathcal{F}_\lambda(\theta)}$.
\end{description}
\end{description}
\end{theorem}
\begin{proof}
{\bf Step 1}.\quad
This is a direct consequence of the stability of critical groups. In fact,
since $(\theta,0)$ is not a bifurcation point of the equation
(\ref{e:Bi.1.2}),  we may find $0<\varepsilon_0\ll 1$ and a small bounded neighborhood $W$
of $\theta\in H$ with $\overline{W}\subset B_H(\theta,\delta)$ such that for each
$\lambda\in(-\varepsilon_0,\varepsilon_0)$ the functional $\mathcal{F}_\lambda$ has a unique
critical point $\theta$ sitting in $\overline{W}$. Note that $\mathcal{F}_\lambda$ is of class $(S)_+$.
We can assume that it satisfies the (PS) condition in $\overline{W}$ by shrinking $W$
(if necessary). After shrinking $\varepsilon_0>0$ (if necessary),
 we may use the stability of critical groups (cf.  \cite[Theorem~III.4]{ChGh} and \cite[Theorem~5.1]{CorH}) to derive
\begin{equation}\label{e:Bi.1.3}
C_\ast(\mathcal{F}_\lambda, \theta;{\bf K})=C_\ast(\mathcal{F}_0, \theta;{\bf K}),\quad\forall\lambda\in (-\varepsilon_0,\varepsilon_0)
\end{equation}
provided that (1) holds. If (2) is satisfied the same claim is obtained by
 \cite[Theorem~3.6]{CiDe}.

\noindent{\bf Step 2}.\quad
By a contradiction, suppose that $(\theta,0)$ is not a bifurcation point of the equation
(\ref{e:Bi.1.2}). Then we have (\ref{e:Bi.1.3}) from (I).
By (a),  $\theta$ is a nondegenerate critical point of $\mathcal{F}_\lambda$. It follows from (\ref{e:Bi.1.3})
and Theorem~\ref{th:S.1.1} that all $\mathcal{F}_\lambda$, $0<|\lambda|<\varepsilon_0$,
have the same Morse index $\mu_\lambda$ at $\theta\in H$, i.e.,  $(-\varepsilon_0, \varepsilon_0)\setminus\{0\}\ni\lambda\mapsto \mu_\lambda$  is constant.

By \cite[Proposition~B.2]{Lu2}, each $\varrho\in\sigma(d^2\mathcal{F}_0(\theta))\cap\{t\in\mathbb{R}^-\,|\,
t\le 0\}$ is an isolated point in $\sigma(d^2\mathcal{F}_0(\theta))$, which is also an
eigenvalue of finite multiplicity. (This can also be derived from \cite[Lemma~2.2]{BoBu}).
Since $0\in\sigma(d^2\mathcal{F}_0(\theta))$ by (c), we may assume $\sigma(d^2\mathcal{F}_0(\theta))\cap\{t\in\mathbb{R}^-\,|\,
t\le 0\}=\{0,\varrho_1,\cdots,\varrho_k\}$, where $\mu_i$ has multiplicity $s_i$ for each $i=1,\cdots,k$.
As above, by this, (b) and the arguments on the pages 107 and 203 in \cite{Ka},
if $0<|\lambda|$ is small enough, $d^2\mathcal{F}_\lambda(\theta)$
 has exactly $s_i$ (possible same) eigenvalues near $\mu_i$, but total dimension
 of corresponding eigensubspaces is equal to that of eigensubspace of $\varrho_i$.
Hence if $\lambda\in (0, \varepsilon_0)$ (resp. $-\lambda\in (0,\varepsilon_0)$) is small enough
we obtain
$$
\mu_\lambda=\mu_0+ r^+_{d^2\mathcal{F}_\lambda(\theta)}\quad\hbox{
(resp. $\mu_{-\lambda}=\mu_0+ r^-_{d^2\mathcal{F}_\lambda(\theta)}$).}
$$
These and (d) imply $\mu_\lambda-\mu_{-\lambda}=r^+_{d^2\mathcal{F}_\lambda(\theta)}-
r^-_{d^2\mathcal{F}_\lambda(\theta)}\ne 0$ for small  $\lambda\in (0, \varepsilon_0)$,
which contradicts the above claim that  $(-\varepsilon_0, \varepsilon_0)\setminus\{0\}\ni\lambda\mapsto \mu_\lambda$ is constant.
\end{proof}

Part (II) in Theorem~\ref{e:Bi.1.1} is
a partial  generalization of a bifurcation theorem due to \cite{ChowLa}.
The latter  requires: 1) $\mathcal{F}\in C^2(I\times V,\mathbb{R})$
(so (b) holds naturally), 2) $0<\dim{\rm Ker}(d^2\mathcal{F}_0(\theta))<\infty$,
3)  $0$ is isolated in $\sigma(d^2\mathcal{F}_0(\theta))$, 4) (d) is satisfied.
Its proof is based on center manifold theory, which is different from ours.

\subsection{Generalizations of Rabinowitz bifurcation theorem}\label{sec:B.2}

Since the birth of the Rabinowitz bifurcation theorem \cite{Rab} some generalizations
and new proofs are given, see \cite{Ch2,ChWa}, \cite{CorH}, \cite{IoSch} and \cite{Wa,Wa1}, etc.

Our generalization will reduce to the following
result, which may be obtained as a corollary of \cite[Theorem~2]{IoSch}.

\begin{theorem}[\hbox{\cite[Theorem~5.1]{Can}}]\label{th:Bi.2.1}
 Let $X$ be a finite dimensional normed space, let $\delta>0$, $\lambda^\ast\in\mathbb{R}$ and
for every $\lambda\in [\lambda^\ast-\delta, \lambda^\ast+\delta]$, let
$\phi_\lambda:B(\theta,\delta)\to\mathbb{R}$ be a function of class $C^1$.
Assume that
\begin{description}
\item[a)] the functions $\{(\lambda,u)\to\phi_\lambda(u)\}$ and
$\{(\lambda,u)\to\phi'_\lambda(u)\}$  are continuous on
$[\lambda^\ast-\delta, \lambda^\ast+\delta]\times B(\theta,\delta)$;
\item[b)] $u=\theta$ is a critical point of $\phi_{\lambda^\ast}$; $\phi_\lambda$ has an isolated local minimum (maximum) at zero for every $\lambda\in (\lambda^\ast,\lambda^\ast+\delta]$ and an isolated local maximum (minimum) at
zero for every $\lambda\in [\lambda^\ast-\delta, \lambda^\ast)$.
\end{description}
Then one at least of the following assertions holds:
\begin{description}
\item[i)] $u=\theta$ is not an isolated critical point of $\phi_{\lambda^\ast}$;

\item[ii)] for every $\lambda\ne\lambda^\ast$ in a neighborhood of $\lambda^\ast$ there is a nontrivial critical point of
$\phi_\lambda$ converging to zero as $\lambda\to\lambda^\ast$;

\item[iii)] there is a one-sided (right or left) neighborhood of $\lambda^\ast$ such that for every
$\lambda\ne\lambda^\ast$ in the neighborhood there are two distinct nontrivial critical points of $\phi_\lambda$
converging to zero as $\lambda\to\lambda^\ast$.
\end{description}
\end{theorem}
It was generalized to infinite dimension spaces in \cite[Theorem~4.2]{CorH}.

The following is a generalization of the  necessity part of Theorem~12 in \cite[Chapter~4, \S4.3]{Skr1}
(including the classical Krasnoselsi potential bifurcation theorem \cite{Kra}).
The  sufficiency  part of Theorem~12 in \cite[Chapter~4, \S4.3]{Skr1} is contained in the case that the condition (a) in Theorem~\ref{th:Bi.2.4} holds.

\begin{theorem}\label{th:Bi.2.2}
 Let $U$ be an open neighborhood of the origin of a real Hilbert space $H$. Suppose
\begin{description}
\item[(i)]  $\mathcal{F}\in C^1(U,\mathbb{R})$ satisfies
 Hypothesis~\ref{hyp:1.1} with $X=H$ as the functional $\mathcal{L}$ there;

\item[(ii)]  $\mathcal{G}_j\in C^1(U,\mathbb{R})$ satisfies
$\mathcal{G}'_j(\theta)=\theta$, $j=1,\cdots,n$, and each gradient $\mathcal{G}'_j$ has the G\^ateaux derivative $\mathcal{G}''_j(u)$ at any $u\in U$, which is  compact linear operators
 and satisfy {\rm (D3)} in Hypothesis~\ref{hyp:1.1} with $X=H$.
\end{description}
If $(\vec{\lambda}^\ast,\theta)\in\mathbb{R}^n\times U$ is a (multiparameter) bifurcation point  for the equation
\begin{equation}\label{e:Bi.2.1}
\mathcal{F}'(u)=\sum^n_{j=1}\lambda_j\mathcal{G}'_j(u),\quad u\in U,
\end{equation}
then $\vec{\lambda}^\ast=(\lambda^\ast_1,\cdots,\lambda^\ast_n)$ is an  eigenvalue  of
\begin{equation}\label{e:Bi.2.2}
\mathcal{F}''(\theta)v-\sum^n_{j=1}\lambda_j\mathcal{G}''_j(\theta)v=0,\quad v\in H,
\end{equation}
 in other words, $\theta$ is
a degenerate critical point of the functional $\mathcal{F}-\sum^n_{j=1}\lambda^\ast_j\mathcal{G}$.
(The solution space  of (\ref{e:Bi.2.2}), denoted by  $H(\vec{\lambda})$,
is of finite dimension because it is the kernel of a Fredholm operator).
\end{theorem}

Theorem~12 in \cite[Chapter~4, \S4.3]{Skr1} also required: (a)  $\mathcal{G}$ is
weakly continuous and uniformly differentiable in $U$, (b)
$\mathcal{F}'$ has uniformly positive definite Frech\`et derivatives and satisfies the condition
$\alpha)$ in \cite[Chapter~3, \S2.2]{Skr1}.
If $\mathcal{G}'$ is completely continuous (i.e., mapping a weakly convergent sequence
into a convergent one in norm) and has Frech\'et  derivative
$\mathcal{G}''(u)$ at $u\in U$, then $\mathcal{G}''(u)\in\mathscr{L}(H)$ is a compact linear operator
(cf. \cite[Remark~2.4.6]{Ber}).

\noindent{\it Proof of Theorem~\ref{th:Bi.2.2}}. \quad Let $(\vec{\lambda}^\ast,\theta)\in\mathbb{R}^n\times U$ be
 a bifurcation point  for (\ref{e:Bi.2.1}).
Then we have a sequence $(\vec{\lambda}_k, u_k)\in\mathbb{R}^n\times(U\setminus\{\theta\})$
such that $\vec{\lambda}_k=({\lambda}_{k,1},\cdots, {\lambda}_{k,n})\to\vec{\lambda}^\ast$, $u_k\to\theta$ and
$$
\mathcal{F}'(u_k)=\sum^n_{j=1}\lambda_{k,j}\mathcal{G}'_j(u_k),\quad
k=1,2,\cdots.
$$
 Passing to a subsequence, if necessary, we can assume $v_k=u_k/\|u_k\|\rightharpoonup v^\ast$.
 By the assumptions, $B=\mathcal{F}''$ has a decomposition $P+ Q$ as in Hypothesis~\ref{hyp:1.1}
 with $X=H$.  {\rm (D4)} and Lemma~\ref{lem:D*} imply that
 there exist positive constants $\eta_0>0$ and  $C'_0>0$ such that
 \begin{equation}\label{e:Bi.2.3}
(P(u)h, h)\ge C'_0\|h\|^2\quad\forall h\in H,\;\forall u\in
B_H(\theta,\eta_0)\subset U.
\end{equation}
 Clearly, we can assume that $\{u_k\}_{k\ge 1}$ is contained in $B_H(\theta,\eta_0)$.    Note that
  \begin{equation}\label{e:Bi.2.4}
\frac{1}{\|u_k\|^2}(\mathcal{F}'(u_k), u_k)=\sum^n_{j=1}\frac{\lambda_{k,j}}{\|u_k\|^2}(\mathcal{G}'_j(u_k), u_k),\quad k=1,2,\cdots.
 \end{equation}
Since $\mathcal{G}'_j(\theta)=\theta$, $j=1,\cdots,n$, using the Mean Value Theorem
 we have a sequence $\{t_k\}_{k\ge 1}\subset (0, 1)$ such that
\begin{eqnarray}\label{e:Bi.2.5}
\sum^n_{j=1}\frac{\lambda_{k,j}}{\|u_k\|^2}(\mathcal{G}'_j(u_k), u_k)&=&\sum^n_{j=1}\lambda_{k,j}
(\mathcal{G}''_j(t_ku_k)v_k, v_k)=
\sum^n_{j=1}\lambda_{k,j}([\mathcal{G}''_j(t_ku_k)- \mathcal{G}''_j(\theta)]v_k, v_k)\nonumber\\
&+&\sum^n_{j=1}\lambda_{k,j}(\mathcal{G}''_j(\theta)v_k, v_k)\to \sum^n_{j=1}\lambda^\ast_j(\mathcal{G}''_j(\theta)v^\ast, v^\ast)
\end{eqnarray}
because all $\mathcal{G}''_j(\theta)$ are compact and $\|\mathcal{G}''_j(t_ku_k)- \mathcal{G}''_j(\theta)\|\to 0$
by {\rm (D3)}. Moreover, since $\mathcal{F}'(\theta)=\theta$,
 we may use the Mean Value Theorem to yield a sequence $\{s_k\}_{k\ge 1}\subset (0, 1)$ such that
 \begin{eqnarray*}
\frac{1}{\|u_k\|^2}(\mathcal{F}'(u_k), u_k)&=&\frac{1}{\|u_k\|^2}(\mathcal{F}''(s_ku_k)u_k, u_k)\\
&=&\frac{1}{\|u_k\|^2}(P(s_ku_k)u_k, u_k)+ \frac{1}{\|u_k\|^2}(Q(s_ku_k)u_k, u_k)\\
&\ge& C_0'+ \frac{1}{\|u_k\|^2}(Q(s_ku_k)u_k, u_k)\quad \forall
k\in\mathbb{N}
\end{eqnarray*}
by (\ref{e:Bi.2.3}). As in (\ref{e:Bi.2.5}) we have also
$$
\frac{1}{\|u_k\|^2}(Q(s_ku_k)u_k, u_k)\to (Q(\theta)v^\ast, v^\ast).
$$
It follows from these and (\ref{e:Bi.2.4}) that
$C_0'\le ([\sum^n_{j=1}\lambda^\ast_j\mathcal{G}''_j(\theta)- Q(\theta)]v^\ast, v^\ast)$
and hence $v^\ast\ne \theta$.

Moreover, for any $h\in H$ we have
 \begin{equation}\label{e:Bi.2.6}
\frac{1}{\|u_k\|}(\mathcal{F}'(u_k), h)=\sum^n_{k=1}\frac{\lambda_k}{\|u_k\|}(\mathcal{G}'_j(u_k), h),\quad k=1,2,\cdots,
 \end{equation}
  and as in (\ref{e:Bi.2.5}) we may prove that
 \begin{equation}\label{e:Bi.2.7}
\sum^n_{j=1}\frac{\lambda_k}{\|u_k\|}(\mathcal{G}'_j(u_k), h)\to \sum^n_{j=1}\lambda^\ast_j(\mathcal{G}''_j(\theta)v^\ast, h),
 \end{equation}
and that for some sequence $\{\tau_k\}_{k\ge 1}\subset (0, 1)$, depending on $\{u_k\}_{k\ge 1}$ and $h$,
 \begin{eqnarray*}
\frac{1}{\|u_k\|}(\mathcal{F}'(u_k), h)=(
\mathcal{F}''(\tau_ku_k)v_k, h)
=(v_k, \mathcal{F}''(\tau_ku_k)h)
\to (v^\ast, \mathcal{F}''(\theta)h )
\end{eqnarray*}
because $v_k\rightharpoonup v^\ast$ and $\|\mathcal{F}''(\tau_ku_k)h-\mathcal{F}''(\theta)h\|\to 0$ by (D2)
and (D3).
This and (\ref{e:Bi.2.6})--(\ref{e:Bi.2.7}) lead to
$$
\sum^n_{j=1}\lambda^\ast_j(\mathcal{G}''_j(\theta)v^\ast, h)=(v^\ast, \mathcal{F}''(\theta)h)\quad\forall h\in H
$$
and hence $\mathcal{F}''(\theta)v^\ast-\sum^n_{j=1}\lambda^\ast_j\mathcal{G}''_j(\theta)v^\ast=0$.
That is, $\vec{\lambda}^\ast$ is an eigenvalue  of (\ref{e:Bi.2.7.4}).
\hfill$\Box$\vspace{2mm}

In the following we give two generalizations of  Rabinowitz bifurcation theorem \cite{Rab}.

\begin{theorem}\label{th:Bi.2.3}
Under the assumptions (i)--(ii) of Theorem~\ref{th:Bi.2.2},
 suppose:
\begin{description}
\item[(iii)]  $\vec{\lambda}^\ast$ is an isolated  eigenvalue  of (\ref{e:Bi.2.2});
\item[(iv)]   the corresponding finite dimension reduction  $\mathcal{L}^\circ_{\vec{\lambda}}$
with the functional $\mathcal{L}=\mathcal{F}-\sum^n_{j=1}\lambda^\ast_j\mathcal{G}_j$ as in Theorem~\ref{th:S.5.3} is of class $C^2$ for each $\vec{\lambda}$ near the origin of $\mathbb{R}^n$;
\item[(v)]  (\ref{e:S.5.12.3}) holds with $H^0=H(\vec{\lambda}^\ast)$
(the solution space  of (\ref{e:Bi.2.2}) with $\vec{\lambda}=\vec{\lambda}^\ast$);
\item[(vi)] either $\dim H(\vec{\lambda}^\ast)$ is odd or there exists  $\vec{\lambda}\in\mathbb{R}^n\setminus\{\vec{0}\}$ such that  the  symmetric bilinear form
\begin{equation}\label{e:Bi.2.2.1}
H(\vec{\lambda}^\ast)\times H(\vec{\lambda}^\ast)\ni (z_1,z_2)\mapsto \mathscr{Q}_{\vec{\lambda}}(z_1,z_2)=
 \sum^n_{j=1}\lambda_j(\mathcal{G}''_j(\theta)z_1,z_2)_H
\end{equation}
has different Morse index and coindex. (In particular, if $n=1$ the latter is equivalent to the fact
that the form  $H({\lambda}^\ast)\times H({\lambda}^\ast)\ni (z_1,z_2)\mapsto
(\mathcal{G}''_1(\theta)z_1,z_2)_H$ has different Morse index and co-index).
\end{description}
Then $(\vec{\lambda}^\ast,\theta)\in\mathbb{R}^n\times U$ is a  bifurcation point  for the equation
(\ref{e:Bi.2.1}). Furthermore, if for some $\vec{\mu}\in\mathbb{R}^n\setminus\{\vec{0}\}$
the form $\mathscr{Q}_{\vec{\mu}}$ defined by (\ref{e:Bi.2.2.1})
is either positive definite or negative one, then one of the following alternatives occurs:
\begin{description}
\item[(A)] $(\vec{\lambda}^\ast,\theta)$ is not an isolated solution of (\ref{e:Bi.2.1}) in
 $\{\vec{\lambda}^\ast\}\times U$.

\item[(B)] there exists a sequence $\{t_k\}_{k\ge 1}\subset\mathbb{R}\setminus\{0\}$
such that $t_k\to 0$ and that for each $t_k$ the equation
(\ref{e:Bi.2.1}) with $\vec{\lambda}=t_k\vec{\mu}+\vec{\lambda}^\ast$ has infinitely many solutions converging to $\theta\in H$.

\item[(C)]  for every $t$ in a small neighborhood of $0\in\mathbb{R}$ there is a nontrivial solution $u_t$ of (\ref{e:Bi.2.1}) with $\vec{\lambda}=t\vec{\mu}+\vec{\lambda}^\ast$
    converging to $\theta$ as $t\to 0$;

\item[(D)] there is a one-sided $\mathfrak{T}$ neighborhood of $0\in\mathbb{R}$ such that
for any $t\in\mathfrak{T}\setminus\{0\}$,
(\ref{e:Bi.2.1}) with $\vec{\lambda}=t\vec{\mu}+\vec{\lambda}^\ast$ has at least two nontrivial solutions converging to zero as $t\to 0$.
\end{description}
\end{theorem}

\begin{proof} {\bf Step 1}.\quad
 By the assumptions we have the conclusions of
Theorem~\ref{th:S.5.3} with  $\mathcal{L}=\mathcal{F}-\sum^n_{j=1}\lambda^\ast_j\mathcal{G}_j$.
  Suppose that $(\vec{\lambda}^\ast,\theta)\in\mathbb{R}^n\times U$ is not a  bifurcation point  for the equation (\ref{e:Bi.2.1}). Then as in the proof of Theorem~\ref{th:Bi.1.1}(I)
we may find $0<\eta\ll 1$  with $\bar{B}_H(\theta,\eta)\subset U$ such that
after shrinking $\delta>0$ in Theorem~\ref{th:S.5.3} for each
$\vec{\lambda}\in[-\delta,\delta]^n$ the following claims hold:\\
$\bullet$ the functional $\mathcal{L}_{\vec{\lambda}}$ has a unique
critical point $\theta$ in $\bar{B}_H(\theta,\eta)$,\\
$\bullet$  $\vec{\lambda}^\ast=(\lambda^\ast_1,\cdots,\lambda^\ast_n)$ is a unique  eigenvalue  of
(\ref{e:Bi.2.2}) in $[-\delta,\delta]^n+\vec{\lambda}^\ast$,\\
$\bullet$ for all $\vec{\lambda}\in [-\delta,\delta]^n$,
\begin{equation}\label{e:Bi.2.7.1}
C_\ast(\mathcal{L}_{\vec{\lambda}}, \theta;{\bf K})=C_\ast(\mathcal{L}_{\vec{0}}, \theta;{\bf K})=C_\ast(\mathcal{F}-\sum^n_{j=1}\lambda_j^\ast
\mathcal{G}_j, \theta;{\bf K}).
\end{equation}
 We may also shrink  $\epsilon>0$, $r>0, s>0$ and $W$ in Theorem~\ref{th:S.5.3} so that
$$
\bar{B}_{H^0}(\theta,\epsilon)\oplus \bar{B}_{H^+}(\theta, r)\oplus
\bar{B}_{H^-}(\theta, s)\subset B_H(\theta,\eta)\quad\hbox{and}\quad \overline{W}\subset B_H(\theta,\eta),
$$
where $H^0=H(\vec{\lambda}^\ast)$. By (\ref{e:Bi.2.7.1}) and Theorem~\ref{th:S.5.4} we have
\begin{equation}\label{e:Bi.2.7.2}
C_\ast(\mathcal{L}^\circ_{\vec{\lambda}}, \theta;{\bf K})=C_\ast(\mathcal{L}^\circ_{\vec{0}}, \theta;{\bf K}),\quad
\forall \vec{\lambda}\in [-\delta,\delta]^n.
\end{equation}
(This can also be derived from the stability of critical groups as before.)
For each $\vec{\lambda}\in [-\delta,\delta]^n\setminus\{\vec{0}\}$, since
$\theta\in H$ is a nondegenerate critical point of $\mathcal{L}_{\vec{\lambda}}$,
Claim below (\ref{e:S.5.12.3}) tells us that
$\theta\in H^0$ is a nondegenerate critical point of $\mathcal{L}^\circ_{\vec{\lambda}}$ too.
Hence (\ref{e:Bi.2.7.2}) implies that the Morse index of $\mathcal{L}^\circ_{\vec{\lambda}}$ at $\theta$
is constant with respect to $\vec{\lambda}\in [-\delta,\delta]^n\setminus\{\vec{0}\}$.

On the other hand, by (vi), if $\dim H(\vec{\lambda}^\ast)$ is odd,
for every $\vec{\lambda}\in [-\delta,\delta]^n\setminus\{\vec{0}\}$ the nondegenerate
quadratic forms $d^2\mathcal{L}^\circ_{\vec{\lambda}}(\theta)$ and $d^2\mathcal{L}^\circ_{-\vec{\lambda}}(\theta)$ on $H(\vec{\lambda}^\ast)$
must have different Morse indexes, where $d^2\mathcal{L}^\circ_{\vec{\lambda}}(\theta):H(\vec{\lambda}^\ast)\times H(\vec{\lambda}^\ast)
\to\mathbb{R}$ given by
$$
d^2\mathcal{L}^\circ_{\vec{\lambda}}(\theta)(z_1,z_2)=(\mathcal{L}''_{\vec{\lambda}}(\theta)z_1,z_2)_H=
  -\sum^n_{j=1}\lambda_j(\mathcal{G}''_j(\theta)z_1,z_2)_H=-\mathscr{Q}_{\vec{\lambda}}(z_1,z_2).
$$
This contradicts (\ref{e:Bi.2.7.2}). Similarly, if there exists  $\vec{\lambda}\in\mathbb{R}^n\setminus\{\vec{0}\}$ such that the form
in (\ref{e:Bi.2.2.1}) has different Morse index and coindex, then for every
$t>0$ with $t\vec{\lambda}\in[-\delta,\delta]^n$ the forms
 $d^2\mathcal{L}^\circ_{t\vec{\lambda}}(\theta)$ and $d^2\mathcal{L}^\circ_{-t\vec{\lambda}}(\theta)$ on $H(\vec{\lambda}^\ast)$  have different Morse indexes, and hence a contradiction is obtained.

\noindent{\bf Step 2}.\quad
By replacing $\vec{\mu}$ by $-\vec{\mu}$ we may assume that
the form $\mathscr{Q}_{\vec{\mu}}$ is positive definite.
Suppose that any one of (A)--(C) does not hold. Then
there exists $\epsilon\in (0,1)$ such that
$\theta\in H(\vec{\lambda}^\ast)$ is an isolated critical point of $\mathcal{L}^\circ_{t\vec{\mu}}$
for each $t\in [-\epsilon,\epsilon]$.
By the assumption $d^2\mathcal{L}^\circ_{t\vec{\mu}}(\theta)$ is
negative (resp. positive) definite for each $t$ in $(0,\epsilon]$ (resp.
 $[-\epsilon, 0)$). Then Theorem~\ref{th:Bi.2.1} implies that
 for some one-sided $\mathfrak{T}$ neighborhood of $0\in [-\epsilon, \epsilon]$ and
 any $t\in\mathfrak{T}\setminus\{0\}$ the functional  $\mathcal{L}^\circ_{t\vec{\mu}}$
 has two distinct nontrivial critical points $z_{t,1}$ and $z_{t,2}$ converging to $\theta\in H(\vec{\lambda}^\ast)$. Then $u_{t,j}=z_{t,j}+\psi(t\vec{\mu}+\vec{\lambda}^\ast,
 z_{t,j})$, $j=1,2$, are two nontrivial solutions of
(\ref{e:Bi.2.1}) with $\vec{\lambda}=t\vec{\mu}+\vec{\lambda}^\ast$, and both
 converge to zero as $t\to 0$.
\end{proof}

 Note that (\ref{e:Bi.2.2}) has no isolated eigenvalues if
$\cap^n_{j=1}{\rm Ker}(\mathcal{G}''_j(\theta))\cap{\rm Ker}(\mathcal{F}''(\theta))\ne\{\theta\}$.
It is natural to ask when  $\vec{\lambda}^\ast$ is an isolated  eigenvalue  of (\ref{e:Bi.2.2}).
For the sake of simplicity let us consider the case $n=1$. Then
$H(\vec{\lambda}^\ast)={\rm Ker}(\mathcal{F}''(\theta))-\lambda^\ast_1\mathcal{G}''_1(\theta))$, and  if $\lambda_1^\ast\ne 0$ we have
$$
d^2\mathcal{L}^\circ_{\vec{\lambda}}(\theta)(z_1,z_2)=-\lambda_1(\mathcal{G}''_j(\theta)z_1,z_2)_H
=-\frac{\lambda_1}{\lambda^\ast_1}(\mathcal{F}''(\theta)z_1,z_2).
$$
In Theorem~\ref{th:Bi.2.3}, the final condition that the form in (\ref{e:Bi.2.2.1})
is either positive definite or negative one suggests that we should require
$\mathcal{F}''(\theta)$ to be nondegenerate on $H(\vec{\lambda}^\ast)$.

Suppose now that  $\mathcal{F}''(\theta)$ is  invertible, $n=1$ and write
$\mathcal{G}=\mathcal{G}_1$. Then (\ref{e:Bi.2.7.3}) and (\ref{e:Bi.2.7.4}) become
\begin{eqnarray}
&&\mathcal{F}'(u)=\lambda\mathcal{G}'(u),\quad u\in U,\label{e:Bi.2.7.3}\\
&&\mathcal{F}''(\theta)v-\lambda\mathcal{G}''(\theta)v=0,\quad v\in H,\label{e:Bi.2.7.4}
\end{eqnarray}
respectively. Moreover,  $0$ is not an eigenvalue  of (\ref{e:Bi.2.7.4}), and $\lambda\in\mathbb{R}\setminus\{0\}$
is an eigenvalue  of (\ref{e:Bi.2.7.4}) if and only if $1/\lambda$
is an eigenvalue  of compact linear self-adjoint operator $L:=[\mathcal{F}''(\theta) ]^{-1}\mathcal{G}''(\theta)\in\mathscr{L}_s(H)$.
By Riesz-Schauder theory, the spectrum of $L$, $\sigma(L)$,
contains a unique accumulation point $0$, and
$\sigma(L)\setminus\{0\}$ is a real countable set of eigenvalues
of finite multiplicity, denoted by $\{1/\lambda_n\}_{n=1}^\infty$.
Let $H_n$ be the eigensubspace corresponding to $1/\lambda_n$ for $n\in\mathbb{N}$.
Then $H=\oplus^\infty_{n=0}H_n$,  $H_0={\rm Ker}(L)={\rm Ker}(\mathcal{G}''(\theta))$ and
\begin{equation}\label{e:Bi.2.8}
H_n={\rm Ker}(I/\lambda_n-L)={\rm Ker}(\mathcal{F}''(\theta)-\lambda_n \mathcal{G}''(\theta)),\quad
n=1,2,\cdots.
\end{equation}

As another generalization of  Rabinowitz bifurcation theorem \cite{Rab} we have the following
improvement of sufficiency of Theorem~12 in \cite[Chap.4, \S4.3]{Skr1}.

\begin{theorem}\label{th:Bi.2.4}
Let $\mathcal{F}, \mathcal{G}=\mathcal{G}_1\in C^1(U,\mathbb{R})$ be as in Theorem~\ref{th:Bi.2.2},
and  $\lambda^\ast$ be an eigenvalue of (\ref{e:Bi.2.7.4}).
Suppose that the operator $\mathcal{F}''(\theta)$ is invertible and
 also satisfies one of the following three conditions:
 (a) positive definite, (b) negative definite, (c)
 each $H_n$ in (\ref{e:Bi.2.8}) with $L=[\mathcal{F}''(\theta)]^{-1}\mathcal{G}''(\theta)$
 is an invariant subspace of $\mathcal{F}''(\theta)$
 (e.g. these are true if $\mathcal{F}''(\theta)$
 commutes with $\mathcal{G}''(\theta)$), and $\mathcal{F}''(\theta)$
 is either positive definite or negative one on $H_{n_0}$ if $\lambda^\ast=\lambda_{n_0}$.
 Then
$(\lambda^\ast,\theta)\in\mathbb{R}\times U$ is a bifurcation point  for the equation
(\ref{e:Bi.2.7.3})
 and  one of the following alternatives occurs:
\begin{description}
\item[(i)] $(\lambda^\ast,\theta)$ is not an isolated solution of (\ref{e:Bi.2.7.3}) in
 $\{\lambda^\ast\}\times U$.

\item[(ii)] there exists a sequence $\{\kappa_n\}_{n\ge 1}\subset\mathbb{R}\setminus\{\lambda^\ast\}$
such that $\kappa_n\to\lambda^\ast$ and that for each $\kappa_n$ the equation
(\ref{e:Bi.2.7.3}) with $\lambda=\kappa_n$ has infinitely many solutions converging to
$\theta\in H$.

\item[(iii)]  for every $\lambda$ in a small neighborhood of $\lambda^\ast$ there is a nontrivial solution
$u_\lambda$ of (\ref{e:Bi.2.7.3}) converging to $\theta$ as $\lambda\to\lambda^\ast$;

\item[(iv)] there is a one-sided $\Lambda$ neighborhood of $\lambda^\ast$ such that
for any $\lambda\in\Lambda\setminus\{\lambda^\ast\}$,
(\ref{e:Bi.2.7.3}) has at least two nontrivial solutions converging to
zero as $\lambda\to\lambda^\ast$.
\end{description}
\end{theorem}

It is easily seen that the functional $\mathcal{F}$ in
\cite[\S4.3, Theorem~4.3]{Skr1} or  in \cite[Chap.1, Theorem~3.4]{Skr2}
satisfies the conditions of this theorem for the case (a).

\noindent{\it Proof of Theorem~\ref{th:Bi.2.4}}.  \quad
\noindent{\bf Case 1}. {\it  $\mathcal{F}''(\theta)$
is either positive definite or negative one}.

Clearly, we only need to consider the first case.
For $\lambda>\lambda_n$  and $h\in H_n\setminus\{\theta\}$ with $n>0$, since
$\lambda_n\mathcal{G}''(\theta)h=\mathcal{F}''(\theta)h$, we have
\begin{equation}\label{e:Bi.2.11}
(\mathcal{F}''(\theta)h-\lambda\mathcal{G}''(\theta)h, h)=
(1-\lambda/\lambda_n)(\mathcal{F}''(\theta)h,h)<0.
 \end{equation}
 Clearly, if $H_0\ne\{\theta\}$ (this is true if $\dim H=\infty$), for $h\in H_0\setminus\{\theta\}$
 it holds that
  $$
  (\mathcal{F}''(\theta)h-\lambda\mathcal{G}''(\theta)h, h)=
(\mathcal{F}''(\theta)h,h)>0.
$$
Let $\mu_\lambda$ denote the Morse index of $\mathcal{L}_\lambda:=\mathcal{F}-\lambda\mathcal{G}$ at $\theta$.
 Then by (\ref{e:Bi.2.11}) we obtain
 \begin{equation}\label{e:Bi.2.12}
\mu_\lambda=\sum_{\lambda_n<\lambda}\dim H_n.
\end{equation}

Assume $\lambda^\ast=\lambda_{n_0}$ for some $n_0\in\mathbb{N}$.
 Then we have $\varepsilon>0$ such that $(\lambda^\ast-2\varepsilon,
\lambda^\ast+ 2\varepsilon)\setminus\{\lambda^\ast\}$ has no intersection with $\{\lambda_n\}^\infty_{n=1}$
since $\lambda_n\to\infty$.
By (\ref{e:Bi.2.12}) it is easy to verify that
\begin{eqnarray}
&&\mu_\lambda=\mu_{\lambda^\ast},\quad\forall \lambda\in (\lambda^\ast- 2\varepsilon, \lambda^\ast],\label{e:Bi.2.13}\\
&&\mu_\lambda=\mu_{\lambda^\ast}+ \nu_{\lambda^\ast},\quad\forall \lambda\in (\lambda^\ast, \lambda^\ast+2\varepsilon),\label{e:Bi.2.14}
\end{eqnarray}
where $\nu_{\lambda^\ast}=\dim H_{n_0}$ is the nullity of  $\mathcal{L}_{\lambda^\ast}$
 at $\theta$.

By Step 1 of proof of Theorem~\ref{th:S.4.3},  we have

 \noindent{\bf Claim 1}.\quad {\it After shrinking $\varepsilon>0$
 we may verify that  the homotopy
$$
 [\lambda^\ast-\varepsilon, \lambda^\ast+\varepsilon]\times \overline{B_H(\theta,\varepsilon)}\to H,\;
 (\lambda, v)\mapsto \nabla\mathcal{L}_\lambda(v)
 $$
is of class $(S)_+$. So if $\{(\kappa_n, v_n)\}_{n\ge 1}\subset
 [\lambda^\ast-\varepsilon, \lambda^\ast+\varepsilon]\times \overline{B_H(\theta,\varepsilon)}$
 satisfies $\nabla\mathcal{L}_{\kappa_n}(v_n)\to\theta$ and $\kappa_n\to \kappa_0$, then
 $\{v_n\}_{n\ge 1}$ has a convergent subsequence in $\overline{B_H(\theta,\varepsilon)}$.}

 If (i) or (ii) holds, then $(\lambda^\ast,\theta)$ is  a bifurcation point  for (\ref{e:Bi.2.7.3}).

{\it Now suppose that neither (i) nor (ii) holds}. Then we have

\noindent{\bf Claim 2}.\quad {\it
 $\theta\in H$ is an isolated critical point of
$\mathcal{L}_\lambda$ for each  $\lambda\in [\lambda^\ast-\varepsilon, \lambda^\ast+\varepsilon]$ by shrinking $\varepsilon>0$ if necessary.}

Writing $H^0=H_{n_0}$ and  applying Theorem~\ref{th:S.5.3}
to $\mathcal{L}_{\lambda}=\mathcal{L}_{\lambda^\ast}-(\lambda^\ast-\lambda)\mathcal{G}$
with $\lambda\in [\lambda^\ast-\varepsilon, \lambda^\ast+\varepsilon]$
and $-\mathcal{G}$,
 we have $\delta\in (0, \varepsilon]$, $\epsilon>0$ and a unique continuous map
 \begin{equation}\label{e:Bi.2.14.1}
 \psi:[\lambda^\ast-\delta, \lambda^\ast+\delta]\times B_H(\theta,\epsilon)\cap H^0\to (H^0)^\bot
 \end{equation}
  such that for each $\lambda\in [\lambda^\ast-\delta, \lambda^\ast+\delta]$,
 $\psi(\lambda,\theta)=\theta$ and
\begin{equation}\label{e:Bi.2.14.2}
 P^\bot\nabla\mathcal{F}(z+ \psi(\lambda, z))-
 \lambda P^\bot\nabla\mathcal{G}(z+ \psi(\lambda, z))=\theta\quad\forall z\in B_{H}(\theta,\epsilon)\cap H^0,
 \end{equation}
 where  $P^\bot$ is the orthogonal projection onto $(H^0)^\bot$, and that
  the functional
 \begin{equation}\label{e:Bi.2.15}
 B_H(\theta, \epsilon)\cap H^0\ni z\mapsto\mathcal{L}^\circ_\lambda(z):=
 \mathcal{F}(z+\psi(\lambda, z))- \lambda\mathcal{G}(z+\psi(\lambda, z))
  \end{equation}
   is of class $C^1$, whose differential is given by
\begin{equation}\label{e:Bi.2.16}
D\mathcal{L}^\circ_\lambda(z)h=D\mathcal{F}(z+\psi(\lambda,z))h- \lambda D\mathcal{G}(z+\psi(\lambda,z))h,\quad\forall h\in H^0.
\end{equation}
Hence the problem is reduced to finding the critical points of $\mathcal{L}^\circ_\lambda$
near $\theta\in H^0$ for fixed $\lambda$ near $\lambda^\ast$.
Note that Claim 2  is equivalent to the following

\noindent{\bf Claim 3}.\quad {\it
 $\theta\in H^0$ is an isolated critical point of
$\mathcal{L}^\circ_\lambda$ for each  $\lambda\in [\lambda^\ast-\delta, \lambda^\ast+\delta]$ by shrinking $\delta>0$ if necessary.}

 Hence if $\theta\in H^0$ is a local  maximizer (resp. minimizer) of $\mathcal{L}^\circ_\lambda$,
 it must be strict.

For a $C^1$ function $\varphi$ on a neighborhood $U$ of the origin $\theta\in\mathbb{R}^N$
we may always find  $\tilde\varphi\in C^1(\mathbb{R}^N,
\mathbb{R})$ such that it agrees with $\varphi$ near $\theta\in\mathbb{R}^N$ and is also
coercive (so satisfies the (PS)-condition).
Suppose that $\theta$ is an isolated critical point of $\varphi$. By Proposition~6.95 and Example~6.45 in \cite{MoMoPa}
we have
\begin{equation}\label{e:Bi.2.17}
\left.
\begin{array}{ll}
&C_k(\varphi,\theta;{\bf K})=\delta_{k0}\;\Longleftrightarrow\;\hbox{ $\theta$ is a local minimizer of $\varphi$},\\
&C_k(\varphi,\theta;{\bf K})=\delta_{kN}\;\Longleftrightarrow\;\hbox{ $\theta$ is a local maximizer of $\varphi$},
\end{array}\right\}
\end{equation}
and $C_0(\varphi,\theta;{\bf K})=0=C_N(\varphi,\theta;{\bf K})$ 
if $\theta\in  \mathbb{R}^N$ is neither a local  maximizer nor a local
minimizer of $\varphi$.

Then by Theorem~\ref{th:S.1.1}, (\ref{e:S.5.19}) and (\ref{e:Bi.2.13})--(\ref{e:Bi.2.14}) we get that for any
$q\in\mathbb{N}\cup\{0\}$,
\begin{eqnarray*}
&&\delta_{q(\mu_{\lambda^\ast}+\nu_{\lambda^\ast})}{\bf K}=C_q(\mathcal{L}_{\lambda},\theta;{\bf K})=C_{q-\mu_{\lambda^\ast}}(\mathcal{L}^\circ_{\lambda},\theta;{\bf K}),\quad
\forall \lambda\in (\lambda^\ast, \lambda^\ast+\delta],\\
&&\delta_{q\mu_{\lambda^\ast}}{\bf K}=C_q(\mathcal{L}_{\lambda},\theta;{\bf K})=C_{q-\mu_{\lambda^\ast}}(\mathcal{L}^\circ_{\lambda},\theta;{\bf K}),\quad
\forall \lambda\in [\lambda^\ast-\delta, \lambda^\ast).
\end{eqnarray*}
It follows that
\begin{eqnarray*}
&&C_{j}(\mathcal{L}^\circ_{\lambda},\theta;{\bf K})=
\delta_{(j+\mu_{\lambda^\ast})(\mu_{\lambda^\ast}+\nu_{\lambda^\ast})}{\bf K}=
\delta_{j\nu_{\lambda^\ast}}{\bf K},\quad
\forall \lambda\in (\lambda^\ast, \lambda^\ast+\delta],\\
&&C_{j}(\mathcal{L}^\circ_{\lambda},\theta;{\bf K})=\delta_{(j+\mu_{\lambda^\ast})\mu_{\lambda^\ast}}{\bf K}=\delta_{j0}{\bf K},\quad
\forall \lambda\in [\lambda^\ast-\delta, \lambda^\ast).
\end{eqnarray*}
These and (\ref{e:Bi.2.17}) imply
\begin{eqnarray}
&&\theta\in H^0\;\hbox{is a local minimizer of}\;\mathcal{L}^\circ_{\lambda},\quad
\forall \lambda\in [\lambda^\ast-\nu, \lambda^\ast),\label{e:Bi.2.18}\\
&&\theta\in H^0\;\hbox{is a local maximizer of}\;\mathcal{L}^\circ_{\lambda},\quad
\forall \lambda\in (\lambda^\ast, \lambda^\ast+\delta].\label{e:Bi.2.19}
\end{eqnarray}

By (\ref{e:Bi.2.18})--(\ref{e:Bi.2.19}) and Theorem~\ref{th:Bi.2.1}, one of the following possibilities occurs:
\begin{description}
\item[(I)]  for every $\lambda$ in a small neighborhood of $\lambda^\ast$,
$\mathcal{L}^\circ_{\lambda}$ has a nontrivial critical point  converging to $\theta\in H^0$ as $\lambda\to\lambda^\ast$;

\item[(II)] there is a one-sided $\Lambda$ neighborhood of $\lambda^\ast$ such that
for any $\lambda\in\Lambda\setminus\{\lambda^\ast\}$,
$\mathcal{L}^\circ_{\lambda}$ has two nontrivial critical points
 converging to zero as $\lambda\to\lambda^\ast$.
\end{description}
Obviously, they lead to (iii) and (iv), respectively.

\noindent{\bf Case 2}. {\it Each $H_n$ is an invariant subspace of $\mathcal{F}''(\theta)$, $n=1,2,\cdots$}.
Note that $H_n$ has an orthogonal decomposition $H_n^+\oplus H_n^-$, where
$H_n^+$ (resp. $H_n^-$) is the positive (resp. negative) definite subspace
of $\mathcal{F}''(\theta)|_{H_n}$.
It is possible that $H_n^+=\{\theta\}$ or $H_n^-=\{\theta\}$.

As in (\ref{e:Bi.2.11}), if $H^+_n\ne\{\theta\}$ (resp. $H^-_n\ne\{\theta\}$) and $\lambda>\lambda_n$ (resp. $\lambda<\lambda_n$) we have
$$
(\mathcal{F}''(\theta)h-\lambda\mathcal{G}''(\theta)h, h)=
(1-\lambda/\lambda_n)(\mathcal{F}''(\theta)h,h)<0
 $$
for $h\in H^+_n\ne\{\theta\}$ (resp. $h\in H^-_n\ne\{\theta\}$).
Then  the Morse index of $\mathcal{L}_\lambda$ at $\theta$,
 \begin{equation}\label{e:Bi.2.20}
\mu_\lambda=\sum_{\lambda_n<\lambda}\dim H_n^+ + \sum_{\lambda_n>\lambda}\dim H_n^-.
\end{equation}
Since $\lambda^\ast=
\lambda_{n_0}$, as in (\ref{e:Bi.2.13})-(\ref{e:Bi.2.14}) it follows from these that
\begin{eqnarray}
&&\mu_\lambda=\mu_{\lambda^\ast}+ \nu^-_{\lambda^\ast},\quad\forall \lambda\in (\lambda^\ast- 2\varepsilon, \lambda^\ast),\label{e:Bi.2.21}\\
&&\mu_\lambda=\mu_{\lambda^\ast}+ \nu^+_{\lambda^\ast},\quad\forall \lambda\in (\lambda^\ast, \lambda^\ast+2\varepsilon),\label{e:Bi.2.22}
\end{eqnarray}
where $\nu^+_{\lambda^\ast}=\dim H^+_{n_0}$ (resp. $\nu^-_{\lambda^\ast}=\dim H^-_{n_0}$)
is the positive (resp. negative) index of inertia of $\mathcal{F}''(\theta)|_{H_{n_0}}$.

 Since $\mathcal{F}''(\theta)$
 is either positive definite or negative one on $H_{n_0}$, we have either
 \begin{eqnarray}
&&\mu_\lambda=\mu_{\lambda^\ast},\quad\forall \lambda\in (\lambda^\ast- 2\varepsilon, \lambda^\ast),\label{e:Bi.2.23}\\
&&\mu_\lambda=\mu_{\lambda^\ast}+ \nu_{\lambda^\ast},\quad\forall \lambda\in (\lambda^\ast, \lambda^\ast+2\varepsilon),\label{e:Bi.2.24}
\end{eqnarray}
 or
 \begin{eqnarray}
&&\mu_\lambda=\mu_{\lambda^\ast}+ \nu_{\lambda^\ast},\quad\forall \lambda\in (\lambda^\ast- 2\varepsilon, \lambda^\ast),\label{e:Bi.2.25}\\
&&\mu_\lambda=\mu_{\lambda^\ast}.
\quad\forall \lambda\in (\lambda^\ast, \lambda^\ast+2\varepsilon),\label{e:Bi.2.26}
\end{eqnarray}

 {\it We also suppose that neither (i) nor (ii) holds}. Then Claim~1 and so Claim~2 holds.
  (\ref{e:Bi.2.23})--(\ref{e:Bi.2.24}) and (\ref{e:Bi.2.25})--(\ref{e:Bi.2.26}) lead, respectively, to
\begin{eqnarray*}
&&C_{j}(\mathcal{L}^\circ_{\lambda},\theta;{\bf K})=
\delta_{j\nu_{\lambda^\ast}}{\bf K},\quad
\forall \lambda\in (\lambda^\ast, \lambda^\ast+\delta],\\
&&C_{j}(\mathcal{L}^\circ_{\lambda},\theta;{\bf K})=\delta_{j0}{\bf K},\quad
\forall \lambda\in [\lambda^\ast-\delta, \lambda^\ast),
\end{eqnarray*}
 and
\begin{eqnarray*}
&&C_{j}(\mathcal{L}^\circ_{\lambda},\theta;{\bf K})=
\delta_{j\nu_{\lambda^\ast}}{\bf K},\quad
\forall \lambda\in [\lambda^\ast-\delta, \lambda^\ast),\\
&&C_{j}(\mathcal{L}^\circ_{\lambda},\theta;{\bf K})=\delta_{j0}{\bf K},\quad
\forall \lambda\in (\lambda^\ast, \lambda^\ast+\delta].
\end{eqnarray*}
The former yields (\ref{e:Bi.2.18})--(\ref{e:Bi.2.19}) as above.
Similarly, the latter and (\ref{e:Bi.2.17}) imply
\begin{eqnarray*}
&&\theta\in H^0\;\hbox{is a local maximizer of}\;\mathcal{L}^\circ_{\lambda},\quad
\forall \lambda\in [\lambda^\ast-\delta, \lambda^\ast),\\
&&\theta\in H^0\;\hbox{is a local minimizer of}\;\mathcal{L}^\circ_{\lambda},\quad
\forall \lambda\in (\lambda^\ast, \lambda^\ast+\delta].
\end{eqnarray*}
They and Theorem~\ref{th:Bi.2.1} show that either (c) or (d) holds.
\hfill$\Box$\vspace{2mm}

\begin{remark}\label{rm:Bi.2.5}
{\rm
 Suppose that  (iv) is replaced by the following weaker conclusion
\begin{description}
\item[(iv')] there is a one-sided $\Lambda$ neighborhood of $\lambda^\ast$ such that
for any $\lambda\in\Lambda\setminus\{\lambda^\ast\}$,
(\ref{e:Bi.2.7.3}) has a nontrivial solution $u_\lambda$ converging to
zero as $\lambda\to\lambda^\ast$.
\end{description}
Then the finite dimension reduction in the proof is not needed;
and the assumption ``$\mathcal{F}''(\theta)$
 is either positive definite or negative one on $H_{n_0}$" in (c)
 may be replaced by the weaker condition
 ``$\mathcal{F}''(\theta)|_{H_{n_0}}$ has nonzero signature".}
\end{remark}

\subsection{Bifurcation for equivariant problems}\label{sec:B.3}

Now let us generalize the above results to the equivariant case.
The first is a partial generalization of \cite[Theorem~4.2]{Wa}
and  \cite[Theorem~2.5]{Wa1}. The proofs of the latter were based on Morse theory ideas of \cite{Ch2}.
Because of our theory in Section~\ref{sec:S}, the same methods may be used with some technical improvements.

\begin{theorem}\label{th:Bi.3.1}
Under the assumptions of Theorem~\ref{th:Bi.2.4},
let $G$ be a compact Lie group acting on $H$ orthogonally, and suppose that
$U$, $\mathcal{F}$ and $\mathcal{G}$  are $G$-invariant.
For an eigenvalue  $\lambda^\ast$  of (\ref{e:Bi.2.7.4}),
suppose that the operator $\mathcal{F}''(\theta)$ is invertible and
 also satisfies one of the following three conditions:
 (a) positive definite, (b) negative definite, (c)
 each $H_n$ in (\ref{e:Bi.2.8}) with $L=[\mathcal{F}''(\theta)]^{-1}\mathcal{G}''(\theta)$
 is an invariant subspace of $\mathcal{F}''(\theta)$
 (e.g. these are true if $\mathcal{F}''(\theta)$
 commutes with $\mathcal{G}''(\theta)$), and $\mathcal{F}''(\theta)$
 is either positive definite or negative one on $H^0=H_{n_0}$ if $\lambda^\ast=\lambda_{n_0}$.
 Then
 \begin{description}
\item[1)] $(\lambda^\ast,\theta)\in\mathbb{R}\times U$ is a bifurcation point  for the equation
(\ref{e:Bi.2.7.3}).

\item[2)] If $\dim H_{n_0}\ge 2$ and the unit sphere in
$H_{n_0}$ is not a $G$-orbit we must get one of the following alternatives:
\begin{description}
\item[(i)] $(\lambda^\ast,\theta)$ is not an isolated solution of (\ref{e:Bi.2.7.3}) in
 $\{\lambda^\ast\}\times U$;

\item[(ii)] there exists a sequence $\{\kappa_n\}_{n\ge 1}\subset\mathbb{R}\setminus\{\lambda^\ast\}$
such that $\kappa_n\to\lambda^\ast$ and that for each $\kappa_n$ the equation
(\ref{e:Bi.2.7.3}) with $\lambda=\kappa_n$ has infinitely many $G$-orbits of solutions converging to
$\theta\in H$;

\item[(iii)]  for every $\lambda$ in a small neighborhood of $\lambda^\ast$ there is a nontrivial solution
$u_\lambda$ of (\ref{e:Bi.2.7.3}) converging to $\theta$ as $\lambda\to\lambda^\ast$;

\item[(iv)] there is a one-sided $\Lambda$ neighborhood of $\lambda^\ast$ such that
for any $\lambda\in\Lambda\setminus\{\lambda^\ast\}$,
(\ref{e:Bi.2.7.3}) has at least two nontrivial critical orbits converging to
zero $\theta$ as $\lambda\to\lambda^\ast$.
\end{description}
\item[3)] Suppose one of the following assumptions holds: 3.a)
$G=T^m$ and ${\rm Fix}(G)\cap  H_{n_0}=\{\theta\}$; 3.b)
$G=T^m$ and every orbit in $H_{n_0}$ is homeomorphic to some $T^s$ for $s\ge 2$;
3.c) $G$ is a finite group and the greatest common divisor $\delta$ of set
$\{\sharp G/\sharp G_x\,|\, x\in H_{n_0}\setminus\{\theta\}\}$ is equal to or bigger than $2$,
where $\sharp S$ denotes the number of elements in a set $S$.
Then either one of the above (i)-(iii) or the following hold:\\
{\bf (iv)'} there is a one-sided $\Lambda$ neighborhood of $\lambda^\ast$ such that
for any $\lambda\in\Lambda\setminus\{\lambda^\ast\}$,
(\ref{e:Bi.2.7.3}) has at least $\dim H_{n_0}$ (resp. $2\dim H_{n_0}$, $\delta\dim H_{n_0}$) nontrivial critical orbits
in the case 3.a) (resp. 3.b), 3.c)), where every orbit is counted with
its multiplicity. (The multiplicity of a critical orbit $\mathcal{O}$
of $\mathcal{L}_\lambda=\mathcal{F}-\lambda \mathcal{G}$
was defined as $c(\mathcal{O})=\sum^\infty_{q=0}{\rm rank}C_q(\mathcal{L}_\lambda, \mathcal{O})$   in \cite[Definition~1.3]{Wa1}).
\end{description}
\end{theorem}

This theorem can also be viewed generalizations of \cite{FaRa1, FaRa2}. Even so, we
still  give a direct generalization version of \cite{FaRa1, FaRa2}
in Theorem~\ref{th:Bi.3.2} below (because different methods need be employed).
Another different point is that the number of critical orbits in 3) is counted
in a non-usual way.  After Theorem~\ref{th:Bi.3.2}
we shall compare these two theorems.
\\

\noindent{\it Proof of Theorem~\ref{th:Bi.3.1}}. \quad 1) follows from Theorem~\ref{th:Bi.2.4}.
For 2), as in the proof of Theorem~\ref{th:Bi.2.4}, we assume that
neither (i) nor (ii) holds. Then
 $\theta\in H^0$ is a unique critical orbit of the functional
$B_H(\theta, \epsilon)\cap H^0\ni z\mapsto\mathcal{L}^\circ_\lambda(z)$ in
  (\ref{e:Bi.2.15}) for each  $\lambda\in [\lambda^\ast-2\delta, \lambda^\ast+ 2\delta]$ by shrinking $\delta>0$ and $\epsilon>0$ if necessary. (Thus if $\theta$ is an extreme point of
 $\mathcal{L}^\circ_\lambda$, it must be strict).
 Moreover, since $\dim H^0<\infty$,  replacing $\epsilon$ by a slightly smaller one
 we can assume that $\{\mathcal{L}^\circ_\lambda\,|\,\lambda\in [\lambda^\ast-2\delta, \lambda^\ast+ 2\delta]\}$ satisfies the (PS) condition. That is, if
$z_k\in B_H(\theta, \epsilon)\cap H^0$ and $\lambda_k\in [\lambda^\ast-2\delta, \lambda^\ast+ 2\delta]$
satisfy $D\mathcal{L}^\circ_{\lambda_k}(z_k)\to 0$ and $\sup_k|\mathcal{L}^\circ_{\lambda_k}(z_k)|<\infty$, then $\{(z_k,\lambda_k)\}^\infty_{k=1}$
has a converging subsequence.

The problem is reduced to finding
 the critical orbits of $\mathcal{L}^\circ_\lambda$
near $\theta\in H^0$ for fixed $\lambda$ near $\lambda^\ast$.

\underline{Firstly, we assume that $\mathcal{F}''(\theta)$
is positive definite.}

Since $\theta\in H^0$ is an isolated critical orbit of
$\mathcal{L}^\circ_{\lambda^\ast}$, by (\ref{e:S.5.19}) and
 (\ref{e:Bi.2.17}), we have three cases:\\
$\bullet$ $C_q(\mathcal{L}_{\lambda^\ast},\theta;{\bf K})=
\delta_{q\mu_{\lambda^\ast}}{\bf K}$ if $\theta\in H^0$ is a local
minimizer of $\mathcal{L}^\circ_{\lambda^\ast}$;\\
$\bullet$ $C_q(\mathcal{L}_{\lambda^\ast},\theta;{\bf K})=
\delta_{q(\mu_{\lambda^\ast}+\nu_{\lambda^\ast})}{\bf K}$ if $\theta\in  H^0$
is a local  maximizer of $\mathcal{L}^\circ_{\lambda^\ast}$;\\
$\bullet$ $C_q(\mathcal{L}_{\lambda^\ast},\theta;{\bf K})=0$ for $q\notin (\mu_{\lambda^\ast},
\mu_{\lambda^\ast}+ \nu_{\lambda^\ast})$ if $\theta\in  H^0$ is neither a local  maximizer nor a local
minimizer of $\mathcal{L}^\circ_{\lambda^\ast}$.

In the third case,  the stability of critical groups implies (iii) as before.

For the first two cases, as in the proofs of (\ref{e:Bi.2.18})--(\ref{e:Bi.2.19}),
if $\theta\in H^0$ is a local minimizer of $\mathcal{L}^\circ_{\lambda^\ast}$ we can obtain
\begin{eqnarray}
&&\theta\in H^0\;\hbox{is a local minimizer of}\;\mathcal{L}^\circ_{\lambda},\quad
\forall \lambda\in [\lambda^\ast-2\delta, \lambda^\ast],\label{e:Bi.3.1}\\
&&\theta\in H^0\;\hbox{is a local maximizer of}\;\mathcal{L}^\circ_{\lambda},\quad
\forall \lambda\in (\lambda^\ast, \lambda^\ast+ 2\delta];\label{e:Bi.3.2}
\end{eqnarray}
 and if $\theta\in H^0$ is a local maximizer of $\mathcal{L}^\circ_{\lambda^\ast}$ we have
\begin{eqnarray}
&&\theta\in H^0\;\hbox{is a local minimizer  of}\;\mathcal{L}^\circ_{\lambda},\quad
\forall \lambda\in [\lambda^\ast-2\delta, \lambda^\ast),\label{e:Bi.3.3}\\
&&\theta\in H^0\;\hbox{is a local maximizer of}\;\mathcal{L}^\circ_{\lambda},\quad
\forall \lambda\in [\lambda^\ast, \lambda^\ast+2\delta].\label{e:Bi.3.4}
\end{eqnarray}

Now let us follow the proof ideas of \cite[page 220]{Wa} to complete
the remaining arguments. But differen from the case therein our
 $\mathcal{L}^\circ_{\lambda}$ is only of class $C^1$.
Fortunately, the standard arguments (cf. \cite[Propsition~5.57]{MoMoPa}) may yield

\begin{lemma}\label{lem:pseudogradient}
There exists a smooth map
$$
(B_H(\theta, \epsilon)\cap H^0\setminus\{\theta\})\times(\lambda^\ast-2\delta, \lambda^\ast+ 2\delta)
\to H,\;(z,\lambda)\mapsto\mathscr{V}_\lambda(z)
$$
such that for each $\lambda\in(\lambda^\ast-2\delta, \lambda^\ast+ 2\delta)$ the map
$\mathscr{V}_\lambda: B_H(\theta, \epsilon)\cap H^0\setminus\{\theta\}\to H$
is a $G$-equivariant pseudo-gradient vector field for $\mathcal{L}^\circ_{\lambda}$, precisely for all
$z\in B_H(\theta, \epsilon)\cap H^0\setminus\{\theta\}$,
$$
\|\mathscr{V}_\lambda(z)\|\le 2\|D\mathcal{L}^\circ_{\lambda}(z)\|\quad\hbox{and}\quad
\langle D\mathcal{L}^\circ_{\lambda}(z), \mathscr{V}_\lambda(z)\rangle\ge\frac{1}{2}
\|D\mathcal{L}^\circ_{\lambda}(z)\|^2.
$$
\end{lemma}

For completeness we are also to give the proof it, which is postponed after the
proof of this theorem.

When $\theta\in H^0$ is a local minimizer of $\mathcal{L}^\circ_{\lambda^\ast}$
(and so a strict local minimizer as noted above),  (\ref{e:Bi.3.1}) and (\ref{e:Bi.3.2}) are satisfied.
Let $c=\mathcal{L}^\circ_{\lambda^\ast}(\theta)$. We can choose $\varepsilon>0$
so small that $\mathcal{W}_\varepsilon=\{z\in B_H(\theta, \epsilon)\cap H^0\,|\,
\mathcal{L}^\circ_{\lambda^\ast}(z)<c+\varepsilon\}$ is a contractible neighborhood of $\theta$.
Obverse that $\mathcal{W}_\varepsilon$ is $G$-invariant and the flow of $-\mathscr{V}_{\lambda^\ast}$
preserves $\mathcal{W}_\varepsilon$. (Actually, it is not hard to construct a deformation contraction from $\mathcal{W}_\varepsilon$  to $\theta$ with the flow of $-\mathscr{V}_{\lambda^\ast}$).
Since $\theta\in H^0$ is a unique critical orbit of the functional
$\mathcal{L}^\circ_\lambda$ in $\mathcal{W}_\varepsilon$
for each  $\lambda\in [\lambda^\ast-2\delta, \lambda^\ast+ 2\delta]$, and
 $\{\mathcal{L}^\circ_\lambda\,|\,\lambda\in [\lambda^\ast-2\delta, \lambda^\ast+ 2\delta]\}$ satisfies the (PS) condition, by the theorem on continuous dependence
of solutions of ordinary differential equations on initial values and parameters we may deduce
that there exists $0<\bar\delta<2\delta$ such that
 the flow of $-\mathscr{V}_{\lambda}$ also preserves $\mathcal{W}_\varepsilon$
 for each $\lambda\in [\lambda^\ast-\bar\delta,\lambda^\ast+\bar\delta]$.
  For any $\lambda\in (\lambda^\ast, \lambda^\ast+ \bar\delta]$,
since $\theta\in H^0$ is a local maximizer of  $\mathcal{L}^\circ_{\lambda}$ by (\ref{e:Bi.3.2}),
  $\mathcal{L}^\circ_{\lambda}$ must have a minimal critical orbit $\mathcal{O}\ne\{\theta\}$.
 Note that Theorems~3.1,~3.2 also hold in our case by Corollaries~\ref{cor:S.6.5},
 ~\ref{cor:S.6.6}. Repeating the other arguments in the proof of \cite[page 220]{Wa}
we may verify that for $\lambda\in (\lambda^\ast, \lambda^\ast+\delta]$ close to $\lambda^\ast$,
$\mathcal{L}^\circ_{\lambda}$ has three critical orbits.

Similarly, if (\ref{e:Bi.3.3}) and (\ref{e:Bi.3.4}) hold, $\mathcal{L}^\circ_{\lambda}$ has three critical orbits
for each $\lambda\in [\lambda^\ast-\delta, \lambda^\ast)$ close to $\lambda^\ast$.

 By considering $-\mathcal{F}''(\theta)$
we get the conclusion if $\mathcal{F}''(\theta)$
is negative definite.

\underline{Next, we consider the case that the condition (c) holds.}

If $\theta\in H^0$ is a local
minimizer of $\mathcal{L}^\circ_{\lambda^\ast}$, by (\ref{e:Bi.2.23}) and (\ref{e:Bi.2.24}) we get
\begin{eqnarray}
&&\theta\in H^0\;\hbox{is a local minimizer of}\;\mathcal{L}^\circ_{\lambda},\quad
\forall \lambda\in [\lambda^\ast-\delta, \lambda^\ast],\label{e:Bi.3.5}\\
&&\theta\in H^0\;\hbox{is a local maximizer of}\;\mathcal{L}^\circ_{\lambda},\quad
\forall \lambda\in (\lambda^\ast, \lambda^\ast+\delta].\label{e:Bi.3.6}
\end{eqnarray}

If $\theta\in H^0$ is a local maximizer of $\mathcal{L}^\circ_{\lambda^\ast}$, by (\ref{e:Bi.2.25}) and (\ref{e:Bi.2.26}) we have
\begin{eqnarray}
&&\theta\in H^0\;\hbox{is a local maximizer of}\;\mathcal{L}^\circ_{\lambda},\quad
\forall \lambda\in [\lambda^\ast-\delta, \lambda^\ast],\label{e:Bi.3.7}\\
&&\theta\in H^0\;\hbox{is a local minimizer of}\;\mathcal{L}^\circ_{\lambda},\quad
\forall \lambda\in (\lambda^\ast, \lambda^\ast+\delta].\label{e:Bi.3.8}
\end{eqnarray}

 If $\theta\in  H^0$ is neither a local  maximizer nor a local
minimizer of $\mathcal{L}^\circ_{\lambda^\ast}$, (iii) will occur by the stability of critical groups implies  as before. The proofs of the first two cases are as above.

Finally, we prove 3).  As in 2) we only consider the case that $\mathcal{F}''(\theta)$
is positive definite. The other cases can be proved similarly. We assume that any one of (i)-(iii) does not occur. Then
we have either (\ref{e:Bi.3.1})-(\ref{e:Bi.3.2}) or (\ref{e:Bi.3.1})-(\ref{e:Bi.3.2}).
In these two cases we obtain that
$\theta\in H^0$ is a local maximizer of $\mathcal{L}^\circ_{\lambda}$
for any $\lambda\in (\lambda^\ast, \lambda^\ast+ 2\delta]$
(resp. for all $\lambda\in [\lambda^\ast, \lambda^\ast+2\delta]$).
Note $\nu_{\lambda^\ast}=\dim H_{n_0}$. It follows (or from Theorem~\ref{th:S.5.4}) that
$$
C_q(\mathcal{L}^\circ_{\lambda},\theta;{\bf K})=
\delta_{q\nu_{\lambda^\ast}}{\bf K},\quad q=0,1,\cdots
$$
for any $\lambda\in (\lambda^\ast, \lambda^\ast+ 2\delta]$
(resp. for all $\lambda\in [\lambda^\ast, \lambda^\ast+2\delta]$).
 These show that $\theta$ is essentially the same as nondegenerate
 critical point of $\mathcal{L}^\circ_{\lambda}$ with Morse index
 $\nu_{\lambda^\ast}=\dim H_{n_0}$. Obverse that the conclusions of
 \cite[Corollary~1.3]{Wa1} also hold in the present case (because
 $X=H_{n_0}$ has finite dimension) though $\mathcal{L}^\circ_{\lambda}$
 is only of class $C^1$. The results in the cases 3.a) and 3.b) follow as in \cite{Wa1}.
 The case 3.c) can be derived from \cite[Theorem~1.3]{Wa1}.
 \hfill$\Box$\vspace{2mm}

\noindent{\it Proof of Lemma~\ref{lem:pseudogradient}}.\quad
 Note that $\lambda\mapsto\mathcal{L}^\circ_\lambda\in C^1(B_H(\theta, \epsilon)\cap H^0)$
is continuous by (\ref{e:S.5.4}), and that $D\mathcal{L}^\circ_\lambda$ has no any zero
point in $(B_H(\theta, \epsilon)\cap H^0\setminus\{\theta\})\times[\lambda^\ast-2\delta, \lambda^\ast+ 2\delta]$. For any given $z\in B_H(\theta, \epsilon)\cap H^0\setminus\{\theta\}$ and $\lambda\in (\lambda^\ast-2\delta, \lambda^\ast+ 2\delta)$ we have an open neighborhood
$O_{(z,\lambda)}$ of $z$ in $B_H(\theta, \epsilon)\cap H^0\setminus\{\theta\}$,
a positive number $r_{(z,\lambda)}$
with $(\lambda-r_{(z,\lambda)}, \lambda+r_{(z,\lambda)})\subset
(\lambda^\ast-2\delta, \lambda^\ast+ 2\delta)$ and
a unit vector $v_{(z,\lambda)}\in H$ such that
for all $(z',\lambda')\in O_{(z,\lambda)}\times (\lambda-r_{(z,\lambda)}, \lambda+r_{(z,\lambda)})$,
$$
\|v_{(z,\lambda)}\|\le 2\|D\mathcal{L}^\circ_{\lambda'}(z')\|\quad\hbox{and}\quad
\langle D\mathcal{L}^\circ_{\lambda'}(z'), v_{(z,\lambda)}\rangle\ge\frac{1}{2}
\|D\mathcal{L}^\circ_{\lambda'}(z')\|^2.
$$
Now all above $O_{(z,\lambda)}\times (\lambda-r_{(z,\lambda)}, \lambda+r_{(z,\lambda)})$
form an open cover $\mathscr{Q}$ of $(B_H(\theta, \epsilon)\cap H^0\setminus\{\theta\})\times(\lambda^\ast-2\delta, \lambda^\ast+ 2\delta)$,
and the latter admits a $C^\infty$-unit decomposition $\{\eta_\alpha\}_{\alpha\in\Xi}$
subordinate to a locally finite
refinement $\{W_\alpha\}_{\alpha\in\Xi}$ of $\mathscr{Q}$.
Since each $W_\alpha$ can be contained in some open subset of form
$O_{(z,\lambda)}\times (\lambda-r_{(z,\lambda)}, \lambda+r_{(z,\lambda)})$, we have
a unit vector $v_\alpha\in H$ such that
$$
\|v_\alpha\|\le 2\|D\mathcal{L}^\circ_{\lambda'}(z')\|\quad\hbox{and}\quad
\langle D\mathcal{L}^\circ_{\lambda'}(z'), v_\alpha\rangle\ge\frac{1}{2}
\|D\mathcal{L}^\circ_{\lambda'}(z')\|^2
$$
for all $(z',\lambda')\in W_\alpha$.
Set $\chi=\sum_{\alpha\in\Xi}\eta_\alpha v_\alpha$. Then it is a smooth map
from $(B_H(\theta, \epsilon)\cap H^0\setminus\{\theta\})\times(\lambda^\ast-2\delta, \lambda^\ast+ 2\delta)$ to $H$, and satisfies
$$
\|\chi(z,\lambda)\|\le 2\|D\mathcal{L}^\circ_{\lambda}(z)\|\quad\hbox{and}\quad
\langle D\mathcal{L}^\circ_{\lambda}(z), \chi(z,\lambda)\rangle\ge\frac{1}{2}
\|D\mathcal{L}^\circ_{\lambda}(z)\|^2
$$
for all $(z,\lambda)\in(B_H(\theta, \epsilon)\cap H^0\setminus\{\theta\})\times(\lambda^\ast-2\delta, \lambda^\ast+ 2\delta)$. Let $d\mu$ denote the right invariant Haar measure on $G$. Define
$$
(B_H(\theta, \epsilon)\cap H^0\setminus\{\theta\})\times(\lambda^\ast-2\delta, \lambda^\ast+ 2\delta)
\ni (z,\lambda)\mapsto\mathscr{V}_\lambda(z)=\int_G g^{-1}\chi(gz,\lambda)d\mu\in H.
$$
It is easily checked that $\mathscr{V}_\lambda$ satisfies requirements.
\hfill$\Box$\vspace{2mm}

The following is a direct generalization of Fadell--Rabinowitz theorems \cite{FaRa1, FaRa2}.

\begin{theorem}\label{th:Bi.3.2}
Under the assumptions of Theorem~\ref{th:Bi.3.1},
if the Lie group  $G$ is equal to $\mathbb{Z}_2$ or $S^1$,
 Then
$(\lambda^\ast,\theta)\in\mathbb{R}\times U$ is a bifurcation point  for the equation
(\ref{e:Bi.2.7.3}),  and if $\dim H_{n_0}\ge 2$ and the unit sphere in
$H^0=H_{n_0}$ is not a $G$-orbit we must get one of the following alternatives:
\begin{description}
\item[(i)] $(\lambda^\ast,\theta)$ is not an isolated solution of (\ref{e:Bi.2.7.3}) in
 $\{\lambda^\ast\}\times U$;

\item[(ii)] there exists a sequence $\{\kappa_n\}_{n\ge 1}\subset\mathbb{R}\setminus\{\lambda^\ast\}$
such that $\kappa_n\to\lambda^\ast$ and that for each $\kappa_n$ the equation
(\ref{e:Bi.2.7.3}) with $\lambda=\kappa_n$ has infinitely many $G$-orbits of solutions converging to
$\theta\in H$;

\item[(iii)] there exist left and right  neighborhoods $\Lambda^-$ and $\Lambda^+$ of $\lambda^\ast$ in $\mathbb{R}$
and integers $n^+, n^-\ge 0$, such that $n^++n^-\ge\dim H^0=\dim H_{n_0}$
and for $\lambda\in\Lambda^-\setminus\{\lambda^\ast\}$ (resp. $\lambda\in\Lambda^+\setminus\{\lambda^\ast\}$),
(\ref{e:Bi.2.7.3}) has at least $n^-$ (resp. $n^+$) distinct critical
$G$-orbits different from $\theta$, which converge to
zero $\theta$ as $\lambda\to\lambda^\ast$.
\end{description}
\end{theorem}

Different from Theorem~\ref{th:Bi.3.1}, if $(\lambda^\ast,\theta)$ is not an isolated solution of (\ref{e:Bi.2.7.3}) in
 $\{\lambda^\ast\}\times U$, Theorem~\ref{th:Bi.3.2} implies that
 there exist sequences $\lambda^+_k\downarrow\lambda^\ast$ and $\lambda^-_k\uparrow\lambda^\ast$
 such that for each $k\in\mathbb{N}$, the numbers of non-trivial critical orbits of
 $\mathcal{L}_{\lambda^+_k}=\mathcal{F}-\lambda^+_k \mathcal{G}$ plus those of
 non-trivial critical orbits of $\mathcal{L}_{\lambda^-_k}$ in $U$ are at least $\dim H^0$.
 Moreover, these non-trivial critical orbits converge $\theta$ as $k\to\infty$.
Note also that the count method for critical orbits in 3) of Theorem~\ref{th:Bi.3.1}
is not usual as the present one.\\

\noindent{\it Proof of Theorem~\ref{th:Bi.3.2}}. \quad We assume that the cases (i) and (ii) do not occur.
Moreover, we only consider cases:
{\bf 1)}  $\mathcal{F}''(\theta)$
is  positive definite, {\bf 2)}  (c) holds and
$\mathcal{F}''(\theta)$
 is  positive definite  on $H^0$.
In these cases,  by the proofs of Theorems~\ref{th:Bi.2.4},\ref{th:Bi.3.1}, we obtain:\\
{\bf A)}  if  $\theta\in H^0$ is a local minimizer of $\mathcal{L}^\circ_{\lambda^\ast}$,
i.e., $C_q(\mathcal{L}^\circ_{\lambda^\ast},\theta;{\bf K})=
\delta_{q0}{\bf K}$, then
\begin{eqnarray*}
&&\theta\in H^0\;\hbox{is a local minimizer of}\;\mathcal{L}^\circ_{\lambda},\quad
\forall \lambda\in [\lambda^\ast-\delta, \lambda^\ast],\\
&&\theta\in H^0\;\hbox{is a local maximizer of}\;\mathcal{L}^\circ_{\lambda},\quad
\forall \lambda\in (\lambda^\ast, \lambda^\ast+\delta];
\end{eqnarray*}
{\bf B)}  if  $\theta\in  H^0$
is a local  maximizer of $\mathcal{L}^\circ_{\lambda^\ast}$, i.e., $C_q(\mathcal{L}^\circ_{\lambda^\ast},\theta;{\bf K})=
\delta_{q\nu_{\lambda^\ast}}{\bf K}$, then
\begin{eqnarray*}
&&\theta\in H^0\;\hbox{is a local minimizer  of}\;\mathcal{L}^\circ_{\lambda},\quad
\forall \lambda\in [\lambda^\ast-\delta, \lambda^\ast),\\
&&\theta\in H^0\;\hbox{is a local maximizer of}\;\mathcal{L}^\circ_{\lambda},\quad
\forall \lambda\in [\lambda^\ast, \lambda^\ast+\delta];
\end{eqnarray*}
{\bf C)}  if  $\theta\in  H^0$
is neither a local  maximizer of $\mathcal{L}^\circ_{\lambda^\ast}$ nor
a local  maximizer of it, i.e., $C_q(\mathcal{L}^\circ_{\lambda^\ast},\theta;{\bf K})=0$
for $q=0, \nu_{\lambda^\ast}$, then
\begin{eqnarray*}
&&\theta\in H^0\;\hbox{is a local minimizer  of}\;\mathcal{L}^\circ_{\lambda},\quad
\forall \lambda\in [\lambda^\ast-\delta, \lambda^\ast),\\
&&\theta\in H^0\;\hbox{is a local maximizer of}\;\mathcal{L}^\circ_{\lambda},\quad
\forall \lambda\in (\lambda^\ast, \lambda^\ast+\delta].
\end{eqnarray*}

Let us shrink $\epsilon>0$  in (\ref{e:Bi.2.15})
such that $\theta$ is the only critical point of
$\mathcal{L}^\circ_{\lambda}$ in $B_H(\theta, \epsilon)\cap H^0$.
 Since $\lambda\mapsto\mathcal{L}^\circ_\lambda\in C^1(B_H(\theta, \epsilon)\cap H^0)$
is continuous by (\ref{e:S.5.4}), it is easy to see that
$$
R_{\delta,\epsilon}:=\{(\lambda, z)\in [\lambda^\ast-\delta, \lambda^\ast+\delta]\times (B_H(\theta, \epsilon)\cap H^0)\,|\,D\mathcal{L}^\circ_\lambda(z)\ne \theta\}$$
is an open subset in $[\lambda^\ast-\delta, \lambda^\ast+\delta]\times (B_H(\theta, \epsilon)\cap H^0)$, and
$$
R_{\delta,\epsilon}=\{(\lambda, z)\in [\lambda^\ast-\delta, \lambda^\ast+\delta]\times B_H(\theta, \epsilon)\cap H^0\,|\,z\in (B_H(\theta, \epsilon)\cap H^0)\setminus K(\mathcal{L}^\circ_\lambda)\},
$$
where $K(\mathcal{L}^\circ_\lambda)$ denotes the critical set of $\mathcal{L}^\circ_\lambda$.
As in the proof of Lemma~\ref{lem:pseudogradient}
we can produce a smooth map, $R_{\delta,\epsilon}\to H^0,\;(\lambda,z)\mapsto\mathscr{V}_\lambda(z)$,
such that each
$$
\mathscr{V}_\lambda: B_H(\theta, \epsilon)\cap H^0\setminus K(\mathcal{L}^\circ_\lambda)
\to H^0,\;z\mapsto\mathscr{V}_\lambda(z)
$$
is a $G$-equivariant $C^{\infty}$ pseudo-gradient vector field for
$\mathcal{L}^\circ_\lambda$.

Replacing  \cite[(11.1)]{Rab} (or \cite[(2.4)]{FaRa1}) for $G=\mathbb{Z}_2$,
and  \cite[(8.19)]{FaRa2} for $G=S^1$ by
\begin{equation}\label{e:Bi.3.9}
\frac{d\varphi_\lambda}{ds}=-\mathscr{V}_{\lambda}(\varphi_\lambda),\quad\varphi_\lambda(0,z)=z,
\end{equation}
we can repeat the constructions in \cite[\S1]{Rab} and \cite[\S8]{FaRa1} to obtain:

\begin{lemma}\label{lem:Bi.3.3}
There is a $G$-invariant open neighborhood $\mathscr{Q}$ of $\theta$ in $H^0$
with compact closure  $\overline{\mathscr{Q}}$ contained in
$B_H(\theta, \epsilon)\cap H^0$ such that
for every $\lambda$ close to $\lambda^\ast$, every $c\in\mathbb{R}$
and every $\tau_1>0$, every $G$-neighborhood $U$ of
$K_{\lambda,c}:=K(\mathcal{L}^\circ_\lambda)\cap\{z\in \overline{\mathscr{Q}}\,|\,
\mathcal{L}^\circ_\lambda(z)\le c\}$ there exists an $\tau\in (0,\tau_1)$ and
a $G$ equivariant homotopy $\eta:[0, 1]\times \overline{\mathscr{Q}}\to \overline{\mathscr{Q}}$
with the following properties:
\begin{description}
\item[$1^\circ$] $\eta(t,z)=z$ if $z\in \overline{\mathscr{Q}}\setminus (\mathcal{L}^\circ_\lambda)^{-1}[c-\tau_1, c+\tau_1]$;
\item[$2^\circ$] $\eta(t,\cdot)$ is homeomorphism of $\overline{\mathscr{Q}}$
to $\eta(t, \overline{\mathscr{Q}})$ for each $t\in [0,1]$;
 \item[$3^\circ$] $\eta(1, A_{\lambda,c+\tau}\setminus U)\subset A_{\lambda,c-\tau})$,
 where $A_{\lambda, d}:=\{z\in \overline{\mathscr{Q}}\,|\, \mathcal{L}^\circ_\lambda(z)\le d\}$;
\item[$4^\circ$] if $K_{\lambda,c}=\emptyset$, $\eta(1, A_{\lambda,c+\tau})\subset A_{\lambda,c-\tau})$.
\end{description}
\end{lemma}

For $\ast=+,-$, let $S^\ast
=\{z\in B_H(\theta, \epsilon)\cap H^0\,|\, \psi(s,z)\in B_H(\theta, \epsilon)\cap H^0,\;\forall \ast s>0\}$
and $T^\ast=S^\ast\cap\partial\overline{\mathscr{Q}}$. For $G=\mathbb{Z}_2$ (resp. $S^1$)
let $i_G$ denote the genus in \cite{Rab} (resp. the index in \cite[\S7]{FaRa2}).

\begin{lemma}\label{lem:Bi.3.4}
Both $T^+$ and $T^-$ are $G$-invariant compact subset of $\partial\overline{\mathscr{Q}}$, and also satisfy
\begin{description}
\item[$1^\circ$] $\min\{\mathcal{L}^\circ_\lambda(z)\,|\,z\in T^+\}>0$ and
$\max\{\mathcal{L}^\circ_\lambda(z)\,|\,z\in T^-\}<0$;
\item[$2^\circ$] $i_{\mathbb{Z}_2}(T^+)+ i_{\mathbb{Z}_2}(T^-)\ge\dim H^0$ and
$i_{S^1}(T^+)+ i_{S^1}(T^-)\ge\frac{1}{2}\dim H^0$.
\end{description}
\end{lemma}

Two inequalities in $2^\circ$ are \cite[Lemma~2.11]{FaRa1} and \cite[Theorem~8.30]{FaRa2}, respectively.

\noindent{\bf Case $G=\mathbb{Z}_2$}.\quad Suppose $i_{\mathbb{Z}_2}(T^-)=k>0$.
Let $c_j$ be defined by \cite[(2.13)]{FaRa1}, but $\bar{Q}$ and $g(\lambda,v)$
are replaced by $\overline{\mathscr{Q}}$ and $\mathcal{L}^\circ_\lambda(z)$,
respectively. We can modify the proof of (i) on the page 54 of \cite{FaRa1} as follows:

In the above three cases {\bf A)}, {\bf B)} and {\bf C)}, for each
$\lambda\in [\lambda^\ast-\delta, \lambda^\ast)$, $\theta\in H^0$ is always a local (strict) minimizer  of $\mathcal{L}^\circ_{\lambda}$. Therefore for arbitrary sufficiently small $\rho>0$, depending on $\lambda$,
$\mathcal{L}^\circ_{\lambda}(z)>0$ for any $0<\|z\|\le\rho$
and so
$$
c_1\ge \min_{\|z\|=\rho}\mathcal{L}^\circ_{\lambda}(z)>0.
$$
Other arguments are same. Hence we obtain: {\it
if $\lambda\in [\lambda^\ast-\delta, \lambda^\ast)$ is close to $\lambda^\ast$,
$\mathcal{L}^\circ_{\lambda}$ has at least $k$ distinct pairs of nontrivial
critical points, which also converge to $\theta$ as $\lambda\to\lambda^\ast$.}

Since for every $\lambda\in (\lambda^\ast, \lambda^\ast+\delta]$,
$\theta\in H^0$ is a local maximizer of $\mathcal{L}^\circ_{\lambda}$, by considering
$-\mathcal{L}^\circ_{\lambda}$ we get:
if $i_{\mathbb{Z}_2}(T^+)=l>0$, for every $\lambda\in (\lambda^\ast, \lambda^\ast+\delta]$  close to $\lambda^\ast$,
$\mathcal{L}^\circ_{\lambda}$ has at least $l$ distinct pairs of nontrivial
critical points converging to $\theta$ as $\lambda\to\lambda^\ast$.

These two claims together yield the desired result.

\noindent{\bf Case $G=S^1$}.\quad Suppose $i_{S^1}(T^-)=k>0$.
Similarly, for $c_j$ defined by \cite[(8.56)]{FaRa2}, we may replace \cite[(8.58), (8.63)]{FaRa2} by
$$
\mathcal{L}^\circ_{\lambda}(x)\ge \min_{\|z\|=\rho}\mathcal{L}^\circ_{\lambda}(z)>0,\quad\hbox{and so}\quad
c_{\gamma+1}\ge \min_{\|z\|=\rho}\mathcal{L}^\circ_{\lambda}(z)>0,
$$
and then repeat the arguments in \cite[\S8]{FaRa2} to complete the final proof.
Of course, we also use the fact that $\mathcal{L}^\circ_\lambda\to\mathcal{L}^\circ_{\lambda^\ast}$
uniformly on $\overline{\mathscr{Q}}$ as $\lambda\to\lambda^\ast$,
which can be derived from  (\ref{e:S.5.4}).
\hfill$\Box$\vspace{2mm}

Fadell--Rabinowitz theorems in \cite{FaRa1, FaRa2} were also generalized to
the case of arbitrary compact Lie groups for potential operators of $C^2$ functionals
by Bartsch and Clapp \cite{BaCl}, Bartsch \cite{Ba1}.
We now give generalizations of their some results.

Fix a set $\mathcal{A}$ of $G$-spaces, a multiplicative equivariant cohomology theory
$h^\ast$ and an ideal $I$ of the coefficient ring $R=h^\ast(pt)$.
Recall in \cite[Definition~4.1]{Ba1} that
the {\it $(\mathcal{A}, h^\ast, I)$-length} of a $G$-space $X$,
 $(\mathcal{A}, h^\ast, I)$-${\rm length}(X)$, was defined to be the smallest integer $k$ such that
there exist $A_1,\cdots, A_k$ in $\mathcal{A}$ with the following property:
For all $\gamma\in h^\ast(X)$ and for all $\omega_i\in I\cap{\rm kern}(h^\ast(pt)\to h^\ast(A_i))$,
$i=1,\cdots,k$, the product $\omega_1\cdot\ldots\cdot\omega_k\cdot\gamma=0$ in $h^\ast(X)$.
Moreover, set $(\mathcal{A}, h^\ast, I)$-${\rm length}(X)=\infty$ if no such $k$ exists.
 Let $SE$ denote the unit sphere in a Hilbert space $E$, and
 $G$ be a compact Lie group acting on $E$ orthogonally.
Denote by $\mathscr{G}$  the set of orbits occurring on $SE$.
Let $h^\ast$ be any continuous, multiplicative, equivariant cohomology theory such that
${\rm kern}(h^\ast(pt)\to h^\ast(G/H))$ is
a finitely generated ideal for all $G/H\in \mathscr{G}$.
Taking $I=R=h^\ast(pt)$ the $(\mathcal{A}, h^\ast, I)$-length of a $G$-space $X$
becomes the {\it $(\mathscr{G}, h^\ast)$-length} $\ell(X)$ defined in \cite{BaCl}.
For a bounded closed $G$-neighborhood $V$ of the origin in a $G$-module
$E$ it was proved in \cite[Lemma~1.6]{BaCl} that $\ell(\partial V)=\ell(SE)$.
If $G=\mathbb{Z}_2$ (resp. $S^1$) and $h^\ast=H^\ast_G$, $\ell$ becomes
${\rm index}_{\mathbb{R}}$ (resp. ${\rm index}_{\mathbb{C}}$) in \cite{FaRa2}.

\begin{hypothesis}\label{hyp:Bi.3.5}
{\rm  Let $G$ be a compact Lie group acting on $H$ orthogonally,
and $U$ a $G$-invariant open neighborhood of the origin of a real Hilbert space $H$.
Let $\mathcal{F}, \mathcal{G}=\mathcal{G}_1\in C^1(U,\mathbb{R})$ be
as in Theorem~\ref{th:Bi.2.2}, and $G$-invariant.
 Let $\lambda^\ast$ be an isolated eigenvalue of (\ref{e:Bi.2.7.4}), i.e.,
 for each $\lambda\ne\lambda^\ast$ near $\lambda^\ast$ the equation (\ref{e:Bi.2.7.4}) has only trivial solution,
and let $H^0$ be the corresponding eigenspace. (Every  eigenvalue of (\ref{e:Bi.2.7.4}) is isolated if $\mathcal{F}''(\theta)$ is invertible.)
Suppose $H^0\cap {\rm Fix}(G)=\{\theta\}$.}
\end{hypothesis}
As in the proof of Theorem~\ref{th:Bi.2.4},   applying Theorem~\ref{th:S.5.3}
to $\mathcal{L}_{\lambda}=\mathcal{F}-\lambda\mathcal{G}=
\mathcal{L}_{\lambda^\ast}-(\lambda^\ast-\lambda)\mathcal{G}$
with $\lambda\in [\lambda^\ast-\varepsilon, \lambda^\ast+\varepsilon]$
and $-\mathcal{G}$,  we have $\delta\in (0, \varepsilon]$, $\epsilon>0$ and a unique continuous map
$$
 \psi:[\lambda^\ast-\delta, \lambda^\ast+\delta]\times (B_H(\theta,\epsilon)\cap H^0)\to (H^0)^\bot
 $$
 as in (\ref{e:Bi.2.14.1}), such that (\ref{e:Bi.2.14.2}) and (\ref{e:Bi.2.15})--(\ref{e:Bi.2.16})
 hold. But the present  $\psi$ is $G$-equivariant and each
 $\mathcal{L}^\circ_\lambda$ is $G$-invariant.

\begin{hypothesis}\label{hyp:Bi.3.6}
{\rm  Under Hypothesis~\ref{hyp:Bi.3.5},
suppose that the continuous map $\psi$ as in (\ref{e:Bi.2.14.1}) is of class $C^1$
with respect to the second variable, which implies that $d\mathcal{L}^\circ_{\lambda}$ has G\^ateaux
derivative
\begin{eqnarray*}
  d^2\mathcal{L}^\circ_{\lambda}(z)(u,v)&=&(\mathcal{F}''_{\lambda}(z+ \psi(\lambda, z))
   (u+ D_z\psi(\lambda, z)u),v)_H\\
   &-&\lambda(\mathcal{G}''_{\lambda}(z+ \psi(\lambda, z))
   (u+ D_z\psi(\lambda, z)u),v)_H\quad\forall u,v\in H^0
  \end{eqnarray*}
  at every $z\in B_H(\theta, \epsilon)\cap H^0$ by (\ref{e:S.5.12.2}), and therefore
  for all $u,v\in H^0$,
  \begin{equation}\label{e:Bi.3.9.1}
d^2\mathcal{L}^\circ_{\lambda}(\theta)(u,v)=(\lambda^\ast-\lambda)(\mathcal{G}''(\theta)u,v)_H
=(\lambda^\ast-\lambda)(P^0\mathcal{G}''(\theta)|_{H^0}u,v)_H
\end{equation}
 because $\psi(\lambda,\theta)=\theta$
and $D_z\psi({\lambda}, \theta)=\theta$.
Furthermore, we assume that each $\mathcal{L}^\circ_{\lambda}$
is of class $C^2$.}
\end{hypothesis}

Since $\lambda^\ast$ is an isolated  eigenvalue  of (\ref{e:Bi.2.7.4}),
$0$ is an  eigenvalue  of $\mathcal{F}''(\theta)-\lambda_0\mathcal{G}''(\theta)$
with finite multiplicity,   isolated in the spectrum $\sigma(\mathcal{F}''(\theta)-\lambda_0\mathcal{G}''(\theta))$.
Hence for $\lambda$ near $\lambda^\ast$  the
 $0$-group ${\rm eig}_0(\mathcal{F}''(\theta)-\lambda \mathcal{G}''(\theta))$ (cf. Section~\ref{sec:B.1})
 is well-defined. Let $E_\lambda$ be the generalized eigenspace of
 $\mathcal{F}''(\theta)-\lambda \mathcal{G}''(\theta)$ belonging to
 ${\rm eig}_0(\mathcal{F}''(\theta)-\lambda \mathcal{G}''(\theta))\cap\mathbb{R}^-$.
 It is $G$-invariant.
 For $\lambda$ and $\lambda'$
near $\lambda^\ast$ the spaces $E_\lambda$ and $E_{\lambda'}$ are $G$-isomorphic
if $(\lambda-\lambda^\ast)(\lambda'-\lambda^\ast)>0$. (When the latter holds
the orthogonal eigenprojection (cf. \cite[page 181]{Ka} for the definition), $P_\lambda:H\to E_\lambda$,
 restricts to a $G$-isomorphism from $E_{\lambda'}$ onto $E_{\lambda}$.)
 Since $\theta\in H^0$ is a nondegenerate critical point of $\mathcal{L}^\circ_{\lambda}$
 for each $\lambda\ne\lambda^\ast$ near $\lambda^\ast$,  $P^0\mathcal{G}''(\theta)|_{H^0}:H^0\to H^0$ is an isomorphism by (\ref{e:Bi.3.9.1}). Let $H^0_+$ and $H^0_-$ be the positive and negative definite subspaces of $P^0\mathcal{G}''(\theta)|_{H^0}$, respectively. Then $H^0=H^0_+\oplus H^0_-$.
Let $F_\lambda^-$ (resp. $F_\lambda^+$) be the eigenspace belonging to
  $\sigma( d^2\mathcal{L}^\circ_{\lambda}(\theta))\cap\mathbb{R}^-$
(resp.  $\sigma( d^2\mathcal{L}^\circ_{\lambda}(\theta))\cap\mathbb{R}^+$).
Clearly,  $F_\lambda^+$ the orthogonal complement of
 $F_\lambda^-$ in $H^0$, and the spaces $F^-_\lambda$ and $F^-_{\lambda'}$ are $G$-isomorphic
if $(\lambda-\lambda^\ast)(\lambda'-\lambda^\ast)>0$. Hence if $\ell$ is the above
$(\mathscr{G}, h^\ast)$-length, for $\lambda<\lambda^\ast<\mu$ close to $\lambda^\ast$, the number
\begin{equation}\label{e:Bi.3.9.2}
d:=\ell(SH^0)-\min\{\ell(SF_{\lambda}^-)+\ell(SF_{\mu}^+),
 \ell(SF_{\lambda}^+)+\ell(SF_{\mu}^-)
\end{equation}
  is well-defined, and  if $\ell(SV)=c\cdot\dim V$ for every $G$-module $V$ with $V^G=\{\theta\}$ we have
 \begin{equation}\label{e:Bi.3.9.3}
 d=|\ell(SF_{\lambda}^-)-\ell(SF_{\mu}^-)|=c|\dim F_{\lambda}^--\dim F_{\mu}^-|=c|\dim E_{\lambda}-\dim E_{\mu}|
 \end{equation}
 (\cite{BaCl}). Here the final equality comes from the fact that $P^0$ defines a $G$-isomorphism from $E_\lambda$ onto $F_\lambda^-$ (cf. \cite[page 353]{BaCl}). Moreover, it is easy to see that
$F_\lambda^-=H^0_-$ and $F_\lambda^+=H^0_+$ for $\lambda<\lambda^\ast$, and
$F_\lambda^-=H^0_+$ and $F_\lambda^+=H^0_-$ for $\lambda>\lambda^\ast$.
(\ref{e:Bi.3.9.2}) and (\ref{e:Bi.3.9.3}), respectively, become
\begin{eqnarray}
&&d=\ell(SH^0)-2\min\{\ell(SH^0_-), \ell(SH^0_+)),\label{e:Bi.3.9.4}\\
&&d=|\ell(SH^0_-)-\ell(SH^0_+)|=c|\dim H^0_--\dim H^0_+|.\label{e:Bi.3.9.5}
\end{eqnarray}
Having these we may state the following partial generalization of \cite[Theorem~3.1]{BaCl}.

\begin{theorem}\label{th:Bi.3.14}
Under Hypothesis~\ref{hyp:Bi.3.6},  if the number $d$ in (\ref{e:Bi.3.9.2}) or (\ref{e:Bi.3.9.4})  is positive, then
$(\lambda^\ast,\theta)\in\mathbb{R}\times U$ is a bifurcation point  for the equation
(\ref{e:Bi.2.7.3})
 and  one of the following alternatives occurs:
\begin{description}
\item[(i)] $(\lambda^\ast,\theta)$ is not an isolated solution of (\ref{e:Bi.2.7.3}) in
 $\{\lambda^\ast\}\times U$.

\item[(ii)] there exist left and right neighborhoods $\Lambda_l$ and $\Lambda_r$ of $\lambda^\ast$ in $\mathbb{R}$
and integers $i_l, i_r\ge 0$ such that $i_l+i_r\ge d$ and for any $\lambda\in\Lambda_l\setminus\{\lambda^\ast\}$ (resp. $\lambda\in\Lambda_r\setminus\{\lambda^\ast\}$,
(\ref{e:Bi.2.7.3}) has at least $i_l$ (resp. $i_r$) distinct  nontrivial solution orbits,
which  converge to $\theta$ in $H$  as $\lambda\to\lambda^\ast$.
\end{description}
\end{theorem}

\begin{proof}
Suppose that (i) does not hold. Then $\theta\in H^0$ is an isolated critical point
of $\mathcal{L}^\circ_{\lambda^\ast}$. We  assume that $\theta\in H^0$ is a unique
critical point in the set $\overline{\mathscr{Q}}$ of Lemma~\ref{lem:Bi.3.3}.
Moreover  the flow $\varphi_\lambda$ of (\ref{e:Bi.3.9}) may be replaced by the negative gradient one $\chi_\lambda$ of $\mathcal{L}^\circ_{\lambda}$
since we have assumed $\mathcal{L}^\circ_{\lambda}$ to be of class $C^2$.
Then we get a corresponding Lemma~\ref{lem:Bi.3.3}.
For $\ast=+,-$, let $S^\ast
=\{z\in B_H(\theta, \epsilon)\cap H^0\,|\, \chi_\lambda(s,z)\in B_H(\theta, \epsilon)\cap H^0,\;\forall \ast s>0\}$
and $T^\ast=S^\ast\cap\partial\overline{\mathscr{Q}}$.
We have the following corresponding result with the part b) of \cite[Lemma~4.1]{BaCl}.

\begin{lemma}\label{lem:Bi.3.15}
Both $T^+$ and $T^-$ are $G$-invariant compact subset of $\partial\overline{\mathscr{Q}}$, and also satisfy
\begin{description}
\item[$1^\circ$] $\min\{\mathcal{L}^\circ_\lambda(z)\,|\,z\in T^+\}>0$ and
$\max\{\mathcal{L}^\circ_\lambda(z)\,|\,z\in T^-\}<0$;
\item[$2^\circ$] $T^+$ and $T^-$ can be deformed inside $\overline{\mathscr{Q}}\setminus\{\theta\}$
into arbitrarily small neighborhoods of $\theta\in \overline{\mathscr{Q}}$ such that
 $\mathcal{L}^\circ_{\lambda^\ast}$ does not change sign during the deformation.
\item[$3^\circ$] $\ell(T^+)+ \ell(T^-)\ge \ell(SH^0)$.
\end{description}
\end{lemma}

Recall $\Lambda=[\lambda^\ast-\delta,\lambda^\ast+\delta]$. Shrinking $\delta>0$
(if necessary) we may assume that
$\mathcal{L}^\circ_{\lambda}(T^+)\subset \mathbb{R}^+$ and
$\mathcal{L}^\circ_{\lambda}(T^-)\subset \mathbb{R}^-$ for all
$\lambda\in\Lambda$. Note that for each $\lambda\ne\lambda^\ast$ near
$\lambda^\ast$, $F_\lambda^+$ is  the tangent space of the stable manifold of the negative
gradient one $\chi_\lambda$ because $\theta\in H^0$ is a nondegenerate critical point
of the $C^2$ function $\mathcal{L}^\circ_{\lambda}$.
We can complete proofs of corresponding results with \cite[Lemmas~4.2,4.3]{BaCl}.
\end{proof}

Clearly, even if $G=S^1$ or $\mathbb{Z}_2$, Theorem~\ref{th:Bi.3.2} cannot be included in
Theorem~\ref{th:Bi.3.14} and the following two results.

Now consider  generalizations of bifurcation results \cite[\S7.5]{Ba1}.
 Once $(\mathcal{A}, h^\ast, I)$ is understood its $(\mathcal{A}, h^\ast, I)$-length is written
 as $\ell$ below.

Under Hypothesis~\ref{hyp:Bi.3.5}, let $\varphi_\lambda$ be the flow of  $\mathscr{V}_\lambda$
given by (\ref{e:Bi.3.9}). Assume that  $\Lambda$ is equipped with
trivial $G$-action. By the theorem on continuous dependence
of solutions of ordinary differential equations on initial values and parameters we obtain
an equivariant product flow parametrized by $\Lambda$ on $\Lambda\times B_{H^0}(\theta,\epsilon)$,
$(\lambda, z, t)\to\varphi(\lambda, z, t):=(\lambda,\varphi_\lambda(z,t))$.
Clearly, $\mathscr{V}_\lambda$ is gradient-like with Lyapunov-function $\mathcal{L}^\circ_\lambda$.
Since $\lambda^\ast$ is an isolated eigenvalue of (\ref{e:Bi.2.7.4}),
 $\theta\in H^0$ is an isolated critical point of $\mathcal{L}^\circ_\lambda$
(and so an isolated invariant set of $\varphi_\lambda$) for each $\lambda$ near $\lambda^\ast$ and
$\lambda\ne\lambda^\ast$. Let $\ell^u(\lambda,\theta)$
(resp. $\ell^s(\lambda,\theta)$) be the exit-length (resp. entry-length) of
$\theta$ with respect to $\varphi_\lambda$, see \cite[\S7.3]{Ba1}.
Moreover, $\ell^u(\lambda,\theta)$  is independent of $\lambda\in\Lambda\cap (\lambda^\ast,\infty)$ (resp. $\Lambda\cap (-\infty,\lambda^\ast)$) close to $\lambda^\ast$,
denoted by $\ell^u_+$ (resp. $\ell^u_-$).  See \cite[\S7.2, \S7.5]{Ba1} for these.
By Theorems~7.10,~7.11 in \cite{Ba1} we immediately obtain the following two theorems.

\begin{theorem}\label{th:Bi.3.16}
Under  Hypothesis~\ref{hyp:Bi.3.5}, suppose that
   $\ell^u_+\ne\ell^u_-$. Then  $(\lambda^\ast,\theta)$ is a bifurcation point of
$\nabla\mathcal{L}^\circ_{\lambda}$.
Moreover, if $\theta\in H^0$ is also an isolated critical point of $\mathcal{L}^\circ_{\lambda^\ast}$,
then there exits $\varepsilon>0$ and integers $d_l, d_r\ge 0$ with $d_l+d_r\ge|\ell^u_+-\ell^u_-|$
so that for each $\lambda\in (\lambda^\ast-\varepsilon,\lambda^\ast)$ respectively
$\lambda\in (\lambda^\ast,\lambda^\ast+\varepsilon)$
at least  one of the following alternatives occurs:
 \begin{description}
\item[(i)] there exist critical $G$-orbits $Gu^i_\lambda$, $i\in\mathbb{Z}\setminus\{0\}$,
of $\mathcal{L}^\circ_{\lambda}$ with $u^i_\lambda\ne\theta$,
$\mathcal{L}^\circ_\lambda(u^{-i}_\lambda)<\mathcal{L}^\circ_\lambda(\theta)
<\mathcal{L}^\circ_\lambda(u^{i}_\lambda)$ for $i\ge 1$,
and $\mathcal{L}^\circ_\lambda(u^{i}_\lambda)\to\mathcal{L}^\circ_\lambda(\theta)$ as $|i|\to\infty$. In particular, $\mathcal{L}^\circ_\lambda$
has infinitely many critical $G$-orbits. They converge to $\theta\in H^0$ as $\lambda\to\lambda^\ast$.
 \item[(ii)] There exists an isolated invariant
  set $S_\lambda\in B_{H^0}(\theta, \epsilon)\setminus\{\theta\}$
  with $\ell(\mathcal{C}(S_\lambda))\ge d_l$ respectively
  $\ell(\mathcal{C}(S_\lambda))\ge d_r$. Moreover, $S_\lambda$
  converge to $\theta\in H^0$ as $\lambda\to\lambda^\ast$, i.e., for any neighborhood
  $W$ of $\theta$ in $H^0$ there exist $\epsilon=\epsilon(W)>0$ such that
  $S_\lambda\in W$ if $|\lambda-\lambda^\ast|<\epsilon$.
\end{description}
The result is also true for $G=\mathbb{Z}/p$ and $\ell=\ell_0+\ell_1-1$
as in \cite[Remark~4.14]{Ba1}.
\end{theorem}

\begin{theorem}\label{th:Bi.3.17}
Under  Hypothesis~\ref{hyp:Bi.3.5},
 suppose that   $G=\mathbb{Z}/p$, $p$ a prime, or $G=S^1\times\Gamma$,
$\Gamma$ a finite group,  and let $\ell$ be any of the lengths defined
in \cite[4.4, 4.14 or 4.16]{Ba1}.
 If $\theta\in H^0$ is also an isolated critical point of $\mathcal{L}^\circ_{\lambda^\ast}$,
 and  $\ell^u_-, \ell^u_+$ are the exit-lengths as above, then
there exits $\varepsilon>0$ and integers $d_l, d_r\ge 0$ with $d_l+d_r\ge|\ell^u_+-\ell^u_-|$
such  that the following holds: For each $\lambda\in (\lambda^\ast-\varepsilon,\lambda^\ast)$ respectively
$\lambda\in (\lambda^\ast,\lambda^\ast+\varepsilon)$
there exists a compact invariant
  set $S_\lambda\in B_{H^0}(\theta, \epsilon)\setminus\{\theta\}$
  with $\ell(\mathcal{C}(S_\lambda))\ge d_l$ respectively
  $\ell(\mathcal{C}(S_\lambda))\ge d_r$.
\end{theorem}

Every $G$-critical orbit of $\mathcal{L}^\circ_\lambda$ produced by
Theorems~\ref{th:Bi.3.16},~\ref{th:Bi.3.17} gives rise to
a $G$-critical orbit of $\mathcal{L}_\lambda$, and different orbits yield
different ones too. In particular, $(\lambda^\ast,\theta)\in\mathbb{R}\times U$ is a bifurcation point  for the equation (\ref{e:Bi.2.7.3}). However, in order to understand $\ell^u_+$ and $\ell^u_-$ we assume that
 Hypothesis~\ref{hyp:Bi.3.6} is satisfied. Then  the flow $\varphi_\lambda$  may be replaced by the negative gradient one $\chi_\lambda$ of $\mathcal{L}^\circ_{\lambda}$. By the arguments below Definition~7.1 in \cite{Ba1} we have
$\ell^u_+=\ell^u(\lambda,\theta)=\ell(SF_\lambda^-)=\ell(SH^0_+)=\ell(SE_\lambda)$
(resp. $\ell^u_+=\ell^u(\lambda,\theta)=\ell(SF_\lambda^-)=\ell(SH^0_-)=\ell(SE_\lambda)$) if
$\lambda>\lambda^\ast$ (resp. $\lambda<\lambda^\ast$) is close to $\lambda^\ast$.
It follows that $\ell(SE_{\lambda^\ast_-}):=\ell(SE_{\lambda^\ast-\rho})$ and
$\ell(SE_{\lambda^\ast_+}):=\ell(SE_{\lambda^\ast+\rho})$
are independent of  small $\rho>0$ and that
$$
|\ell^u_+-\ell^u_-|=|\ell(SH^0_+)-\ell(SH^0_-)|=|\ell(SE_{\lambda^\ast_+})-\ell(SE_{\lambda^\ast_-})|,
$$
which may be chosen as  $|\dim E_{\lambda^\ast_+}-\dim E_{\lambda^\ast_-}|$ (resp.
$\frac{1}{2}|\dim E_{\lambda^\ast_+}-\dim E_{\lambda^\ast_-}|$) if
$G=\mathbb{Z}/p$ with a prime $p$ (resp.  $G=S^1\times\Gamma$ with a finite group
$\Gamma$). By these, under Hypothesis~\ref{hyp:Bi.3.6}, Theorem~\ref{th:Bi.3.16},~\ref{th:Bi.3.17}
may be directly transformed two results about bifurcation information of
(\ref{e:Bi.2.7.3})  from  $(\lambda^\ast,\theta)$. We here omit them.

\begin{remark}\label{rem:Bi.3.18}
{\rm (i) If $n=\dim\Omega=1$, and  $\mathcal{F}, \mathcal{G}$
are defined by (\ref{e:1.3}) under \textsf{Hypothesis} $\mathfrak{F}_{2,N}$,
 then they  can satisfy Hypothesis~\ref{hyp:Bi.3.6} on  $W^{m,2}_0(\Omega,\mathbb{R}^N)$; see \cite{Lu1,Lu9}.
Hence the last three theorems may be applied in this case.\\
(ii) The proof key of \cite[Theorem~7.12]{Ba1} is to use the (local) center manifold theorem instead of the Lyapunov-Schmidt reduction. This method was firstly used by Chow and Lauterbach \cite{ChowLa}
in the non-equivariant case.  Our potential operators are neither strictly Fr\'echlet differentiable at $\theta$ nor of class $C^{1}$.  Hence the  center manifold theorem seems unable to be used in our situation.
But, from the constructions of center manifolds by  Vanderbauwhede and Iooss \cite{VanI}
this method is also possible if the restrictions of our potential operators
to a Banach space $X$ continuously and densely embedding in $H$ are of class $C^1$
as in the framework of \cite{Lu1,Lu2}; see \cite{Lu7}.}
\end{remark}


The bifurcations in the previous theorems are all from a trivial critical orbit.
Finally, let us give  a result about bifurcations starting a nontrivial critical orbit.

\begin{hypothesis}\label{hyp:Bi.3.19}
{\rm Under Hypothesis~\ref{hyp:S.6.2},
let for some $x_0\in\mathcal{ O}$ the pair
$(\mathcal{L}\circ\exp|_{N\mathcal{O}(\varepsilon)_{x_0}},  N\mathcal{O}(\varepsilon)_{x_0})$
satisfy the corresponding conditions with Hypothesis~\ref{hyp:1.1} with $X=H$.
Let $\mathcal{G}\in C^1(\mathcal{H},\mathbb{R})$  be  $G$-invariant, have
a critical orbit $\mathcal{O}$, and also satisfy:
\begin{description}
\item[(i)] the gradient $\nabla\mathcal{G}$ is G\^ateaux differentiable
near $\mathcal{O}$, and every derivative $\mathcal{G}''(u)$ is also a compact
linear operator;
\item[(ii)] $\mathcal{G}''$ are continuous at each point $u\in\mathcal{O}$.
\end{description}

The assumptions on $\mathcal{G}$ assure that the functional $\mathcal{L}-\lambda\mathcal{G}$,
$\lambda\in\mathbb{R}$, also satisfy the conditions of
Theorems~\ref{th:S.6.0},~\ref{th:S.6.1}.

 Let $\lambda^\ast$ be an eigenvalue  of
\begin{equation}\label{e:Bi.3.10}
\mathcal{L}''(x_0)v-\lambda\mathcal{G}''(x_0)v=0,\quad v\in T_{x_0}\mathcal{H},
\end{equation}
and $\dim{\rm Ker}(\mathcal{L}''(x_0)-\lambda^\ast\mathcal{G}''(x_0))>\dim \mathcal{O}$.}
\end{hypothesis}

We say $\mathcal{O}$ to be a {\it bifurcation $G$-orbit with parameter $\lambda^\ast$} of
 the equation
\begin{equation}\label{e:Bi.3.11}
\mathcal{L}'(u)=\lambda\mathcal{G}'(u),\quad u\in \mathcal{H}
\end{equation}
if for any $\varepsilon>0$ and neighborhood $\mathscr{U}$ of $\mathcal{O}$ in $\mathcal{H}$
there exists a solution $G$-orbit $\mathcal{O}'\ne \mathcal{O}$ in $\mathscr{U}$ of
(\ref{e:Bi.3.11}) with some $\lambda\in (-\varepsilon,\varepsilon)$.

Note that  the orthogonal
complementary of $T_{x_0}\mathcal{O}$ in $T_{x_0}\mathcal{H}$, $N\mathcal{O}_{x_0}$,
 is an invariant subspace of $\mathcal{L}''(x_0)$ and $\mathcal{G}''(x_0)$.
Let $\mathcal{L}''(x_0)^\bot$ (resp. $\mathcal{G}''(x_0)^\bot$) denote
the restriction self-adjoint operator of $\mathcal{L}''(x_0)$ (resp. $\mathcal{G}''(x_0)$)
from  $N\mathcal{O}_{x_0}$ to itself. Then
$\mathcal{L}''(x_0)^\bot=d^2(\mathcal{L}\circ\exp|_{N\mathcal{O}(\varepsilon)_{x_0}})(\theta)$
and $\mathcal{G}''(x_0)^\bot=d^2(\mathcal{G}\circ\exp|_{N\mathcal{O}(\varepsilon)_{x_0}})(\theta)$.
Suppose that $\mathcal{L}''(x_0)^\bot$ is invertible, or equivalently
 ${\rm Ker}(\mathcal{L}''(x_0))=T_{x_0}\mathcal{O}$. Then
$0$ is not an eigenvalue  of
\begin{equation}\label{e:Bi.3.12}
\mathcal{L}''(x_0)^\bot v-\lambda\mathcal{G}''(x_0)^\bot v=0,\quad v\in N\mathcal{O}_{x_0},
\end{equation}
and $\lambda\in\mathbb{R}\setminus\{0\}$
is an eigenvalue  of (\ref{e:Bi.3.12}) if and only if $1/\lambda$
is an eigenvalue  of compact linear self-adjoint operator $L_{x_0}:=[\mathcal{L}''(x_0)^\bot ]^{-1}\mathcal{G}''(x_0)^\bot\in\mathscr{L}_s(N\mathcal{O}_{x_0})$.
Hence $\sigma(L_{x_0})\setminus\{0\}=\{1/\lambda_n\}_{n=1}^\infty\subset\mathbb{R}$
with $\lambda_n\to 0$, and each $1/\lambda_n$ has finite multiplicity.
Let $N\mathcal{O}_{x_0}^n$ be the eigensubspace corresponding to $1/\lambda_n$ for $n\in\mathbb{N}$.
Then  $N\mathcal{O}_{x_0}^0={\rm Ker}(L_{x_0})={\rm Ker}(\mathcal{G}''(x_0)^\bot)$ and
\begin{equation}\label{e:Bi.3.13}
N\mathcal{O}_{x_0}^n={\rm Ker}(I/\lambda_n-L_{x_0})={\rm Ker}(\mathcal{L}''(x_0)^\bot-\lambda_n \mathcal{G}''(x_0)^\bot),\quad
n=1,2,\cdots,
\end{equation}
and $N\mathcal{O}_{x_0}=\oplus^\infty_{n=0}N\mathcal{O}_{x_0}^n$.

\begin{theorem}\label{th:Bi.3.20}
Under Hypothesis~\ref{hyp:Bi.3.19}, suppose that ${\rm Ker}(\mathcal{L}''(x_0))=T_{x_0}\mathcal{O}$
(so the operator $\mathcal{L}''(x_0)^\bot$
is invertible) and  $\lambda^\ast=\lambda_{n_0}$ for some $n_0\in\mathbb{N}$.
 Then $\mathcal{O}$ is a bifurcation $G$-orbit with parameter $\lambda^\ast$ of
 (\ref{e:Bi.3.11}) if one of the following two conditions holds:\\
 {\rm a)} $\mathcal{L}''(x_0)^\bot$ is either positive definite or negative one, and
 \begin{equation}\label{e:Bi.3.14}
C_{l}(\mathcal{O};\mathbb{Z}_2)\ne C_{l-\nu_{\lambda^\ast}}(\mathcal{O};\mathbb{Z}_2)
\end{equation}
 for some $l\in\mathbb{Z}$, where  $\nu_{\lambda^\ast} =\dim N\mathcal{O}_{x_0}^{n_0}$
 (is more than zero because $\mathcal{O}$ is a degenerate
critical orbit of $\mathcal{L}_{\lambda^\ast}$ by the assumption in Hypothesis~\ref{hyp:Bi.3.19});  \\
 {\rm b)}  each $N\mathcal{O}_{x_0}^n$ in (\ref{e:Bi.3.13})
 is an invariant subspace of $\mathcal{L}''(x_0)^\bot$
 (e.g. these are true if $\mathcal{L}''(x_0)^\bot$
 commutes with $\mathcal{G}''(x_0)^\bot$), and
  \begin{equation}\label{e:Bi.3.15}
 C_{l-\nu^-_{\lambda^\ast}}(\mathcal{O};\mathbb{Z}_2)\ne
 C_{l-\nu^+_{\lambda^\ast}}(\mathcal{O};\mathbb{Z}_2)
 \end{equation}
 for some $l\in\mathbb{Z}$, where
$\nu^+_{\lambda^\ast}$ (resp. $\nu^-_{\lambda^\ast}$) is the dimension of
the positive (resp. negative) definite space of
$\mathcal{L}''(x_0)^\bot$  on $N\mathcal{O}_{x_0}^{n_0}$.
\end{theorem}

 From the following proof it is easily seen that
 (\ref{e:Bi.3.15}) may be replaced by (\ref{e:Bi.3.14})
 if we add a condition ``$\mathcal{L}''(x_0)^\bot$ is either positive definite or negative one on
 $N\mathcal{O}_{x_0}^{n_0}$" in b).\\

\noindent{\it Proof of Theorem~\ref{th:Bi.3.20}}. \quad Let $\mu_\lambda$ denote the Morse index of $\mathcal{L}_\lambda:=\mathcal{L}-\lambda\mathcal{G}$ at $\mathcal{O}$,
 $\lambda^\ast=\lambda_{n_0}$ for some $n_0\in\mathbb{N}$, and
 let $\nu_{\lambda^\ast}$ be the nullity of  $\mathcal{L}_{\lambda^\ast}$
 at $\mathcal{O}$, i.e., $\nu_{\lambda^\ast} =\dim N\mathcal{O}_{x_0}^{n_0}$.
As in the proof of Theorem~\ref{th:Bi.2.4} we have $\varepsilon>0$ such that
\begin{equation}\label{e:Bi.3.16}
\mu_\lambda=\sum_{\lambda_n<\lambda}\dim N\mathcal{O}_{x_0}^n=\left\{\begin{array}{ll}
\mu_{\lambda^\ast}, &\quad\forall \lambda\in (\lambda^\ast- 2\varepsilon, \lambda^\ast],\\
\mu_{\lambda^\ast}+ \nu_{\lambda^\ast}, &\quad\forall \lambda\in (\lambda^\ast, \lambda^\ast+2\varepsilon)
\end{array}\right.
\end{equation}
if  $\mathcal{L}''(x_0)^\bot$ is  positive definite,
\begin{equation}\label{e:Bi.3.17}
\mu_\lambda=\sum_{\lambda_n>\lambda}\dim N\mathcal{O}_{x_0}^n=\left\{\begin{array}{ll}
\mu_{\lambda^\ast}+ \nu_{\lambda^\ast}, &\quad\forall \lambda\in (\lambda^\ast- 2\varepsilon, \lambda^\ast),\\
\mu_{\lambda^\ast}, &\quad\forall \lambda\in [\lambda^\ast, \lambda^\ast+2\varepsilon)
\end{array}\right.
\end{equation}
if  $\mathcal{L}''(x_0)^\bot$ is negative definite, and
\begin{equation}\label{e:Bi.3.18}
\mu_\lambda=\left\{\begin{array}{ll}
\mu_{\lambda^\ast}+ \nu^-_{\lambda^\ast}, &\quad\forall \lambda\in (\lambda^\ast- 2\varepsilon, \lambda^\ast),\\
\mu_{\lambda^\ast}+ \nu^+_{\lambda^\ast}, &\quad\forall \lambda\in (\lambda^\ast, \lambda^\ast+2\varepsilon)
\end{array}\right.
\end{equation}
if b) holds.

Corresponding to Claim 1 in the proof of Theorem~\ref{th:Bi.2.4} we may also prove

 \noindent{\bf Claim 1}.\quad {\it After shrinking $\varepsilon>0$,
  if $\{(\kappa_n, v_n)\}_{n\ge 1}\subset
 [\lambda^\ast-\varepsilon, \lambda^\ast+\varepsilon]\times \overline{N\mathcal{O}(\varepsilon)}$
 satisfies $\nabla\mathcal{L}_{\kappa_n}(v_n)\to\theta$ and $\kappa_n\to \kappa_0$, then
 $\{v_n\}_{n\ge 1}$ has a convergent subsequence in $\overline{N\mathcal{O}(\varepsilon)}$.}

By a contradiction, assume that
$\mathcal{O}$ is not a bifurcation $G$-orbit with parameter $\lambda^\ast$ of
 (\ref{e:Bi.3.11}). Then we have $\delta\in (0, \varepsilon]$ such that
 for each  $\lambda\in [\lambda^\ast-\delta, \lambda^\ast+\delta]$, $\mathcal{O}$
 is an unique critical orbit of $\mathcal{L}_\lambda$
in $\overline{N\mathcal{O}(\delta)}$.
By \cite[Theorem~5.1.21]{Ch1} (or as in \cite{ChWa}) we deduce
\begin{equation}\label{e:Bi.3.19}
C_\ast(\mathcal{L}_{\lambda'}, \mathcal{O};{\bf K})=
C_\ast(\mathcal{L}_{\lambda''}, \mathcal{O};{\bf K}),\quad\forall\lambda', \lambda''\in
[\lambda^\ast-\delta,\lambda^\ast+\delta].
\end{equation}
Since $\mathcal{O}$ is a nondegenerate critical orbit of $\mathcal{L}_\lambda$
for each $\lambda\in [\lambda^\ast-\delta,\lambda^\ast+\delta]\setminus\{\lambda^\ast\}$,
as in the proof of (\ref{e:S.6.18.2}) we derive from (\ref{e:S.6.26}) that
\begin{eqnarray}\label{e:Bi.3.20}
C_\ast(\mathcal{L}_{\lambda'}, \mathcal{O};\mathbb{Z}_2)=
C_{\ast-\mu_{\lambda'}}(\mathcal{O};\mathbb{Z}_2)\quad\hbox{and}\quad
C_\ast(\mathcal{L}_{\lambda''}, \mathcal{O};\mathbb{Z}_2)=
C_{\ast-\mu_{\lambda''}}(\mathcal{O};\mathbb{Z}_2)
\end{eqnarray}
for any $\lambda'\in [\lambda^\ast-\delta,\lambda^\ast)$ any $\lambda''\in (\lambda^\ast,\lambda^\ast+\delta]$.

If  $\mathcal{L}''(x_0)^\bot$ is  positive definite,
by (\ref{e:Bi.3.14}), (\ref{e:Bi.3.16}) and (\ref{e:Bi.3.20}) we deduce
$$
C_{l+ \mu_{\lambda^\ast}}(\mathcal{L}_{\lambda'}, \mathcal{O};\mathbb{Z}_2)
=C_{l}(\mathcal{O};\mathbb{Z}_2)\ne C_{l-\nu_{\lambda^\ast}}(\mathcal{O};\mathbb{Z}_2)
=C_{l+ \mu_{\lambda^\ast}}(\mathcal{L}_{\lambda''}, \mathcal{O};\mathbb{Z}_2)
$$
for any $\lambda'\in [\lambda^\ast-\delta,\lambda^\ast)$ any $\lambda''\in (\lambda^\ast,\lambda^\ast+\delta]$.
This contradicts (\ref{e:Bi.3.19}).

Similarly, if $\mathcal{L}''(x_0)^\bot$ is  negative definite,
 we deduce
$$
C_{l+ \mu_{\lambda^\ast}}(\mathcal{L}_{\lambda''}, \mathcal{O};\mathbb{Z}_2)
=C_{l}(\mathcal{O};\mathbb{Z}_2)\ne C_{l-\nu_{\lambda^\ast}}(\mathcal{O};\mathbb{Z}_2)
=C_{l+ \mu_{\lambda^\ast}}(\mathcal{L}_{\lambda'}, \mathcal{O};\mathbb{Z}_2)
$$
for any $\lambda'\in [\lambda^\ast-\delta,\lambda^\ast)$ any $\lambda''\in (\lambda^\ast,\lambda^\ast+\delta]$,
 and also arrive at a contradiction to (\ref{e:Bi.3.19}).

If {\rm b)} holds, by (\ref{e:Bi.3.17}), (\ref{e:Bi.3.20}) and (\ref{e:Bi.3.15}) we have
\begin{eqnarray*}
C_{l+ \mu_{\lambda^\ast}}(\mathcal{L}_{\lambda'}, \mathcal{O};\mathbb{Z}_2)
=C_{l-\nu^-_{\lambda^\ast}}(\mathcal{O};\mathbb{Z}_2)\ne
C_{l-\nu^+_{\lambda^\ast}}(\mathcal{O};\mathbb{Z}_2)=
C_{l+ \mu_{\lambda^\ast}}(\mathcal{L}_{\lambda''}, \mathcal{O};\mathbb{Z}_2)
\end{eqnarray*}
for any $\lambda'\in [\lambda^\ast-\delta,\lambda^\ast)$ any $\lambda''\in (\lambda^\ast,\lambda^\ast+\delta]$,
which contradicts (\ref{e:Bi.3.19}).
\hfill$\Box$\vspace{2mm}

Using Theorem~\ref{th:S.6.7*} many results above can be generalized the case of bifurcations at a nontrivial critical orbit.

\part{Applications to quasi-linear elliptic systems of higher order}\label{part:2}

\section{Fundamental analytic properties for functionals $\mathcal{F}$ and $\mathfrak{F}$}\label{sec:Funct}
\setcounter{equation}{0}

\subsection{Results and preliminaries}\label{sec:Funct.1}

 A bounded domain $\Omega$ in $\R^n$ is said to be a {\it Sobolev domain}
 if for each integer $0\le k\le m-1$  the Sobolev space embeddings hold for it, i.e.,
 \begin{eqnarray*}
 && W^{m,p}(\Omega)\hookrightarrow W^{k,q}(\Omega)\qquad\hbox{if}\qquad \frac{1}{q}\ge\frac{1}{p}-\frac{m-k}{n}>0,\\
  && W^{m,p}(\Omega)\hookrightarrow\hookrightarrow W^{k,q}(\Omega)\qquad\hbox{if}\qquad \frac{1}{q}>\frac{1}{p}-\frac{m-k}{n}>0,\\
 && W^{m,p}(\Omega)\hookrightarrow\hookrightarrow W^{k,q}(\Omega)\qquad\hbox{if}\qquad q<\infty,\;\frac{1}{p}=\frac{m-k}{n},\\
 && W^{m,p}(\Omega)\hookrightarrow\hookrightarrow C^{k,\sigma}(\overline{\Omega})\qquad\hbox{if}\qquad \frac{n}{p}<m-(k+\sigma),\quad 0\le\sigma<1,
 \end{eqnarray*}
 where $\hookrightarrow\hookrightarrow$ denotes the compact embedding.
 Each bounded domain $\Omega$ in $\R^n$ with suitable smooth boundary
 $\partial\Omega$ is a Sobolev domain.

The key result of this section is the following theorem.
Most claims of it come from the auxiliary theorem 16 in Section 3.4 of
 Chapter 3 in \cite{Skr1} or Lemma~3.2 on the page 112 of \cite{Skr3}.
 Since only partial claims were  proved with $V=W^{m,p}_0(\Omega)$ therein,
 for completeness we give a detailed proof of it.

 \begin{theorem}\label{th:3.1}
Let $\Omega\subset\R^n$ be a bounded  Sobolev domain,
    $p\in [2,\infty)$ and let $V$ be a closed subspace of $W^{m,p}(\Omega)$.
   Suppose that (i)-(ii) in \textsf{Hypothesis} $\mathfrak{f}_p$ hold.
   Then we have\\
     {\bf A)}. On $V$ the functional $\mathcal{F}$
in (\ref{e:1.8}) is bounded on any bounded subset, of class $C^1$, and the derivative
$\mathcal{F}'(u)$ of $\mathcal{F}$ at $u$ is given by
\begin{equation}\label{e:3.1}
\langle \mathcal{F}'(u), v\rangle=\sum_{|\alpha|\le m}\int_\Omega f_\alpha(x,
u(x),\cdots, D^m u(x))D^\alpha v dx,\quad\forall v\in V.
\end{equation}
Moreover, the map $u\to \mathcal{F}'(u)$ also maps bounded subset into bounded ones.\\
{\bf B)}. The map $\mathcal{F}'$  is of class $C^1$ on $V$ if $p>2$,  G\^ateaux differentiable on $V$ if $p=2$,
  and for each $u\in V$ the derivative  $D\mathcal{F}'(u)\in \mathscr{L}(V, V^\ast)$ is given by
  \begin{equation}\label{e:3.2}
  \langle D\mathcal{F}'(u)v,\varphi\rangle=\sum_{|\alpha|,|\beta|\le m}\int_\Omega
  f_{\alpha\beta}(x, u(x),\cdots, D^m u(x))D^\beta v\cdot D^\alpha\varphi dx.
  \end{equation}
(In the case $p=2$,  equivalently,   the gradient map of $\mathcal{F}$, $V\ni u\mapsto\nabla \mathcal{F}(u)\in V$,
 given by
\begin{equation}\label{e:3.3}
(\nabla \mathcal{F}(u), v)_{m,2}=\langle \mathcal{F}'(u), v\rangle\quad \forall v\in V,
\end{equation}
has a G\^ateaux derivative $D(\nabla \mathcal{F})(u)\in\mathscr{L}_s(V)$ at every $u\in V$.)
Moreover,   $D\mathcal{F}'$ also satisfies the following properties:
\begin{description}
\item[(i)] For every given $R>0$, $\{D\mathcal{F}'(u)\,|\, \|u\|_{m,p}\le R\}$
is bounded in $\mathscr{L}_s(V)$.
Consequently, when $p=2$, $F$ is on $V$ of class $C^{2-0}$.
\item[(ii)] For any $v\in V$, $u_n\to
u_0$ implies
$D\mathcal{F}'(u_n)v\to D\mathcal{F}'(u_0)v$ in $V^\ast$.
\item[(iii)] If $p=2$ and $f(x,\xi)$ is independent of all variables $\xi_\alpha$, $|\alpha|=m$,
then $V\ni u\mapsto D\mathcal{F}'(\bar{u})\in\mathscr{L}(V, V^\ast)$ is continuous,
i.e., $\mathcal{F}$ is of class $C^2$,
and  $D(\nabla\mathcal{F})(u): V\to V$ is completely continuous
for each $u\in V$.
\end{description}
In addition,  if (iii) in \textsf{Hypothesis} $\mathfrak{f}_p$ is also satisfied, we further have\\
{\bf C)}.   $\mathcal{F}'$ satisfies condition $(S)_+$.\\
{\bf D)}. Suppose  $p=2$.   For $u\in V$, let $D(\nabla\mathcal{F})(u)$, $P(u)$ and $Q(u)$ be
  operators in $\mathscr{L}(V)$ defined by
  \begin{eqnarray}\label{e:3.4}
  (D(\nabla\mathcal{F})(u)v,\varphi)_{m,2}&=&\sum_{|\alpha|,|\beta|\le m}\int_\Omega
  f_{\alpha\beta}(x, u(x),\cdots, D^m u(x))D^\beta v\cdot D^\alpha\varphi dx,\\
   (P(u)v,\varphi)_{m,2}&=&\sum_{|\alpha|=|\beta|=m}\int_\Omega
  f_{\alpha\beta}(x, u(x),\cdots, D^m u(x))D^\beta v\cdot D^\alpha\varphi dx\nonumber\\
  &&+ \sum_{|\alpha|\le m-1}\int_\Omega  D^\alpha v\cdot D^\alpha\varphi dx,\label{e:3.5}\\
   (Q(u)v,\varphi)_{m,2}&=&\sum_{|\alpha|+|\beta|<2m}\int_\Omega
  f_{\alpha\beta}(x, u(x),\cdots, D^m u(x))D^\beta v\cdot D^\alpha\varphi dx\nonumber\\
  &&-\sum_{|\alpha|\le m-1}\int_\Omega  D^\alpha v\cdot D^\alpha\varphi dx,\label{e:3.6}
  \end{eqnarray}
  respectively.
    Then $D(\nabla\mathcal{F})=P+ Q$,  and
 \begin{description}
\item[(i)]  for any $v\in V$, the map $V\ni u\mapsto P(u)v\in W^{m,2}(\Omega)$ is continuous;
\item[(ii)] for every given $R>0$ there exist positive constants $C(R, n, m, \Omega)$ such that
$$
(P(u)v,v)_{m,2}\ge C\|v\|^2_{m,2}\qquad\forall v\in V
$$
if $u\in W^{m,2}(\Omega)$ satisfies $\|u\|_{m,2}\le R$;
\item[(iii)] $V\ni u\mapsto Q(u)\in\mathscr{L}(V)$ is continuous,
and  $Q(u): V\to V$ is completely continuous
for each $u$;
\item[(iv)] for every given $R>0$ there exist positive constants $C_j(R, n, m, \Omega), j=1,2$ such that
$$
(D(\nabla\mathcal{F})(u)v,v)_{m,2}\ge C_1\|v\|^2_{m,2}-C_2\|v\|^2_{m-1,2}\qquad\forall v\in V
$$
if $u\in V$ satisfies $\|u\|_{m,2}\le R$.
\end{description}
\end{theorem}

This is a special case of the following result.

 \begin{theorem}\label{th:3.2}
Let $\Omega\subset\R^n$ and $p\in [2,\infty)$ be as in Theorem~\ref{th:3.1},
$N\ge 1$ an integer, and $V$  a closed subspace of $W^{m,p}(\Omega, \mathbb{R}^N)$.
Suppose that (i)-(ii) in \textsf{Hypothesis} $\mathfrak{F}_{p,N}$ hold.
   Then corresponding conclusions to A) and B) in Theorem~\ref{th:3.1} are also true
   if the letters $u,v,\mathcal{F}$ therein are replaced by $\vec{u}, \vec{v}, \mathfrak{F}$, respectively,
   and (\ref{e:3.1})--(\ref{e:3.2}) are changed into
   \begin{eqnarray}\label{e:3.6.1}
\langle \mathfrak{F}'(\vec{u}), \vec{v}\rangle=\sum^N_{i=1}\sum_{|\alpha|\le m}\int_\Omega F^i_\alpha(x,
\vec{u}(x),\cdots, D^m \vec{u}(x))D^\alpha v^i dx,\quad\forall \vec{v}\in V,\\
  \langle D\mathfrak{F}'(\vec{u})\vec{v},\vec{\varphi}\rangle=\sum^N_{i,j=1}\sum_{|\alpha|,|\beta|\le m}\int_\Omega
  F^{ij}_{\alpha\beta}(x, \vec{u}(x),\cdots, D^m \vec{u}(x))D^\beta v^j\cdot D^\alpha\varphi^i dx.\nonumber
 \end{eqnarray}
  Moreover,  if (iii) in \textsf{Hypothesis} $\mathfrak{F}_{p,N}$ is also satisfied, then
 corresponding conclusions to C) and D) in Theorem~\ref{th:3.1} still remain true
   if the letters $u,v,\mathcal{F}$ therein are replaced by $\vec{u}, \vec{v}, \mathfrak{F}$, respectively,
   and (\ref{e:3.4})--(\ref{e:3.6}) are  changed into
   \begin{eqnarray*}
  (D(\nabla\mathfrak{F})(\vec{u})\vec{v},\vec{\varphi})_{m,2}&=&\sum^N_{i,j=1}\sum_{|\alpha|,|\beta|\le m}\int_\Omega
  F^{ij}_{\alpha\beta}(x, \vec{u}(x),\cdots, D^m \vec{u}(x))D^\beta v^j\cdot D^\alpha\varphi^i dx,\\
   (P(\vec{u})\vec{v}, \vec{\varphi})_{m,2}&=&\sum^N_{i,j=1}\sum_{|\alpha|=|\beta|=m}\int_\Omega
  F^{ij}_{\alpha\beta}(x, \vec{u}(x),\cdots, D^m \vec{u}(x))D^\beta v^j\cdot D^\alpha\varphi^i dx\\
  &&+ \sum^N_{i=1}\sum_{|\alpha|\le m-1}\int_\Omega  D^\alpha v^i\cdot D^\alpha\varphi^i dx,\\
   (Q(\vec{u})\vec{v},\vec{\varphi})_{m,2}&=&\sum^N_{i,j=1}\sum_{|\alpha|+|\beta|<2m}\int_\Omega
  F^{ij}_{\alpha\beta}(x, \vec{u}(x),\cdots, D^m \vec{u}(x))D^\beta v^j\cdot D^\alpha\varphi^i dx\\
  &&-\sum^N_{i=1}\sum_{|\alpha|\le m-1}\int_\Omega  D^\alpha v^i\cdot D^\alpha\varphi^i dx.
  \end{eqnarray*}
   \end{theorem}

Theorem~\ref{th:3.2} can be proved as that of Theorem~\ref{th:3.1}, only more terms
are added or estimated in each step. For the sake of simplicity,
we only  prove Theorem~\ref{th:3.1}. To this goal the following preliminary results are needed.

\begin{proposition}\label{prop:3.4}
For the function $\mathfrak{g}_1$ in \textsf{Hypothesis},
let continuous positive nondecreasing functions $\mathfrak{g}_k:[0,\infty)\to\mathbb{R}$, $k=3,4,5$, be given by
\begin{eqnarray*}
&&\mathfrak{g}_3(t):=1+ \mathfrak{g}_1(t)[t^2 M(m)+ t(M(m)+1)^2]+\mathfrak{g}_1(t)t(M(m)+1)+\mathfrak{g}_1(t)(M(m)+1)^2,\\
&&\mathfrak{g}_4(t):=\mathfrak{g}_1(t)t+\mathfrak{g}_1(t)\qquad\hbox{and}\qquad
\mathfrak{g}_5(t):=(M(m)+1)\mathfrak{g}_1(t)(t+1).
\end{eqnarray*}
Then (ii) in \textsf{Hypothesis} $\mathfrak{f}_p$ implies that for all $(x,\xi)$,
\begin{eqnarray}\label{e:3.7}
|f(x,\xi)|&\le& |f(x,0)|+
|\xi_\circ|\sum_{|\alpha|<m-n/p}|f_{\alpha}(x, 0)|+\sum_{m-n/p\le|\alpha|\le m}|f_{\alpha}(x, 0)|^{q_\alpha}
\nonumber\\
&&+ \mathfrak{g}_3(|\xi_\circ|)\Bigg(1+\sum_{m-n/p\le|\alpha|\le m}|\xi_\alpha|^{p_\alpha}\Bigg),
\end{eqnarray}
\begin{eqnarray}\label{e:3.8}
|f_\alpha(x,\xi)|&\le&|f_\alpha(x,0)|
+ \mathfrak{g}_4(|\xi_\circ|)\sum_{|\beta|<m-n/p}\Bigg(1+
\sum_{m-n/p\le |\gamma|\le
m}|\xi_\gamma|^{p_\gamma }\Bigg)^{p_{\alpha\beta}}\nonumber\\
&+&\mathfrak{g}_4(|\xi_\circ|)\sum_{m-n/p\le |\beta|\le m} \Bigg(1+
\sum_{m-n/p\le |\gamma|\le
m}|\xi_\gamma|^{p_\gamma }\Bigg)^{p_{\alpha\beta}}|\xi_\beta|;
\end{eqnarray}
for the latter we further have
\begin{eqnarray}\label{e:3.9}
|f_\alpha(x,\xi)|&\le&|f_\alpha(x,0)|
+ \mathfrak{g}_5(|\xi_\circ|)\Bigg(1+
\sum_{m-n/p\le |\gamma|\le m}|\xi_\gamma|^{p_\gamma }\Bigg),
\end{eqnarray}
if $|\alpha|<m-n/p$,  and
\begin{eqnarray}\label{e:3.10}
|f_\alpha(x,\xi)|&\le&|f_\alpha(x,0)|+\mathfrak{g}_5(|\xi_\circ|)
+ \mathfrak{g}_5(|\xi_\circ|)\Bigg(\sum_{m-n/p\le |\gamma|\le
m}|\xi_\gamma|^{p_\gamma }\Bigg)^{1/q_{\alpha}}
\end{eqnarray}
if $m-n/p\le|\alpha|\le m$.
\end{proposition}

Its proof will be given in Appendix A.  The following standard
result concerning the continuity of the Nemytski operator (cf. \cite[Lemma 3.2]{Bro-} and \cite[Proposition 1.1, page 3]{Skr3})
will be used many times.

\begin{proposition}\label{prop:3.5}
Let $G$ be a measurable set of positive measure in $\R^N$ and let
$f:G\times\R^N\to\R$ satisfy the following conditions:
\begin{description}
\item[(a)] $f(x,\xi_1,\cdots, \xi_N)$ is continuous in $(\xi_1,\cdots, \xi_N)$
for almost all $x\in G$;

\item[(b)] $f(x,\xi_1,\cdots, \xi_N)$ is measurable in $x$ for any fixed $(\xi_1,\cdots,
\xi_N)\in\R^N$;

\item[(c)] there exist positive numbers $C$, $1<p, p_1,\cdots,
p_N<\infty$ and a function $g\in L^p(G)$ such that
$$
|f(x,\xi_1,\cdots,\xi_N)|\le C\sum^N_{i=1}|\xi_i|^{\frac{p_i}{p}}+ g(x),\quad
\forall (x,\xi)\in\overline{\Omega}\times\R^N.
$$
\end{description}
Then the Nemytskii operator $F: \prod^N_{i=1}L_{p_i}(G)\to L_p(G)$
defined by the formula
$$
F(u_1,\cdots,u_N)(x)=f(x,u_1(x),\cdots,u_N(x))
$$
is bounded (i.e. mapping bounded sets into bounded sets) and
continuous.
\end{proposition}

The following basic inequalities are standard.

\begin{lemma}\label{lem:3.6}
{\bf (i)} {\it There exists a positive constant
$C$ only depending on $p\ge 2$ such that}
$$
\int^1_0(1+|ta+(1-t)b|)^{p-2}dt\ge C(1+|a|+|b|)^{p-2},\quad\forall
a,b\in\mathbb{R}.
$$
{\bf (ii)} $(|x_1|+\cdots+|x_n|)^q\le |x_1|^q+\cdots+|x_n|^q$
for any $q\in (0,1)$ and numbers $x_j$, $j=1,\cdots$.
\end{lemma}

 We shall prove Theorem~\ref{th:3.1} in four subsections. Clearly,
 it suffices to prove the case $V=W^{m,p}(\Omega)$.

\subsection{Proof for A) of Theorem~\ref{th:3.1}}\label{sec:Funct.2}

\noindent{\bf Step 1}.\quad{\it Prove the continuity of $\mathcal{F}$.}

 By (\ref{e:3.7}) it is easy to see that
the functional $\mathcal{F}$ in (\ref{e:1.8}) is well-defined on $W^{m,p}(\Omega)$
and is bounded on any bounded subset of $W^{m,p}(\Omega)$.

Next, we prove that $\mathcal{F}$ is continuous at a fixed
 $u_0\in W^{m,p}(\Omega)$. For $\delta>0$ let $B(u_0,\delta)=
 \{u\in W^{m,p}(\Omega)\,|\, \|u-u_0\|_{m,p}\le \delta\}$.
 By the Sobolev embedding theorem
 there exist $R=R(u_0)>0$ such that
$\sup\{|D^\alpha u(x)|:\,|\alpha|<m-n/p,\;x\in\Omega\}\le R$ for all $u\in B(u_0, 1)$.
Take a continuous function $\chi:\mathbb{R}\to\mathbb{R}$ such that
$\chi(t)=t\;\forall |t|\le 2R$, $\chi(t)=\pm 3R\;\forall \pm t\ge 3R$, and $|\chi(t)|\le 3R$.
Define a function $\tilde{f}:\overline\Omega\times\R^{M(m)}\to \R,\; (x,
\xi)\mapsto \tilde{f}(x,\xi)=f(x, \tilde\xi)$,
where $\tilde\xi_\alpha=\chi(\xi_\alpha)$ if $|\alpha|<m-n/p$, and
$\tilde\xi_\alpha=\xi_\alpha$ if $m-n/p\le |\alpha|\le m$.
Then $\tilde{0}=0$, $|\tilde{\xi}_\circ|\le 3M(m)R$. It follows that
$\tilde{f}$ also satisfies the Caratheodory condition, and therefore
(\ref{e:3.7}) leads to
\begin{eqnarray}\label{e:3.11}
|\tilde{f}(x,\xi)|&\le& |f(x,0)|+
3M(m)R\sum_{|\alpha|<m-n/p}|f_{\alpha}(x, 0)|+\sum_{m-n/p\le|\alpha|\le m}|f_{\alpha}(x, 0)|^{q_\alpha}
\nonumber\\
&&+ \mathfrak{g}_3(3M(m)R)\Bigg(1+\sum_{m-n/p\le|\alpha|\le m}|\xi_\alpha|^{p_\alpha}\Bigg).
\end{eqnarray}
Let the functional $\tilde{\mathcal{F}}:W^{m,p}(\Omega)\to \R$ be defined by
$$
\tilde{\mathcal{F}}(u)=\int_\Omega \tilde{f}(x, u(x),\cdots, D^m u(x))dx.
$$
Clearly, it is equal to $\mathcal{F}$ on ball $B(u_0, 1)$.
Hence we only need to prove that $\tilde{\mathcal{F}}$ is continuous at $u_0$.
This can follow from (\ref{e:3.11}) and Proposition~\ref{prop:3.5}.

\noindent{\bf Step 2}.\quad{\it Prove the $C^1$-smoothness of $\mathcal{F}$.}

Fix $u,\varphi\in W^{m,p}(\Omega)$. Let
$C_{u,\varphi}=\sup_{k<m-\frac{n}{p}}[\|u\|_{C^k}+\|\varphi\|_{C^k}]$.
For any $t\in [-1,1]\setminus\{0\}$
and a.a. $x\in\Omega$, using the intermediate value theorem and (\ref{e:3.9})-(\ref{e:3.10}) we deduce
{\small
\begin{eqnarray*}
&&\biggl|\frac{1}{t}[f(x, u(x)+ t\varphi(x),\cdots, D^m u(x)+ tD^m\varphi(x))-f(x, u(x),\cdots, D^m u(x))]
\biggr|\\
&\le&\sup_{0\le\vartheta\le 1}\sum_{|\alpha|<m-n/p}
|f_\alpha(x, u(x)+ \vartheta t\varphi(x),\cdots, D^m u(x)+ \vartheta tD^m\varphi(x))|\cdot
|D^\alpha\varphi(x)|\\
&&+\sup_{0\le\vartheta\le 1}\sum_{m-n/p\le |\alpha|\le m}
|f_\alpha(x, u(x)+ \vartheta t\varphi(x),\cdots, D^m u(x)+ \vartheta tD^m\varphi(x))|\cdot
|D^\alpha\varphi(x)|\\
&\le&C_{u,\varphi}\sup_{0\le\vartheta\le 1}\sum_{|\alpha|<m-n/p}
\Bigg[|f_\alpha(x,0)|+\mathfrak{g}_5(C_{u,\varphi})\bigg(1+
\sum_{m-n/p\le |\gamma|\le
m}|D^\gamma u(x)+ \vartheta tD^\gamma\varphi(x)|^{p_\gamma }\bigg)\Bigg]\\
&&+\sup_{0\le\vartheta\le 1}\sum_{m-n/p\le |\alpha|\le m}
\biggl[|f_\alpha(x,0)|+ \mathfrak{g}_5(C_{u,\varphi})\\
&&\hspace{30mm}+ \mathfrak{g}_5(C_{u,\varphi})\bigg(
\sum_{m-n/p\le |\gamma|\le m}|D^\gamma u(x)+ \vartheta tD^\gamma\varphi(x)|^{p_\gamma }\bigg)^{1/q_{\alpha}}\biggr]
\cdot|D^\alpha\varphi(x)|\\
&\le&C_{u,\varphi}\sum_{|\alpha|<m-n/p}
\biggl[|f_\alpha(x,0)|+\mathfrak{g}_5(C_{u,\varphi})\\
&&\hspace{40mm}+ \mathfrak{g}_5(C_{u,\varphi})\bigg(\sum_{m-n/p\le |\gamma|\le
m}2^{p_\gamma }\big[|D^\gamma u(x)|^{p_\gamma }+ |D^\gamma\varphi(x)|^{p_\gamma}\big]\bigg)\biggr]\\
&&+\sum_{m-n/p\le |\alpha|\le m}
\big[|f_\alpha(x,0)|+ \mathfrak{g}_5(C_{u,\varphi})\big]\cdot |D^\alpha\varphi(x)|\\
&&+\mathfrak{g}_5(C_{u,\varphi})\sum_{m-n/p\le |\alpha|\le m}
\bigg(\sum_{m-n/p\le |\gamma|\le m}2^{p_\gamma}\big[|D^\gamma u(x)|^{p_\gamma}+ |D^\gamma\varphi(x)|^{p_\gamma}\big]\bigg)^{1/q_{\alpha}}
\cdot |D^\alpha\varphi(x)|.
\end{eqnarray*}}
It follows from the assumptions on $p_{\alpha\beta}$ that the right side
is integrable and thus from the Lebesgue dominated convergence theorem
that the functional
$\mathcal{F}$  is G\^ateaux differentiable.
Moreover, the G\^ateaux differential of $\mathcal{F}$ at $u$, $D\mathcal{F}(u)\in[W^{m,p}(\Omega)]^\ast$,
is given by
$$
D\mathcal{F}(u)\varphi=\langle D\mathcal{F}(u),\varphi\rangle=
\sum_{|\alpha|\le m}\int_\Omega
f_\alpha(x, u(x),\cdots, D^m u(x))
D^\alpha\varphi(x)dx.
$$
Let $D_\alpha\mathcal{F}(u)\in[W^{m,p}(\Omega)]^\ast$ be defined by
\begin{equation}\label{e:3.12}
\langle D_\alpha\mathcal{F}(u),\varphi\rangle=
\int_\Omega
f_\alpha(x, u(x),\cdots, D^m u(x))
D^\alpha\varphi(x)dx.
\end{equation}

\noindent{\it Claim 1}.  The map $D_\alpha\mathcal{F}:W^{m,p}(\Omega)\to [W^{m,p}(\Omega)]^\ast$
is continuous.

$\bullet$ Case $|\alpha|<m-n/p$.

 Then $\|D^\alpha\varphi\|_{C^0}\le C\|\varphi\|_{m,p},\;\forall
 \varphi\in W^{m,p}(\Omega)$, where $C>0$ is a constant coming from the Sobolev embedding theorem.
 Fix $u\in W^{m,p}(\Omega)$. For any $v\in B(u,1)$,
 we have
\begin{eqnarray}\label{e:3.13}
&&\|D_\alpha\mathcal{F}(v)-D_\alpha\mathcal{F}(u)\|=\sup_{\|\varphi\|_{m,p}=1}|
\langle D_\alpha\mathcal{F}(v)-D_\alpha\mathcal{F}(u),\varphi\rangle|\nonumber\\
&\le&\sup_{\|\varphi\|_{m,p}=1}\int_\Omega
|f_\alpha(x, v(x),\cdots, D^m v(x))-f_\alpha(x, u(x),\cdots, D^m u(x))|\cdot
|D^\alpha\varphi(x)|dx\nonumber\\
&\le&\sup_{\|\varphi\|_{m,p}=1}\|D^\alpha\varphi\|_{C^0}\int_\Omega
|f_\alpha(x, v(x),\cdots, D^m v(x))-f_\alpha(x, u(x),\cdots, D^m u(x))|dx\nonumber\\
&\le& C\int_\Omega
|f_\alpha(x, v(x),\cdots, D^m v(x))-f_\alpha(x, u(x),\cdots, D^m u(x))|dx.
\end{eqnarray}
 Because of (\ref{e:3.9}),
by a standard method as in Step~1 we may also assume that for some constant $C'>0$,
\begin{eqnarray*}
|f_\alpha(x,\xi)|&\le&|f_\alpha(x,0)|
+ C'\bigg(1+
\sum_{m-n/p\le |\gamma|\le
m}|\xi_\gamma|^{p_\gamma }\bigg),\quad\forall (x,\xi).
\end{eqnarray*}
 Note that
$f_\alpha(\cdot,0)\in L^1(\Omega)$.
 Using Proposition~\ref{prop:3.5} we deduce that
\begin{eqnarray*}
\prod_{|\beta|<n/p}L^1(\Omega)\times\prod_{m-\frac{n}{p}\le|\beta|\le
m}L^{p_\beta}(\Omega)\to L^{1}(\Omega),\;{\bf u}=\{u_\beta:\,|\beta|\le m\}\to f_\alpha(\cdot,{\bf u})
\end{eqnarray*}
is continuous, which implies the continuity of the map
$W^{m,p}(\Omega)\ni u\mapsto f_\alpha(\cdot,u,\cdots,D^mu)\in L^{1}(\Omega)$.
The latter claim and (\ref{e:3.13}) yields that
$\|D_\alpha\mathcal{F}(v)-D_\alpha\mathcal{F}(u)\|\to 0$ as $\|v-u\|_{m,p}\to 0$.
That is, $D_\alpha\mathcal{F}$ is continuous in this case.

$\bullet$ Case $m-n/p\le |\alpha|\le m$. Then we have
\begin{eqnarray*}
&&\|D_\alpha\mathcal{F}(u)-D_\alpha\mathcal{F}(v)\|=\sup_{\|\varphi\|_{m,p}=1}|
\langle D_\alpha\mathcal{F}(u)-D_\alpha\mathcal{F}(v),\varphi\rangle|\\
&\le&\sup_{\|\varphi\|_{m,p}=1}\int_\Omega
|f_\alpha(x, u(x),\cdots, D^m u(x))-f_\alpha(x, v(x),\cdots, D^m v(x))|\cdot
|D^\alpha\varphi(x)|dx\\
&\le&\sup_{\|\varphi\|_{m,p}=1}\|D^\alpha\varphi\|_{p_\alpha}\bigg(\int_\Omega
|f_\alpha(x, u(x),\cdots, D^m u(x))-f_\alpha(x, v(x),\cdots, D^m v(x))|^{q_\alpha}dx
\bigg)^{1/q_\alpha}\\
&\le&C\bigg(\int_\Omega
|f_\alpha(x, u(x),\cdots, D^m u(x))-f_\alpha(x, v(x),\cdots, D^m v(x))|^{q_\alpha}dx
\bigg)^{1/q_\alpha}.
\end{eqnarray*}
Since $v\in B(u,1)$, as in Step~1, by (\ref{e:3.10}) we may also assume that for some constant $C>0$,
\begin{eqnarray*}
|f_\alpha(x,\xi)|&\le&|f_\alpha(x,0)|
+ C\bigg(1+
\sum_{m-n/p\le |\gamma|\le
m}|\xi_\gamma|^{p_\gamma/q_{\alpha}}\bigg),\quad\forall (x,\xi).
\end{eqnarray*}
 Note that
$f_\alpha(\cdot,0)\in L^{q_\alpha}(\Omega)$ with $q_\alpha=\frac{p_\alpha}{p_\alpha-1}$.
 We derive from Proposition~\ref{prop:3.5} that
\begin{eqnarray*}
\prod_{|\beta|<n/p}L^1(\Omega)\times\prod_{m-\frac{n}{p}\le|\beta|\le
m}L^{p_\beta}(\Omega)\to L^{q_\alpha}(\Omega),\;{\bf u}=\{u_\beta:\,|\beta|\le m\}\to f_\alpha(\cdot,{\bf u})
\end{eqnarray*}
is continuous.   This leads to the continuity of the maps
$W^{m,p}(\Omega)\ni u\mapsto f_\alpha(\cdot,u,\cdots,D^mu)\in L^{q_\alpha}(\Omega)$
and hence $D_\alpha\mathcal{F}$ as above.

To sum up,  the  map $D\mathcal{F}:W^{m,p}(\Omega)\to [W^{m,p}(\Omega)]^\ast$
is continuous. As usual this implies that
$\mathcal{F}$ has the Fr\'echet derivative $\mathcal{F}'(u)=D\mathcal{F}(u)$
at each point $u\in W^{m,p}(\Omega)$ and thus is of class $C^1$.

Moreover, from the above proof we see that for any $u\in W^{m,p}(\Omega)$,
\begin{eqnarray*}
\|D_\alpha\mathcal{F}(u)\|=\sup_{\|\varphi\|_{m,p}=1}|
\langle D_\alpha\mathcal{F}(u),\varphi\rangle|
\le C\int_\Omega
|f_\alpha(x, u(x),\cdots, D^m u(x))|dx
\end{eqnarray*}
if $|\alpha|<m-n/p$, and
\begin{eqnarray*}
\|D_\alpha\mathcal{F}(u)\|=\sup_{\|\varphi\|_{m,p}=1}|
\langle D_\alpha\mathcal{F}(u),\varphi\rangle|
\le C\bigg(\int_\Omega
|f_\alpha(x, u(x),\cdots, D^m u(x))|^{q_\alpha}dx
\bigg)^{1/q_\alpha}
\end{eqnarray*}
if $m-n/p\le |\alpha|\le m$. It follows from these and (\ref{e:3.9})-(\ref{e:3.10}) that
$\mathcal{F}'$ maps a bounded subset into a bounded set.

\subsection{Proof for B) of Theorem~\ref{th:3.1}}\label{sec:Funct.3}

 \noindent{\bf Step 1}.\quad{\it Prove that the right side of (\ref{e:3.2}) determines
an operator in $\mathscr{L}(W^{m,p}(\Omega), [W^{m,p}(\Omega)]^\ast)$.}

 By (\ref{e:1.5}) we deduce that the right side of (\ref{e:3.2})
satisfies
\begin{eqnarray*}
  &&\bigg|\sum_{|\alpha|,|\beta|\le m}\int_\Omega
  f_{\alpha\beta}(x, u(x),\cdots, D^m u(x))D^\beta v\cdot D^\alpha\varphi dx\bigg|\\
  &\le& \sup_{k<m-n/p}\mathfrak{g}_1(\|u\|_{C^k})\sum_{|\alpha|,|\beta|\le m}\int_\Omega
  \bigg(1+\sum_{m-n/p\le |\gamma|\le
m}|D^\gamma u(x)|^{p_\gamma}\bigg)^{p_{\alpha\beta}}|D^\beta v(x)|\cdot |D^\alpha\varphi(x)|dx.
    \end{eqnarray*}
It suffices to prove that there exists a constant $C=C(u,\alpha,\beta)$ such that
\begin{eqnarray}\label{e:3.14}
 \int_\Omega\bigg(1+\sum_{m-n/p\le |\gamma|\le
m}|D^\gamma u(x)|^{p_\gamma}\bigg)^{p_{\alpha\beta}}|D^\beta v(x)|\cdot |D^\alpha\varphi(x)|dx
\le C\|v\|_{m,p}\|\varphi\|_{m,p}.
    \end{eqnarray}

{\it Case $|\alpha|=|\beta|=m$}.\quad
Then $p_{\alpha\beta}= 1-\frac{1}{p_\alpha}-\frac{1}{p_\beta}$
and thus the H\"older inequality leads to
\begin{eqnarray*}
 &&\int_\Omega
  \bigg(1+\sum_{m-n/p\le |\gamma|\le
m}|D^\gamma u(x)|^{p_\gamma}\bigg)^{p_{\alpha\beta}}|D^\beta v(x)|\cdot |D^\alpha\varphi(x)|dx\\
&\le&\bigg(\int_\Omega
  \bigg(1+\sum_{m-n/p\le |\gamma|\le
m}|D^\gamma u(x)|^{p_\gamma}\bigg)\bigg)^{p_{\alpha\beta}}
\bigg(\int_\Omega|D^\beta v(x)|^{p_\beta}\bigg)^{1/p_\beta}
\bigg(\int_\Omega|D^\alpha \varphi(x)|^{p_\alpha}\bigg)^{1/p_\alpha}\\
&\le&  C\|v\|_{m,p}\|\varphi\|_{m,p}.
    \end{eqnarray*}

{\it Case $m-n/p\le |\alpha|\le
 m,\; |\beta|<m-n/p$}.\quad 
Then $p_{\alpha\beta}=  1-\frac{1}{p_\alpha}$ and $\sup_x|D^\beta v(x)|\le C(m,n,p)\|v\|_{m,p}$.
It follows from these that
\begin{eqnarray*}
 &&\int_\Omega
  \bigg(1+\sum_{m-n/p\le |\gamma|\le
m}|D^\gamma u(x)|^{p_\gamma}\bigg)^{p_{\alpha\beta}}|D^\beta v(x)|\cdot |D^\alpha\varphi(x)|dx\\
&\le&C(m,n,p)\|v\|_{m,p}\int_\Omega
  \bigg(1+\sum_{m-n/p\le |\gamma|\le
m}|D^\gamma u(x)|^{p_\gamma}\bigg)^{p_{\alpha\beta}}\cdot |D^\alpha\varphi(x)|dx\\
&\le&C(m,n,p)\|v\|_{m,p}\bigg(\int_\Omega
  \bigg(1+\sum_{m-n/p\le |\gamma|\le
m}|D^\gamma u(x)|^{p_\gamma}\bigg)\bigg)^{p_{\alpha\beta}}
\bigg(\int_\Omega|D^\alpha \varphi(x)|^{p_\alpha}\bigg)^{1/p_\alpha}\\
&\le&  C\|v\|_{m,p}\|\varphi\|_{m,p}.
    \end{eqnarray*}

{\it Case $m-n/p\le |\beta|\le
 m,\; |\alpha|<m-n/p$}.\quad The proof is the same  as the last case.

{\it Case $|\alpha|, |\beta|<m-n/p$}.\quad 
We have $p_{\alpha\beta}= 1$, $\sup_x|D^\beta v(x)|\le C(m,n,p)\|v\|_{m,p}$
and $\sup_x|D^\alpha\varphi(x)|\le C(m,n,p)\|\varphi\|_{m,p}$. Hence
\begin{eqnarray*}
 &&\int_\Omega
  \left(1+\sum_{m-n/p\le |\gamma|\le
m}|D^\gamma u(x)|^{p_\gamma}\right)^{p_{\alpha\beta}}|D^\beta v(x)|\cdot |D^\alpha\varphi(x)|dx\\
&\le&C(m,n,p)\|v\|_{m,p}\|\varphi\|_{m,p}\int_\Omega
  \left(1+\sum_{m-n/p\le |\gamma|\le
m}|D^\gamma u(x)|^{p_\gamma}\right)dx\\
&\le&  C\|v\|_{m,p}\|\varphi\|_{m,p}.
    \end{eqnarray*}

{\it Case $|\alpha|,\;|\beta|\ge
 m-n/p,\;|\alpha|+|\beta|<2m$}.\quad 
Then $0<p_{\alpha\beta}<1-\frac{1}{p_\alpha}-\frac{1}{p_\beta}$.
Let $q_{\alpha\beta}>1$ be determined by
$q_{\alpha\beta}+\frac{1}{p_{\alpha\beta}}+\frac{1}{p_\alpha}+\frac{1}{p_\beta}=1$.
Using the H\"older inequality we get
{\small
\begin{eqnarray*}
 &&\int_\Omega
  \bigg(1+\sum_{m-n/p\le |\gamma|\le
m}|D^\gamma u(x)|^{p_\gamma}\bigg)^{p_{\alpha\beta}}|D^\beta v(x)|\cdot |D^\alpha\varphi(x)|dx\\
&\le&|\Omega|^{1/q_{\alpha\beta}}\bigg(\int_\Omega
  \bigg(1+\sum_{m-n/p\le |\gamma|\le
m}|D^\gamma u(x)|^{p_\gamma}\bigg)\bigg)^{p_{\alpha\beta}}
\bigg(\int_\Omega|D^\beta v(x)|^{p_\beta}\bigg)^{1/p_\beta}
\bigg(\int_\Omega|D^\alpha \varphi(x)|^{p_\alpha}\bigg)^{1/p_\alpha}\\
&\le&  C\|v\|_{m,p}\|\varphi\|_{m,p}.
    \end{eqnarray*}}

In summary, (\ref{e:3.14}) is proved.
Hence the right side of (\ref{e:3.2}) determines
an operator $A(u)\in \mathscr{L}(W^{m,p}(\Omega), [W^{m,p}(\Omega)]^\ast)$  by
  \begin{equation}\label{e:3.15}
  \langle A(u)v,\varphi\rangle=\sum_{|\alpha|,|\beta|\le m}\int_\Omega
  f_{\alpha\beta}(x, u(x),\cdots, D^m u(x))D^\beta v\cdot D^\alpha\varphi dx.
  \end{equation}
In particular, each term $\sum_{|\beta|\le m}\int_\Omega
  f_{\alpha\beta}(x, u(x),\cdots, D^m u(x))D^\beta v\cdot D^\alpha\varphi dx$
determines
an operator $A_\alpha(u)\in \mathscr{L}(W^{m,p}(\Omega), [W^{m,p}(\Omega)]^\ast)$  by
  \begin{equation}\label{e:3.16}
  \langle A_\alpha(u)v,\varphi\rangle=\sum_{|\beta|\le m}\int_\Omega
  f_{\alpha\beta}(x, u(x),\cdots, D^m u(x))D^\beta v\cdot D^\alpha\varphi dx.
  \end{equation}

\noindent{\bf Step 2}.\quad{\it Prove that the map $D_\alpha\mathcal{F}:W^{m,p}(\Omega)\to [W^{m,p}(\Omega)]^\ast$
defined by (\ref{e:3.12}) is  of class $C^1$ for $p>2$ or $p=2$ and $|\alpha|<m$, but only
G\^ateaux differentiable for $p=2$ and $|\alpha|=m$.}

For any $t\in [-1,1]\setminus\{0\}$ and $v\in W^{m,p}(\Omega)$ we derive from (\ref{e:3.12})
and (\ref{e:3.16}) that
\begin{eqnarray}\label{e:3.17}
&&\langle \frac{1}{t}[D_\alpha\mathcal{F}(u+tv)-D_\alpha\mathcal{F}(u)],\varphi\rangle-\langle A_\alpha(u)v,\varphi\rangle\nonumber\\
&=&
\int_\Omega\frac{1}{t}
[f_\alpha(x, u(x)+ tv(x),\cdots, D^m u(x)+tD^m v(x))
-f_\alpha(x, u(x),\cdots, D^m u(x))]
D^\alpha\varphi(x)dx\nonumber\\
&&\hspace{10mm}-\sum_{|\beta|\le m}\int_\Omega
  f_{\alpha\beta}(x, u(x),\cdots, D^m u(x))D^\beta v\cdot D^\alpha\varphi dx\nonumber\\
  &=&\sum_{|\beta|\le m}\int_\Omega\int^1_0
  f_{\alpha\beta}(x, u(x)+ stv(x),\cdots, D^m u(x)+ stD^m v(x))D^\beta v\cdot D^\alpha\varphi dsdx\nonumber\\
  &&\hspace{10mm}-\sum_{|\beta|\le m}\int_\Omega
  f_{\alpha\beta}(x, u(x),\cdots, D^m u(x))D^\beta v\cdot D^\alpha\varphi dx\nonumber\\
   &=&\sum_{|\beta|\le m}\int_\Omega\int^1_0
  [f_{\alpha\beta}(x, u(x)+ stv(x),\cdots, D^m u(x)+ stD^m v(x))\nonumber\\
  &&\hspace{30mm}- f_{\alpha\beta}(x, u(x),\cdots, D^m u(x))]D^\beta v\cdot D^\alpha\varphi dsdx\nonumber\\
  &=&\sum_{|\beta|\le m}I_{\alpha\beta}.
\end{eqnarray}

Firstly, we consider the case $p>2$.

$\bullet$ {\it Case $|\alpha|=|\beta|=m$}.\quad 
Then $p_{\alpha\beta}= 1-\frac{1}{p_\alpha}-\frac{1}{p_\beta}\in (0, 1)$
and  we obtain
\begin{eqnarray*}
\qquad |I_{\alpha\beta}|&\le&\int^1_0ds\Bigl(\int_\Omega
  |f_{\alpha\beta}(x, u(x)+ stv(x),\cdots, D^m u(x)+ stD^m v(x))\\
  &&\hspace{6mm}- f_{\alpha\beta}(x, u(x),\cdots, D^m u(x))|^{1/p_{\alpha\beta}}\Bigl)^{p_{\alpha\beta}}
  \left(\int_\Omega|D^\beta v|^{p_\beta}\right)^{1/p_\beta}
  \left(\int_\Omega|D^\alpha\varphi|^{p_\alpha}\right)^{1/p_\alpha}\\
  &\le&C\|v\|_{m,p}\|\varphi\|_{m,p}\int^1_0ds\Bigl(\int_\Omega
  |f_{\alpha\beta}(x, u(x)+ stv(x),\cdots, D^m u(x)+ stD^m v(x))\\
  &&\hspace{60mm}- f_{\alpha\beta}(x, u(x),\cdots, D^m u(x))|^{1/p_{\alpha\beta}}\Bigl)^{p_{\alpha\beta}}.
\end{eqnarray*}

$\bullet$ {\it Case $m-n/p\le |\alpha|\le
 m,\; |\beta|<m-n/p$}.\quad 
Then $p_{\alpha\beta}=  1-\frac{1}{p_\alpha}$ and $\sup_x|D^\beta v(x)|\le C(m,n,p)\|v\|_{m,p}$.
 We can also deduce
\begin{eqnarray*}
  \qquad |I_{\alpha\beta}|&\le&C\|v\|_{m,p}\int^1_0ds\Bigl(\int_\Omega
  |f_{\alpha\beta}(x, u(x)+ stv(x),\cdots, D^m u(x)+ stD^m v(x))\\
  &&\hspace{30mm}- f_{\alpha\beta}(x, u(x),\cdots, D^m u(x))|^{1/p_{\alpha\beta}}\Bigl)^{p_{\alpha\beta}}
    \bigg(\int_\Omega|D^\alpha\varphi|^{p_\alpha}\bigg)^{1/p_\alpha}\\
  &\le&C\|v\|_{m,p}\|\varphi\|_{m,p}\int^1_0ds\Bigl(\int_\Omega
  |f_{\alpha\beta}(x, u(x)+ stv(x),\cdots, D^m u(x)+ stD^m v(x))\\
  &&\hspace{60mm}- f_{\alpha\beta}(x, u(x),\cdots, D^m u(x))|^{1/p_{\alpha\beta}}\Bigl)^{p_{\alpha\beta}}.
\end{eqnarray*}

$\bullet$ {\it Case $m-n/p\le |\beta|\le
 m,\; |\alpha|<m-n/p$}.\quad 
Then $p_{\beta\alpha}=  1-\frac{1}{p_\beta}$ and $\sup_x|D^\alpha\varphi(x)|\le C(m,n,p)\|\varphi\|_{m,p}$.
 We can also deduce
\begin{eqnarray*}
\qquad |I_{\alpha\beta}|&\le&C\|\varphi\|_{m,p}\int^1_0ds\Bigl(\int_\Omega
  |f_{\alpha\beta}(x, u(x)+ stv(x),\cdots, D^m u(x)+ stD^m v(x))\\
  &&\hspace{20mm}- f_{\alpha\beta}(x, u(x),\cdots, D^m u(x))|^{1/p_{\beta\alpha}}\Bigl)^{p_{\beta\alpha}}
    \bigg(\int_\Omega|D^\beta v|^{p_\beta}\bigg)^{1/p_\alpha}\\
  &\le&C\|v\|_{m,p}\|\varphi\|_{m,p}\int^1_0ds\Bigl(\int_\Omega
  |f_{\alpha\beta}(x, u(x)+ stv(x),\cdots, D^m u(x)+ stD^m v(x))\\
  &&\hspace{50mm}- f_{\alpha\beta}(x, u(x),\cdots, D^m u(x))|^{1/p_{\beta\alpha}}\Bigl)^{p_{\beta\alpha}}.
\end{eqnarray*}

$\bullet$ {\it Case $|\alpha|,\;|\beta|\ge
 m-n/p,\;|\alpha|+|\beta|<2m$}.\quad 
Then $0<p_{\alpha\beta}<1-\frac{1}{p_\alpha}-\frac{1}{p_\beta}$.
Let $q_{\alpha\beta}>1$ be determined by
$q_{\alpha\beta}+\frac{1}{p_{\alpha\beta}}+\frac{1}{p_\alpha}+\frac{1}{p_\beta}=1$. Then
\begin{eqnarray*}
&&\int^1_0ds\int_\Omega
  |f_{\alpha\beta}(x, u(x)+ stv(x),\cdots, D^m u(x)+ stD^m v(x))\\
  &&\hspace{30mm}- f_{\alpha\beta}(x, u(x),\cdots, D^m u(x))|\cdot|D^\beta v|\cdot|D^\alpha\varphi| dsdx\\
  &\le&|\Omega|^{1/q_{\alpha\beta}}\int^1_0ds\Bigl(\int_\Omega
  |f_{\alpha\beta}(x, u(x)+ stv(x),\cdots, D^m u(x)+ stD^m v(x))\\
  &&\hspace{20mm}- f_{\alpha\beta}(x, u(x),\cdots, D^m u(x))|^{1/p_{\alpha\beta}}\Bigl)^{p_{\alpha\beta}}
  \bigg(\int_\Omega|D^\beta v|^{p_\beta}\bigg)^{1/p_\beta}
  \bigg(\int_\Omega|D^\alpha\varphi|^{p_\alpha}\bigg)^{1/p_\alpha}\\
  &\le&C\|v\|_{m,p}\|\varphi\|_{m,p}\int^1_0ds\Bigl(\int_\Omega
  |f_{\alpha\beta}(x, u(x)+ stv(x),\cdots, D^m u(x)+ stD^m v(x))\\
  &&\hspace{50mm}- f_{\alpha\beta}(x, u(x),\cdots, D^m u(x))|^{1/p_{\alpha\beta}}\Bigl)^{p_{\alpha\beta}}.
\end{eqnarray*}

$\bullet$ {\it Case $|\alpha|<m-n/p,\; |\beta|<m-n/p$}.\quad 
Then $p_{\beta\alpha}=  1$, $\sup_x|D^\beta v(x)|\le C(m,n,p)\|\varphi\|_{m,p}$  and $\sup_x|D^\alpha\varphi(x)|\le C(m,n,p)\|\varphi\|_{m,p}$.
 We can also deduce
\begin{eqnarray*}
    |I_{\alpha\beta}|&\le&C\|v\|_{m,p}\|\varphi\|_{m,p}\int^1_0ds\Bigl(\int_\Omega
  |f_{\alpha\beta}(x, u(x)+ stv(x),\cdots, D^m u(x)+ stD^m v(x))\\
  &&\hspace{60mm}- f_{\alpha\beta}(x, u(x),\cdots, D^m u(x))|\Bigl).
\end{eqnarray*}

Summarizing the above five cases, by (\ref{e:3.17}) we obtain that for any $t\in [-1,1]\setminus\{0\}$,
\begin{eqnarray}\label{e:3.18}
&&\|[D_\alpha\mathcal{F}(u+tv)-D_\alpha\mathcal{F}(u)]/t-
A_\alpha(u)v\|\nonumber\\
   &\le&C\|v\|_{m,p}\sum_{|\beta|\le m}\int^1_0ds\Bigl(\int_\Omega
  |f_{\alpha\beta}(x, u(x)+ stv(x),\cdots, D^m u(x)+ stD^m v(x))\nonumber\\
  &&\hspace{50mm}- f_{\alpha\beta}(x, u(x),\cdots, D^m u(x))|^{1/p_{\alpha\beta}}\Bigl)^{p_{\alpha\beta}}.
 \end{eqnarray}

 Fix $u,v\in W^{m,p}(\Omega)$. Because of (\ref{e:1.5}), after treating as in Step~1 of Section~\ref{sec:Funct.2}
 we assume that for some constant $C=C_{u,v}>0$ and all $(x,\xi)$,
\begin{eqnarray}\label{e:3.19}
 |f_{\alpha\beta}(x,\xi)|\le
C\bigg(1+
\sum_{m-n/p\le |\gamma|\le
m}|\xi_\gamma|^{p_\gamma}\bigg)^{p_{\alpha\beta}}
\le C\bigg(1+\sum_{m-n/p\le |\gamma|\le m}|\xi_\gamma|^{p_\gamma p_{\alpha\beta}}\bigg).
\nonumber\\
\end{eqnarray}
Then we derive from  Proposition~\ref{prop:3.5}  that
\begin{eqnarray*}
\prod_{|\beta|<n/p}L^1(\Omega)\times\prod_{m-\frac{n}{p}\le|\gamma|\le
m}L^{p_\gamma}(\Omega)\to L^{1/p_{\alpha\beta}}(\Omega),\;{\bf u}=\{u_\gamma:\,|\gamma|\le m\}\to f_{\alpha\beta}(\cdot,{\bf u})
\end{eqnarray*}
is continuous. This implies the continuity of the map
$$
W^{m,p}(\Omega)\ni w\mapsto f_{\alpha\beta}(\cdot, w,\cdots, D^mw)\in L^{1/p_{\alpha\beta}}(\Omega).
$$
It follows from (\ref{e:3.18}) that
\begin{eqnarray}\label{e:3.20}
\lim_{t\to 0}\|[D_\alpha\mathcal{F}(u+tv)-D_\alpha\mathcal{F}(u)]/t-
A_\alpha(u)v\|=0,
 \end{eqnarray}
 Namely, $D_\alpha\mathcal{F}$ has G\^ateaux derivative $A_\alpha(u)$
 at $u$.

 If $v$ is allowed to varies in the ball $B(u,1)\subset W^{m,p}(\Omega)$,
 then we may assume that (\ref{e:3.19}) also holds for another constant $C=C_{u,1}>0$,
 and thus that $C$ in (\ref{e:3.18}) may be changed into $C_{u,1}$.
Taking $t=1$ and letting $\|v\|_{m,p}\to 0$ in (\ref{e:3.18}) we get that
 \begin{eqnarray}\label{e:3.21}
\|D_\alpha\mathcal{F}(u+v)-D_\alpha\mathcal{F}(u)-
A_\alpha(u)v\|=o(\|v\|_{m,p}).
 \end{eqnarray}
That is,   $D_\alpha\mathcal{F}$ has Fr\'echet derivative $A_\alpha(u)$
 at $u$.

Moreover, using a similar method to the above one we can prove:
\textsf{for any $u,v\in W^{m,p}(\Omega)$,  we have $C=C(m,n,p,\Omega)>0$ such that}
\begin{eqnarray}\label{e:3.22}
&&\|A_\alpha(u)-A_\alpha(v)\|\nonumber\\
   &\le&C\sum_{|\beta|\le m}\Bigl(\int_\Omega
  |f_{\alpha\beta}(x, u(x),\cdots, D^m u(x))- f_{\alpha\beta}(x, v(x),\cdots, D^m v(x))|^{1/p_{\alpha\beta}}\Bigl)^{p_{\alpha\beta}}.\nonumber\\
 \end{eqnarray}
This shows that $A_\alpha$ is continuous. Hence for $p>2$ we have proved:
\begin{center}
\textsf{ $A$ is continuous, and $\mathcal{F}$ is of class $C^2$ on $W^{m,p}(\Omega)$.}
\end{center}

Next, we consider the case  $p=2$.

 If $|\alpha|+|\beta|<2m$ the above arguments  also work, and so $A_\alpha$ is of class $C^1$.

 If $|\alpha|=|\beta|=m$, then  $p_{\alpha\beta}= 1-\frac{1}{p_\alpha}-\frac{1}{p_\beta}=0$
(since $p_\alpha=p_\beta=2$).  We can only obtain
\begin{eqnarray*}
&&\int^1_0ds\int_\Omega
  |f_{\alpha\beta}(x, u(x)+ stv(x),\cdots, D^m u(x)+ stD^m v(x))\\
  &&\hspace{30mm}- f_{\alpha\beta}(x, u(x),\cdots, D^m u(x))|\cdot|D^\beta v|\cdot|D^\alpha\varphi| dsdx\\
  &\le&\int^1_0ds\Bigl(\int_\Omega
  |f_{\alpha\beta}(x, u(x)+ stv(x),\cdots, D^m u(x)+ stD^m v(x))\\
  &&\hspace{20mm}- f_{\alpha\beta}(x, u(x),\cdots, D^m u(x))|^2
  |D^\beta v|^2\Bigl)^{1/2}
    \bigg(\int_\Omega|D^\alpha\varphi|^{2}\bigg)^{1/2}\\
  &\le&C\|\varphi\|_{m,2}\int^1_0ds\Bigl(\int_\Omega
  |f_{\alpha\beta}(x, u(x)+ stv(x),\cdots, D^m u(x)+ stD^m v(x))\\
  &&\hspace{20mm}- f_{\alpha\beta}(x, u(x),\cdots, D^m u(x))|^2
  |D^\beta v|^2\Bigl)^{1/2}.
\end{eqnarray*}
And therefore
\begin{eqnarray*}
&&\|[D_\alpha\mathcal{F}(u+tv)-D_\alpha\mathcal{F}(u)]/t-
A_\alpha(u)v\|\nonumber\\
&\le&C\int^1_0ds\Bigl(\int_\Omega
  |f_{\alpha\beta}(x, u(x)+ stv(x),\cdots, D^m u(x)+ stD^m v(x))\nonumber\\
  &&\hspace{20mm}- f_{\alpha\beta}(x, u(x),\cdots, D^m u(x))|^2
  |D^\beta v|^2\Bigl)^{1/2}.
 \end{eqnarray*}
Note that for fixed $u,v\in W^{m,p}(\Omega)$ and all $s,t\in [0,1]$ the functions
$$
|f_{\alpha\beta}(x, u(x)+ stv(x),\cdots, D^m u(x)+ stD^m v(x))
  - f_{\alpha\beta}(x, u(x),\cdots, D^m u(x))|^2
  $$
  are uniformly bounded. It follows
\begin{eqnarray*}
\lim_{t\to 0}\|[D_\alpha\mathcal{F}(u+tv)-D_\alpha\mathcal{F}(u)]/t-
A_\alpha(u)v\|=0.
 \end{eqnarray*}
Hence $D_\alpha\mathcal{F}$ has G\^ateaux derivative $A_\alpha(u)$
 at $u$.

\noindent{\bf Step 3}.\quad{\it Prove that $D\mathcal{F}'$ is bounded on any ball in $W^{m,p}(\Omega)$.}\quad
 From the above arguments we obtain $C=C(m,n,p,\Omega)$ such that
 \begin{eqnarray}\label{e:3.23}
 |\langle A_\alpha(u)v,\varphi\rangle|&=&\left|\int_\Omega
  f_{\alpha\beta}(x, u(x),\cdots, D^m u(x))D^\beta v D^\alpha\varphi(x) dx\right|\nonumber\\
 &\le&C\|v\|_{m,p}\|\cdot\varphi\|_{m,p}\Bigl(\int_\Omega
  |f_{\alpha\beta}(x, u(x),\cdots, D^m u(x))|^{1/p_{\alpha\beta}}\Bigl)^{p_{\alpha\beta}}
 \end{eqnarray}
 for $|\alpha|+|\beta|<2m$ or $p>2$ and $|\alpha|=|\beta|=m$, and
 \begin{eqnarray}\label{e:3.24}
|\langle A_\alpha(u)v,\varphi\rangle|&=&\left|\int_\Omega
  f_{\alpha\beta}(x, u(x),\cdots, D^m u(x))D^\beta v D^\alpha\varphi(x)dx\right|\nonumber\\
 &\le&C\cdot\varphi\|_{m,p}\Bigl(\int_\Omega
  |f_{\alpha\beta}(x, u(x),\cdots, D^m u(x))|^2
  |D^\beta v|^2\Bigl)^{1/2}
 \end{eqnarray}
 for $p=2$ and $|\alpha|=|\beta|=m$.

 For (\ref{e:3.24}), by (\ref{e:1.5}) we derive
 $$
 |f_{\alpha\beta}(x,u(x),\cdots, D^mu(x))|\le\sup_{k<m-n/2}
\mathfrak{g}_1(\|u\|_{C^k})\quad\forall u\in W^{m,2},
 $$
 and hence
 \begin{eqnarray*}
|\langle A_\alpha(u)v,\varphi\rangle| &=&\left|\int_\Omega
  f_{\alpha\beta}(x, u(x),\cdots, D^m u(x))D^\beta v dx\right|\nonumber\\
 &\le&C\sup_{k<m-n/2}
\mathfrak{g}_1(\|u\|_{C^k})\Bigl(\int_\Omega
    |D^\beta v|^2\Bigl)^{1/2}
 \end{eqnarray*}
 This and (\ref{e:3.23}), (\ref{e:1.5}) yield the desired claim.

  \noindent{\bf Step 4}.\quad{\it Prove (ii) in {\bf B)}.}\quad
 By (\ref{e:3.2}) we derive
 \begin{eqnarray*}
&& \langle \varphi, [D\mathcal{F}'(u_k)]^\ast v-[D\mathcal{F}'(u)]^\ast v\rangle
= \langle D\mathcal{F}'(u_k)\varphi-D\mathcal{F}'(u)\varphi, v\rangle\nonumber\\
 &=&\sum_{|\alpha|,|\beta|\le m}\int_\Omega
  [f_{\alpha\beta}(x, u_k(x),\cdots, D^m u_k(x))\nonumber\\
  &&\hspace{40mm}-f_{\alpha\beta}(x, u(x),\cdots, D^m u(x))]D^\beta \varphi\cdot D^\alpha v dx.
  \end{eqnarray*}

 Let $u_k\to u$ in $W^{m,p}(\Omega)$.
 If $p>2$, with the same reasoning as in (\ref{e:3.23}) we can derive from this that
  \begin{eqnarray*}
&& |\langle \varphi, [D\mathcal{F}'(u_k)]^\ast v-[D\mathcal{F}'(u)]^\ast v\rangle|\nonumber\\
  &\le&\sum_{|\alpha|,|\beta|\le m}\int_\Omega
  |f_{\alpha\beta}(x, u_k(x),\cdots, D^m u_k(x))\nonumber\\
  &&\hspace{30mm}-f_{\alpha\beta}(x, u(x),\cdots, D^m u(x))|\cdot|D^\beta \varphi|
  \cdot |D^\alpha v| dx\nonumber\\
  &\le&C\|v\|_{m,p}\|\varphi\|_{m,p}\sum_{|\alpha|,|\beta|\le m}\Bigl(\int_\Omega
  |f_{\alpha\beta}(x, u_k(x),\cdots, D^m u_k(x))\nonumber\\
  &&\hspace{50mm}-f_{\alpha\beta}(x, u(x),\cdots, D^m u(x))|^{1/p_{\alpha\beta}} dx\Bigr)^{p_{\alpha\beta}}.
  \end{eqnarray*}
  Because of (\ref{e:1.5}), as treated in Step~1 of Section~\ref{sec:Funct.2}
  we deduce
 \begin{eqnarray*}
 \sum_{|\alpha|,|\beta|\le m}\Bigl(\int_\Omega
  |f_{\alpha\beta}(x, u_k(x),\cdots, D^m u_k(x))
  -f_{\alpha\beta}(x, u(x),\cdots, D^m u(x))|^{1/p_{\alpha\beta}} dx\Bigr)^{p_{\alpha\beta}}
  \to 0
  \end{eqnarray*}
  and so  $\|[D\mathcal{F}'(u_k)]^\ast v-[D\mathcal{F}'(u)]^\ast v\|\to 0$ as $k\to\infty$.

 For $p=2$ we can use (\ref{e:3.24}) to arrive at the same conclusion.

 \noindent{\bf Step 5}. The proof of (iii) in {\bf B)} is the same as that of (iii) in {\bf D)} later on.

\subsection{Proof for C) of Theorem~\ref{th:3.1}}\label{sec:Funct.4}

Let $u_j\rightharpoonup u$ in $W^{m,p}(\Omega)$ and satisfy
$\overline{\lim}_{j\to\infty}\langle \mathcal{F}'(u_j), u_j-u\rangle\le 0$, i.e.,
\begin{eqnarray}\label{e:3.25}
\overline{\lim}_{j\to\infty}\sum_{|\alpha|\le m}\int_\Omega f_\alpha(x,u_j(x),\cdots, D^m u_j(x))D^\alpha (u_j-u)dx=0.
\end{eqnarray}
By Sobolev embedding theorem we have strong convergence
\begin{eqnarray*}
&&D^\alpha u_j\to D^\alpha u\quad\hbox{in}\;L^q(\Omega)\;\hbox{if}\;q<\frac{np}{n-p},\;m-n/p\le |\alpha|\le
m-1,\\
&&D^\alpha u_j\to D^\alpha u\quad\hbox{in}\;C^0(\Omega)\;\hbox{if}\; |\alpha|<m-n/p.
\end{eqnarray*}
Since $\sup_j\|u_j\|_{m,p}<\infty$,  $C=\sup\{\|u\|_{C^k}+\|u_j\|_{C^k}:\, k<m-n/p,\;j\in\mathbb{N}\}<\infty$.
These, (\ref{e:3.9}) and (\ref{e:3.10}) lead to
\begin{eqnarray*}
&&\sum_{|\alpha|<m-n/p}\int_\Omega|f_\alpha(x,u_j(x),\cdots, D^m u_j(x))D^\alpha (u_j-u)|dx\\
&\le&\sum_{|\alpha|<m-n/p}\int_\Omega|f_\alpha(x,0)|\cdot|D^\alpha (u_j-u)|dx
+ \mathfrak{g}_5(C)\sum_{|\alpha|<m-n/p}\int_\Omega|D^\alpha (u_j-u)|dx\\
&&+\mathfrak{g}_5(C)\sum_{|\alpha|<m-n/p}\;
\sum_{m-n/p\le |\gamma|\le
m}\int_\Omega|D^\gamma u_j|^{p_\gamma }\cdot|D^\alpha (u_j-u)|dx\to 0
\end{eqnarray*}
and
\begin{eqnarray*}
&&\sum_{m-n/p\le|\alpha|\le m-1}\int_\Omega|f_\alpha(x,u_j(x),\cdots, D^m u_j(x))D^\alpha (u_j-u)|dx\\
&\le&\sum_{m-n/p\le|\alpha|\le m-1}\int_\Omega|f_\alpha(x,0)|\cdot|D^\alpha (u_j-u)|dx
+ \mathfrak{g}_5(C)\sum_{m-n/p\le|\alpha|\le m-1}\int_\Omega|D^\alpha (u_j-u)|dx\\
&&+\mathfrak{g}_5(C)\sum_{m-n/p\le|\alpha|\le m-1}
\int_\Omega\Bigg(\sum_{m-n/p\le |\gamma|\le
m}|D^\gamma u_j|^{p_\gamma}\Bigg)^{1/q_\alpha}\cdot|D^\alpha (u_j-u)|dx\to 0.
\end{eqnarray*}
Here the final limit is because Lemma~\ref{lem:3.6}(ii) implies  that
\begin{eqnarray*}
&&\int_\Omega\Bigg(\sum_{m-n/p\le |\gamma|\le
m}|D^\gamma u_j|^{p_\gamma}\Bigg)^{1/q_\alpha}\cdot|D^\alpha (u_j-u)|dx\to 0\\
&\le&\sum_{m-n/p\le |\gamma|\le m}
\int_\Omega|D^\gamma u_j|^{p_\gamma/q_\alpha}\cdot|D^\alpha (u_j-u)|dx\to 0\\
&\le&\sum_{m-n/p\le |\gamma|\le m}
\left(\int_\Omega|D^\gamma u_j|^{p_\gamma}dx\right)^{1/q_\alpha}\left(\int_\Omega
|D^\alpha (u_j-u)|^{p_\alpha}dx\right)^{1/p_\alpha}\to 0.
\end{eqnarray*}
It follows that
\begin{equation}\label{e:3.26}
\sum_{|\alpha|\le m-1}\int_\Omega |f_\alpha(x,
u_j(x),\cdots, D^m u_j(x))D^\alpha (u_j-u)| dx\to 0\quad\hbox{as}\;j\to\infty.
\end{equation}
This and (\ref{e:3.25}) yield
\begin{equation}\label{e:3.27}
\sum_{|\alpha|=m}\int_\Omega f_\alpha(x,
u_j(x),\cdots, D^m u_j(x))D^\alpha (u_j-u) dx\to 0\quad\hbox{as}\;j\to\infty.
\end{equation}

Next, we claim
\begin{eqnarray}\label{e:3.28}
&&\sum_{|\alpha|=m}\int_\Omega |f_\alpha(x,
u_j(x),\cdots, D^{m-1} u_j(x), D^{m} u(x))D^\alpha (u_j-u)| dx\nonumber\\
&\le &\sum_{|\alpha|=m}\int_\Omega \Big|[f_\alpha(x,
u_j(x),\cdots, D^{m-1} u_j(x), D^{m} u(x))\nonumber\\
&&\hspace{20mm}-f_\alpha(x,
u(x),\cdots, D^{m-1} u(x), D^{m} u(x))]D^\alpha (u_j-u)\Big| dx\nonumber\\
&&+\sum_{|\alpha|=m}\int_\Omega |f_\alpha(x,
u(x),\cdots, D^{m-1} u(x), D^{m} u(x))D^\alpha (u_j-u)| dx\nonumber\\
&=&I_{1,j}+I_{2,j}\to 0.
\end{eqnarray}
Indeed, by (\ref{e:3.10}) it is easily checked that
$f_\alpha(x,u(x),\cdots, D^{m-1} u(x), D^{m} u(x))$ belongs to
$L^{q_\alpha}(\Omega)$ for $|\alpha|\ge m-n/p$. Since
 $u_j\rightharpoonup u$ in $W^{m,p}(\Omega)$ it follows from that
 $I_{2,j}\to 0$ as $j\to\infty$.

 By the H\"older inequality we have
\begin{eqnarray*}
I_{1,j}&\le&\sum_{|\alpha|=m}\Bigl(\int_\Omega |f_\alpha(x,
u_j(x),\cdots, D^{m-1} u_j(x), D^{m} u(x))\nonumber\\
&&\hspace{10mm}-f_\alpha(x,
u(x),\cdots, D^{m-1} u(x), D^{m} u(x))|^{q_\alpha}dx\Bigr)^{1/q_\alpha}
\Bigl(\int_\Omega |D^\alpha (u_j-u)|^{p_\alpha}dx\Bigr)^{1/p_\alpha}.
\end{eqnarray*}
As before,  using (\ref{e:3.10}) and Proposition~\ref{prop:3.5} it is easy to derive
\begin{eqnarray*}
&&\Bigl(\int_\Omega |f_\alpha(x,
u_j(x),\cdots, D^{m-1} u_j(x), D^{m} u(x))\nonumber\\
&&\hspace{20mm}-f_\alpha(x,
u(x),\cdots, D^{m-1} u(x), D^{m} u(x))|^{q_\alpha}dx\Bigr)^{1/q_\alpha}\to 0.
\end{eqnarray*}
Note that  the sequence $\Bigl(\int_\Omega |D^\alpha (u_j-u)|^{p_\alpha}dx\Bigr)^{1/p_\alpha}$
is bounded. We get $I_{1,j}\to 0$.

Hence (\ref{e:3.27}) and (\ref{e:3.28}) yield
 \begin{eqnarray}\label{e:3.29}
&&\sum_{|\alpha|=m}\int_\Omega [f_\alpha(x,
u_j(x),\cdots, D^{m-1} u_j(x), D^{m} u_j(x))\nonumber\\
&&\hspace{10mm}-f_\alpha(x,
u_j(x),\cdots, D^{m-1} u_j(x), D^{m} u(x))]D^\alpha (u_j-u) dx\to 0.
\end{eqnarray}
Note that the integrand in each term is non-negative, which may be derived from the mean value theorem and (\ref{e:1.6})
as seen below. This implies that for any subset $E\subset\Omega$,
 \begin{eqnarray}\label{e:3.30}
&&\lambda_j(E):=\int_E\sum_{|\alpha|=m} [f_\alpha(x,
u_j(x),\cdots, D^{m-1} u_j(x), D^{m} u_j(x))\nonumber\\
&&\hspace{20mm}-f_\alpha(x,
u_j(x),\cdots, D^{m-1} u_j(x), D^{m} u(x))]D^\alpha (u_j-u) dx\to 0
\end{eqnarray}
as $j\to\infty$.  We claim that for $E\subset\Omega$,
 \begin{eqnarray}\label{e:3.31}
\lim_{{\rm meas}(E)\to 0}\int_E\sum_{|\alpha|=m}|D^\alpha u_j(x)|^pdx=0
\end{eqnarray}
 uniformly with respect to $j$.
In fact, by (\ref{e:3.30})  we have with $v_j=u_j-u$,
{\small 
\begin{eqnarray*}
\lambda_j(E)=\sum_{|\alpha|=|\beta|=m}\int_E\int^1_0
f_{\alpha\beta}\left(x,u_j(x),\cdots, D^{m-1} u_j(x),
D^{m} u(x)+ sD^{m} v_j(x)\right)D^\beta v_j\cdot D^\alpha v_j ds dx.
\end{eqnarray*}}
From (\ref{e:1.6}) and Lemma~\ref{lem:3.6}(i) we deduce
{\small
\begin{eqnarray}\label{e:3.32}
&&\sum_{|\alpha|=|\beta|=m}\int_E\int^1_0
f_{\alpha\beta}(x,u_j(x),\cdots, D^{m-1} u_j(x),
D^{m} u(x)+ sD^{m} v_j(x))D^\beta v_j\cdot D^\alpha v_j ds dx\nonumber\\
&\ge&\int_E\int^1_0
\mathfrak{g}_2(\sum_{k<m-n/p}\|u_j\|_{C^k})\biggl(1+ \sum_{|\gamma|=m}|D^{\gamma} u(x)+ sD^{\gamma} v_j(x))|\biggr)^{p-2}
\sum_{|\alpha|= m}|D^\alpha v_j|^2dsdx\nonumber\\
&\ge&\mathfrak{g}_2(\sum_{k<m-n/p}\|u_j\|_{C^k})
\int_E\int^1_0\bigl(1+ |D^{\gamma} u(x)+ sD^{\gamma} v_j(x))|\bigr)^{p-2}
\sum_{|\alpha|= m}|D^\alpha v_j|^2dsdx\nonumber\\
&\ge&\mathfrak{g}_2(\sum_{k<m-n/p}\|u_j\|_{C^k})
\int_E\bigl(1+ |D^{\gamma} u(x)|+ |D^{\gamma} u_j(x)|\bigr)^{p-2}
\sum_{|\alpha|= m}|D^\alpha v_j|^2dx\nonumber\\
&\ge&\mathfrak{g}_2(\sum_{k<m-n/p}\|u_j\|_{C^k})
\int_E\bigl(1+  |D^{\gamma} v_j(x)|\bigr)^{p-2}
|D^\gamma v_j|^2dx\nonumber\\
&\ge&\mathfrak{g}_2(\sum_{k<m-n/p}\|u_j\|_{C^k})
\int_E|D^\gamma v_j|^pdx
\end{eqnarray}}
for each $\gamma$ of length $m$. It follows
\begin{eqnarray}\label{e:3.33}
\sum_{|\gamma|=m}\int_E|D^\gamma v_j|^pdx&\le& \frac{M_0(m)}{\mathfrak{g}_2(\sum_{k<m-n/p}\|u_j\|_{C^k})}\lambda_j(E)\nonumber\\
&\le&\frac{M_0(m)}{\mathfrak{g}_2(\sup_j\sum_{k<m-n/p}\|u_j\|_{C^k})}\lambda_j(E).
\end{eqnarray}
The final inequality is because
 $u_j\rightharpoonup u$ and thus that $\sup_j\|u_j\|_{m,p}<\infty$, which implies that $\sup_j\sum_{k<m-n/p}\|u_j\|_{C^k}<\infty$. Moreover
 \begin{eqnarray*}
\bigg(\sum_{|\gamma|=m}\int_E|D^\gamma u_j|^pdx\bigg)^{1/p}\le
\bigg(\sum_{|\gamma|=m}\int_E|D^\gamma u|^pdx\bigg)^{1/p}
+\bigg(\sum_{|\gamma|=m}\int_E|D^\gamma v_j|^pdx\bigg)^{1/p}.
\end{eqnarray*}
Given any  $\varepsilon>0$, from (\ref{e:3.30}) and (\ref{e:3.33})
 there exist $j_0\in\mathbb{N}$ such that
$$
\bigg(\sum_{|\gamma|=m}\int_E|D^\gamma v_j|^pdx\bigg)^{1/p}<\varepsilon^p/2
$$
for all $j\ge j_0$ and all $E\subset \Omega$.
Using the absolute continuity of the integral we have $\delta>0$ such that
for any $E\subset \Omega$ with ${\rm mes}(E)<\delta$,
$$
\bigg(\sum_{|\gamma|=m}\int_E|D^\gamma u|^pdx\bigg)^{1/p}+
\bigg(\sum_{|\gamma|=m}\int_E|D^\gamma v_j|^pdx\bigg)^{1/p}<\varepsilon^p/2,\;j=1,\cdots,j_0.
$$
These lead to
$$
\sum_{|\gamma|=m}\int_E|D^\gamma u_j|^pdx<\varepsilon
$$
for all $j\in\mathbb{N}$ and all $E\subset\Omega$
with ${\rm mes}(E)<\delta$.
(\ref{e:3.31}) is proved.

Next, by (\ref{e:3.33}) and (\ref{e:3.30}), for any given $\sigma>0$ we have
\begin{eqnarray*}
\sigma^p{\rm mes}(\{|D^\alpha v_j|\ge\sigma\})&\le&\int_{\{|D^\alpha v_j|\ge\sigma\}}|D^\alpha v_j|^pdx
\le\int_{\Omega}|D^\alpha v_j|^pdx
\nonumber\\
&\le&\frac{M_0(m)}{\mathfrak{g}_2(\sup_j\sum_{k<m-n/p}\|u_j\|_{C^k})}\lambda_j(\Omega)\to 0
\end{eqnarray*}
as $j\to\infty$. This means that the sequence $D^\alpha u_j$ converges to $D^\alpha u$ in measure for $|\alpha|=m$.
Combing with (\ref{e:3.31}) we obtain that $D^\alpha u_j\to D^\alpha u$ in $L^p(\Omega)$
for $|\alpha|=m$. Moreover, $u_j\rightharpoonup u$ implies that
 $u_j\to u$ in $W^{m-1,p}(\Omega)$. Hence $\|u_j-u\|_{m,p}\to 0$ as $j\to\infty$.
\hfill$\Box$\vspace{2mm}

\subsection{Proof for D) of Theorem~\ref{th:3.1}}\label{sec:Funct.5}

 \noindent{\bf Step 1}. (i) of {\bf D)} can be proved as in (ii) of {\bf B)}.

\noindent{\bf Step 2}. \quad{\it Prove (ii) of {\bf D)}}. By Sobolev embedding theorem
$$
\sup\{\|D^\gamma u\|_{C^k}:\, k<m-n/2,\;u\in W^{m,2}(\Omega),\;\|u\|_{m,2}\le R\}\in (0, \infty).
$$
Let $C$ be equal to the value of $\mathfrak{g}_2$ at this number. We derive from (\ref{e:1.6}) that
\begin{eqnarray*}
(P(u)v, v)_{m,2}&=&\sum_{|\alpha|=|\beta|=m}\int_\Omega
  f_{\alpha\beta}(x, u(x),\cdots, D^m u(x))D^\beta v\cdot D^\alpha v dx\nonumber\\
  &&+ \sum_{|\alpha|\le m-1}\int_\Omega  D^\alpha v\cdot D^\alpha v dx\\
  &\ge & C\sum_{|\alpha|= m}\int_\Omega |D^\alpha v|^2dx+ \sum_{|\alpha|\le m-1}\int_\Omega  D^\alpha v\cdot D^\alpha v dx\\
  &\ge&\min\{C,1\}\|v\|^2_{m,2}\quad\hbox{for any $v\in W^{m,2}(\Omega)$.}
\end{eqnarray*}

\noindent{\bf Step 3}.\quad{\it Prove (iii) of {\bf D)}}.\quad
 As in the proof of (\ref{e:3.23}) we can
 obtain $C=C(m,n,p,\Omega)>0$ with $p=2$ such that
 \begin{eqnarray*}
 &&([Q(u)-Q(\bar{u})]v, \varphi)_{m,2}\nonumber\\
 &=&\bigg|\sum_{|\alpha|+|\beta|<2m}\int_\Omega
  [f_{\alpha\beta}(x, u(x),\cdots, D^m u(x))-f_{\alpha\beta}(x, \bar{u}(x),\cdots, D^m \bar{u}(x))]D^\beta v D^\alpha\varphi(x) dx\bigg|\nonumber\\
 &\le&C\sum_{|\alpha|+|\beta|<2m}\Bigl(\int_\Omega
  |f_{\alpha\beta}(x, u(x),\cdots, D^m u(x))-f_{\alpha\beta}(x, \bar{u}(x),\cdots, D^m \bar{u}(x))|^{1/p_{\alpha\beta}}\Bigl)^{p_{\alpha\beta}}\times\\
  &&\times\|v\|_{m,p}\|\cdot\varphi\|_{m,p}.
 \end{eqnarray*}
 So it follows from  (\ref{e:1.5}) and Proposition~\ref{prop:3.5} that
 {\small
\begin{eqnarray*}
 &&\|Q(u)-Q(\bar{u})\|_{\mathscr{L}(W^{m,2}(\Omega))}\nonumber\\
  &\le&C\sum_{|\alpha|+|\beta|<2m}\Bigl(\int_\Omega
  |f_{\alpha\beta}(x, u(x),\cdots, D^m u(x))-f_{\alpha\beta}(x, \bar{u}(x),\cdots, D^m \bar{u}(x))|^{1/p_{\alpha\beta}}\Bigl)^{p_{\alpha\beta}}\to 0.
 \end{eqnarray*}}

To prove the second claim let us
decompose $Q(u)$ into $Q_1(u)+ Q_2(u)+ Q_3(u)$, where
\begin{eqnarray*}
(Q_1(u)v,\varphi)_{m,2}&=&\sum_{|\alpha|\le m-1,|\beta|\le m-1}\int_\Omega
  f_{\alpha\beta}(x, u(x),\cdots, D^m u(x))D^\beta v\cdot D^\alpha\varphi dx\nonumber\\
  &&-\sum_{|\alpha|\le m-1}\int_\Omega  D^\alpha v\cdot D^\alpha\varphi dx,\\
  (Q_2(u)v,\varphi)_{m,2}&=&\sum_{|\alpha|=m,|\beta|\le m-1}\int_\Omega
  f_{\alpha\beta}(x, u(x),\cdots, D^m u(x))D^\beta v\cdot D^\alpha\varphi dx,\nonumber\\
  (Q_3(u)v,\varphi)_{m,2}&=&\sum_{|\alpha|\le m-1,|\beta|=m}\int_\Omega
  f_{\alpha\beta}(x, u(x),\cdots, D^m u(x))D^\beta v\cdot D^\alpha\varphi dx.
  \end{eqnarray*}
Clearly, $(Q_2(u)v,\varphi)_{m,2}=(Q_3(u)\varphi,v)_{m,2}$, that is, they are adjoint each other.
Let $v_j\rightharpoonup v$ in $W^{m,2}(\Omega)$. By the proof of (\ref{e:3.23}) we can get
 $C=C(m,n,p,\Omega)>0$ with $p=2$ such that
 \begin{eqnarray*}
 &&|(Q_1(u)(v_j-v),\varphi)_{m,2}|\\
 &\le&\sum_{|\alpha|\le m-1,|\beta|\le m-1}\left|\int_\Omega
  f_{\alpha\beta}(x, u(x),\cdots, D^m u(x))D^\beta (v_j-v) D^\alpha\varphi(x) dx\right|\nonumber\\
  &&+\sum_{|\alpha|\le m-1}\int_\Omega  |D^\alpha (v_j-v)|\cdot |D^\alpha\varphi| dx\\
   &\le&C\|v_j-v\|_{m-1,2}\cdot\|\varphi\|_{m-1,2}\sum_{|\alpha|\le m-1,|\beta|\le m-1}\Bigl(\int_\Omega
  |f_{\alpha\beta}(x, u(x),\cdots, D^m u(x))|^{1/p_{\alpha\beta}}\Bigl)^{p_{\alpha\beta}}\\
  &&+\|v_j-v\|_{m-1,2}\cdot\|\varphi\|_{m-1,2}
 \end{eqnarray*}
and hence
\begin{eqnarray*}
 &&\|Q_1(u)(v_j-v)\|_{m,2}\\
    &\le&C\|v_j-v\|_{m-1,2}\sum_{|\alpha|\le m-1,|\beta|\le m-1}\Bigl(\int_\Omega
  |f_{\alpha\beta}(x, u(x),\cdots, D^m u(x))|^{1/p_{\alpha\beta}}\Bigl)^{p_{\alpha\beta}}\\
  &&+\|v_j-v\|_{m-1,2}\to 0
 \end{eqnarray*}
because $\|v_j-v\|_{m-1,2}\to 0$ by the compactness of the embedding $W^{m,2}(\Omega)\hookrightarrow W^{m-1,2}(\Omega)$.

Similarly, we have
\begin{eqnarray*}
&&|(Q_2(u)(v_j-v),\varphi)_{m,2}|\\
 &\le&\sum_{|\alpha|=m,|\beta|\le m-1}\left|\int_\Omega
  f_{\alpha\beta}(x, u(x),\cdots, D^m u(x))D^\beta (v_j-v) D^\alpha\varphi(x) dx\right|\nonumber\\
    &\le&C\|v_j-v\|_{m-1,2}\cdot\|\varphi\|_{m,2}\sum_{|\alpha|=m,|\beta|\le m-1}\Bigl(\int_\Omega
  |f_{\alpha\beta}(x, u(x),\cdots, D^m u(x))|^{1/p_{\alpha\beta}}\Bigl)^{p_{\alpha\beta}}
   \end{eqnarray*}
and hence
\begin{eqnarray*}
 &&\|Q_2(u)(v_j-v)\|_{m,2}\\
    &\le&C\|v_j-v\|_{m-1,2}\sum_{|\alpha|=m,|\beta|\le m-1}\Bigl(\int_\Omega
  |f_{\alpha\beta}(x, u(x),\cdots, D^m u(x))|^{1/p_{\alpha\beta}}\Bigl)^{p_{\alpha\beta}}
  \to 0.
 \end{eqnarray*}

Since $Q_3(u)$ is the adjoint operator of $Q_2(u)$, it is completely continuous too.

\noindent{\bf Step 4}.\quad{\it Prove (iv) of {\bf D)}}.\quad
 By the arguments in Step 3 we see:
for every given $R>0$ there exist positive constants $\hat{C}(R, n, m, \Omega)$ such that
if $u\in W^{m,2}(\Omega)$ satisfies $\|u\|_{m,2}\le R$ then
\begin{eqnarray*}
&&|(Q_1(u)v,\varphi)_{m,2}|\le \hat{C}\|v\|_{m-1,2}\cdot\|\varphi\|_{m-1,2}\qquad\forall v,\varphi\in W^{m,2}(\Omega),\\
&&|(Q_2(u)v,\varphi)_{m,2}|\le \hat{C}\|v\|_{m-1,2}\cdot\|\varphi\|_{m,2}\qquad\forall v,\varphi\in W^{m,2}(\Omega),\\
&&|(Q_3(u)v,\varphi)_{m,2}|\le \hat{C}\|v\|_{m,2}\cdot\|\varphi\|_{m-1,2}\qquad\forall v,\varphi\in W^{m,2}(\Omega)
\end{eqnarray*}
and therefore
$$
|(Q(u)v,v)_{m,2}|\le 3\hat{C}\|v\|_{m-1,2}\cdot\|v\|_{m,2}\qquad\forall v\in W^{m,2}(\Omega).
$$
 For the constant $C$ in (ii) of {\bf D)}, using
the inequality $ab\le\frac{1}{2\varepsilon}a^2+ \frac{\varepsilon}{2}b^2$ for any $\varepsilon>0$
and $a,b\ge 0$, we derive with $\varepsilon=C/(3\hat{C})$,
$$
|(Q(u)v,v)_{m,2}|\le  3\hat{C}\|v\|_{m-1,2}\cdot\|v\|_{m,2}\le\frac{C}{2}\|v\|^2_{m,2}+ \frac{9\hat{C}^2}{2C}\|v\|^2_{m-1,2}
$$
for all $v\in W^{m,2}(\Omega)$. Taking $C_1=C/2$ and $C_2=\frac{9\hat{C}^2}{2C}$, this and (ii) of {\bf D)}
give the desired result.


\section{(PS)- and (C)-conditions}\label{sec:PS}
\setcounter{equation}{0}

A $C^1$ functional $\varphi$ on a Banach-Finsler manifold $\mathcal{M}$ is said to satisfy
the {\it Palais-Smale condition at the level $c\in\mathbb{R}$} ({$(PS)_c$-condition}, for short)
if every sequence $\{x_j\}_{j\ge 1}\subset X$ such that
$\varphi(x_j)\to c\in\mathbb{R}$ and $\varphi'(x_j)\to 0$ in $X^\ast$
has a convergent subsequence in $\mathcal{M}$. When $\varphi$ satisfies the $(PS)_c$-condition
at every level $c\in\mathbb{R}$ we say that it satisfies the {\it Palais-Smale condition}
($(PS)$-{\it condition}, for short).

When $\mathcal{M}$ a Banach space there is weaker condition.
Call a $C^1$ functional $\varphi$ on a Banach space $X$  to satisfy
the {\it Cerami condition at the level $c\in\mathbb{R}$} ({\it $(C)_c$-condition}, for short)
if every sequence $\{x_j\}_{j\ge 1}\subset X$ such that
$\varphi(x_j)\to c\in\mathbb{R}$ and $(1+\|x_j\|)\varphi'(x_j)\to 0$ in $X^\ast$
has a convergent subsequence in $X$. When $\varphi$ satisfies the $(C)_c$-condition
at every level $c\in\mathbb{R}$ we say that it satisfies the {\it Cerami condition}
($(C)$-{\it condition}, for short).

Actually, if a $C^1$ functional $\varphi$ on a Banach space $X$ is bounded below,
 Caklovic, Li and  Willem \cite{CaLiWi} showed that $\varphi$ satisfies the $(PS)$-condition if and only if it
 is coercive. It was further proved in \cite[Proposition~5.23]{MoMoPa}
 that $\varphi$ satisfies the $(PS)$-condition if and only if it does the $(C)$-condition.
 Recently, it was proved in \cite[Theorem~6]{Suz} that
 if a continuous  functional $\varphi$ on $X$ is G\^ateaux differentiable
 and satisfies the weak Palais-Smale condition then $|\varphi|$ is
 coercive provided  $\{x\in X\,|\,\varphi(x)=c\}$ is bounded for some $c\in\mathbb{R}$.

\begin{theorem}\label{th:3.7}
Let $\Omega\subset\R^n$, $N\in\mathbb{N}$, $p\in [2,\infty)$ and $V\subset W^{m,p}(\Omega,\mathbb{R}^N)$ be as in Theorem~\ref{th:3.2}.
  Suppose that  \textsf{Hypothesis} $\mathfrak{F}_{p,N}$ hold and that
   $\mathfrak{F}$ is coercive, i.e., $\mathfrak{F}(\vec{u})\to\infty$
as $\|\vec{u}\|_V\to\infty$.  Then $\mathfrak{F}$ satisfies the {\rm (PS)}- and {\rm (C)}-conditions on $V$.
In particular, the same conclusions hold with the functional $\mathcal{F}$ on $V\subset W^{m,p}(\Omega)$
under the condition $\mathfrak{f}_p$.
\end{theorem}

\begin{proof}
 Since $\mathfrak{F}$ is coercive, it is bounded below.
By \cite[Proposition~5.23]{MoMoPa} it suffices to prove that
$\mathfrak{F}$ satisfies the (PS)-condition.

Let a sequence $\{\vec{u}_j\}_{j\ge 1}$ such that
$\mathfrak{F}(\vec{u}_j)\to c\in\mathbb{R}$ and $\mathfrak{F}'(\vec{u}_j)\to 0$ as $j\to\infty$.
Since $\mathfrak{F}$ is coercive, the sequence $\{\vec{u}_j\}_{j\ge 1}$ must be bounded.
Note that $V$ is a self-reflexive Banach space. After passing to a subsequence we may assume $\vec{u}_j\rightharpoonup \vec{u}$ in $V$. Moreover, $\mathfrak{F}'(\vec{u}_j)\to 0$ implies $\overline{\lim}_{j\to\infty}\langle\mathfrak{F}'(\vec{u}_j), \vec{u}_j-\vec{u}\rangle=0$.
By Theorem~\ref{th:3.2} (the corresponding conclusion to {\bf C)} of Theorem~\ref{th:3.1}) we know that $\mathfrak{F}'$ is of class $(S)_+$.
Hence $\vec{u}_j\to \vec{u}$ in $V$.
\end{proof}

Clearly, the coercivity of $\mathfrak{F}$ implies that it is bounded below.
On the other hand,  for a $C^1$ functional $\varphi$ on a Banach space $X$ which is bounded below,
Li Shujie showed that it is coercive if $\varphi$
satisfies the $(PS)$-condition.

There exist some explicit conditions on $F$ under which  $\mathfrak{F}$ is coercive on $W^{m,p}_0(\Omega,\mathbb{R}^N)$,
for example,  there exist  some two positive constants $c_0, c_1$ such that
$$
  F(x,\xi)\ge c_0\sum^N_{i=1}\sum_{|\alpha|=m}|\xi^i_\alpha|^p-c_1\quad\forall (x,\xi).
  $$
The coercivity requirement is too strong. In fact, the proof of Theorem~\ref{th:3.7}
shows that under  \textsf{Hypothesis} $\mathfrak{F}_{p,N}$ we only need to add some conditions
so that
$$
\sup_j|\mathfrak{F}(\vec{u}_j)|<\infty\quad\hbox{and}\quad \mathfrak{F}'(\vec{u}_j)\to 0\;\Longrightarrow\;
\sup_j\|\vec{u}_j\|_{m,p}<\infty.
$$

\begin{theorem}\label{th:3.8}
Under  \textsf{Hypothesis} $\mathfrak{F}_{p,N}$, suppose that there exist $\kappa\in\mathbb{R}$
   and $\Upsilon\in L^1(\Omega)$ such that
  $$
  F(x,\xi)-\kappa\sum^N_{i=1}\sum_{|\alpha|\le m}F^i_\alpha(x,\xi)\xi^i_\alpha\ge c_0\sum^N_{i=1}\sum_{|\alpha|=m}|\xi^i_\alpha|^p-c_1\sum^N_{i=1}|\xi^i_{\bf 0}|^p
  -\Upsilon(x)  \quad\forall (x,\xi),
  $$
 where  $c_0>0$ and $c_0-c_1S_{m,p}>0$ for the best constant $S_{m,p}>0$ with
 $$
 \int_\Omega |{u}|^p dx\le S_{m,p}\int_\Omega|D^m{u}|^p dx=S_{m,p}\sum_{|\alpha|=m}
 \int_\Omega|D^\alpha{u}|^p\quad\forall {u}\in W^{m,p}_0(\Omega).
 $$
 Then $\mathfrak{F}$ satisfies the {\rm (PS)}- and {\rm (C)}-conditions  on $W^{m,p}_0(\Omega,\mathbb{R}^N)$.
\end{theorem}

\begin{proof}
 Let $\{\vec{u}_k\}_{k\ge 1}\subset W^{m,p}_0(\Omega,\mathbb{R}^N)$ be a sequence such that
$|\mathfrak{F}(\vec{u}_k)|\le M\;\forall k$ for some $M>0$, and
$\mathfrak{F}'(\vec{u}_k)\to 0$ (resp. $(1+\|\vec{u}_k\|)\mathfrak{F}'(\vec{u}_k)\to 0$).
 By (\ref{e:3.6.1}) the latter means
\begin{equation}\label{e:3.34*}
\left.\begin{array}{cr}
&\left|\sum^N_{i=1}\sum_{|\alpha|\le m}\int_\Omega F^i_\alpha(x,
\vec{u}_k(x),\cdots, D^m \vec{u}_k(x))D^\alpha {u}^i_k dx\right|\le \varepsilon_k\|\vec{u}_k\|_{m,p}\\
&\\
&\hbox{(resp.}\;
\left|\sum^N_{i=1}\sum_{|\alpha|\le m}\int_\Omega F^i_\alpha(x,
\vec{u}_k(x),\cdots, D^m \vec{u}_k(x))D^\alpha {u}^i_k dx\right|\le \varepsilon_k\frac{\|\vec{u}_k\|_{m,p}}{1+
\|\vec{u}_k\|_{m,p}}\;{\rm )}
\end{array}\right.
\end{equation}
where $\varepsilon_k\to 0$ and $\|\vec{u}_k\|_{m,p}=\|D^m\vec{u}_k\|_{p}$ as usual. By the assumption we have
\begin{eqnarray*}
&&\mathfrak{F}(\vec{u}_k)-\kappa\sum^N_{i=1}\sum_{|\alpha|\le m}\int_\Omega F^i_\alpha(x,
\vec{u}_k(x),\cdots, D^m \vec{u}_k(x))D^\alpha u^i_k dx\\
&\ge& c_0\sum^N_{i=1}\sum_{|\alpha|=m}\int_\Omega|D^\alpha u^i_k|^p dx-c_1\sum^N_{i=1}\int_\Omega |u^i_k|^p dx-\int_\Omega\Upsilon(x) dx\\
&\ge&c_0\sum^N_{i=1}\int_\Omega|D^m u^i_k|^p dx-c_1S_{m,p}\sum^N_{i=1}\int_\Omega|D^m u^i_k|^p dx-\int_\Omega\Upsilon(x) dx\\
\end{eqnarray*}
and therefore
\begin{eqnarray*}
\begin{array}{cr}
&(c_0-c_1S_{m,p})\sum^N_{i=1}\int_\Omega|D^m u_k^i|^p dx\le \int_\Omega\Upsilon(x) dx+ M+ |\kappa|\varepsilon_k\|\vec{u}_k\|_{m,p}\\
&\\
&{\rm (resp.}\;(c_0-c_1S_{m,p})\sum^N_{i=1}\int_\Omega|D^m u^i_k|^p dx\le \int_\Omega\Upsilon(x) dx+ M+ |\kappa|\varepsilon_k
\;{\rm )}.
\end{array}
\end{eqnarray*}
This implies that $\|\vec{u}_k\|_{m,p}$ is bounded. Passing to a subsequence if necessary,
we may assume $\vec{u}_k\rightharpoonup \vec{u}$. The remainder is the same as that of Theorem~\ref{th:3.7}.
\end{proof}

\begin{theorem}\label{th:3.9}
Let $\Omega\subset\R^n$ be a bounded Sobolev domain. Suppose that
   \textsf{Hypothesis} $\mathfrak{F}_{2,N}$ is satisfied with the constant function $\mathfrak{g}_2$, and  that
  for any $(x, \xi)\in\overline\Omega\times\prod^{m-1}_{k=0}\mathbb{R}^{N\times M_0(k)}$,
   $$
  F(x,\xi, {\bf 0})\le \varphi(x)+ C\sum^N_{i=1}\sum_{|\alpha|\le m-1}|\xi^i_\alpha|^{r},
  $$
 where  $\varphi\in L^1(\Omega)$ and $1\le r<2$.
 Then $\mathfrak{F}$ satisfies the {\rm (PS)}- and {\rm (C)}-conditions  on $W^{m,2}_0(\Omega, \mathbb{R}^N)$.
\end{theorem}

\begin{proof}
 For any $x\in\Omega$ and $\xi=(\xi^1,\cdots,\xi^m)\in\prod^{m}_{k=0}\mathbb{R}^{N\times M_0(k)}$
let $\hat{\xi}=(\xi^1,\cdots,\xi^{m-1})\in\prod^{m-1}_{k=0}\mathbb{R}^{N\times M_0(k)}$.
By the mean value theorem we get
\begin{eqnarray*}
&&\sum^N_{i=1}\sum_{|\alpha|=m}F^i_\alpha(x,\xi)\xi^i_\alpha - F(x,\xi)\\
&=&\sum^N_{i=1}\sum_{|\alpha|=m}F^i_\alpha(x,\xi)\xi^i_\alpha
-[F(x,\xi)-F(x,\hat{\xi}, {\bf 0})]- F(x,\hat{\xi}, {\bf 0})\\
&=&\sum^N_{i=1}\sum_{|\alpha|=m}F^i_\alpha(x,\xi)\xi^i_\alpha-\sum^N_{i=1}\sum_{|\alpha|=m}\int^1_0
F^i_\alpha(x,\hat{\xi}, t\xi^m)\xi^i_\alpha  dt- F(x,\hat{\xi}, {\bf 0}) \\
&=&\sum^N_{i,j=1}\sum_{|\alpha|=|\beta|=m}\int^1_0dt\int^1_0F^{ij}_{\alpha\beta}(x,\hat{\xi},
(t+s-st)\xi^{m})(1-t)\xi^i_\alpha \xi^j_\beta ds- F(x,\hat{\xi}, {\bf 0})\\
&\ge&\frac{1}{2}\mathfrak{g}_2\sum^N_{i=1}\sum_{|\alpha|= m}|\xi^i_\alpha |^2- F(x,\hat{\xi}, {\bf 0}).
\end{eqnarray*}
It follows that for any $\vec{u}\in W^{m,2}_0(\Omega,\mathbb{R}^N)$,
\begin{eqnarray*}
\frac{1}{2}\mathfrak{g}_2
\sum^N_{i=1}\sum_{|\alpha|= m}\int_\Omega|D^\alpha u^i|^2\le \int_\Omega F(x,\vec{u},\cdots, D^{m-1}\vec{u},0)- \mathfrak{F}(\vec{u})+ \langle\mathfrak{F}'(\vec{u}),
\vec{u}\rangle.
\end{eqnarray*}
  By the assumption and the Young inequality we derive
\begin{eqnarray*}
 && \int_\Omega F(x,\vec{u},\cdots,D^{m-1}\vec{u}, 0)\le \int_\Omega\varphi(x)+ C\sum^N_{i=1}\sum_{|\alpha|\le m-1}\int_\Omega|D^\alpha u^i|^{r}\\
 &\le&\int_\Omega\varphi(x)+ C\sum^N_{i=1}\sum_{|\alpha|\le m-1}\int_\Omega\left(\frac{r\varepsilon}{2}
 |D^\alpha u^i|^{2}+ \frac{2-r}{2}\varepsilon^{-r/(2-r)}\right)\\
 &\le&\int_\Omega\varphi(x)+ C\varepsilon\|\vec{u}\|^2_{m,2}+ C\varepsilon^{-r/(2-r)}
  \end{eqnarray*}
and hence
\begin{eqnarray*}
\frac{1}{2}\mathfrak{g}_2
\|\vec{u}\|^2_{m,2}\le\int_\Omega\varphi(x)+ C\varepsilon\|\vec{u}\|^2_{m,2}+ C\varepsilon^{-r/(2-r)}- \mathfrak{F}(u)+ \langle\mathfrak{F}'(\vec{u}),
\vec{u}\rangle.
\end{eqnarray*}
Taking $\varepsilon=\frac{\mathfrak{g}_2}{4C}$ leads to
\begin{eqnarray}\label{e:3.35*}
\frac{1}{4}\mathfrak{g}_2
\|\vec{u}\|^2_{m,2}\le\int_\Omega\varphi(x)+ C\left(\frac{4C}{\mathfrak{g}_2}\right)^{r/(2-r)}- \mathfrak{F}(\vec{u})+ \langle\mathfrak{F}'(\vec{u}),
\vec{u}\rangle
\end{eqnarray}
for any $\vec{u}\in W^{m,2}_0(\Omega,\mathbb{R}^N)$.
Let the sequence $\{\vec{u}_k\}_{k\ge 1}\subset W^{m,2}_0(\Omega,\mathbb{R}^N)$ such that
$\sup_k|\mathfrak{F}(\vec{u}_k)|\le M$ for some $M>0$, and
$\mathfrak{F}'(\vec{u}_k)\to 0$ (resp.  $(1+\|\vec{u}_k\|_{m,2})\mathfrak{F}'(\vec{u}_k)\to 0$).  Then (\ref{e:3.35*})
implies (because of $|\langle\mathfrak{F}'(\vec{u}),
\vec{u}\rangle|\le \|\mathfrak{F}'(\vec{u})\|\cdot\|\vec{u}\|_{m,2}\le (1+\|\vec{u}\|_{m,2})\|\mathfrak{F}'(\vec{u})\|$) that $\{\vec{u}_k\}_{k\ge 1}$ is bounded in $W^{m,2}_0(\Omega,\mathbb{R}^N)$
and thus has a convergent subsequence as above.
\end{proof}

\section{Morse inequalities}\label{sec:Morse}
\setcounter{equation}{0}

Firstly, we show that Theorem~\ref{th:3.2} and Theorems~\ref{th:S.6.0},~\ref{th:S.6.1} imply
the generalized Morse lemma.

\begin{theorem}\label{th:Morse.1}
Let $\Omega\subset\R^n$ be a bounded  Sobolev domain, $N\in\mathbb{N}$, and $H$  a closed subspace of $W^{m,2}(\Omega,\mathbb{R}^N)$. Let $G$ be a compact Lie group which acts on $H$ in a $C^3$-smooth
isometric way.  Suppose that \textsf{Hypothesis} $\mathfrak{F}_{2,N}$ is satisfied and that
the functional $\mathfrak{F}$ given by (\ref{e:1.3}) is $G$-invariant.
Let $\mathcal{O}$ be an isolated critical orbit of $\mathfrak{F}$ (always understanding as
$\mathfrak{F}|_H$). It is a compact $C^3$ submanifold, whose normal bundle $N\mathcal{O}$ has fiber at $\vec{u}\in\mathcal{O}$,
$$
N\mathcal{O}_{\vec{u}}=\{\vec{v}\in H\,|\, (\vec{v}, \vec{w})_{m,2}=0\;\forall\vec{w}\in T_{\vec{u}}\mathcal{O}\subset H\}.
$$
Let $N^+\mathcal{O}_{\vec{u}},  N^0\mathcal{O}_{\vec{u}}$ and $N^-\mathcal{O}_{\vec{u}}$ be the positive definite, null and negative definite spaces of the bounded linear self-adjoint operator associated with the bilinear form
   $$
  N\mathcal{O}_{\vec{u}}\times N\mathcal{O}_{\vec{u}}\ni (\vec{v}, \vec{w})\mapsto \sum^N_{i=1}\sum_{|\alpha|,|\beta|\le m}\int_\Omega
  F^{ij}_{\alpha\beta}(x, \vec{u}(x),\cdots, D^m \vec{u}(x))D^\beta v^j\cdot D^\alpha w^i dx.
  $$
Then $\dim N^0\mathcal{O}_{\vec{u}}$ and $\dim N^-\mathcal{O}_{\vec{u}}$ are finite and independent of choice
of $\vec{u}\in\mathcal{O}$. They are called nullity and Morse index of $\mathcal{O}$, denoted by $\nu_{\mathcal{O}}$
and $\mu_{\mathcal{O}}$, respectively. Moreover, the following holds.
\begin{description}
\item[(i)] If $\nu_{\mathcal{O}}=0$ (i.e., the critical orbit $\mathcal{O}$ is nondegenerate),
there exist  $\epsilon>0$
and a  $G$-equivariant homeomorphism onto an open neighborhood of
the zero section preserving fibers
$$
\Phi:  N^+\mathcal{O}(\epsilon)\oplus N^-\mathcal{ O}(\epsilon)\to N\mathcal{ O}
$$
such that for any $\vec{u}\in\mathcal{O}$ and  $(\vec{v}_+, \vec{v}_-)\in
N^+\mathcal{ O}(\epsilon)_{\vec{u}}\times N^-\mathcal{O}(\epsilon)_{\vec{u}}$,
\begin{eqnarray}\label{e:Morse.1}
\mathfrak{F}\circ E\circ\Phi(\vec{u},  \vec{v}_++
\vec{v}_-)=\|\vec{v}_+\|^2_{m,2}-\|\vec{v}_-\|^2_{m,2}+ \mathfrak{F}|_{\mathcal{O}},
\end{eqnarray}
where $E:N\mathcal{ O}\to H$ is given by $E(\vec{u},\vec{v})=\vec{u}+\vec{v}$.
\item[(ii)] If $\nu_{\mathcal{O}}\ne 0$ there exist $\epsilon>0$,
a  $G$-equivariant topological  bundle
morphism that preserves the zero section,
 $$
\mathfrak{h}:N^0\mathcal{O}(3\epsilon)\to N^+\mathcal{O}\oplus N^-\mathcal{O}\subset H,\;(\vec{u},\vec{v})
\mapsto \mathfrak{h}_{\vec{u}}(\vec{v}),
$$
and a  $G$-equivariant homeomorphism onto an open neighborhood of
the zero section preserving fibers,
$\Phi: N^0\mathcal{ O}(\epsilon)\oplus N^+\mathcal{
O}(\epsilon)\oplus N^-\mathcal{ O}(\epsilon)\to N\mathcal{O}$,
such that  the following properties hold:\\
\noindent{\bf (ii.1)} for any $\vec{u}\in\mathcal{O}$ and  $(\vec{v}_0, \vec{v}_+, \vec{v}_-)\in N^0\mathcal{O}(\epsilon)_{\vec{u}}\times N^+\mathcal{ O}(\epsilon)_{\vec{u}}\times N^-\mathcal{
O}(\epsilon)_{\vec{u}}$,
\begin{eqnarray}\label{e:Morse.2}
\mathfrak{F}\circ E\circ\Phi(\vec{u}, \vec{v}_0, \vec{v}_++
\vec{v}_-)=\|\vec{v}_+\|^2_{m,2}-\|\vec{v}_-\|^2_{m,2}+ \mathfrak{F}(\vec{u}+\vec{v}_0+
\mathfrak{h}_{\vec{u}}(\vec{v}_0));
\end{eqnarray}
\noindent{\bf (ii.2)} for each $\vec{u}\in\mathcal{O}$ the function
\begin{eqnarray}\label{e:Morse.3}
N^0\mathcal{O}(\epsilon)_{\vec{u}}\to\R,\;\vec{v}\mapsto \mathfrak{F}_{\vec{u}}^\circ(\vec{v}):=
\mathfrak{F}(\vec{u}+\vec{v}+ \mathfrak{h}_{\vec{u}}(\vec{v}))
\end{eqnarray}
is $G_{\vec{u}}$-invariant, of class $C^{1}$,  and satisfies
$$
D\mathfrak{F}_{\vec{u}}^\circ(\vec{v})\vec{w}:=
(\nabla\mathfrak{F}(\vec{u}+\vec{v}+ \mathfrak{h}_{\vec{u}}(\vec{v})),\vec{w}),\quad\;\forall \vec{w}\in N^0\mathcal{O}_{\vec{u}}.
$$
\end{description}
\end{theorem}

\begin{proof}
Since $\mathfrak{F}|_H$ satisfies
Hypothesis~\ref{hyp:1.1} with $X=H$ around each critical point by Theorem~\ref{th:3.2},
it follows from Lemma~\ref{lem:S.2.4} that
for each $\vec{u}\in\mathcal{O}$ the restriction
$\mathfrak{F}_{N\mathcal{O}_{\vec{u}}}$ satisfies Hypothesis~\ref{hyp:1.1} with $X=N\mathcal{O}_{\vec{u}}$
around the origin of $N\mathcal{O}_{\vec{u}}$.
Note that the exponential map on $H$, $\exp:TH=H\times H\to H$,
 is given by $\exp(\vec{u},\vec{v})=\vec{u}+\vec{v}$.
Then Theorems~\ref{th:S.6.0},~\ref{th:S.6.1} lead to the desired conclusions immediately.
\end{proof}

By Corollary~\ref{cor:S.6.5} and (\ref{e:S.6.20}), for any commutative ring ${\bf K}$ we get
\begin{eqnarray}
C_q(\mathfrak{F}, \mathcal{ O};{\bf K})\cong
\oplus^q_{j=0}C_{q-j-\mu_\mathcal{O}}(\mathfrak{F}^{\circ}_{\vec{u}},
\theta; {\bf K})\otimes H_j(\mathcal{ O};{\bf K})\quad\forall q=0,
1,2,\cdots\label{e:Morse.4}
\end{eqnarray}
if  $\vec{u}\in\mathcal{O}$, $\nu_{\mathcal{O}}\ne 0$ and  $\mathcal{ O}$ has
trivial normal bundle;  and $C_\ast(\mathfrak{F}, \mathcal{O};\mathbb{Z}_2)\cong
C_{\ast-\mu_\mathcal{O}}(\mathcal{O};\mathbb{Z}_2)$,
\begin{eqnarray}\label{e:Morse.5}
C_\ast(\mathfrak{F}, \mathcal{O};{\bf K})\cong
C_{\ast-\mu_\mathcal{O}}(\mathcal{O};\theta^-\otimes{\bf K})
\quad\hbox{and}\quad C_G^\ast(\mathfrak{F}, \mathcal{ O};{\bf K})\cong
H_G^{\mu_{\cal O}-1}(\mathcal{O};\theta^-\otimes{\bf K})
\end{eqnarray}
if $\nu_{\mathcal{O}}=0$, where $\mu_{\cal O}$ is the Morse index of $\mathcal{O}$
and $\theta^-$ is the orientation bundle of $N^-\mathcal{O}$.

From the second equality in (\ref{e:Morse.5}) and \cite[Chapter I, Theorem~7.6]{Ch} we immediately arrive at

\begin{theorem}\label{th:Morse.2}
Under the assumptions of Theorem~\ref{th:Morse.1},
Let  $a<b$ be two regular values of $\mathfrak{F}$ and
  $\mathfrak{F}^{-1}([a,b])$ contains only nondegenerate  critical orbits $\mathcal{O}_j$
  with Morse indexes $\mu_j$, $j=1,\cdots,k$.
 Suppose that  $\mathfrak{F}$ satisfies the $(PS)_c$ condition
for each $c\in [a,b)$. (For example, this is true if  either  $\mathfrak{F}$ is coercive or  one of Theorems~\ref{th:3.8},~\ref{th:3.9}  holds in case $H=W^{m,2}_0(\Omega,\mathbb{R}^N)$.)
Then there exists a polynomial with nonnegative
integral coefficients $Q(t)$ such that
\begin{eqnarray}\label{e:Morse.6}
\sum^\infty_{i=0}\sum^k_{j=1}{\rm rank}H^i_G(\mathcal{O}_j, \theta^-_j\otimes{\bf K})t^{\mu_j+i}
=\sum^\infty_{i=0}{\rm rank}H^i_G(\mathfrak{F}_b, \mathfrak{F}_a;{\bf K})t^i+ (1+t)Q(t),
\end{eqnarray}
where $\theta^-_j$ is the orientation bundle of $N^-\mathcal{O}_j$, $j=1,\cdots,k$.
In particular, if  $G$ is trivial and each $\mathcal{O}_j$ becomes a nondegenerate
 critical point $\vec{u}_j$, then  the following Morse inequalities hold:
\begin{eqnarray}\label{e:Morse.7}
\sum^l_{j=0}(-1)^{l-j}N_j(a,b)\ge\sum^l_{j=0}(-1)^{l-j}\beta_j(a,b),\quad l=0,1,\cdots,
\end{eqnarray}
where for each $q\in\mathbb{N}\cup\{0\}$, $N_q(a,b)=\sharp\{1\le i\le k\,|\,\mu_i=q\}$ (the number of points in $\{\vec{u}_j\}^k_{j=1}$ with
Morse index $q$)  and
$$
\beta_q(a,b)=\sum^k_{i=1}{\rm rank} H_q(\mathfrak{J}_b, \mathfrak{J}_a; {\bf K}).
$$
Furthermore, if  $\mathfrak{F}$ is coercive,  has only nondegenerate critical points,  and
  for each $q\in\{0\}\cup\mathbb{N}$ there exist only  finitely many critical points with Morse index $q$,
 then the following relations hold:
\begin{eqnarray}\label{e:Morse.8}
 \sum^q_{i=0}(-1)^{q-i}N_i\ge (-1)^q,\;q=0,1,2,\cdots,\quad\hbox{and}\quad
  \sum^\infty_{i=0}(-1)^iN_i=1,
 \end{eqnarray}
 where $N_i$ is the number of critical points of $\mathfrak{F}$ with Morse index $i$.
\end{theorem}

The proof of (\ref{e:Morse.8}) is standard, see the proof of \cite[Corollary~6.5.10]{Ber}.
 When $H=W^{m,2}_0(\Omega)$ and $\mathfrak{F}$ is
coercive, (\ref{e:Morse.7}) was first obtained by Skrypnik in \cite[\S5.2]{Skr1}.

\section{Bifurcations for Quasi-linear elliptic systems}\label{sec:BifE}
\setcounter{equation}{0}


\begin{hypothesis}\label{hyp:BifE.1}
{\rm Let $\Omega\subset\R^n$ be a bounded  Sobolev domain, $N\in\mathbb{N}$, and
  functions
$$
F:\overline\Omega\times\prod^m_{k=0}\mathbb{R}^{N\times M_0(k)}\to \R\quad\hbox{and}\quad
\textsf{G}:\overline\Omega\times\prod^{m-1}_{k=0}\mathbb{R}^{N\times M_0(k)}\to \R
$$
satisfy \textsf{Hypothesis} $\mathfrak{F}_{p,N}$ and (i)-(ii) in \textsf{Hypothesis} $\mathfrak{F}_{p,N}$,
respectively.  Let $V$ be a closed subspace of $W^{m,p}(\Omega,\mathbb{R}^N)$
containing $W^{m,p}_0(\Omega,\mathbb{R}^N)$.
}
\end{hypothesis}

 We consider
(generalized) bifurcation solutions of the boundary value problem
corresponding to the subspace $V$:
\begin{eqnarray}\label{e:BifE.1}
&&\sum_{|\alpha|\le m}(-1)^{|\alpha|}D^\alpha F^i_\alpha(x, \vec{u},\cdots, D^m\vec{u})=
\lambda\sum_{|\alpha|\le m-1}(-1)^{|\alpha|}D^\alpha \textsf{G}^i_\alpha(x, \vec{u},\cdots, D^{m-1}\vec{u}),\nonumber\\
&&\hspace{20mm}i=1,\cdots,N.
\end{eqnarray}
Call $\vec{u}\in V$ a {\it generalized solution} of (\ref{e:BifE.1}) if it is a
 critical point on $V$ of the variational problem
\begin{equation}\label{e:BifE.2}
\mathfrak{F}(\vec{u})-\lambda \mathfrak{G}(\vec{u})=\int_\Omega F(x, \vec{u},\cdots, D^m\vec{u})dx-\lambda
\int_\Omega \textsf{G}(x, \vec{u},\cdots, D^{m-1}\vec{u})dx.
\end{equation}

As a generalization of  \cite[Theorem~7.2, Chapter~4]{Skr3}, we may derive from  Theorem~\ref{th:3.2} and
Theorem~7.1 of \cite[Chapter~4]{Skr3}:

\begin{theorem}\label{th:BifE.2}
Under Hypothesis~\ref{hyp:BifE.1}, assume
\begin{description}
\item[(i)] the functionals $\mathfrak{F}$ and $\mathfrak{G}$ are even, $\mathfrak{F}(\theta)=\mathfrak{G}(\theta)=0$,
$\mathfrak{G}(\vec{u})\ne 0\;\forall \vec{u}\in V\setminus\{\theta\}$, and $\mathfrak{G}'(\vec{u})\ne \theta\;\forall \vec{u}\in V\setminus\{\theta\}$;
\item[(ii)] $\langle\mathfrak{F}'(\vec{u}), \vec{u}\rangle\ge\nu(\|\vec{u}\|_{m,p})$, where $\nu(t)$ is a continuous
function and positive for $t>0$;
\item[(iii)] $\mathfrak{F}(\vec{u})\to +\infty$ as $\|\vec{u}\|_{m,p}\to\infty$.
\end{description}
Then for any $c>0$ there exists at least a sequence $\{(\lambda_j,\vec{u}_j)\}_j\subset\mathbb{R}\times\{\vec{u}\in V\,|\,
\mathfrak{F}(\vec{u})=c\}$ satisfying (\ref{e:BifE.1}).
\end{theorem}

 By Theorems~\ref{th:Bi.2.2},~\ref{th:3.2} we have

\begin{theorem}\label{th:BifE.3}
Under Hypothesis~\ref{hyp:BifE.1} with $p=2$, let $\vec{u}_0\in V$ satisfy $\mathfrak{F}'(\vec{u}_0)=0$ and $\mathfrak{G}'(\vec{u}_0)=0$.
If  $(\lambda^\ast, \vec{u}_0)$ with certain $\lambda^\ast\in\mathbb{R}$
is a bifurcation point for (\ref{e:BifE.1}), then the linear problem
 \begin{eqnarray}\label{e:BifE.3}
 &&\sum^N_{i,j=1}\sum_{|\alpha|,|\beta|\le m}
(-1)^{\alpha}\bigl[F^{ij}_{\alpha\beta}(x, \vec{u}_0(x),\cdots, D^m \vec{u}_0(x))D^\beta v^j\bigr]
\nonumber\\
&&=\lambda \sum^N_{i,j=1}\sum_{|\alpha|,|\beta|\le m-1}
(-1)^{\alpha}\bigl[\textsf{G}^{ij}_{\alpha\beta}(x, \vec{u}_0(x),\cdots, D^{m-1}\vec{u}_0(x))D^\beta v^j\bigr]
 \end{eqnarray}
with $\lambda=\lambda^\ast$ has a nontrivial solution in $V$, i.e., $\vec{u}_0$ is a degenerate
critical point of the functional $\mathfrak{F}-\lambda^\ast\mathfrak{G}$ on $V$.
\end{theorem}

Conversely, if $\dim\Omega=1$ and $\vec{u}_0$ is a degenerate critical point of the functional $\mathfrak{F}-\lambda^\ast\mathfrak{G}$, using Theorem~\ref{th:Bi.2.3} we may obtain
the corresponding bifurcation results. These will be given in more general forms, see \cite{Lu9}.

\begin{hypothesis}\label{hyp:BifE.4}
{\rm Let Hypothesis~\ref{hyp:BifE.1} hold with $p=2$,  $\vec{u}_0\in V$ satisfy $\mathfrak{F}'(\vec{u}_0)=0$ and $\mathfrak{G}'(\vec{u}_0)=0$, and the linear problem
 \begin{eqnarray}\label{e:BifE.4}
 \sum^N_{i,j=1}\sum_{|\alpha|,|\beta|\le m}
(-1)^{\alpha}\bigl[F^{ij}_{\alpha\beta}(x, \vec{u}_0(x),\cdots, D^m \vec{u}_0(x))D^\beta v^j\bigr]
=0
 \end{eqnarray}
have no nontrivial solutions in $V$.}
\end{hypothesis}

The final condition in this hypothesis means that $\mathfrak{F}''(\vec{u})$ has
a bounded linear inverse.
Under Hypothesis~\ref{hyp:BifE.4}, by the arguments above Theorem~\ref{th:Bi.2.4},
the all eigenvalues of (\ref{e:BifE.3}) form a discrete subset of $\mathbb{R}$,
$\{\lambda_j\}^\infty_{j=1}$, which contains no zero and satisfies $|\lambda_j|\to\infty$
as $j\to\infty$; moreover, each $\lambda_j$ has finite multiplicity. Let $V_j$ be the eigensubspace of  (\ref{e:BifE.3}) corresponding to
the eigenvalue $\lambda_j$, $j=1,2,\cdots$.
By Theorems~\ref{th:Bi.2.4},~\ref{th:3.2} we directly obtain

\begin{theorem}\label{th:BifE.5}
Under Hypothesis~\ref{hyp:BifE.4}, for an eigenvalue $\lambda_k$
of (\ref{e:BifE.3}) as above, assume that one of the following three conditions holds:
\begin{description}
\item[(a)] $\mathfrak{F}''(\vec{u}_0)$ is positive definite, i.e., for each $v\in V\setminus\{\theta\}$,
 \begin{eqnarray}\label{e:BifE.5}
 \sum^N_{i,j=1}\sum_{|\alpha|,|\beta|\le m}
\int_\Omega F^{ij}_{\alpha\beta}(x, \vec{u}_0(x),\cdots, D^m \vec{u}_0(x))D^\beta v^j(x)D^\alpha v^i(x)dx>0;
 \end{eqnarray}
 \item[(b)] $\mathfrak{F}''(\vec{u}_0)$ is negative definite, i.e., for each $v\in V\setminus\{\theta\}$,
 \begin{eqnarray}\label{e:BifE.6}
 \sum^N_{i,j=1}\sum_{|\alpha|,|\beta|\le m}
\int_\Omega F^{ij}_{\alpha\beta}(x, \vec{u}_0(x),\cdots, D^m \vec{u}_0(x))D^\beta v^j(x)D^\alpha v^i(x)dx<0;
 \end{eqnarray}

\item[(c)] each $V_j$ is an invariant subspace of $\mathfrak{F}''(\vec{u}_0)$ in $V$,
$j=1,2,\cdots$,  and either (\ref{e:BifE.5}) holds for all $v\in V_k\setminus\{\theta\}$,
or (\ref{e:BifE.6}) does for all $v\in V_k\setminus\{\theta\}$.
\end{description}
 Then $(\lambda_k, \vec{u}_0)\in\mathbb{R}\times V$ is a bifurcation point  for the problem
(\ref{e:BifE.1}),  and  one of the following alternatives occurs:
\begin{description}
\item[(i)] $(\lambda_k, \vec{u}_0)$ is not an isolated solution of (\ref{e:BifE.1}) in
 $\{\lambda_k\}\times V$.

\item[(ii)] there exists a sequence $\{\kappa_j\}_{j\ge 1}\subset\mathbb{R}\setminus\{\lambda_k\}$
such that $\kappa_j\to\lambda_k$ and that for each $\kappa_j$ the problem
(\ref{e:BifE.1}) with $\lambda=\kappa_j$ has infinitely many solutions converging to
$\vec{u}_0\in V$.

\item[(iii)]  for every $\lambda$ in a small neighborhood of $\lambda_k$ there is a nontrivial solution $\vec{u}_\lambda$ of (\ref{e:BifE.1}) converging to $\vec{u}_0$ as $\lambda\to\lambda_k$;

\item[(iv)] there is a one-sided $\Lambda$ neighborhood of $\lambda_k$ such that
for any $\lambda\in\Lambda\setminus\{\lambda_k\}$,
(\ref{e:BifE.1}) has at least two nontrivial solutions converging to
$\vec{u}_0$ as $\lambda\to\lambda_k$.
\end{description}
\end{theorem}

\begin{remark}\label{rmk:BifE.6}
{\rm (i) When $N=1$, $V=W_0^{m,2}(\Omega)$, ${u}=\theta$ and $\mathfrak{F}$
also satisfies
\begin{eqnarray}\label{e:BifE.7}
(\mathfrak{F}'({u}),{u})_{m,2}\ge c\|u\|^2_{m,2}
\end{eqnarray}
for some $c>0$ and all sufficiently small $\|{u}\|_{m,2}$,
if $\lambda^\ast$ is  an eigenvalue of (\ref{e:BifE.3}) with ${u}=\theta$,
it was proved in \cite[Chap.1, Theorem~3.5]{Skr2} that
$(\lambda,\theta)$ is a bifurcation point of (\ref{e:BifE.1}).
Since $\mathfrak{F}(\theta)=\theta$, it is clear that (\ref{e:BifE.7})
implies $\mathfrak{F}''(\theta)$ to be positive definite. Hence
Theorem~\ref{th:BifE.5} contains \cite[Chap.1, Theorem~3.5]{Skr2}
as a special example.\\
(ii) When $N=1$, $n\ge 3$, $V=H^1_0(\Omega)$, $\textsf{G}(x,\xi_0,\cdots,\xi_n)=
\frac{1}{2}\xi_0^2$
and
\begin{eqnarray}\label{e:BifE.8}
F(x,\xi_0,\cdots,\xi_n)=\frac{1}{2}\sum^n_{i,j=1}a_{ij}(x,\xi_0)\xi_i\xi_j-\int^{\xi_0}_0g(x,t)dt,
\end{eqnarray}
Canino \cite[Theorem~1.3]{Can} obtained a corresponding result
provided that functions $a_{ij}=a_{ji}, g:\Omega\times\mathbb{R}\to\mathbb{R}$
satisfy the following assumptions:\\
a.0) $a_{ij}$ is of class $C^1$, and is of class $C^2$ in $\xi_0$ for $a.e.\;x\in\Omega$;\\
a.1) there exists $C>0$ such that for $a.e.\;x\in\Omega$, for all $\xi_0\in\mathbb{R}$ and
for all $i,j,k$,
\begin{eqnarray*}
&&|a_{ij}(x,\xi_0)|\le C,\quad |D_{\xi_0}a_{ij}(x,\xi_0)|\le C,\\
&&|D_{x_k}a_{ij}(x,\xi_0)|\le C,\quad |D^2_{\xi_0\xi_0}a_{ij}(x,\xi_0)|\le C;
\end{eqnarray*}
a.2) there exists $\nu>0$ such that for $a.e.\;x\in\Omega$, for all $\xi_i\in\mathbb{R}$,
$i=0,\cdots,n$,
$$
\sum^n_{i,j=1}a_{ij}(x,\xi_0)\xi_i\xi_j\ge\nu\sum^n_{\i=1}\xi_i^2;
$$
a.3) for $a.e.\;x\in\Omega$, for all $\xi_i\in\mathbb{R}$,
$i=0,\cdots,n$,
$$
\sum^n_{i,j=1}sD_{\xi_0}a_{ij}(x,\xi_0)\xi_i\xi_j\ge 0;
$$
g) for every $\xi_0\in\mathbb{R}$, $g(x,\xi_0)$ is measurable with respect to $x$,
for a.e. $x\in\Omega$, $g(x,\xi_0)$ is of class $C^1$ with respect to $s$,
 $g(x,0) = 0$; moreover,  there exist $b\in\mathbb{R}$ and $0 < p <
4/(n-2)$ such that, for a.e. $x\in\Omega$ and all $s\in\mathbb{R}$,
$$
|D_{\xi_0}g(x,\xi_0)| \le b(1 + |\xi_0|^p ).
$$
It is easily checked that $F$ in (\ref{e:BifE.8}) satisfies \textsf{Hypothesis} $\mathfrak{F}_{2,1}$.
Thus \cite[Theorem~1.3]{Can} is implied in Theorem~\ref{th:BifE.5} with $N=1$ and $\vec{u}=\theta$.
}
\end{remark}


By Theorems~\ref{th:Bi.3.1},~\ref{th:Bi.3.2} we deduce

\begin{theorem}\label{th:BifE.7}
Under Hypothesis~\ref{hyp:BifE.1} with $p=2$, let $G$ be a compact Lie group acting on $V$ in a $C^3$-smooth
and isometric (and so orthogonal) way.  Suppose that both $\mathfrak{F}$ and $\mathfrak{G}$  are $G$-invariant, and that $\vec{u}_0\in {\rm Fix}(G)$ satisfies $\mathfrak{F}'(\vec{u}_0)=0$ and $\mathfrak{G}'(\vec{u}_0)=0$. For an eigenvalue $\lambda_k$
of (\ref{e:BifE.3}) as above, assume that one of the three conditions (a),(b) and (c) in
Theorem~\ref{th:BifE.5} holds. Then
$(\lambda_k,\vec{u}_0)\in\mathbb{R}\times V$ is a bifurcation point  for the equation
(\ref{e:Bi.2.7.3}),  and if $\dim V_k\ge 2$ and the unit sphere in
$V_k$ is not a $G$-orbit we must get one of the following alternatives:
\begin{description}
\item[(i)] $(\lambda_k, \vec{u}_0)$ is not an isolated solution of (\ref{e:BifE.1}) in
 $\{\lambda_k\}\times V$;

\item[(ii)] there exists a sequence $\{\kappa_j\}_{j\ge 1}\subset\mathbb{R}\setminus\{\lambda_k\}$
such that $\kappa_j\to\lambda_k$ and that for each $\kappa_j$ the problem
(\ref{e:BifE.1}) with $\lambda=\kappa_j$ has infinitely many $G$-orbits of solutions converging to
$\vec{u}_0\in V$;

\item[(iii)]  for every $\lambda$ in a small neighborhood of $\lambda_k$ there is a nontrivial solution $\vec{u}_\lambda$ of (\ref{e:BifE.1}) converging to $\vec{u}_0$ as $\lambda\to\lambda_k$;

\item[(iv)] there is a one-sided $\Lambda$ neighborhood of $\lambda_k$ such that
for any $\lambda\in\Lambda\setminus\{\lambda_k\}$,
(\ref{e:BifE.1}) has at least two nontrivial critical orbits converging to
$\vec{u}_0$ as $\lambda\to\lambda_k$.
\end{description}
Furthermore, if the Lie group  $G$ is equal to $\mathbb{Z}_2$ or $S^1$, then
the above (iii)-(iv) can be replaced by the following
\begin{description}
\item[(iii')] there exist left and right  neighborhoods $\Lambda^-$ and $\Lambda^+$ of $\lambda_k$ in $\mathbb{R}$
and integers $n^+, n^-\ge 0$, such that $n^++n^-\ge\dim V$
and for $\lambda\in\Lambda^-\setminus\{\lambda^\ast\}$ (resp. $\lambda\in\Lambda^+\setminus\{\lambda^\ast\}$),
(\ref{e:BifE.1}) has at least $n^-$ (resp. $n^+$) distinct critical
$G$-orbits different from $\vec{u}_0$, which converge to
 $\vec{u}_0$ as $\lambda\to\lambda_k$.
\end{description}
The corresponding claims to 3) of Theorem~\ref{th:Bi.3.1} can be easily written.
\end{theorem}

If $\vec{u}_0\notin {\rm Fix}(G)$,  Theorem~\ref{th:Bi.3.20} may yield a result.
If $n=\dim\Omega=1$, by Theorems~\ref{th:Bi.3.14},~\ref{th:Bi.3.16},\ref{th:Bi.3.17}
we can also obtain more results. They will be given in \cite{Lu7,Lu9}.

\subsection*{Example}\label{ex:BifE.7}

 For $0<T_i<\infty$, $i=1,\cdots,n$, and $\Omega:=\prod^n_{i=1}(0, T_i)$ let
$C^{m}(\prod^n_{i=1}\mathbb{R}/(T_i\mathbb{Z}),\mathbb{R}^N)$
be the set of all $\vec{u}\in C^m(\mathbb{R}^n,\mathbb{R}^N)$ which are $T_i$-periodic
with respect to the $i$-th variable, $i=1,\cdots,n$. It may be viewed as a subspace
of $W^{m,2}(\Omega,\mathbb{R}^N)$.   Denote by
$W^{m,2}(\prod^n_{i=1}\mathbb{R}/(T_i\mathbb{Z}),\mathbb{R}^N)$
the closure of $C^{m}(\prod^n_{i=1}\mathbb{R}/(T_i\mathbb{Z}),\mathbb{R}^N)$ in
$W^{m,2}(\Omega,\mathbb{R}^N)$. Clearly, $W^{m,2}_0(\Omega, \mathbb{R}^N)$ is contained
in $W^{m,2}(\prod^n_{i=1}\mathbb{R}/(T_i\mathbb{Z}),\mathbb{R}^N)$.
The compact connected Lie group $G=\prod^n_{i=1}\mathbb{R}/(T_i\mathbb{Z})$,
which is isomorphic to $\mathbb{T}^n=\mathbb{R}^n/\mathbb{Z}^n$, acts on
$W^{m,2}(\prod^n_{i=1}\mathbb{R}/(T_i\mathbb{Z}),\mathbb{R}^N)$ via the following
isometric linear representation:
\begin{equation}\label{e:T^n-action}
([t_1,\cdots,t_n]\cdot\vec{u})(x_1,\cdots,x_n)=(u_1(x_1+t_1),\cdots, u_n(x_n+t_n))
\end{equation}
for $[t_1,\cdots,t_n]\in G$ and $\vec{u}=(u_1,\cdots,u_n)\in W^{m,2}(\prod^n_{i=1}\mathbb{R}/(T_i\mathbb{Z}),\mathbb{R}^N)$.
The set of fixed points of this action,
${\rm Fix}(G)$, consist of all constant vector functions from $\overline{\Omega}$ to $\mathbb{R}^N$.
Under Hypothesis~\ref{hyp:BifE.1} with $\Omega=\prod^n_{i=1}(0, T_i)\subset\mathbb{R}^n$,  assume also
that  $F(x, \xi)$ and $\textsf{G}(x, \xi)$ satisfy
\begin{eqnarray*}
&&F(x_1,\cdots,x_{\check{i}},\cdots, x_n, \xi)=F(x_1,\cdots,x_{\hat{i}},\cdots, x_n,\xi),\\
&&\textsf{G}(x_1,\cdots,x_{\check{i}},\cdots, x_n, \xi)=\textsf{G}(x_1,\cdots,x_{\hat{i}},\cdots, x_n,\xi),
\end{eqnarray*}
where $x_{\check{i}}=0$ and $x_{\hat{i}}=T_i$, $i=1,\cdots,n$. In other words, $F$ and $\textsf{G}$
may be viewed as functions on $\mathbb{R}^n\times\prod^m_{k=0}\mathbb{R}^{N\times M_0(k)}$ with
period $T_i$ in variables $x_i$, $i=1,\cdots,n$. Then the functionals $\mathfrak{F}$ and $\mathfrak{G}$ are $G$-invariant,
and every critical orbit different from points in ${\rm Fix}(G)$ must be homeomorphic to some $T^s$, $1\le s\le n$.
Clearly, if some $\vec{u}\in W^{m,2}(\prod^n_{i=1}\mathbb{R}/(T_i\mathbb{Z}),\mathbb{R}^N)$
are constant with respect to variables $x_{i_r}$, $r=1,\cdots,k<n$, but not
with respect to any other variable $x_i$, then the orbit
$G(\vec{u})$ of $\vec{u}$ is homeomorphic to some $T^k$.
However, it is possible that the orbit of $\vec{u}$ is of $T^k$-type even if
$\vec{u}$ is not constant with respect to each variable, see
the proof of \cite[Proposition~3.4]{Van}.

Clearly,  Hypothesis~\ref{hyp:BifE.4} is satisfied for each constant map
$\vec{u}_0: \overline{\Omega}\to\mathbb{R}^N$. So Theorem~\ref{th:BifE.5},
and Theorem~\ref{th:BifE.7} with $G=T^n$ can be directly applied.
By Theorem~\ref{th:Bi.3.1} we also get

\begin{theorem}\label{th:BifE.8}
Let $\vec{u}_0: \overline{\Omega}\to\mathbb{R}^N$ be a constant map. Assume  that $F$ and $\textsf{G}$ satisfy
the conditions of Theorem~\ref{th:BifE.5} with $V=W^{m,2}(\prod^n_{i=1}\mathbb{R}/(T_i\mathbb{Z}),\mathbb{R}^N)$.
Then $(\lambda_k,\vec{u}_0)\in\mathbb{R}\times V$ is a bifurcation point  for the equation
(\ref{e:BifE.1}). Moreover,  if  $V_k$ (the eigensubspace of  (\ref{e:BifE.3}) corresponding to
the eigenvalue $\lambda_k$) satisfies one of the following assumptions: A)
 ${\rm Fix}(T^n)\cap  V_k=\{\theta\}$, B) every orbit in $V_k$ is homeomorphic to some $T^s$ for $s\ge 2$;
then either one of the above (i)-(iii) in Theorem~\ref{th:BifE.7} with $G=T^n$ or the following hold:\\
{\bf (iv)'} there is a one-sided $\Lambda$ neighborhood of $\lambda_k$ such that
for any $\lambda\in\Lambda\setminus\{\lambda_k\}$,
the problem (\ref{e:BifE.1}) with $\lambda=\lambda_k$  has at least $\dim V_k$ (resp. $2\dim V_k$) nontrivial critical orbits
in the case A) (resp. B)), where every orbit is counted with
its multiplicity (defined by \cite[Definition~1.3]{Wa1}).
\end{theorem}

The theory in this paper provides necessary tools for generalizing \cite{Lu1, Van}
to the functional considered in the above example. They will be investigated
in the latter paper.

Finally, we present a result associated with Theorems~5.4.2,~5.7.4 in \cite{EKBB}.

\begin{theorem}\label{th:BifE.9}
 Let $\Omega\subset\R^n$ be a bounded  Sobolev domain, $N\in\mathbb{N}$. Suppose that
 $$
\overline\Omega\times\prod^m_{k=0}\mathbb{R}^{N\times M_0(k)}\times [0, 1]\ni (x,
\xi,\lambda)\mapsto F(x,\xi;\lambda)\in\R
$$
is differentiable with respect to $\lambda$, and satisfies the following conditions:
 \begin{description}
 \item[(i)] All $F(\cdot;\lambda)$ satisfy Hypothesis~$\mathfrak{F}_{2,N}$ uniformly with respect to $\lambda\in [0,1]$, i.e., the inequalities (\ref{e:1.1}) and (\ref{e:1.2}) are uniformly satisfied for all $\lambda\in [0,1]$.
 \item[(ii)] If $q_\alpha=1$ for $|\alpha|<m-n/2$, and
 $q_\alpha=2_\alpha/(2_\alpha-1)$ for $m-n/2\le|\alpha|\le m$,
 then
 $$
\sup_{|\alpha|\le m}\sup_{1\le i\le N}\sup_\lambda\int_\Omega \left[|D_\lambda F(x,0;\lambda)|+ |D_\lambda F^i_{\alpha}(x, 0;\lambda)|^{q_\alpha}\right]dx<\infty.
$$
 \item[(iii)] For all $i=1,\cdots,N$ and $|\alpha|\le m$,
 \begin{eqnarray*}
&&|D_\lambda F^i_\alpha(x,\xi;\lambda)|\le|D_\lambda F^i_\alpha(x,0;\lambda)|\\
&&+ \mathfrak{g}(\sum^N_{k=1}|\xi_0^k|)\sum_{|\beta|<m-n/2}\bigg(1+
\sum^N_{k=1}\sum_{m-n/2\le |\gamma|\le
m}|\xi^k_\gamma|^{2_\gamma }\bigg)^{2_{\alpha\beta}}\nonumber\\
&&+\mathfrak{g}(\sum^N_{k=1}|\xi^k_0|)\sum^N_{l=1}\sum_{m-n/2\le |\beta|\le m} \bigg(1+
\sum^N_{k=1}\sum_{m-n/2\le |\gamma|\le m}|\xi^k_\gamma|^{2_\gamma }\bigg)^{2_{\alpha\beta}}|\xi^l_\beta|;
\end{eqnarray*}
 where $\mathfrak{g}:[0,\infty)\to\mathbb{R}$ is a continuous, positive, nondecreasing function.
 \end{description}
  Let $V$ be a closed subspace of $W^{m,2}(\Omega,\mathbb{R}^N)$, and for each $\lambda\in [0,1]$
  let $\vec{u}_\lambda$ be a critical point of the functional
$$
\mathfrak{F}_\lambda(\vec{u})=\int_\Omega F(x, \vec{u},\cdots, D^m\vec{u};\lambda)dx
$$
on $V$. Suppose that $[0, 1]\ni\lambda\mapsto \vec{u}_\lambda\in V$ is continuous. Then
one of the following alternatives occurs:
\begin{description}
\item[(I)] There exists certain $\lambda_0\in [0,1]$ such that $(\lambda_0, \vec{u}_{\lambda_0})$
           is a bifurcation point of $\nabla\mathfrak{F}_\lambda(\vec{u})=0$.
\item[(II)] Each $\vec{u}_\lambda$ is an isolated critical point of $\mathfrak{F}_\lambda$ and
$C_\ast(\mathfrak{F}_\lambda, \vec{u}_\lambda;{\bf K})=C_\ast(\mathfrak{F}_0, \vec{u}_0;{\bf K})$
for all $\lambda\in [0,1]$; moreover $\vec{u}_\lambda$ is a local minimizer of $\mathfrak{F}_\lambda$ if and only if $\vec{u}_0$ is a local minimizer of $\mathfrak{F}_0$.
 \end{description}
\end{theorem}

If $D_\lambda F(\cdot;\lambda)$ uniformly satisfy  the inequalities (\ref{e:1.1}) and (\ref{e:1.2})  for all $\lambda\in [0,1]$. Then these and (ii) can yield (iii).

\begin{proof} Suppose that (I) does not hold. Then each $\vec{u}_\lambda$ is an isolated critical point of $\mathfrak{F}_\lambda$. Since $[0, 1]\ni\lambda\mapsto \vec{u}_\lambda\in V$ is continuous,
we may find a bounded open subset $\mathscr{O}$ in $V$ such that
$\vec{u}_\lambda$ is a unique  critical point of $\mathfrak{F}_\lambda$
contained in the closure $\overline{\mathscr{O}}$ of $\mathscr{O}$.
Take $R>0$ such that $\overline{\mathscr{O}}\subset B_V(\theta, R)$.

As in the proof of (\ref{e:3.7}) we may derive from (iii) that with $q_\alpha$ in (ii),
\begin{eqnarray*}
&&|D_\lambda F(x,\xi;\lambda)|\le |D_\lambda F(x,0;\lambda)|+
\Big(\sum^N_{k=1}|\xi^k_0|\Big)\sum^N_{i=1}\sum_{|\alpha|<m-n/2}|D_\lambda F^i_{\alpha}(x, 0;\lambda)|\\
&&+\sum^N_{i=1}\sum_{m-n/2\le|\alpha|\le m}|D_\lambda F^i_{\alpha}(x, 0;\lambda)|^{q_\alpha}
+ \widehat{\mathfrak{g}}(\sum^N_{k=1}|\xi^k_0|)\bigg(1+\sum^N_{l=1}\sum_{m-n/2\le|\alpha|\le m}|\xi^l_\alpha|^{2_\alpha}\bigg)
\end{eqnarray*}
for all $(x,\xi,\lambda)$ and some  continuous, positive, nondecreasing function
$\widehat{\mathfrak{g}}:[0,\infty)\to\mathbb{R}$. As before we have a constant
$C=C(m,n,N, R)>0$ such that
$$
\sup\bigg\{\sum_{|\alpha|<m-n/2}|D^\alpha\vec{u}(x)|\,|\, x\in\Omega\bigg\}<C\quad
\hbox{for all $\vec{u}\in W^{m,2}(\Omega,\mathbb{R}^N)$ with $\|\vec{u}\|_{m,2}\le R$}.
$$
It follows that for any $\lambda_i\in [0,1]$, $i=1,2$,
\begin{eqnarray*}
&&|\mathfrak{F}_{\lambda_1}(\vec{u})- \mathfrak{F}_{\lambda_2}(\vec{u})|\le |\lambda_2-\lambda_1|\int_\Omega\sup_\lambda
|D_\lambda F(x, \vec{u},\cdots, D^m\vec{u};\lambda)|dx\\
&&\le |\lambda_2-\lambda_1|\biggl[\sup_\lambda\int_\Omega |D_\lambda F(x,0;\lambda)|dx +
C\sum^N_{i=1}\sum_{|\alpha|<m-n/2}\sup_\lambda\int_\Omega |D_\lambda F^i_{\alpha}(x, 0;\lambda)|dx\\
&&+\sum^N_{i=1}\sum_{m-n/2\le|\alpha|\le m}\sup_\lambda\int_\Omega
|D_\lambda F^i_{\alpha}(x, 0;\lambda)|^{q_\alpha}dx\\
&&+ \widehat{\mathfrak{g}}(C)\int_\Omega\bigg(1+\sum^N_{l=1}\sum_{m-n/2\le|\alpha|\le m}|D^\alpha u^l|^{2_\alpha}\bigg)dx
\biggr].
\end{eqnarray*}
This implies that $[0,1]\ni\lambda\mapsto \mathfrak{F}_\lambda$
is continuous in $C^0(\bar{B}_V(\theta, R))$. Similarly,
 (ii)--(iii) yield the continuity of the map $[0,1]\ni\lambda\mapsto \nabla\mathfrak{F}_\lambda$
 in $C^0(\bar{B}_V(\theta, R), V)$. Hence the map
 $[0,1]\ni\lambda\mapsto \mathfrak{F}_\lambda$
is continuous in $C^1(\bar{B}_V(\theta, R))$.
As in Step~1 of the proof of Theorem~\ref{th:Bi.1.1},
the stability of critical groups (cf.  \cite[Theorem~III.4]{ChGh} and \cite[Theorem~5.1]{CorH})
leads to the first claim in (II).

 For the second claim, it suffices to prove that
 $\vec{u}_0$ is a local minimizer of  $\mathfrak{F}_0$
 provided $\vec{u}_\lambda$ is a local minimizer of $\mathfrak{F}_\lambda$.
 Since  $\vec{u}_\lambda$ is an isolated critical point of $\mathfrak{F}_\lambda$,
 by Example~1 in \cite[page 33]{Ch} we have $C_q(\mathfrak{F}_\lambda, \vec{u}_\lambda;{\bf K})=\delta_{q0}{\bf K}$
 for $q=0,1,\cdots$.   It follows that
 $C_q(\mathfrak{F}_0, \vec{u}_0;{\bf K})=\delta_{q0}{\bf K}$ for $q=0,1,\cdots$.
  By Theorem~\ref{th:S.1.3}, this means that the Morse index of $\mathfrak{F}_0$ at $\vec{u}_0$
 must be zero. We can assume $\vec{u}_0=\theta$ after replaceing $\mathfrak{F}_0$
 by $\mathfrak{F}_0(\vec{u}_0+\cdot)$.  So $C_q(\mathfrak{F}^\circ_0, \theta;{\bf K})=\delta_{q0}{\bf K}$
 for $q=0,1,\cdots$. Then $\theta$ is a local minimizer of
 $\mathfrak{F}^\circ_0$  by Example~4 in \cite[page 43]{Ch}.
 It follows from Theorem~\ref{th:S.1.2} (or Theorem~\ref{th:Morse.1} with $\mathcal{O}=\theta$) that
 $\vec{u}_0=\theta$ must be  a local minimizer of  $\mathfrak{F}_0$.
 \end{proof}

\section{Concluding remarks}\label{sec:CR}
\setcounter{equation}{0}

 In Section~\ref{sec:Bi} we only generalize some bifurcation theorems
 for potential operators with the splitting theorem obtained in this paper.
 Once some splitting theorems are proved,  the same ideas can be used
to generalize some past  bifurcation theorems.
 For example, we may obtain corresponding extended versions of \cite{BaCl,Ba1}
in the variational frames of  \cite{Lu1,Lu2} and \cite{BoBu, JM}.
 These and applications will be given in \cite{Lu7}. 

We here do not consider easy generalizations of the contents in  Part~\ref{part:2} to
a larger framework as in  \cite{PaSm,Sma,Pa2} because they are developed in a more general setting as in \cite{KoKu1, KoKu2},
see \cite{Lu8}.  Moreover, both the theory in Part I and that of \cite{Lu1,Lu2} are  applicable
to one-dimensional variational problem of higher order, see \cite{Lu9}.


\appendix
\section{Proof of Proposition~\ref{prop:3.4}}\label{app:A}\setcounter{equation}{0}

Recall that we have written  $\xi\in\R^{M(m)}$ as  $\xi=\{\xi_\alpha:\,|\alpha|\le m\}$
and denote by $\xi_\circ=\{\xi_\alpha:\,|\alpha|<m-n/p\}$.
By the mean value theorem and (\ref{e:1.5}) we get
a collect of numbers $\{t_\beta\in (0,1):\,|\beta|\le m\}$ such that
\begin{eqnarray*}
&&|f_\alpha(x,\xi)|-|f_\alpha(x,0)|\le \sum_{|\beta|\le m}|f_{\alpha\beta}(x, t_\beta\xi)|\cdot|\xi_\beta|\\
&\le& \sum_{|\beta|\le m} \mathfrak{g}_1(|t_\beta\xi_\circ|)\Bigg(1+
\sum_{m-n/p\le |\gamma|\le
m}|t_\beta\xi_\gamma|^{p_\gamma}\Bigg)^{p_{\alpha\beta}}|\xi_\beta|\\
&\le& \sum_{|\beta|\le m} \mathfrak{g}_1(|\xi_\circ|)\Bigg(1+
\sum_{m-n/p\le |\gamma|\le
m}|\xi_\gamma|^{p_\gamma}\Bigg)^{p_{\alpha\beta}}|\xi_\beta|\\
&\le& \sum_{|\beta|\le m} \mathfrak{g}_1(|\xi_\circ|)\Bigg(1+
\sum_{m-n/p\le |\gamma|\le
m}|\xi_\gamma|^{p_\gamma }\Bigg)^{p_{\alpha\beta}}|\xi_\beta|
\end{eqnarray*}
\begin{eqnarray*}
&=&\sum_{|\beta|<m-n/p} \mathfrak{g}_1(|\xi_\circ|)\Bigg(1+
\sum_{m-n/p\le |\gamma|\le
m}|\xi_\gamma|^{p_\gamma }\Bigg)^{p_{\alpha\beta}}|\xi_\beta|\\
&&+\sum_{m-n/p\le |\beta|\le m} \mathfrak{g}_1(|\xi_\circ|)\Bigg(1+
\sum_{m-n/p\le |\gamma|\le
m}|\xi_\gamma|^{p_\gamma }\Bigg)^{p_{\alpha\beta}}|\xi_\beta|\\
&\le& \mathfrak{g}_1(|\xi_\circ|)|\xi_\circ|\sum_{|\beta|<m-n/p}\Bigg(1+
\sum_{m-n/p\le |\gamma|\le
m}|\xi_\gamma|^{p_\gamma }\Bigg)^{p_{\alpha\beta}}\\
&&+\mathfrak{g}_1(|\xi_\circ|)\sum_{m-n/p\le |\beta|\le m} \Bigg(1+
\sum_{m-n/p\le |\gamma|\le
m}|\xi_\gamma|^{p_\gamma }\Bigg)^{p_{\alpha\beta}}|\xi_\beta|.
\end{eqnarray*}
It follows that
\begin{eqnarray}\label{e:A.1}
|f_\alpha(x,\xi)|&\le&|f_\alpha(x,0)|
+ \mathfrak{g}_1(|\xi_\circ|)|\xi_\circ|\sum_{|\beta|<m-n/p}\Bigg(1+
\sum_{m-n/p\le |\gamma|\le
m}|\xi_\gamma|^{p_\gamma }\Bigg)^{p_{\alpha\beta}}\nonumber\\
&&+\mathfrak{g}_1(|\xi_\circ|)\sum_{m-n/p\le |\beta|\le m} \Bigg(1+
\sum_{m-n/p\le |\gamma|\le
m}|\xi_\gamma|^{p_\gamma }\Bigg)^{p_{\alpha\beta}}|\xi_\beta|,
\end{eqnarray}
which lead to (\ref{e:3.8}) with $\mathfrak{g}_4(|\xi_\circ|):=\mathfrak{g}_1(|\xi_\circ|)|\xi_\circ|+\mathfrak{g}_1(|\xi_\circ|)$.

Suppose $|\alpha|<m-n/p$. Then $p_{\alpha\beta}=1-1/p_\beta=1/q_\beta$
if $m-n/p\le |\beta|\le m$, and hence the second and third terms in (\ref{e:A.1}), respectively, becomes
\begin{eqnarray}
&& \mathfrak{g}_1(|\xi_\circ|)|\xi_\circ|\sum_{|\beta|<m-n/p}\Bigg(1+
\sum_{m-n/p\le |\gamma|\le
m}|\xi_\gamma|^{p_\gamma }\Bigg)^{1/q_{\beta}}\nonumber\\
&\le&\mathfrak{g}_1(|\xi_\circ|)|\xi_\circ|M(m)\Bigg(1+
\sum_{m-n/p\le |\gamma|\le
m}|\xi_\gamma|^{p_\gamma }\Bigg),\label{e:A.2}\\
&& \mathfrak{g}_1(|\xi_\circ|)\sum_{|\beta|<m-n/p}\Bigg(1+
\sum_{m-n/p\le |\gamma|\le
m}|\xi_\gamma|^{p_\gamma }\Bigg)^{1/q_{\beta}}|\xi_\beta|\nonumber\\
&\le&\mathfrak{g}_1(|\xi_\circ|)\sum_{m-n/p\le |\beta|\le m}\Bigg[ \bigg(1+
\sum_{m-n/p\le |\gamma|\le
m}|\xi_\gamma|^{p_\gamma }\bigg)+ |\xi_\beta|^{p_{\beta}}\Biggr]\nonumber\\
&\le&\mathfrak{g}_1(|\xi_\circ|)(M(m)+1) \Bigg(1+
\sum_{m-n/p\le |\gamma|\le
m}|\xi_\gamma|^{p_\gamma }\Bigg).\label{e:A.3}
\end{eqnarray}
 These and (\ref{e:A.1})-(\ref{e:A.2}) give rise to
(\ref{e:3.9}) with  $\mathfrak{g}_5(|\xi_0|):=(M(m)+1)\mathfrak{g}_1(|\xi_\circ|)(|\xi_\circ|+1)$.

Suppose $m-n/p\le|\alpha|\le m$. Then
$0<p_{\alpha\beta}\le 1-1/p_\alpha-1/p_\beta$
if $m-n/p\le |\beta|\le m$.
In this case the second and third terms in (\ref{e:A.1}), respectively, becomes
\begin{eqnarray}\label{e:A.4}
&& \mathfrak{g}_1(|\xi_\circ|)|\xi_\circ|\sum_{|\beta|<m-n/p}\Bigg(1+
\sum_{m-n/p\le |\gamma|\le
m}|\xi_\gamma|^{p_\gamma }\Bigg)^{p_{\alpha\beta}}\nonumber\\
&\le& \mathfrak{g}_1(|\xi_\circ|)|\xi_\circ|\sum_{|\beta|<m-n/p}\Bigg(1+
\sum_{m-n/p\le |\gamma|\le
m}|\xi_\gamma|^{p_\gamma }\Bigg)^{1/q_{\alpha}}\nonumber\\
&\le& \mathfrak{g}_1(|\xi_\circ|)|\xi_\circ|M(m)\Bigg(1+
\sum_{m-n/p\le |\gamma|\le
m}|\xi_\gamma|^{p_\gamma/q_{\alpha}}\Bigg)
\end{eqnarray}
and
\begin{eqnarray}\label{e:A.5}
&& \mathfrak{g}_1(|\xi_\circ|)\sum_{|\beta|<m-n/p}\Bigg(1+
\sum_{m-n/p\le |\gamma|\le
m}|\xi_\gamma|^{p_\gamma }\Bigg)^{p_{\alpha\beta}}|\xi_\beta|\nonumber\\
&\le& \mathfrak{g}_1(|\xi_\circ|)\sum_{|\beta|<m-n/p}\Bigg(1+
\sum_{m-n/p\le |\gamma|\le
m}|\xi_\gamma|^{p_\gamma }\Bigg)^{1/q_{\alpha}-1/p_{\beta}}|\xi_\beta|\nonumber\\
&\le& \mathfrak{g}_1(|\xi_\circ|)\sum_{|\beta|<m-n/p}\Biggl[\Bigg(\bigg(1+
\sum_{m-n/p\le |\gamma|\le
m}|\xi_\gamma|^{p_\gamma }\bigg)^{1/q_{\alpha}-1/p_{\beta}}\Bigg)^{p_\beta/(p_\beta-q_\alpha)}
+|\xi_\beta|^{p_\beta/q_\alpha}\Biggr]\nonumber\\
&\le& \mathfrak{g}_1(|\xi_\circ|)\sum_{|\beta|<m-n/p}\Biggl[\bigg(1+
\sum_{m-n/p\le |\gamma|\le
m}|\xi_\gamma|^{p_\gamma }\bigg)^{1/q_{\alpha}}
+|\xi_\beta|^{p_\beta/q_\alpha}\Biggr]\nonumber\\
&\le& \mathfrak{g}_1(|\xi_\circ|)(M(m)+1)\Bigg(1+
\sum_{m-n/p\le |\gamma|\le
m}|\xi_\gamma|^{p_\gamma }\Bigg)^{1/q_{\alpha}}.
\end{eqnarray}
 These lead to (\ref{e:3.10}).

With a similar argument to (\ref{e:A.1}) we obtain
\begin{eqnarray*}
&&|f(x,\xi)|-|f(x,0)|\le \sum_{|\alpha|\le m}|f_{\alpha}(x, s_\alpha\xi)|\cdot|\xi_\alpha|
\le\sum_{|\alpha|\le m}|f_{\alpha}(x, 0)|\cdot|\xi_\alpha|\\
&&+ \sum_{|\alpha|\le m}\mathfrak{g}_1(|s_\alpha\xi_\circ|)|s_\alpha\xi_\circ| \sum_{|\beta|<m-n/p}
\Bigg(1+\sum_{m-n/p\le |\gamma|\le m}|s_\alpha\xi_\gamma|^{p_\gamma }\Bigg)^{p_{\alpha\beta}}|\xi_\alpha|\\
&&+\sum_{|\alpha|\le m}\mathfrak{g}_1(|s_\alpha\xi_\circ|) \sum_{m-n/p\le |\beta|\le m}
\Bigg(1+\sum_{m-n/p\le |\gamma|\le m}|s_\alpha\xi_\gamma|^{p_\gamma }\Bigg)^{p_{\alpha\beta}}|s_\alpha\xi_\beta|\cdot|\xi_\alpha|\\
&\le&\sum_{|\alpha|\le m}|f_{\alpha}(x, 0)|\cdot|\xi_\alpha|\\
&&+ \mathfrak{g}_1(|\xi_\circ|)|\xi_\circ| \sum_{|\alpha|\le m}\sum_{|\beta|<m-n/p}
\Bigg(1+\sum_{m-n/p\le |\gamma|\le
m}|\xi_\gamma|^{p_\gamma }\Bigg)^{p_{\alpha\beta}}|\xi_\alpha|\\
&&+\mathfrak{g}_1(|\xi_\circ|) \sum_{|\alpha|\le m}\sum_{m-n/p\le |\beta|\le m}
\Bigg(1+\sum_{m-n/p\le |\gamma|\le
m}|\xi_\gamma|^{p_\gamma }\Bigg)^{p_{\alpha\beta}}|\xi_\beta|\cdot|\xi_\alpha|\\
&=&T_1+ T_2+ T_3.
\end{eqnarray*}
We can estimate these three terms as follows.
\begin{eqnarray*}
T_1&=&\sum_{|\alpha|<m-n/p}|f_{\alpha}(x, 0)|\cdot|\xi_\alpha|+\sum_{m-n/p\le|\alpha|\le m}|f_{\alpha}(x, 0)|\cdot|\xi_\alpha|\\
&=&|\xi_\circ|\sum_{|\alpha|<m-n/p}|f_{\alpha}(x, 0)|+\sum_{m-n/p\le|\alpha|\le m}|f_{\alpha}(x, 0)|\cdot|\xi_\alpha|\\
&\le&|\xi_\circ|\sum_{|\alpha|<m-n/p}|f_{\alpha}(x, 0)|+\sum_{m-n/p\le|\alpha|\le m}|f_{\alpha}(x, 0)|^{q_\alpha}
+\sum_{m-n/p\le|\alpha|\le m}|\xi_\alpha|^{p_\alpha};
\end{eqnarray*}
\begin{eqnarray*}
T_2&=&\mathfrak{g}_1(|\xi_0|)|\xi_\circ| \sum_{|\alpha|<m-n/p}\sum_{|\beta|<m-n/p}\Bigg(1+
\sum_{m-n/p\le |\gamma|\le
m}|\xi_\gamma|^{p_\gamma }\Bigg)^{p_{\alpha\beta}}|\xi_\alpha|\\
&&+\mathfrak{g}_1(|\xi_\circ|)|\xi_\circ| \sum_{m-n/p\le |\alpha|\le m}\sum_{|\beta|<m-n/p}
\Bigg(1+\sum_{m-n/p\le |\gamma|\le
m}|\xi_\gamma|^{p_\gamma }\Bigg)^{p_{\alpha\beta}}|\xi_\alpha|\\
&\le&\mathfrak{g}_1(|\xi_\circ|)|\xi_\circ|^2 M(m-n/p+1)\Bigg(1+
\sum_{m-n/p\le |\gamma|\le m}|\xi_\gamma|^{p_\gamma }\Bigg)\\
&&+\mathfrak{g}_1(|\xi_\circ|)|\xi_\circ|M(m-n/p+1)\sum_{m-n/p\le |\alpha|\le m}
\Bigg(1+\sum_{m-n/p\le |\gamma|\le m}|\xi_\gamma|^{p_\gamma }\Bigg)^{1/q_{\alpha}}|\xi_\alpha|\\
&\le&\mathfrak{g}_1(|\xi_\circ|)|\xi_\circ|^2 M(m-n/p+1)
\Bigg(1+ \sum_{m-n/p\le |\gamma|\le m}|\xi_\gamma|^{p_\gamma }\Bigg)\\
&&+\mathfrak{g}_1(|\xi_\circ|)|\xi_\circ|M(m-n/p+1)\sum_{m-n/p\le |\alpha|\le m}
\Bigg[\bigg(1+ \sum_{m-n/p\le |\gamma|\le m}|\xi_\gamma|^{p_\gamma }\bigg)+ |\xi_\alpha|^{p_{\alpha}}\Bigg]\\
&\le&\mathfrak{g}_1(|\xi_\circ|)|\xi_\circ|^2 M(m)
\Bigg(1+ \sum_{m-n/p\le |\gamma|\le m}|\xi_\gamma|^{p_\gamma }\Bigg)\\
&&+\mathfrak{g}_1(|\xi_\circ|)|\xi_\circ|(M(m)+1)^2
\Bigg(1+ \sum_{m-n/p\le |\gamma|\le m}|\xi_\gamma|^{p_\gamma }\Bigg)\\
&=&\mathfrak{g}_1(|\xi_\circ|)\big[|\xi_\circ|^2 M(m)+|\xi_\circ|(M(m)+1)^2\big]
\Bigg(1+ \sum_{m-n/p\le |\gamma|\le m}|\xi_\gamma|^{p_\gamma }\Bigg);
\end{eqnarray*}
\begin{eqnarray*}
T_3&=&\mathfrak{g}_1(|\xi_\circ|) \sum_{|\alpha|<m-n/p}\sum_{m-n/p\le |\beta|\le m}
\Bigg(1+\sum_{m-n/p\le |\gamma|\le m}|\xi_\gamma|^{p_\gamma }\Bigg)^{p_{\alpha\beta}}|\xi_\beta|\cdot|\xi_\alpha|\\
&+&\mathfrak{g}_1(|\xi_\circ|) \sum_{m-n/p\le|\alpha|\le m}\sum_{m-n/p\le |\beta|\le m}
\Bigg(1+\sum_{m-n/p\le |\gamma|\le m}|\xi_\gamma|^{p_\gamma }\Bigg)^{p_{\alpha\beta}}|\xi_\beta|\cdot|\xi_\alpha|
\end{eqnarray*}
\begin{eqnarray*}
&\le&\mathfrak{g}_1(|\xi_\circ|)|\xi_\circ| \sum_{m-n/p\le |\beta|\le m}
\Bigg(1+\sum_{m-n/p\le |\gamma|\le m}|\xi_\gamma|^{p_\gamma }\Bigg)^{1/q_{\beta}}|\xi_\beta|\\
&+&\mathfrak{g}_1(|\xi_\circ|) \sum_{m-n/p\le|\alpha|\le m}\sum_{m-n/p\le |\beta|\le m}
 \Bigg(1+ \sum_{m-n/p\le |\gamma|\le m}|\xi_\gamma|^{p_\gamma }\Bigg)^{p_{\alpha\beta}}|\xi_\beta|\cdot|\xi_\alpha|\\
&\le&\mathfrak{g}_1(|\xi_\circ|)|\xi_\circ|\sum_{m-n/p\le |\beta|\le m}
\Bigg[\Bigg(1+ \sum_{m-n/p\le |\gamma|\le m}|\xi_\gamma|^{p_\gamma }\Bigg)+
 |\xi_\beta|^{p_{\beta}}\Bigg]\\
&+&\mathfrak{g}_1(|\xi_\circ|) \sum_{m-n/p\le|\alpha|\le m}\sum_{m-n/p\le |\beta|\le m}
\Bigg[|\xi_\beta|^{p_\beta}+|\xi_\alpha|^{p_\alpha}\\
&&\hspace{30mm}+ \Bigg(1+\sum_{m-n/p\le |\gamma|\le
m}|\xi_\gamma|^{p_\gamma }\Bigg)^{p_{\alpha\beta}(1-p_\alpha^{-1}-p_\beta^{-1})}\Biggr]\\
&\le&\mathfrak{g}_1(|\xi_\circ|)|\xi_\circ|(M(m)+1)\Bigg(1+
\sum_{m-n/p\le |\gamma|\le m}|\xi_\gamma|^{p_\gamma }\Bigg)\\
&&+\mathfrak{g}_1(|\xi_\circ|)(M(m)+1)^2\Bigg(1+
\sum_{m-n/p\le |\gamma|\le m}|\xi_\gamma|^{p_\gamma }\Bigg)
\end{eqnarray*}
because $0<p_{\alpha\beta}<1-\frac{1}{p_\alpha}-\frac{1}{p_\beta}$ for
$m-n/p\le|\alpha|\le m$ and $m-n/p\le |\beta|\le m$.

In summary we get
\begin{eqnarray*}
|f(x,\xi)|&\le& |f(x,0)|+
|\xi_\circ|\sum_{|\alpha|<m-n/p}|f_{\alpha}(x, 0)|+\sum_{m-n/p\le|\alpha|\le m}|f_{\alpha}(x, 0)|^{q_\alpha}\\
&&+\sum_{m-n/p\le|\alpha|\le m}|\xi_\alpha|^{p_\alpha}\\
&+&\mathfrak{g}_1(|\xi_\circ|)[|\xi_\circ|^2 M(m)+|\xi_\circ|(M(m)+1)^2]\Bigg(1+
\sum_{m-n/p\le |\gamma|\le m}|\xi_\gamma|^{p_\gamma }\Bigg)\\
&+&\mathfrak{g}_1(|\xi_\circ|)|\xi_\circ|(M(m)+1)\Bigg(1+
\sum_{m-n/p\le |\gamma|\le m}|\xi_\gamma|^{p_\gamma }\Bigg)\\
&&+\mathfrak{g}_1(|\xi_\circ|)(M(m)+1)^2\Bigg(1+
\sum_{m-n/p\le |\gamma|\le m}|\xi_\gamma|^{p_\gamma }\Bigg)\\
&\le& |f(x,0)|+
|\xi_\circ|\sum_{|\alpha|<m-n/p}|f_{\alpha}(x, 0)|+\sum_{m-n/p\le|\alpha|\le m}|f_{\alpha}(x, 0)|^{q_\alpha}\\
&&+ \mathfrak{g}_3(|\xi_\circ|)\Bigg(1+\sum_{m-n/p\le|\alpha|\le m}|\xi_\alpha|^{p_\alpha}\Bigg),
\end{eqnarray*}
where
\begin{eqnarray*}
\mathfrak{g}_3(|\xi_\circ|)&=&1+ \mathfrak{g}_1(|\xi_\circ|)[|\xi_\circ|^2 M(m)+|\xi_\circ|(M(m)+1)^2]\\
&&+\mathfrak{g}_1(|\xi_\circ|)|\xi_\circ|(M(m)+1)+\mathfrak{g}_1(|\xi_\circ|)(M(m)+1)^2.
\end{eqnarray*}
\hfill$\Box$\vspace{2mm}




\medskip
\quad
\medskip

\end{document}